\documentclass[12pt,a4paper]{article}
\usepackage{amsmath,amssymb,amsfonts,bbm,calc,cases,verbatim,multirow,color,xcolor}
\usepackage[mathscr]{euscript}
\usepackage{epic,eepic,psfrag,epsfig}
\usepackage[sf,bf,SF,footnotesize]{subfigure}
\usepackage{graphicx}
\usepackage{amsthm,breakcites}
\usepackage{amscd}
\usepackage{epsfig}
\usepackage{natbib}
\usepackage{setspace}
\usepackage{enumerate,enumitem}
\usepackage{array}
\usepackage{epstopdf}
\usepackage[small]{caption}
\RequirePackage[colorlinks,citecolor={blue!65!black},urlcolor={blue!70!black},linkcolor={red!70!black},breaklinks,hypertexnames=false]{hyperref}

\newcommand{\abs}[1]{\lvert #1 \rvert}

\newcommand{\norm}[1]{\| #1 \|}
\newcommand{\Norm}[1]{\left\lVert #1 \right\rVert}
\newcommand{\floor}[1]{\lfloor #1 \rfloor}

\newcommand{\ceil}[1]{\lceil #1 \rceil}

\newcommand{\rbr}[1]{\left(#1\right)}

\newcommand{\cbr}[1]{\left\{#1\right\}}

\newcommand{\inv}[1]{#1^{-1}}

\newcommand{\N}{\mathbb{N}}
\newcommand{\Z}{\mathbb{Z}}

\newcommand{\R}{\mathbb{R}}

\newcommand{\E}{\mathbb{E}}

\renewcommand{\Pr}{\mathbb{P}}

\newcommand{\ipr}[2]{\langle #1,#2 \rangle}

\newcommand{\restr}[2]{\left.#1\right\vert_{#2}}
\newcommand{\tm}[1]{#1^\top}
\newcommand{\cm}[1]{#1^{c}}
\newcommand{\xst}[1]{#1^*}

\newcommand{\Ind}{\mathbbm{1}}

\DeclareMathOperator\Var{Var}

\DeclareMathOperator\Img{Im}

\DeclareMathOperator\Span{span}

\DeclareMathOperator\diag{diag}
\let\det\relax 
\DeclareMathOperator{\det}{det}
\DeclareMathOperator*\argmax{argmax}
\DeclareMathOperator*\argmin{argmin}
\DeclareMathOperator\tr{tr}

\DeclareMathOperator\Int{Int}
\DeclareMathOperator\relint{relint}
\DeclareMathOperator\Cl{Cl}
\DeclareMathOperator\conv{conv}
\DeclareMathOperator\aff{aff}
\DeclareMathOperator\supp{supp}
\DeclareMathOperator\dom{dom}
\DeclareMathOperator\Exp{Exp}
\DeclareMathOperator\Ext{Ext}
\DeclareMathOperator\rec{rec}

\DeclareMathOperator\Po{Po}
\DeclareMathOperator\KL{KL}

\newcommand{\iid}{\overset{\mathrm{iid}}{\sim}}

\newcommand{\dhell}{d_{\mathrm{H}}}

\newcommand{\dex}{d_{\!X}}

\theoremstyle{plain} 
\newtheorem{theorem}{Theorem} 
\newtheorem{proposition}[theorem]{Proposition}
\newtheorem*{proposition*}{Proposition}
\newtheorem{lemma}[theorem]{Lemma}
\newtheorem*{lemma*}{Lemma}
\newtheorem{corollary}[theorem]{Corollary}
\newtheorem{claim}{Claim}
\newtheorem*{claim*}{Claim}
\newtheorem{definition}{Definition}
\newtheorem*{definition*}{Definition}

\theoremstyle{definition} 

\newtheorem*{remark*}{Remark}

\newtheorem*{aside*}{Aside}
\newtheorem{example}{Example}
\newtheorem*{example*}{Example}

\newcommand{\vertiii}[1]{{\left\vert\kern-0.25ex\left\vert\kern-0.25ex\left\vert #1 
\right\vert\kern-0.25ex\right\vert\kern-0.25ex\right\vert}}

\begin{document}
\title{Adaptation in multivariate log-concave density estimation}
\author{Oliver Y. Feng$^\ast$, Adityanand Guntuboyina$^\dag$, Arlene K. H. Kim$^\ddag$ \\
and Richard J. Samworth$^{\ast,\mathsection}$ \\
$^{\ast}$University of Cambridge, $^\dag$University of
California Berkeley \\ and $^{\ddag}$Korea University}
\footnotetext[2]{Research supported by NSF CAREER Grant DMS-16-54589.}
\footnotetext[3]{Research supported by National Research Foundation of Korea (NRF) grant
2017R1C1B501734.}
\footnotetext[4]{Research supported by EPSRC Fellowship EP/P031447/1 and grant RG81761 from the Leverhulme Trust.} 
\date{\today}

\maketitle

\begin{abstract}
We study the adaptation properties of the multivariate log-concave maximum likelihood estimator over three subclasses of log-concave densities. The first consists of densities with polyhedral support whose logarithms are piecewise affine.  The complexity of such densities~$f$ can be measured in terms of the sum $\Gamma(f)$ of the numbers of facets of the subdomains in the polyhedral subdivision of the support induced by $f$.  Given $n$ independent observations from a $d$-dimensional log-concave density with $d \in \{2,3\}$, we prove a sharp oracle inequality, which in particular implies that the Kullback--Leibler risk of the log-concave maximum likelihood estimator for such densities is bounded above by $\Gamma(f)/n$, up to a polylogarithmic factor.  Thus, the rate can be essentially parametric, even in this multivariate setting. For the second type of adaptation, we consider densities that are bounded away from zero on a polytopal support; we show that up to polylogarithmic factors, the log-concave maximum likelihood estimator attains the rate $n^{-4/7}$ when $d=3$, which is faster than the worst-case rate of $n^{-1/2}$.  Finally, our third type of subclass consists of densities whose contours are well-separated; these new classes are constructed to be affine invariant and turn out to contain a wide variety of densities, including those that satisfy H\"older regularity conditions.  Here, we prove another sharp oracle inequality, which reveals in particular that the log-concave maximum likelihood estimator attains a risk bound of order $n^{-\min\bigl(\frac{\beta+3}{\beta+7},\frac{4}{7}\bigr)}$ when $d=3$ over the class of $\beta$-H\"older log-concave densities with $\beta\in (1,3]$, again up to a polylogarithmic factor.       
\end{abstract}

\addtocontents{toc}{\protect\setcounter{tocdepth}{0}}
\section{Introduction}

The field of nonparametric inference under shape constraints has witnessed remarkable progress on several fronts over the last decade or so.  For instance, the area has been enriched by methodological innovations in new research problems, including convex set estimation \citep{GKM2006,Guntuboyina2012,Brunel2013}, shape-constrained dimension reduction \citep{ChenSamworth2014,XCL2016,GroeneboomHendrickx2018} and ranking and pairwise comparisons \citep{Shah17}.  Algorithmic advances together with increased computing power now mean that certain estimators have become computationally feasible for much larger sample sizes \citep{KoenkerMizera2014, MCIS2018}.  On the theoretical side, new tools developed in recent years have allowed us to make progress in understanding how shape-constrained procedures behave \citep{DSS11,GuntuboyinaSen2013,CaiLow2015}.  Moreover, minimax rates of convergence are now known\footnote{In the interests of transparency, we note that in some of our examples, there remain gaps between the known minimax lower and upper bounds that are polylogarithmic in the sample size.} for a variety of core problems in the area, including decreasing density estimation on the non-negative half-line \citep{Birge1987}, isotonic regression \citep{Zhang2002,CGS2018,DengZhang2018,HWCS2018} and convex regression \citep{HanWellner2016}. \citet{GroeneboomJongbloed2014} provide a book-length introduction to the field; many recent developments are also surveyed in a 2018 special issue of \emph{Statistical Science} devoted to the topic.

One of the most intriguing aspects of many shape-constrained estimators is their ability to \emph{adapt} to unknown features of the underlying data generating mechanism.  To illustrate what we mean by this, consider a general setting in which the goal is to estimate a function or parameter that belongs to a class $\mathcal{D}$.  
Given a subclass $\mathcal{D}' \subseteq \mathcal{D}$, we say that our estimator adapts to $\mathcal{D}'$ with respect to a given loss function if its worst-case rate of convergence over $\mathcal{D}'$ is an improvement on its corresponding worst-case rate over $\mathcal{D}$; in the best case, it may even attain the minimax rates of convergence over both $\mathcal{D}'$ and $\mathcal{D}$, at least up to polylogarithmic factors in the sample size.  As a concrete example of this phenomenon, consider independent observations $Y_1,\ldots,Y_n$ with $Y_i \sim N(\theta_{0i},1)$, where $\theta_0 := (\theta_{01},\ldots,\theta_{0n})$ belongs to the monotone cone $\mathcal{D} := \{\theta = (\theta_1,\ldots,\theta_n) \in \mathbb{R}^n:\theta_1 \leq \ldots \leq \theta_n\}$.  \citet{Zhang2002} established that the least squares estimator $\hat{\theta}_n$ over $\mathcal{D}$ satisfies the worst-case $\ell_2$-risk bound that 
\[
\mathbb{E}\bigl\{\|\hat{\theta}_n - \theta_0\|^2\bigr\} \leq C\,\biggl\{\biggl(\frac{\theta_{0n} - \theta_{01}}{n}\biggr)^{2/3} + \frac{\log n}{n}\biggr\}
\]
for some universal constant $C > 0$; thus, in particular, it attains the minimax rate of $O(n^{-2/3})$ for signals $\theta_0 \in \mathcal{D}$ of bounded uniform norm. On the other hand, the fact that the least squares estimator is piecewise constant motivates the thought that $\hat{\theta}_n$ might adapt to piecewise constant signals.  More precisely, letting $\mathcal{D}' \equiv \mathcal{D}_k'$ denote the subset of $\mathcal{D}$ consisting of signals with at most $k$ constant pieces, a consequence of~\citet[Theorem~3.2]{Bellec2018} is that 
\[
\sup_{\theta_0 \in \mathcal{D}_k'} \mathbb{E}\bigl\{\|\hat{\theta}_n - \theta_0\|^2\bigr\} \leq \frac{k}{n}\log\Bigl(\frac{en}{k}\Bigr).
\]
Note that, up to the logarithmic factor, this rate of convergence (which is parametric when $k$ is a constant) is the same as could be attained by an `oracle' estimator that had access to the locations of the jumps in the signal.  The proof of this beautiful result relies on the characterisation of the least squares estimator as an $\ell_2$-projection onto the closed, convex cone $\mathcal{D}$, as well as the notion of such a cone's statistical dimension, which can be computed exactly in the case of the monotone cone \citep{Amelunxenetal2014,SGP19}.


As a result of intensive work over the past decade, the adaptive behaviour of shape-constrained estimators is now fairly well understood in a variety of univariate problems \citep{BRW2009,DumbgenRufibach2009,Jankowski2014,CGS2015,KGS18,ChatterjeeLafferty2018}. Moreover, in the special cases of isotonic and convex regression, very recent work has shown that shape-constrained least squares estimators exhibit an even richer range of adaptation properties in multivariate settings \citep{HanWellner2016,CGS2018,DengZhang2018,HWCS2018,Han2019}. For instance,~\citet{CGS2018} showed that the least squares estimator in bivariate isotonic regression continues to enjoy parametric adaptation up to polylogarithmic factors when the signal is constant on a small number of rectangular pieces.  On the other hand,~\citet{HWCS2018} proved that, in general dimensions $d \geq 3$, the least squares estimator in fixed, lattice design isotonic regression\footnote{Here and below, the $\tilde{O}$ notation is used to denote rates that hold up to polylogarithmic factors in $n$.} adapts at rate $\tilde{O}(n^{-2/d})$ for constant signals, and that it is not possible to obtain a faster rate for this estimator.  This is still an improvement on the minimax rate of $\tilde{O}(n^{-1/d})$ over all isotonic signals (in the lexicographic ordering) with bounded uniform norm, but is strictly slower than the parametric rate.  We remark that, in addition to the ideas employed by~\citet{Bellec2018}, these higher-dimensional results rely on an alternative characterisation of the least squares estimator due to~\citet{Chatterjee2014}, as well as an argument that controls the statistical dimension of the $d$-dimensional monotone cone by induction on $d$; see~\citet[Theorem~3.9]{Han2019} for an alternative approach to the latter. Given the surprising nature of these results, it is of great interest to understand the extent to which adaptation is possible in other shape-constrained estimation problems.

This paper concerns multivariate adaptation behaviour in log-concave density estimation.  The class of log-concave densities lies at the heart of modern shape-constrained nonparametric inference, due to both the modelling flexibility it affords and its attractive stability properties under operations such as marginalisation, conditioning, convolution and linear transformations \citep{Walther2009,SaumardWellner2014,Samworth2018}. 
However, the class of log-concave densities is not convex, so the maximum likelihood estimator cannot be regarded as a projection onto a convex set, and the results of~\citet{Amelunxenetal2014},~\citet{Chatterjee2014} and~\citet{Bellec2018} cannot be applied.

To set the scene, let $\mathcal{F}_d$ denote the class of upper semi-continuous, log-concave densities on $\mathbb{R}^d$, and suppose that $X_1,\ldots,X_n$ are independent and identically distributed random vectors with density $f_0 \in \mathcal{F}_d$.  Also, we write $d_{\mathrm{H}}(f,g) := \bigl\{\int_{\mathbb{R}^d}\,(f^{1/2} - g^{1/2})^2\bigr\}^{1/2}$ for the Hellinger distance between two densities $f$ and $g$. \citet{KS16} proved the following minimax lower bound\footnote{In fact, more recently, \citet{DaganKur2019} proved that $c_d$ may be chosen independently of the dimension $d$.}: for each $d \in \mathbb{N}$, there exists $c_d > 0$ such that 
\begin{equation}
\label{Eq:MinimaxLB}
\inf_{\tilde{f}_n} \sup_{f_0 \in \mathcal{F}_d} \mathbb{E}\{d_\mathrm{H}^2(\tilde{f}_n,f_0)\} \geq \left\{ \begin{array}{ll} c_1\,n^{-4/5} & \mbox{if $d=1$} \\
c_d\,n^{-2/(d+1)} & \mbox{if $d \geq 2$,} \end{array} \right.
\end{equation}
where the infimum is taken over all estimators $\tilde{f}_n$ of $f_0$ based on $X_1,\ldots,X_n$. Thus, when $d\geq 3$, there is a more severe curse of dimensionality than for the problem of estimating a density with two bounded derivatives and exponentially decaying tails, for which the corresponding minimax rate is $n^{-4/(d+4)}$ in all dimensions \citep{GL14}. See Section~\ref{Subsec:HolderRev} in the supplementary material \citep{FGKS2018} for further details and discussion. The reason why this comparison is interesting is because any concave function is twice differentiable Lebesgue almost everywhere on its effective domain, while a twice differentiable function is concave if and only if its Hessian matrix is non-positive definite at every point.  This observation had led to the prediction that the rates in these problems ought to coincide \citep[e.g.][page~3778]{SereginWellner2010}.

The result~\eqref{Eq:MinimaxLB} is relatively discouraging as far as high-dimensional log-concave density estimation is concerned, and has motivated the definition of alternative procedures that seek improved rates when $d$ is large under additional structure, such as independent component analysis \citep{SamworthYuan2012} or symmetry \citep{XS19}.  Nevertheless, in lower-dimensional settings, the performance of the log-concave maximum likelihood estimator $\hat{f}_n := \argmax_{f \in \mathcal{F}_d} \sum_{i=1}^n \log f(X_i)$ has been studied with respect to the divergence $\dex^2(\hat{f}_n,f_0) := n^{-1}\sum_{i=1}^n \log \frac{\hat{f}_n(X_i)}{f_0(X_i)}$~\citep[cf.][page~2281]{KGS18}. This loss function is closely related to the Kullback--Leibler divergence $\KL(f,g) := \int_{\R^d}f\log (f/g)$ and Hellinger distance. Indeed, we have $\dhell^2(\hat{f}_n,f_0) \leq\KL(\hat{f}_n,f_0) \leq\dex^2(\hat{f}_n,f_0)$, where the the first bound is standard and the second inequality follows by applying~\citet[Remark~2.3]{DSS11} to the function $x\mapsto\log\bigl(f_0(x)/\hat{f}_n(x)\bigr)$. A small modification of the proof of~\citet[Theorem~5]{KS16} yields the following result, which is stated as Theorem~\ref{Thm:WorstCaseRates} in the supplementary material~\citep{FGKS2018} for convenience: 
\begin{equation}
\label{Eq:WorstCaseRate}
\sup_{f_0 \in \mathcal{F}_d} \mathbb{E}\{d_{\!X}^2(\hat{f}_n,f_0)\}=
\left\{ \begin{array}{ll} O(n^{-4/5}) & \mbox{if $d=1$} \\
O(n^{-2/3}\log n) & \mbox{if $d=2$} \\
O(n^{-1/2}\log n) & \mbox{if $d=3$;}\end{array} \right.
\end{equation}
see also~\citet{DossWellner2016} for a related result in the univariate case. Moreover, very recently, \citet{DaganKur2019} proved that\footnote{Here and below, the $O_d(\cdot)$ notation is used as shorthand for an upper bound that holds up to a dimension-dependent quantity.}
\begin{equation}
\label{Eq:DaganKur}
\sup_{f_0 \in \mathcal{F}_d} \mathbb{E}\{d_{\mathrm{H}}^2(\hat{f}_n,f_0)\} = O_d(n^{-2/(d+1)}\log n)
\end{equation}
for $d\geq 4$, so that, at least in squared Hellinger loss, it follows from~\eqref{Eq:MinimaxLB},~\eqref{Eq:WorstCaseRate} and~\eqref{Eq:DaganKur} that $\hat{f}_n$ attains the minimax optimal rate in all dimensions, up to a logarithmic factor.  

Our goal is to explore the potential of the log-concave maximum likelihood estimator to adapt to two different types of subclass of $\mathcal{F}_d$.  The definition of the first of these is motivated by the observation that $\log \hat{f}_n$ is piecewise affine on the convex hull of $X_1,\ldots,X_n$, a polyhedral subset of $\mathbb{R}^d$. It is therefore natural to consider, for $k\in \mathbb{N}$ and $m\in\N\cup\{0\}$, the subclass $\mathcal{F}^k(\mathcal{P}^m) \equiv \mathcal{F}_d^k(\mathcal{P}^m) \subseteq \mathcal{F}_d$ consisting of densities that are both log-$k$-affine on their support (see Section~\ref{Subsec:Notation}), and have the property that this support is a polyhedral set with at most $m$ facets. Note that this class contains densities with unbounded support. By Proposition~\ref{Prop:Logkaff} in Section~\ref{Sec:Logkaffine} below, the complexity of such densities~$f$ can be measured in terms of the sum $\Gamma(f)$ of the numbers of facets of the subdomains in the polyhedral subdivision of the support induced by $f$.  A consequence of our first main result, Theorem~\ref{Thm:MainLogkaffine}, is that for all $f_0 \in \mathcal{F}^k(\mathcal{P}^m)$, we have
\begin{equation}
\label{Eq:OracleRiskBound}
\mathbb{E}\{d_{\!X}^2(\hat{f}_n,f_0)\} = \tilde{O}\biggl(\frac{\Gamma(f_0)}{n}\biggr)
\end{equation}
when $d \in \{2,3\}$; moreover, we also show that $\Gamma(f_0)$ is at most of order $k+m$ when $d=2$, and at most of order $k(k+m)$ when $d=3$.  Thus, when $k$ and $m$ may be regarded as constants, \eqref{Eq:OracleRiskBound} reveals that, up to the polylogarithmic term, the log-concave maximum likelihood estimator adapts at a parametric rate to $\mathcal{F}^k(\mathcal{P}^m)$ when $d \in \{2,3\}$.  Moreover, Theorem~\ref{Thm:MainLogkaffine} offers a complete picture for this type of adaptation by providing a sharp oracle inequality that covers the case where $f_0$ is well approximated (in a Kullback--Leibler sense) by a density in $\mathcal{F}^k(\mathcal{P}^m)$ for some $k,m$.  Unsurprisingly, the proof of this inequality is much more delicate and demanding than the corresponding univariate result given in~\citet{KGS18}, owing to the greatly increased geometric complexity of both the boundaries of convex subsets of $\mathbb{R}^d$ for $d \geq 2$ and the structure of the polyhedral subdivisions induced by the densities in $\mathcal{F}^k(\mathcal{P}^m)$.  In particular, the parameter $m$ plays no role in the univariate problem, since the boundary of a convex subset of the real line has at most two points,  but it turns out to be crucial in this multivariate setting.  Indeed, no form of adaptation would be achievable in the absence of restrictions on the shape of the support of $f_0 \in \mathcal{F}_d$; for instance, when $f_0$ is the uniform density on a closed Euclidean ball in $\mathbb{R}^d$ with $d \geq 2$, consideration of the volume of the convex hull of $X_1,\dotsc,X_n$ yields that $\mathbb{E}\{d_{\mathrm{H}}^2(\hat{f}_n,f_0)\} \geq \tilde{c}_d\,n^{-2/(d+1)}$ for some $\tilde{c}_d > 0$ depending only on $d$ \citep{Wieacker1987}. 

In contrast to the isotonic regression problem described above, Theorem~\ref{Thm:MainLogkaffine} indicates that even when $d=3$, the log-concave maximum likelihood estimator also enjoys essentially parametric adaptation when $f_0$ is close to a density in $\mathcal{F}^k(\mathcal{P}^m)$ for small $k$ and $m$.  Unfortunately, our arguments do not allow us to extend our results to dimensions $d \geq 4$, where the relevant bracketing entropy integral diverges at a polynomial rate.  Recent work by \citet{CDSS2018} derived worst-case rates in squared Hellinger loss for the log-concave maximum likelihood estimator when $d \geq 4$; the crux of their argument involved using Vapnik--Chervonenkis theory to bound
\[
\E\,\biggl(\sup_{K \in \mathcal{K}_d^*}\;\biggl|\frac{1}{n}\sum_{i=1}^n\Ind_{\{X_i \in K\}} - \Pr(X_1 \in K)\,\biggr|\biggr),
\]
where $\mathcal{K}_d^*$ denotes the set of all closed, convex subsets of $\mathbb{R}^d$. \citet{DaganKur2019} obtained an improved bound on this quantity of $O_d(n^{-2/(d+1)})$ using a general chaining argument, and this allowed them to deduce the worst-case guarantees on the performance of the log-concave maximum likelihood estimator stated in~\eqref{Eq:DaganKur}.  Unfortunately, it is unclear whether this approach can provide any adaptation guarantees.

Sections~\ref{Sec:ThetaPolytope} and~\ref{Sec:Smoothness} consider different subclasses of $\mathcal{F}_d$, and are motivated by the hope that if we rule out `bad' log-concave densities such as the uniform densities with smooth boundaries mentioned above, then we may be able to achieve faster rates of convergence, up to the $n^{-4/(d+4)}$ rate conjectured by \citet{SereginWellner2010}.  Since this rate already coincides with the worst-case rate for the log-concave maximum likelihood estimator given in~\eqref{Eq:WorstCaseRate} when $d = 1,2$ (up to a logarithmic factor), and since the same entropy integral divergence issues mentioned above apply when $d \geq 4$, we focus on the case $d=3$ in these sections.  In Section~\ref{Sec:ThetaPolytope}, we restrict attention to densities with polytopal support (that need not satisfy the log-$k$-affine condition of Section~\ref{Sec:Logkaffine}).  Theorem~\ref{Thm:ThetaRisk} therein provides a sharp oracle inequality, which reveals that in such cases, the log-concave maximum likelihood estimator attains the rate $\tilde{O}(n^{-4/7})$ with respect to $d_{\!X}^2$ divergence, at least when the density is bounded away from zero on its support.

In Section~\ref{Sec:Smoothness}, we introduce an alternative way to exclude the bad uniform densities mentioned above, namely by considering subclasses of $\mathcal{F}_d$ consisting of densities~$f$ whose contours are well-separated in regions where $f$ is small.  A major advantage of working with contour separation, as opposed to imposing a conventional smoothness condition such as H\"older regularity, is that we are able to exhibit adaptation over much wider classes of densities, as we illustrate through several examples in Section~\ref{Sec:Smoothness}. A consequence of our main theorem in this section (Theorem~\ref{Thm:Smoothness}) is that the log-concave maximum likelihood estimator attains the rate $\tilde{O}(n^{-4/7})$ with respect to $d_{\!X}^2$ divergence over the class of Gaussian densities; again, one can think of this result as partially restoring the original conjecture of \citet{SereginWellner2010}, in that their rate is achieved with additional restrictions on the class of log-concave densities.  A key feature of our definition of contour separation is that it is affine invariant; since the log-concave maximum likelihood estimator is affine equivariant and our loss functions $\dhell^2$, $\KL$ and $\dex^2$ are affine invariant, this allows us to obtain rates that are uniform over classes without any scale restrictions.

We mention that alternative estimators have also been studied for the class of log-concave densities. One such is the smoothed log-concave maximum likelihood estimator \citep{DumbgenRufibach2009,ChenSamworth2013}, which matches the first two moments of the empirical distribution of the data, but for which results on rates of convergence are less developed. Another proposal is the $\rho$-estimation framework of \citet{BaraudBirge2016}, for which similar adaptation properties as for the log-concave maximum likelihood estimator are known in the univariate case.

Proofs of most of our main results are given in the Appendix.  The remaining proofs, as well as numerous auxiliary results, are presented in the supplementary material \citep{FGKS2018}; these results appear with an `S' before the relevant label number.


\vspace{-0.2cm}
\subsection{Notation and background}
\label{Subsec:Notation}

First, we set up some notation and definitions that will be used throughout the main text as well as in the proofs later on. For a fixed $d \in \mathbb{N}$, we write $\{e_1,\dotsc,e_d\}$ for the standard basis of $\R^d$ and denote the $\ell_2$ norm of $x=(x_1,\dotsc,x_d)=\sum_{j=1}^d x_je_j\in\R^d$ by $\norm{x}\equiv\norm{x}_2=\bigl(\sum_{j=1}^d x_j^2\bigr)^{1/2}$. For $x,y\in\R^d$, let $[x,y]:=\{tx+(1-t)y:t\in [0,1]\}$ denote the closed line segment between them, and define $(x,y)$, $[x,y)$, $(x,y]$ analogously. For $x\in\R^d$ and $r>0$, let $\bar{B}(x,r):=\{w\in\R^d:\norm{w-x}\leq r\}$. For $A\subseteq\R^d$, we write $\dim(A)$ for the \emph{affine dimension} of $A$, i.e.\ the dimension of the affine hull of $A$, and for Lebesgue-measurable $A\subseteq\R^d$, we write $\mu_d(A)$ for the $d$-dimensional Lebesgue measure of $A$. If $0<\dim(A)=k<d$, we can view $A$ as a subset of its affine hull and define $\mu_k(A)$ analogously, whilst also setting $\mu_l(A)=0$ for each integer $l>k$. In addition, we denote the set of positive definite $d \times d$ matrices by $\mathbb{S}^{d\times d}$ and the $d\times d$ identity matrix by $I\equiv I_d$. 

Next, let $\Phi\equiv\Phi_d$ be the set of all upper semi-continuous, concave functions $\phi\colon\R^d\to [-\infty,\infty)$ and let $\mathcal{G}\equiv\mathcal{G}_d:=\{e^\phi:\phi\in\Phi\}$. For $\phi\in\Phi$, we write $\dom\phi:=\{x\in\R^d:\phi(x)>-\infty\}$ for the \textit{effective domain} of $\phi$, and for a general $f\colon\R^d\to\R$, we write $\supp f:=\{x\in\R^d:f(x)\neq 0\}$ for the \textit{support} of $f$. For $k\in\N$, we say that $f\in\mathcal{G}_d$ is \textit{log-$k$-affine} if there exist closed sets $E_1,\dotsc,E_k$ such that $\supp f=\bigcup_{j=1}^{\,k}E_j$ and $\log f$ is affine on each $E_j$. Moreover, let $\mathcal{F}\equiv\mathcal{F}_d$ be the family of all \emph{densities} $f\in\mathcal{G}_d$, and let $\mu_f:=\int_{\R^d}\,xf(x)\,dx$ and $\Sigma_f:=\int_{\R^d}\,(x-\mu_f)\tm{(x-\mu_f)}dx$ for each $f\in\mathcal{F}_d$. In addition, we write $\mathcal{F}^{0,I}\equiv\mathcal{F}_d^{0,I}:=\{f\in\mathcal{F}_d:\mu_f=0,\,\Sigma_f=I\}$ for the class of \emph{isotropic} log-concave densities.

Henceforth, for real-valued functions $a$ and $b$, we write $a\lesssim b$ if there exists a universal constant $C>0$ such that $a\leq Cb$, and we write $a\asymp b$ if $a\lesssim b$ and $b\lesssim a$. More generally, for a finite number of parameters $\alpha_1,\dotsc,\alpha_r$, we write $a\lesssim_{\alpha_1,\dotsc,\alpha_r}\!b$ if there exists $C\equiv C_{\alpha_1,\dotsc,\alpha_r}>0$, depending only on $\alpha_1,\dotsc,\alpha_r$, such that $a\leq Cb$. Also, for $x\in\R$, we write $x^+:=x\vee 0$ and $x^-:=(-x)^+$, and for $x>0$, we define $\log_+x:=1\vee\log x$.

To facilitate the exposition in Section~\ref{Sec:Smoothness}, we now introduce some additional terminology. We say that the densities $f$ and $g$ on $\R^d$ are \emph{affinely equivalent} if there exist an $\R^d$-valued random variable $X$ and an invertible affine transformation $T\colon\R^d\to\R^d$ such that $X$ has density $f$ and $T(X)$ has density $g$; in other words, there exist $b\in\R^d$ and an invertible $A\in\R^{d\times d}$ such that $g(x)=\inv{\abs{\det A}}f(\inv{A}(x-b))$ for all $x\in\R^d$. Thus, each $f\in\mathcal{F}_d$ is affinely equivalent to a unique $f_0\in\mathcal{F}_d^{0,I}$. A class $\mathcal{D}$ of densities is said to be \emph{affine invariant} if it is closed under affine equivalence; in other words, if $f$ belongs to $\mathcal{D}$, then so does every density $g$ that is affinely equivalent to $f$.

The rest of this subsection is devoted to a review of some convex analysis background used in Section~\ref{Sec:Logkaffine}. A \textit{closed half-space} is a set of the form $\{x\in\R^d:\tm{\alpha}x\leq u\}$, where $\alpha\in\R^d\setminus\{0\}$ and $u\in\R$, and the interiors and boundaries of closed half-spaces are known as \textit{open half-spaces} and \textit{affine hyperplanes} respectively. For a non-empty and convex $E\subseteq\R^d$, we say that an affine hyperplane $H$ \emph{supports} $E$ if $H\cap E\neq\emptyset$ and $H$ is the boundary of a closed half-space that contains $E$. A \textit{face} $F\subseteq E$ is a convex set with the property that if $u,v\in E$ and $tu+(1-t)v\in F$ for some $t\in (0,1)$, then $u,v\in F$. We say that $x\in E$ is an \textit{extreme point} if $\{x\}$ is a face of $E$. Also, we say that $F\subseteq E$ is an \textit{exposed face} of $E$ if $F=E\cap H$ for some affine hyperplane $H$ that supports $E$. Exposed faces of affine dimensions 0, 1 and $\dim(E)-1$ are also known as \textit{exposed points} (or \textit{vertices}), \textit{edges} and \textit{facets} respectively. We write $\mathscr{F}(E)$ for the set of all facets of $E$.

A \textit{polyhedral set} is a subset of $\R^d$ that can be expressed as the intersection of finitely many closed half-spaces, and a \textit{polytope} is a bounded polyhedral set, or equivalently the convex hull of a finite subset of $\R^d$; see Theorems~2.4.3 and~2.4.6 in~\citet{Sch14}. As a special case, we also view $\R^d$ as a polyhedral set with 0 facets. Let $\mathcal{P}\equiv\mathcal{P}_d$ denote the collection of all polyhedral sets in $\R^d$ with non-empty interior, and for $m\in\N_0:=\N\cup\{0\}$, let $\mathcal{P}^m\equiv\mathcal{P}_d^m$ denote the collection of all $P\in\mathcal{P}$ with at most $m$ facets. For $1\leq k\leq d$, a $k$-\textit{parallelotope} is the image of $[0,1]^k$ under an injective affine transformation from $\R^k$ to $\R^d$, i.e.\ a polytope of the form $\{v_0+\sum_{\ell=1}^k\lambda_\ell v_\ell:0\leq\lambda_\ell\leq 1\text{ for all }\ell\}$, where $v_0,v_1,\dotsc,v_k\in\R^d$ and $v_1,\dotsc,v_k$ are linearly independent. Recall also that a \textit{$k$-simplex} is the convex hull of $k+1$ affinely independent points in $\R^d$. Finally, for $P\in\mathcal{P}_d$, a \textit{(polyhedral) subdivision} of $P$ is a finite collection of sets $E_1,\dotsc,E_\ell\in\mathcal{P}_d$ such that $P=\bigcup_{j=1}^{\,\ell}E_j$ and $E_i\cap E_j$ is a common face of $E_i$ and $E_j$ for all $i,j \in \{1,\ldots,\ell\}$. A \textit{triangulation} of a polytope $P\in\mathcal{P}_d$ is a subdivision of $P$ consisting solely of $d$-simplices.

\section{Adaptation to log-\texorpdfstring{$k$}{k}-affine densities with polyhedral support}
\label{Sec:Logkaffine}

In order to present the main result of this section, we first need to understand the structure of log-$k$-affine functions $f \in \mathcal{G}_d$ with polyhedral support. Due to the global nature of the constraints on $f$, namely that $\log f$ is concave on $\supp f\in\mathcal{P}$ and affine on each of $k$ closed subdomains, the function $f$ necessarily has a simple and rigid structure. More precisely, Proposition~\ref{Prop:Logkaff} below shows that there is a minimal representation of $f$ in which the subdomains are polyhedral sets that form a subdivision of $\supp f$, and the restrictions of $\log f$ to these sets are distinct affine functions. The proof of this result is deferred to Section~\ref{Subsec:LogkaffineDensities}.
\begin{proposition}
\label{Prop:Logkaff}
Suppose that $f\in\mathcal{G}_d$ is log-$k$-affine for some $k\in\N$ and that $\supp f\in\mathcal{P}$. Then there exist $\kappa(f)\leq k$, $\alpha_1,\dotsc,\alpha_{\kappa(f)}\in\R^d$, $\beta_1,\dotsc,\beta_{\kappa(f)}\in\R$ and a polyhedral subdivision $E_1,\dotsc,E_{\kappa(f)}$ of $\supp f$ such that $f(x)=\exp(\tm{\alpha_j}x+\beta_j)$ for all $x\in E_j$, and $\alpha_i\neq\alpha_j$ whenever $i\neq j$. Moreover, the triples $(\alpha_j,\beta_j,E_j)_{j=1}^{\kappa(f)}$ are unique up to reordering. In addition, if $\supp f\in\mathcal{P}^m$, then $E_j\in\mathcal{P}^{k+m-1}$ for all $j$.
\end{proposition}
In particular, for each such $f$, the sum of the numbers of facets of the polyhedral subdomains $E_1,\ldots,E_{\kappa(f)}$, which we denote by
\begin{equation}
\label{Eq:Gammaf}
\Gamma(f):=\sum_{j=1}^{\kappa(f)}\,|\mathscr{F}(E_j)|,
\end{equation}
is well-defined and can be viewed as a parameter that measures the complexity of $f$. Now for $k\in\N$ and $P\in\mathcal{P}$, let $\mathcal{F}^k(P)$ denote the collection of all $f\in\mathcal{F}_d$ for which $\kappa(f) \leq k$ and $\supp f=P$, so that $\mathcal{F}^k(\mathcal{P}^m)=\bigcup_{P\in\mathcal{P}^m}\mathcal{F}^k(P)$ for $m\in\N_0$. It is shown in Proposition~\ref{Prop:FkPmempty} that $\mathcal{F}^k(\mathcal{P}^m)$ is non-empty if and only if $k+m\geq d+1$. We remark here that it is more appropriate to quantify the complexity of a polyhedral support in terms of $m$, which refers to the number of facets of the support, rather than in terms of the number of vertices. Indeed, the former quantity may be much greater than the latter when the support is unbounded; for example, a polyhedral convex cone has just a single vertex but may have arbitrarily many facets. That said, if the support is a polytope with $v$ vertices and $m$ facets, it can be shown that $v=m$ when $d=2$, and that $v \leq 2m-4$ and $m \leq 2v-4$ when $d=3$; see the proof of Lemma~\ref{Lem:euler} and the subsequent remark. 

We are now in a position to state our sharp oracle inequality for the risk of the log-concave maximum likelihood estimator when the true $f_0 \in \mathcal{F}_d$ is close to some element of $\mathcal{F}^k(\mathcal{P}^m)$.
\begin{theorem}
\label{Thm:MainLogkaffine}
Fix $d\in\{2,3\}$. Let $X_1,\ldots,X_n \iid f_0 \in \mathcal{F}_d$ with $n \geq d+1$, and let $\hat{f}_n$ denote the corresponding log-concave maximum likelihood estimator.  Then there exists a universal constant $C > 0$ such that 
\begin{equation}
\label{Eq:OracleIneq}
\E\{\dex^2(\hat{f}_n,f_0)\}\leq \inf_{\substack{k\in\N,\,m\in\N_0:\\ k+m\geq d+1}} \;\inf_{f\in\mathcal{F}^k(\mathcal{P}^m)}\,\Biggl\{\frac{C\,\Gamma(f)}{n}\log^{\gamma_d} n+\KL(f_0,f)\Biggr\},
\end{equation}
where $\gamma_2:=9/2$ and $\gamma_3:=8$.  Moreover, for $d\in\{2,3\}$, we have $\Gamma(f)\lesssim k^{d-2}(k+m)$ for all $f \in \mathcal{F}^k(\mathcal{P}^m)$.
\end{theorem}
The `sharpness' in this oracle inequality refers to the fact that the approximation term $\KL(f_0,f)$ has leading constant~1.  
A consequence of Theorem~\ref{Thm:MainLogkaffine} is that if $d=2$ and $f_0 \in \mathcal{F}^k(\mathcal{P}^m)$ with $k+m$ small by comparison with $n^{1/3}\log^{-7/2}n$, then the log-concave maximum likelihood estimator attains an adaptive rate that is faster than the rate of decay of the worst-case risk bounds~\eqref{Eq:WorstCaseRate} of~\citet{KS16}.  When $d=3$, the same conclusion holds when $k(k+m)$ is small by comparison with $n^{1/2}\log^{-7} n$.

Theorem~\ref{Thm:MainLogkaffine} is proved in Section~\ref{Sec:LogkaffineProofs} by first considering the case $k=1$, where it turns out that we can prove a slightly stronger version of our result.  We therefore state it separately for convenience:
\begin{theorem}
\label{Thm:k1}
Fix $d \in \{2,3\}$. Let $X_1,\ldots,X_n \iid f_0 \in \mathcal{F}_d$ with $n \geq d+1$, and let $\hat{f}_n$ denote the corresponding log-concave maximum likelihood estimator.  Then there exists a universal constant $\bar{C}>0$ such that 
\begin{equation}
\label{Eq:Oraclek1}
\E\{\dex^2(\hat{f}_n,f_0)\}\leq \inf_{m \geq d}\;\Biggl\{\frac{\bar{C}m}{n}\log^{\gamma_d} n\,+\!\inf_{\substack{f\in\mathcal{F}^1(\mathcal{P}^m)\\\supp f_0 \subseteq \supp f}}\dhell^2(f_0,f)\Biggr\}.
\end{equation}
\end{theorem}
We suspect that the restriction on the support of the approximating density $f$ in~\eqref{Eq:Oraclek1} is an artefact of our proof.  Indeed, in the case $d=1$, \citet{BaraudBirge2016} obtain an oracle inequality for their $\rho$-estimator where the approximating density $f$ need not have this property (although their result is stated for $\dhell^2$ rather than $\dex^2$); moreover, we have been able to strengthen the corresponding univariate result for the log-concave maximum likelihood estimator \citep[][Theorem~5]{KGS18} by removing this restriction.

The proof of Theorem~\ref{Thm:k1} in fact constitutes the main technical challenge in deriving Theorem~\ref{Thm:MainLogkaffine}.  This entails deriving upper bounds on the (local) Hellinger bracketing entropies of classes of log-concave functions that lie in small Hellinger neighbourhoods of densities in $f\in \mathcal{F}^1(\mathcal{P}^m)$ for each $m\in\N$ with $m\geq d$.  Our argument proceeds via a series of steps, the first of which deals with the case where $f$ is a uniform density on a simplex (Proposition~\ref{Prop:SimplexEntropy}); it turns out that any density in a small Hellinger ball around such an $f$ satisfies a uniform upper bound (Lemma~\ref{Lem:hellunbds}(ii)), and a pointwise lower bound whose contours are characterised geometrically in Lemma~\ref{Lem:invelopesimp} (and illustrated in Figure~\ref{Fig:invelopesimp}). We proceed by considering a finite nested sequence of polytopal subsets of the simplex, each of which has a controlled number of vertices and approximates the region enclosed by one of the aforementioned contours; see the accompanying Figure~\ref{Fig:SimplexEntropy}. After constructing suitable triangulations of the regions between successive polytopes (Corollary~\ref{Cor:polyapproxsimp}), we exploit existing bracketing entropy results for classes of bounded log-concave functions (Proposition~\ref{Prop:bebds}).

In the next step, we consider the uniform density on a polytope in $\mathcal{P}^m$; here, using the fact that there is a triangulation of the support into $O(m)$ simplices (Lemma~\ref{Lem:euler}), we apply our earlier bracketing entropy bounds in conjunction with an additional argument which handles carefully the fact that these simplices may have very different volumes (Proposition~\ref{Prop:PolytopeEntropy}).  

Finally, in the proof of Proposition~\ref{Prop:Log1affEntropy} in Section~\ref{Sec:LogkaffineProofs}, we generalise to settings where $f$ is an arbitrary (not necessarily uniform) log-affine density whose polyhedral support may be unbounded.  There, we subdivide the domain by intersecting it with a sequence of parallel half-spaces whose normal vectors are in the direction of the negative log-gradient of the density.  Our characterisation of such log-affine densities in Section~\ref{Subsec:LogkaffineDensities} ultimately allows us to apply our earlier results to transformations of the original density and thereby obtain the desired local bracketing entropy bounds (Proposition~\ref{Prop:Log1affEntropy}).  The conclusion of Theorem~\ref{Thm:k1} then follows from standard empirical process theory arguments~\citep[e.g.][Corollary~7.5]{vdG00}; see Section~\ref{Sec:LogkaffineProofs}.

We do not claim any optimality of the polylogarithmic factors in Theorems~\ref{Thm:MainLogkaffine} and~\ref{Thm:k1}.  In fact, we can improve these exponents in the special case where $f_0$ is well-approximated by a uniform density $f_P:=\mu_d(P)^{-1}\Ind_P$ on a polytope $P\in\mathcal{P}\equiv\mathcal{P}_d$. Note that every polytope in $\mathcal{P}_d$ has at least as many facets as a $d$-simplex, namely $d+1$; see for example Lemma~\ref{Lem:minfacets}.
\begin{proposition}
\label{Prop:PolytopeRisk}
Fix $d \in \{2,3\}$, and for $m\geq d+1$, denote by $\mathcal{F}^{[1]}(\mathcal{P}^m)$ the subclass of all uniform densities on polytopes in $\mathcal{P}^m$. Let $X_1,\ldots,X_n\iid f_0\in\mathcal{F}_d$ with $n\geq d+1$, and let $\hat{f}_n$ denote the corresponding log-concave maximum likelihood estimator. Then there exists a universal constant $C'>0$ such that
\begin{equation}
\label{Eq:OraclePm}
\E\{\dex^2(\hat{f}_n,f_0)\}\leq \inf_{m \geq d+1}\;\Biggl\{\frac{C'm}{n}\log^{\gamma_d'} n\,+\!\inf_{\substack{f\in\mathcal{F}^{[1]}(\mathcal{P}^m)\\\supp f_0 \subseteq \supp f}}\dhell^2(f_0,f)\Biggr\},
\end{equation}
where $\gamma_2':=3$ and $\gamma_3':=6$.
\end{proposition}

\section{Adaptation to densities bounded away from zero on a polytopal support}
\label{Sec:ThetaPolytope}
Recall from the discussion in the introduction that in order to observe adaptive behaviour for the log-concave maximum likelihood estimator, we need to exclude uniform densities supported on convex sets with smooth boundaries. In fact, we will see from Proposition~\ref{Prop:LSC} below that we also need to rule out subclasses containing sequences of elements of $\mathcal{F}_d$ that approximate such uniform densities. In this section, we continue to work with densities in $\mathcal{F}_d$ that are close to a log-concave density with polyhedral support, but, in contrast to Section~\ref{Sec:Logkaffine}, now drop the requirement that this approximating density be log-$k$-affine.  In fact, we do not impose any extra structural constraints or smoothness conditions that would regulate further the behaviour of the densities on the interiors of their supports.  It will turn out, however, that we will only be able to improve on the worst-case risk bounds of Theorem~\ref{Thm:WorstCaseRates} when the approximating density is also bounded away from zero on its support, which must therefore necessarily be a polytope.  The generality of the resulting new classes means that we can no longer expect near-parametric adaptation, and moreover, for the reasons explained in the introduction, our main result of this section (Theorem~\ref{Thm:ThetaRisk} below) is restricted to the case $d=3$.  As an example of a density that will be covered by this result, we can consider the density of a trivariate Gaussian random vector conditioned to lie in $[-1,1]^3$.

The proofs of both results in this section are given in Section~\ref{Subsec:ThetaProofs}.

Following on from Proposition~\ref{Prop:PolytopeRisk}, we now extend the definition of $\mathcal{F}^{[1]}(\mathcal{P}^m)$ given above and introduce our new family of subclasses of $\mathcal{F}_d$. For $\theta\in (0,\infty)$ and a polytope $P\in\mathcal{P}_d$, let $\mathcal{F}^{[\theta]}(P)\equiv\mathcal{F}_d^{[\theta]}(P)$ denote the collection of all $f\in\mathcal{F}_d$ for which $\supp f=P$ and $f\geq\inv{\theta}f_P$ on $P$. Then $\mathcal{F}^{[1]}(P)=\{f_P\}$ and $\mathcal{F}^{[\theta]}(P)$ is non-empty if and only if $\theta\geq 1$. For $\theta\in [1,\infty)$ and $m\in\N$ with $m\geq d+1$, denote by $\mathcal{F}^{[\theta]}(\mathcal{P}^m)\equiv\mathcal{F}_d^{[\theta]}(\mathcal{P}_d^m)$ the union of those $\mathcal{F}^{[\theta]}(P)$ for which $P$ is a polytope in $\mathcal{P}^m\equiv\mathcal{P}_d^m$, and note that this is a non-empty affine invariant subclass of $\mathcal{F}_d$. Indeed, fix $b\in\R^d$ and an invertible $A\in\R^{d\times d}$, and let $T\colon\R^d\to\R^d$ be the invertible affine transformation defined by $T(x):=Ax+b$. If $X\sim f\in\mathcal{F}^{[\theta]}(P)$ for some polytope $P\in\mathcal{P}^m$, then $\mu_d(T(P))=\abs{\det A}\,\mu_d(P)$, and so the density $g$ of $T(X)$ satisfies $g(x)=\abs{\det A}^{-1}f(\inv{T}(x))\geq\inv{\{\theta\,\abs{\det A}\,\mu_d(P)\}}=\inv{\{\theta\mu_d(T(P))\}}$ for all $x\in T(P)$. Since $\supp g=T(P)$ is also a polytope in $\mathcal{P}^m$, this shows that $g\in\mathcal{F}^{[\theta]}(\mathcal{P}^m)$, as required.

The sharp oracle inequality~\eqref{Eq:OracleTheta} below may be viewed as 
complementary to Theorem~\ref{Thm:k1} and Proposition~\ref{Prop:PolytopeRisk}.
\begin{theorem}
\label{Thm:ThetaRisk}
Let $X_1,\ldots,X_n \iid f_0\in\mathcal{F}_3$ with $n\geq 4$, and let $\hat{f}_n$ denote the corresponding log-concave maximum likelihood estimator. Then there exists a universal constant $C>0$ such that 
\begin{alignat}{3}
&\E\{\dex^2(\hat{f}_n,f_0)\}&&\notag\\
&\hspace{0.1cm}\leq\!\!\inf_{\substack{m\geq 4\\ \theta\in (1,\infty)}}\,\Biggl\{C\,\biggl(\log^{6/7}\theta\, \rbr{\frac{m}{n}}^{4/7}\log_+^{17/7}\rbr{\frac{n}{\log^{3/2}\theta}}&&+\rbr{\frac{m}{n}}^{20/29}\log^{85/29}n+\theta\log^3(e\theta)\,\frac{m\log^6 n}{n}\biggl)\notag\\
\label{Eq:OracleTheta}
&&&\hspace{2.2cm}+\!\!\inf_{\substack{f\in\mathcal{F}_3^{[\theta]}(\mathcal{P}^m)\\\supp f_0\subseteq\supp f}}\dhell^2(f_0,f)\Biggr\}.
\end{alignat}
\end{theorem}
For a fixed $\theta\in (1,\infty)$, note that if $n/m$ is sufficiently large, then the dominant contribution to the right-hand side of~\eqref{Eq:OracleTheta} comes from the first term. It follows that for fixed $\theta,m$, the log-concave maximum likelihood estimator $\hat{f}_n$ of $f_0\in\mathcal{F}_3^{[\theta]}(\mathcal{P}^m)$ converges at rate $\tilde{O}(n^{-4/7})$ as $n\to\infty$, which was the rate originally conjectured by \citet{SereginWellner2010}.

Despite the attractions of the adaptation mentioned in the previous paragraph, it is worth considering the bound~\eqref{Eq:OracleTheta} in the limits as $\theta \searrow 1$ and $\theta \rightarrow \infty$.  In the first case, owing to the presence of the second term on the right-hand side of~\eqref{Eq:OracleTheta}, we do not recover the bound~\eqref{Eq:OraclePm} from Proposition~\ref{Prop:PolytopeRisk} when we take the limit of the right-hand side of~\eqref{Eq:OracleTheta}; see Section~\ref{Subsec:LocalBE} for further discussion.  We also mention here that for a fixed $n$, the bound in~\eqref{Eq:OraclePm} may be stronger than that in~\eqref{Eq:OracleTheta} if for example $f_0\in\mathcal{F}_3^{[\theta]}(\mathcal{P}^m)$ for some $\theta\equiv\theta_n\in (1,\infty)$ sufficiently close to 1. To substantiate this remark, we note that if $\theta\in [1,\infty)$ and $P\in\mathcal{P}_3$ is a polytope, then it follows from the proof of Lemma~\ref{Lem:hellunbds}(iii) that every $f\in\mathcal{F}_3^{[\theta]}(P)$ satisfies $\theta^{-1}f_P\leq f\lesssim\log^3(e\theta)f_P$ on $P$.  Thus, if $f_0\in\mathcal{F}_3^{[\theta]}(P)$, then
\[
\dhell^2(f_0,f_P)=\int_P\,\bigl(\sqrt{f_0}-\sqrt{f_P}\bigr)^2\lesssim (1-\theta^{-1})\vee \bigl(\log^3(e\theta)-1\bigr)\lesssim\theta-1
\]
when $\theta\leq 2$. Consequently, if $n$ is fixed and $\theta,m$ are such that $\theta\leq 1+n^{-20/29}$ and $m\leq n^{9/29}\log^{-6}n$, then for any $f_0\in\mathcal{F}_3^{[\theta]}(P)$ with $P\in\mathcal{P}^m$, the bound in~\eqref{Eq:OraclePm} is at most a universal constant multiple of $(m/n)\log^6 n+(\theta-1)\lesssim n^{-20/29}$, whereas the bound in~\eqref{Eq:OracleTheta} is at least a universal constant multiple of $n^{-20/29}\log^{85/29}n$.

It is also notable that the bound in~\eqref{Eq:OracleTheta} diverges to infinity as $\theta\to\infty$.  In fact, we will deduce from Proposition~\ref{Prop:LSC} below that this is not just an artefact of our analysis; more precisely, the log-concave maximum likelihood estimator does not adapt uniformly over $\bigcup_{\,\theta \geq 1} \mathcal{F}_d^{[\theta]}(P)$, or indeed over any subclass of $\mathcal{F}_d$ containing an approximating sequence for a uniform density on a closed Euclidean ball.
\begin{proposition}
\label{Prop:LSC}
Fix $d \in \mathbb{N}$ and $n \geq d+1$. Let $(f^{(\ell)})$ be a sequence of densities in $\mathcal{F}_d$ for which the corresponding sequence of probability measures $(P^{(\ell)})$ converges weakly to a distribution $P^{(0)}$ with density $f^{(0)}\colon\R^d\to [0,\infty)$. For each $\ell \in \mathbb{N}_0$, let $X_1^{(\ell)},\ldots,X_n^{(\ell)} \stackrel{\mathrm{iid}}{\sim} f^{(\ell)}$, and let $\hat{f}_n^{(\ell)}$ denote the corresponding log-concave maximum likelihood estimator. Then
\[
\liminf_{\ell \rightarrow \infty}\,\mathbb{E} \bigl\{\dex^2\bigl(\hat{f}_n^{(\ell)},f^{(\ell)}\bigr)\bigr\} \geq \mathbb{E} \bigl\{\dex^2\bigl(\hat{f}_n^{(0)},f^{(0)}\bigr)\bigr\}.
\]
\end{proposition}
To understand the consequences of this lower semi-continuity result, fix any polytope $P\in\mathcal{P}_d$ and a closed Euclidean ball $B\subseteq \Int P$.  We can find a sequence $(f^{(\ell)})$ in $\bigcup_{\,\theta \geq 1} \mathcal{F}_d^{[\theta]}(P)$ such that the corresponding probability measures converge weakly to the uniform distribution on $B$.  Such a sequence must necessarily satisfy $\inf_{x \in P} f^{(\ell)}(x) \rightarrow 0$, and Proposition~\ref{Prop:LSC}, together with the result of~\citet{Wieacker1987} mentioned in the introduction, then ensures that $\liminf_{\ell \rightarrow \infty} \mathbb{E} \bigl\{\dex^2\bigl(\hat{f}_n^{(\ell)},f^{(\ell)}\bigr)\bigr\} \gtrsim_d n^{-2/(d+1)}$ for $d \geq 2$.  Thus, indeed, no adaptation is possible.

The proof of Theorem~\ref{Thm:ThetaRisk} follows a similar approach to that set out after the statement of Theorem~\ref{Thm:k1}. The key intermediate results are the local bracketing entropy bounds in Propositions~\ref{Prop:ThetaSimplexEntropy} and~\ref{Prop:ThetaEntropy} in Section~\ref{Subsec:LocalBE}, which are analogous to the Propositions~\ref{Prop:SimplexEntropy} and~\ref{Prop:PolytopeEntropy} that prepare the ground for the proof of Theorem~\ref{Thm:k1}. As we explain in the discussion before the proof of Proposition~\ref{Prop:SimplexEntropy}, some modifications to the previous arguments are necessary, but we once again draw heavily on the technical apparatus developed in Section~\ref{Subsec:EntropyCalcs}. The key reason we are able to apply these techniques here is that the densities in $\mathcal{F}^{[\theta]}(\mathcal{P}^m)$ are bounded away from zero, as evidenced by the fact that the bound~\eqref{Eq:OracleTheta} diverges as $\theta\to\infty$. Once we have obtained Proposition~\ref{Prop:ThetaEntropy}, all that remains is to appeal to standard empirical process theory~\citep[Corollary~7.5]{vdG00}, from which the desired conclusion~\eqref{Eq:OracleTheta} follows readily; see Section~\ref{Subsec:ThetaProofs}.  
In contrast to the proof of Theorem~\ref{Thm:k1}, we do not require an additional argument along the lines of the proof of Proposition~\ref{Prop:Log1affEntropy} given in Section~\ref{Sec:LogkaffineProofs}, which is specific to the log-1-affine densities (with possibly unbounded polyhedral support) studied in Section~\ref{Sec:Logkaffine}.

\section{Adaptation to densities with well-separated contours}
\label{Sec:Smoothness}
In this section, we consider adaptation of the log-concave maximum likelihood estimator over yet further subclasses of $\mathcal{F}_d$.  As discussed in Examples~\ref{Ex:Holder1} and~\ref{Ex:Holder2} below, these are designed to generalise notions of H\"older smoothness, while at the same time satisfying our key property of affine invariance.  Given $S \in \mathbb{S}^{d \times d}$ and $x \in \mathbb{R}^d$, we write $\|x\|_S := (x^\top S^{-1} x)^{1/2}$ for its \emph{$S$-Mahalanobis norm}. 
\begin{definition}
\label{Def:Separation}
For $\beta\geq 1$ and $\Lambda,\tau>0$, let $\mathcal{F}^{(\beta,\Lambda,\tau)}\equiv\mathcal{F}_d^{(\beta,\Lambda,\tau)}$ denote the collection of all $f\in\mathcal{F}_d$ that are continuous on $\R^d$ and satisfy
\begin{equation}
\label{Eq:Separation}
\norm{x-y}_{\Sigma_f}\geq\frac{\{f(x)-f(y)\}\det^{1/2}\Sigma_f}{\Lambda\,\bigl\{f(x)\det^{1/2}\Sigma_f\}^{1-1/\beta}}
\end{equation}
whenever $x,y\in\R^d$ are such that $f(y)<f(x)<\tau\det^{-1/2}\Sigma_f$. In addition, we define $\mathcal{F}^{(\beta,\Lambda)}:=\bigcap_{\tau>0}\mathcal{F}^{(\beta,\Lambda,\tau)}$.
\end{definition}
The defining condition~\eqref{Eq:Separation} imposes a separation condition on contours below some fixed level. For instance, when $f$ is isotropic, the condition asks that for all small $t>0$, the contours of $f$ at levels $t$ and $2t$ are at least a distance of order $\inv{\Lambda}t^{1/\beta}$ apart. See below for further discussion and motivating examples. We now collect together some basic properties of the classes $\mathcal{F}^{(\beta,\Lambda,\tau)}$. 
\begin{proposition}
\label{Prop:SepProperties}
For $\beta\geq 1$ and $\Lambda,\tau>0$, we have the following:
\begin{enumerate}[label=(\roman*)]
\item $\mathcal{F}^{(\beta,\Lambda,\tau)}$ is \emph{affine invariant}; i.e.\ if $X\sim f\in \mathcal{F}^{(\beta,\Lambda,\tau)}$ and $T\colon\R^d\to\R^d$ is an invertible affine transformation, then the density $g$ of $T(X)$ also lies in $\mathcal{F}^{(\beta,\Lambda,\tau)}$.
\item $\mathcal{F}^{(\beta,\Lambda,\tau)}\subseteq\mathcal{F}^{(\beta,\Lambda^*)}$ for all $\Lambda^*\geq \Lambda(B_d/\tau)^{1/\beta}$, where $B_d:=\sup_{h\in\mathcal{F}_d^{0,I}}\sup_{x\in\R^d}h(x) \in (0,\infty)$.
\item If $\alpha\in [1,\beta)$, then $\mathcal{F}^{(\beta,\Lambda,\tau)}\subseteq\mathcal{F}^{(\alpha,\Lambda',\tau)}$ for all $\Lambda'\geq B_d^{1/\alpha-1/\beta}\Lambda$.
\item There exists $\Lambda_{0,d}>0$, depending only on $d$, such that $\mathcal{F}^{(\beta,\Lambda)}$ is non-empty only if $\Lambda\geq\Lambda_{0,d}$.
\end{enumerate}
\end{proposition}
Note in particular that since the log-concave maximum likelihood estimator $\hat{f}_n$ is affine equivariant \citep[Remark~2.4]{DSS11}, and since our loss functions $\dhell^2$, $\KL$ and $\dex^2$ are affine invariant, property (i) above allows us to restrict attention to \emph{isotropic} $f\in\mathcal{F}^{(\beta,\Lambda,\tau)}$, namely those belonging to $\mathcal{F}_d^{0,I}$. Property (iii) indicates that the classes $\mathcal{F}^{(\beta,\Lambda,\tau)}$ are nested with respect to the exponent $\beta\geq 1$. 

In addition, by taking $\alpha=1$ in (iii) and then applying (ii), we deduce that the densities in $\mathcal{F}^{(\beta,\Lambda,\tau)}$ are all Lipschitz on $\R^d$, but as we will see in Examples~\ref{Ex:Laplace} and~\ref{Ex:Holder1}, they need not be differentiable everywhere. In cases where $f\in\mathcal{F}_d$ is differentiable on an open set of the form $\{x\in\R^d:f(x)<\tau^*\}$ for some $\tau^*>0$, the necessary and sufficient condition in the following proposition provides us with a simpler way of checking whether $f$ lies in $\mathcal{F}^{(\beta,\Lambda,\tau)}$.  For $w \in \mathbb{R}^d$ and $S \in \mathbb{S}^{d \times d}$, let $\|w\|_S' := (\tm{w}\inv{S}w)^{1/2}\det^{-1/2}S$ denote its \emph{scaled $S$-Mahalanobis norm}.
\begin{proposition}
\label{Prop:Differentiable}
Suppose that there exists $\tau^*>0$ such that $f\in\mathcal{F}_d$ is continuous on $\R^d$ and differentiable at every $x\in\R^d$ satisfying $f(x)<\tau^*$. Then for $\beta\geq 1$ and any $\tau\leq\tau^*\det^{1/2}\Sigma_f$, we have $f\in\mathcal{F}^{(\beta,\Lambda,\tau)}$ if and only if
\begin{equation}
\label{Eq:GradBd}
\norm{\nabla f(x)}_{\inv{\Sigma_f}}'\leq\Lambda\,\bigl\{f(x)\textstyle\det^{1/2}\Sigma_f\bigr\}^{1-1/\beta}
\end{equation}
for all $x\in\R^d$ with $f(x) < \tau\det^{-1/2}\Sigma_f$.
\end{proposition}
Our main result in this section is a sharp oracle inequality for the performance of the log-concave maximum likelihood estimator when the true log-concave density is close to $\mathcal{F}_d^{(\beta,\Lambda)}$ when $d=3$. In view of Proposition~\ref{Prop:SepProperties}(ii), we work here with the classes $\mathcal{F}_3^{(\beta,\Lambda)}$ rather than the more general classes $\mathcal{F}_3^{(\beta,\Lambda,\tau)}$ for ease of presentation. Let $\Lambda_0\equiv\Lambda_{0,3}>0$ be the universal constant from Proposition~\ref{Prop:SepProperties}(iv) and its proof, and for each $\beta\geq 1$, let $r_\beta:=\frac{\beta+3}{\beta+7}\wedge\frac{4}{7}$.
\begin{theorem}
\label{Thm:Smoothness}
Let $X_1,\ldots,X_n \iid f_0 \in \mathcal{F}_3$ for some $n \geq 4$, and let $\hat{f}_n$ denote the corresponding log-concave maximum likelihood estimator. Then there exists a universal constant $C > 0$ such that 
\begin{equation}
\label{Eq:Smoothness}
\mathbb{E}\{\dex^2(\hat{f}_n,f_0)\}\leq\inf_{\beta\geq 1,\,\Lambda\geq\Lambda_0}\; \biggl\{C\Lambda^{\frac{4\beta}{\beta+7}\,\wedge\, 1}\,n^{-r_\beta}\log^{\frac{16\beta+39}{2(\beta+3)}r_\beta}n\,+\!\!\inf_{f \in \mathcal{F}_3^{(\beta,\Lambda)}}\dhell^2(f_0,f)\biggr\}.
\end{equation}
\end{theorem}
Ignoring polylogarithmic factors and focusing on the case where $f_0 \in \mathcal{F}_3^{(\beta,\Lambda)}$ for some $\beta\geq 1$ and $\Lambda>0$, Theorem~\ref{Thm:Smoothness} presents a continuum of rates that interpolate between the worst-case rate of $\tilde{O}(n^{-1/2})$, corresponding to the rate when $\beta=1$, and $\tilde{O}(n^{-4/7})$, again matching the rate conjectured by~\citet{SereginWellner2010}. 

The proofs of all results in this section are given in Section~\ref{Subsec:SmoothnessProofs}.

As mentioned in the introduction, the main attraction of working with the general contour separation condition~\eqref{Eq:Separation} is that we can give several examples of classes of densities contained within $\mathcal{F}^{(\beta,\Lambda,\tau)}$ for suitable $\beta$, $\Lambda$ and $\tau$.  Since each of the conditions~\eqref{Eq:Separation} and~\eqref{Eq:GradBd} are affine invariant, it suffices to check these conditions for the isotropic elements of the relevant classes (or for any other convenient choice of scaling). Moreover, to verify~\eqref{Eq:Separation} for densities that are spherically symmetric, it suffices to consider pairs $x,y$ of the form $y = \lambda x$ for some $\lambda > 0$; in other words, if $f(x) = g(\|x\|)$, then it is enough to verify the contour separation condition~\eqref{Eq:Separation} for $g$.
\begin{example}[Gaussian densities]
Writing $f\colon x\mapsto(2\pi)^{-d/2}\,e^{-\|x\|^2/2}$ for the standard Gaussian density on $\mathbb{R}^d$ and fixing an arbitrary $\beta\geq 1$, we have
\begin{align*}
\|\nabla f(x)\|_{I}' = \|\nabla f(x)\| = \frac{\|x\|}{(2\pi)^{d/2}}\,e^{-\|x\|^2/2} &= 2^{1/2}f(x)\log^{1/2}\biggl(\frac{1}{(2\pi)^{d/2}f(x)}\biggr) \\
&\leq \frac{\beta^{1/2}}{(2\pi)^{d/(2\beta)}}\,e^{-1/2}f(x)^{1-1/\beta}
\end{align*}
for all $x \in \mathbb{R}^d$.  Hence, it follows from Proposition~\ref{Prop:Differentiable} that $f \in \mathcal{F}^{(\beta,\Lambda)}$ for all $\beta\geq 1$, with $\Lambda = \beta^{1/2}e^{-1/2}\,(2\pi)^{-d/(2\beta)}$. Thus, Theorem~\ref{Thm:Smoothness} implies that when $d=3$, the log-concave maximum likelihood estimator attains the rate $\tilde{O}(n^{-4/7})$ in $d_{\!X}^2$ divergence uniformly over the class of Gaussian densities.
\end{example}
\begin{example}[Spherically symmetric Laplace density]
\label{Ex:Laplace}
Let $V_d:=\mu_d(\bar{B}(0,1))=\pi^{d/2}/\Gamma(1+d/2)$. Then $f\colon x\mapsto (d!\,V_d)^{-1}e^{-\|x\|}$ is a density in $\mathcal{F}_d$ with corresponding covariance matrix $\Sigma\equiv\Sigma_f = (d+1)I$. For $\tau \leq (d+1)^{d/2}\,(d!\,V_d)^{-1}$ and any $\beta\geq 1$, we have
\[
\|\nabla f(x)\|_{\Sigma^{-1}}' = (d+1)^{(d+1)/2}f(x) \leq \frac{(d+1)^{(d+1)/2}}{(d!\,V_d)^{1/\beta}}f(x)^{1-1/\beta}
\]
for all $x \in \mathbb{R}^d$ with $f(x) < \tau\det^{-1/2}\Sigma=\tau(d+1)^{-d/2}$.  Hence, when $d=3$, the log-concave maximum likelihood estimator attains the rate $\tilde{O}(n^{-4/7})$ in $d_{\!X}^2$ divergence uniformly over the class of densities that are affinely equivalent to $f$, even though $f$ is not differentiable at 0. A similar conclusion holds for the densities $f_1,f_2$ satisfying $f_1(x)\propto\exp(-e^{\norm{x}})$ and $f_2(x)\propto\exp\bigl(-e^{e^{\norm{x}}}\bigr)$.
\end{example}
\begin{example}[Spherically symmetric bump function density]
\label{Ex:bump}
Consider the smooth density $f\colon x\mapsto Ce^{-1/(1-\|x\|^2)}\mathbbm{1}_{\{\|x\| < 1\}}$, where $C > 0$ is a normalisation constant. By~\citet[][Proposition~2]{XS19}, $f$ is log-concave. Writing $\Sigma\equiv\Sigma_f = \sigma^2I$ for the covariance matrix corresponding to $f$, and again fixing an arbitrary $\beta\geq 1$, we see that each $x\in\R^d$ with $\|x\| < 1$ satisfies
\begin{align*}
\|\nabla f(x)\|_{\Sigma^{-1}}' &= \sigma^{d+1}\|\nabla f(x)\| = \sigma^{d+1}\frac{2C\|x\|}{(1-\|x\|^2)^2}\,e^{-1/(1-\|x\|^2)} \\
&\leq 2\sigma^{d+1}f(x)\log^2\!\rbr{\frac{C}{f(x)}} \leq \Lambda_\beta\, \bigl\{f(x)\det^{1/2}\Sigma\}^{1-1/\beta},
\end{align*}
where $\Lambda_\beta := 8C^{1/\beta}\beta^2e^{-2}\sigma^{1+d/\beta}$.  Thus, again by Proposition~\ref{Prop:Differentiable}, we deduce that $f \in \mathcal{F}^{(\beta,\Lambda_\beta)}$ for all $\beta\geq 1$.  Consequently, when $d=3$, the log-concave maximum likelihood estimator attains the rate $\tilde{O}(n^{-4/7})$ in $d_{\!X}^2$ divergence uniformly over the class of densities that are affinely equivalent to $f$.
\end{example}

\begin{example}[H\"older condition on the log-density]
\label{Ex:Holder1}
For $\gamma \in (1,2]$ and $L > 0$, let $\tilde{\mathcal{H}}^{\gamma,L}\equiv\tilde{\mathcal{H}}_d^{\gamma,L}$ denote the subset of densities $f\in\mathcal{F}_d$ such that $\phi:=\log f$ is differentiable and
\begin{equation}
\label{Eq:LogDensity1}
\norm{\nabla \phi(y)-\nabla \phi(x)}_{\inv{\Sigma_f}} \leq L\norm{y-x}_{\Sigma_f}^{\gamma-1}
\end{equation}
for all $x,y\in\mathbb{R}^d$. We extend this definition to $\gamma=1$ by writing $\tilde{\mathcal{H}}^{1,L}$ for the subset of densities $f\in\mathcal{F}_d$ for which $\phi=\log f$ satisfies
\begin{equation}
\label{Eq:LogDensity2}
|\phi(y) - \phi(x)| \leq L\norm{y-x}_{\Sigma_f}.
\end{equation}
for all $x,y\in\mathbb{R}^d$. Note that the densities in $\tilde{\mathcal{H}}^{1,L}$ can have points of non-differentiability for arbitrarily small values of the density. For instance, if we define $f\in\mathcal{F}_d$ by 
\[
f(x) \propto \exp\biggl(-\sum_{r=0}^\infty \frac{\|x\|-r}{2^r}\,\Ind_{\{\|x\| \geq r\}}\biggr),
\]
which is not differentiable at any $x\in\R^d$ with integer Euclidean norm, then $f\in\tilde{\mathcal{H}}^{1,L}$ for suitably large $L>0$.  

The careful and non-standard choice of norms in~\eqref{Eq:LogDensity1} and~\eqref{Eq:LogDensity2} ensures that the classes $\tilde{\mathcal{H}}^{\gamma,L}$ are affine invariant. Moreover, Proposition~\ref{Prop:Holder}(iv) in Section~\ref{Subsec:SmoothnessSupp} shows that for each $\beta\geq 1$, there exists $\Lambda'\equiv\Lambda'(\beta,L)$ such that $\bigcup_{\gamma \in [1,2]}\tilde{\mathcal{H}}^{\gamma,L}\subseteq\mathcal{F}^{(\beta,\Lambda')}$. 
Thus, when $d=3$, the log-concave maximum likelihood estimator attains the rate $\tilde{O}(n^{-4/7})$ in $\dex^2$ divergence uniformly over $\bigcup_{\gamma\in [1,2]}\tilde{\mathcal{H}}^{\gamma,L}$. 

A related result in the literature is~\citet[Theorem~4.1]{DumbgenRufibach2009}, which applies when $d=1$, $\gamma\in (1,2]$ and the logarithm of the true fixed $f_0\in\mathcal{F}_1$ is $\gamma$-H\"older on some compact subinterval $T$ of the interior of $\supp f_0$. In this case, the corresponding $\hat{f}_n$ is shown to achieve an adaptive rate of order $\bigl(\frac{\log n}{n}\bigr)^{\frac{\gamma}{2\gamma+1}}$ with respect to the supremum norm over certain compact subintervals of the interior of $T$. We remark that this is not entirely comparable with the rate we obtain in the paragraph above, especially since our loss function $\dex^2$ is rather different.

Observe that the densities in the classes $\tilde{\mathcal{H}}^{\gamma,L}$ must be supported on the whole of $\R^d$, and that conditions~\eqref{Eq:LogDensity1} and~\eqref{Eq:LogDensity2} imply that the rate of tail decay of $f$ is `super-Gaussian'. This is quite a stringent restriction; note for example that the density $f$ satisfying $f(x)\propto\exp(-e^{\norm{x}})$ does not feature in any of the classes $\tilde{\mathcal{H}}^{\gamma,L}$.  Another drawback of this definition of smoothness is that the classes are not nested with respect to the H\"older exponent $\gamma\in (1,2]$; this can be seen by considering a density $f$ satisfying $f(x) \propto \exp(-\|x\|^\gamma)$, which belongs to $\tilde{\mathcal{H}}^{\tilde{\gamma},L}$ for some $L > 0$ if and only if $\tilde{\gamma} = \gamma$.
\end{example}

\begin{example}[H\"older condition on the density]
\label{Ex:Holder2}
To remedy the issues mentioned in the previous example, fix $\beta\in (1,2]$ and $L>0$, and let $\mathcal{H}^{\beta,L} \equiv\mathcal{H}_d^{\beta,L}$ denote the set of $f\in\mathcal{F}_d$ such that $f$ is differentiable on $\R^d$ and
\begin{equation}
\label{Eq:HolderOriginal}
\norm{\nabla f(y)-\nabla f(x)}_{\inv{\Sigma_f}}'\leq L\norm{y-x}_{\Sigma_f}^{\beta-1}
\end{equation}
for all $x,y\in\R^d$.  Again, it can be shown that the classes $\mathcal{H}^{\beta,L}$ are affine invariant, and if $f \in \mathcal{F}_d$ is $\beta$-H\"older in the usual Euclidean sense, i.e.\ $\|\nabla f(y)-\nabla f(x)\| \leq L\|y-x\|^{\beta-1}$ for all $x,y \in \mathbb{R}^d$, then $f \in \mathcal{H}^{\beta,\tilde{L}}$ with $\tilde{L} := L\,\lambda_{\mathrm{max}}^{\beta/2}(\Sigma_f) \det^{1/2}\Sigma_f$, where $\lambda_{\mathrm{max}}(\Sigma_f)$ denotes the maximum eigenvalue of $\Sigma_f$. This is because $\norm{w}_{\Sigma_f}\geq\norm{w}\,\lambda_{\mathrm{max}}^{-1/2}(\Sigma_f)$ and $\norm{w}_{\inv{\Sigma_f}}'\leq\norm{w}\,\lambda_{\mathrm{max}}^{1/2}(\Sigma_f)\det^{1/2}\Sigma_f$  for all $w\in\R^d$. Moreover, Proposition~\ref{Prop:Nested} shows that the classes $\mathcal{H}^{\beta,L}$ are nested with respect to the H\"older exponent $\beta$; more precisely, if $\beta,L$ are as above, then there exists $\tilde{L}\equiv\tilde{L}(d,\beta,L)>0$ such that $\mathcal{H}^{\beta,L}\subseteq\mathcal{H}^{\alpha,\tilde{L}}$ for all $\alpha\in (1,\beta]$.

The condition~\eqref{Eq:HolderOriginal} can in fact be extended to an affine invariant notion of $\beta$-H\"older regularity for all $\beta>1$; see Section~\ref{Subsec:Holder} for full technical details. Here, we present the analogue of~\eqref{Eq:HolderOriginal} for $\beta\in (2,3]$ and $L>0$, for which we require the following additional notation. First, if $g\colon\R^d\to\R$ is twice differentiable at $x\in\R^d$, then denote by $Hg(x)\in\R^{d\times d}$ the Hessian of $g$ at $x$. In addition, for each $S\in\mathbb{S}^{d\times d}$, define a norm $\norm{{\cdot}}_S'$ on $\R^{d\times d}$ by $\norm{M}_S':=\norm{S^{-1/2}MS^{-1/2}}_{\mathrm{F}}\det^{-1/2}S$, where $\norm{A}_{\mathrm{F}}:=\tr(\tm{A}A)^{1/2}$ denotes the Frobenius norm of $A\in\R^{d\times d}$. We now define $\mathcal{H}^{\beta,L}$ to be the collection of $f\in\mathcal{F}_d$ for which $f$ is twice differentiable on $\R^d$ and
\begin{equation}
\label{Eq:HolderFrob}
\norm{Hf(y)-Hf(x)}_{\inv{\Sigma}_f}'\leq L\norm{y-x}_{\Sigma_f}^{\beta-2}
\end{equation}
for all $x,y\in\R^d$. In Section~\ref{Subsec:Holder}, we present a unified argument that establishes the affine invariance of the classes $\mathcal{H}^{\beta,L}$ defined by~\eqref{Eq:HolderOriginal} and~\eqref{Eq:HolderFrob}; see the proof of Lemma~\ref{Lem:holderaffinv}.

In addition, for each $\beta\in(1,3]$ and $L>0$, parts (i) and (iii) of Proposition~\ref{Prop:Holder} in Section~\ref{Subsec:SmoothnessSupp} imply that $\mathcal{H}^{\beta,L} \subseteq \mathcal{F}^{(\beta,\Lambda)}$ for some $\Lambda\equiv\Lambda(\beta,L)$; when $\beta\in (1,2]$, we can take $\Lambda(\beta,L) := L^{1/\beta}(1-1/\beta)^{-1+1/\beta}$.  It was this fact that motivated our choice of parametrisation in~$\beta$ in~\eqref{Eq:Separation}.  Theorem~\ref{Thm:Smoothness} therefore yields the rate $\tilde{O}\bigl(n^{-\min\cbr{\frac{\beta+3}{\beta+7},\frac{4}{7}}}\bigr)$ for the log-concave maximum likelihood estimator, uniformly over $\mathcal{H}^{\beta,L}$. An interesting feature of this rate is that, when $\beta \in (1,9/5)$, it is faster than the rate $O(n^{-\frac{2\beta}{2\beta+3}})$ that can be obtained in squared Hellinger distance for $\beta$-H\"older densities that satisfy a `tail dominance' condition~\citep[Section 4]{GL14}. For further details of this comparison, see Section~\ref{Subsec:HolderRev}. Thus, in this range of $\beta$, the log-concavity shape constraint results in a strict improvement in the rates attainable.
\end{example}

\section{Appendix: Proofs}
\label{Sec:Proofs}

The following notation is used in this section and in the supplementary material.

To define bracketing entropy, let $S \subseteq \mathbb{R}^d$ and let $\mathcal{G}$ be a class of non-negative functions whose domains contain $S$. For $\varepsilon > 0$ and a semi-metric $\rho$ on $\mathcal{G}$, let $N_{[\,]}(\varepsilon, \mathcal{G}, \rho,S)$ denote the smallest $M \in \mathbb{N}$ for which there exist pairs of functions $\{[g_{j}^L, g_{j}^U]: j = 1, \dots, M\}$ such that $\rho(g_j^U,g_j^L) \leq \varepsilon$    
for every $j=1,\ldots,M$, and such that for every $g \in \mathcal{G}$, there exists $j^* \in \{1,\ldots,M\}$ with $g_{j^*}^L(x) \leq g(x) \leq g_{j^*}^U(x)$ for every $x \in S$.  We then define the \emph{$\varepsilon$-bracketing entropy of $\mathcal{G}$ over $S$ with respect to $\rho$} by $H_{[\,]}(\varepsilon, \mathcal{G}, \rho,S) := \log N_{[\,]}(\varepsilon, \mathcal{G}, \rho,S)$ and write $H_{[\,]}(\varepsilon, \mathcal{G}, \rho) := H_{[\,]}(\varepsilon, \mathcal{G}, \rho,\R^d)$ when $S=\R^d$.

For each $f_0\in\mathcal{F}_d$ and $\delta>0$, let $\mathcal{G}(f_0,\delta)\equiv\mathcal{G}_d(f_0,\delta):=\{f\Ind_{\supp f_0}:f\in\mathcal{G}_d,\,\dhell(f,f_0)\leq\delta\}$.  In addition, let $\mathcal{F}(f_0,\delta)\equiv\mathcal{F}_d(f_0,\delta)=\mathcal{F}_d\cap\mathcal{G}_d(f_0,\delta)$ and let $\tilde{\mathcal{F}}(f_0,\delta)\equiv\tilde{\mathcal{F}}_d(f_0,\delta):=\{f\in\mathcal{F}_d:\dhell(f,f_0)\leq\delta\}$. Writing $\norm{M}\equiv\norm{M}_{\mathrm{op}}:=\sup_{\norm{u}\leq 1}\norm{Mu}$ for the operator norm of a matrix $M\in\R^{d\times d}$, we denote by $\tilde{\mathcal{F}}^{1,\eta}\equiv\tilde{\mathcal{F}}_d^{1,\eta_d}:=\{f\in\mathcal{F}_d:\norm{\mu_f}\leq 1,\,\norm{\Sigma_f-I}\leq\eta_d\}$ the class of \emph{`near-isotropic'} log-concave densities, 
where the constant $\eta\equiv\eta_d\in (0,1)$ is taken from \citet[Lemma~6]{KS16} and depends only on $d$. Finally, we define $h_2,\,h_3\colon (0,\infty)\to  (0,\infty)$ by $h_2(x):=\inv{x}\log_+^{3/2}(\inv{x})$ and $h_3(x):=x^{-2}$ respectively.

\subsection{Proofs of main results in Section~\ref{Sec:Logkaffine}}
\label{Sec:LogkaffineProofs}

The proof of Proposition~\ref{Prop:Logkaff} is lengthy and is deferred to Section~\ref{Subsec:LogkaffineDensities}.  The main goal of this subsection, therefore, is to prove Theorem~\ref{Thm:MainLogkaffine}, which proceeds via several intermediate results, including Theorem~\ref{Thm:k1}.  We begin by stating our main local bracketing entropy result, whose proof is summarised at the end of Section~\ref{Sec:Logkaffine}. Note that by Proposition~\ref{Prop:FkPmempty}, the subclass $\mathcal{F}^1(\mathcal{P}^m)$ is non-empty if and only if $m\geq d$.

\begin{proposition}
\label{Prop:Log1affEntropy}
Let $d \in \{2,3\}$ and fix $m\in\N$ with $m\geq d$. Then there exist universal constants $\varrho_2,\varrho_3>0$ such that whenever $0<\varepsilon<\delta<\varrho_d$ and $f_0 \in\mathcal{F}^1(\mathcal{P}^m)$, we have
\begin{equation}
\label{Eq:Log1affEntropy2}
H_{[\,]}(2^{1/2}\varepsilon,\mathcal{G}(f_0,\delta),\dhell)\lesssim m\rbr{\frac{\delta}{\varepsilon}}\log^3\rbr{\frac{1}{\delta}}\log^{3/2}\rbr{\frac{\log(1/\delta)}{\varepsilon}}
\end{equation}
when $d=2$ and 
\begin{equation}
\label{Eq:Log1affEntropy3}
H_{[\,]}(2^{1/2}\varepsilon,\mathcal{G}(f_0,\delta),\dhell)\lesssim m\,\biggl\{\rbr{\frac{\delta}{\varepsilon}}^2\log^6\rbr{\frac{1}{\delta}}+\rbr{\frac{\delta}{\varepsilon}}^{3/2}\log^7\rbr{\frac{1}{\delta}}\biggr\}
\end{equation}
when $d=3$.
\end{proposition}
See Propositions~\ref{Prop:SimplexEntropy} and~\ref{Prop:PolytopeEntropy} for details of the initial stages of the proof, which deal with the case where $f_0$ is the uniform density $f_K:=\mu_d(K)^{-1}\Ind_K$ on some polytope $K\in\mathcal{P}^m$. Here, we turn our attention to the general non-uniform case, where the support of $f_0$ may be unbounded. Writing $\mathcal{F}^1$ for the subclass of all log-1-affine densities in $\mathcal{F}_d$, we note that any $f\in\mathcal{F}^1$ must take the form $x\mapsto f_{K,\alpha}(x):=c_{K,\alpha}^{-1}\exp(-\tm{\alpha}x)\Ind_{\{x\in K\}}$, where $K\subseteq\R^d$ and $\alpha\in\R^d$ are the support and negative log-gradient of $f$ respectively, and $c_{K,\alpha}:=\int_K\,\exp(-\tm{\alpha}x)\,dx \in (0,\infty)$; see~\eqref{Eq:F1}. It follows from the characterisation of $\mathcal{F}^1$ given in Proposition~\ref{Prop:log1aff} that $K$ and $\alpha$ satisfy the conditions of Proposition~\ref{Prop:kslices}(ii), which in turn implies that $m_{K,\alpha}:=\inf_{x\in K}\tm{\alpha}x$ is finite. In addition, let $M_{K,\alpha}:=\sup_{x\in K}\tm{\alpha}x\in (-\infty,\infty]$, and for $t\in\R$, define the convex sets 
\begin{align*}
&K_{\alpha,t}:=K\cap\{x\in\R^d:\tm{\alpha}x=t\},\\
&K_{\alpha,t}^+:=K\cap\{x\in\R^d:\tm{\alpha}x\leq t\},\\
&\breve{K}_{\alpha,t}:=K\cap\{x\in\R^d:t-1\leq\tm{\alpha}x\leq t\},
\end{align*}
which are all compact by Proposition~\ref{Prop:kslices}; see Figure~\ref{Fig:Kalphat} for an illustration. Finally, we denote by $\mathcal{F}_\star^1$ the collection of all $f=f_{K,\alpha}\in\mathcal{F}^1$ for which $m_{K,\alpha}=0$.
\begin{proof}[Proof of Proposition~\ref{Prop:Log1affEntropy}]
For a fixed $d\in\{2,3\}$, let $\upsilon\equiv\upsilon_d:=2^{-3/2}\wedge \{d^{-1/2}(d+1)^{-(d-1)/2}\}$, $C\equiv C_d:=8d+7$ and $\varrho\equiv\varrho_d:=\nu_d\,\wedge\,\upsilon_d\,e^{-C/2}\,\gamma(d,C)^{1/2}$, where $\gamma\equiv\gamma(d,C)$ and $\nu\equiv\nu_d$ are taken from Lemmas~\ref{Lem:log1affbds} and~\ref{Lem:hellnunbd} respectively. For $0<\varepsilon<\delta<\upsilon$, the important quantity $H_d(\delta,\varepsilon)$ is defined in Proposition~\ref{Prop:SimplexEntropy}. 

Fix $0<\varepsilon<\delta<\varrho$ and $m\in\N$ with $m\geq d$. It follows from Corollary~\ref{Cor:log1affcanon} and the affine invariance of the Hellinger distance that we need only consider densities $f_0=f_{K,\alpha}\in\mathcal{F}_\star^1\cap\mathcal{F}^1(\mathcal{P}^m)$, which have the property that $K\in\mathcal{P}^m$ and $m_{K,\alpha}=0$. Since Proposition~\ref{Prop:PolytopeEntropy} handles the case $\alpha=0$, we fix an arbitrary $f_{K,\alpha}\in\mathcal{F}_\star^1\cap\mathcal{F}^1(\mathcal{P}^m)$ with $\alpha\neq 0$, and set $L:=\ceil{M_{K,\alpha}}\in\N\cup\{\infty\}$. Now define
\[K_j':=
\begin{cases}
K_{\alpha,C}^+&\text{for }j=C\\
\breve{K}_{\alpha,j}\qquad&\text{for each }j\in\N\text{ with }C+1\leq j\leq L,
\end{cases}
\]
which is compact for all integers $C\leq j\leq L$. Note also that since $K\in\mathcal{P}^m$, it follows from~\citet[Theorem~1.6]{BG09} that $K_C'\in\mathcal{P}^{m+1}$ and $K_j'\in\mathcal{P}^{m+2}$ for all integers $C+1\leq j\leq L$. 

Also, let $a_+$ be the smallest integer $C+1\leq j\leq L$ such that $\delta^2e^{j+1}\inv{\mu_d(\breve{K}_{\alpha,j})}c_{K,\alpha}\geq \upsilon^2$ if such a $j$ exists, and let $a_+=L+1$ otherwise. Since $(1/\tilde{\delta})^{d-1}\geq\log^{d-1}(1/\tilde{\delta})\geq d(d+1)^{d-1}\,\upsilon^2\log^{d-1}(1/\tilde{\delta})$ for all $\tilde{\delta}\in (0,1)$, we deduce from~\eqref{Eq:ratioslices} in Lemma~\ref{Lem:log1affbds} that
\[\frac{\delta^2e^{t+1}c_{K,\alpha}}{\textcolor{white}{\bigl(}\mu_d(\breve{K}_{\alpha,t})\textcolor{white}{\bigr)}}\geq\frac{\delta^2e^{t+1}c_{K,\alpha}}{\textcolor{white}{\bigl(}dt^{d-1}\mu_d(K_{\alpha,1}^+)\textcolor{white}{\bigr)}}\geq\frac{\delta^2e^t}{dt^{d-1}}\geq\upsilon^2\]
for all $t\geq (d+1)\log(1/\delta)$, and hence that $a_+\lesssim\log(1/\delta)$. 
Next, set $u_j^2:=c\exp\{-(j-a_+)/2\}$ for each integer $a_+\leq j\leq L$, where $c:=1-e^{-1/2}$ is chosen to ensure that $\sum_{j=a_+}^L u_j^2\leq 1$, and also define
\begin{align*}
\varepsilon_j^2:=
\begin{cases}
2\varepsilon^2/3\quad&\text{for $j=C$}\\
2\varepsilon^2\inv{(a_+-C)}/3\quad&\text{for $j=C+1,\dotsc,a_+-1$}\\
2u_j^2\,\varepsilon^2/3\quad&\text{for $j=a_+,\dotsc,L$}.
\end{cases}
\end{align*}
Since $K=\bigcup_{j=C}^{\,L}\,K_j'$ and $\sum_{j=C}^L\,\varepsilon_j^2\leq 2\varepsilon^2$, we can write
\begin{align}
\label{Eq:bn1}
H_{[\,]}(2^{1/2}\varepsilon,\mathcal{G}(f_{K,\alpha},\delta),\dhell)&\leq H_{[\,]}(\varepsilon_C,\mathcal{G}(f_{K,\alpha},\delta),\dhell,K_C')\\
\label{Eq:bn2}
&\hspace{0.5cm}+
\textstyle\sum_{j=C+1}^{a_+-1}H_{[\,]}(\varepsilon_j,\mathcal{G}(f_{K,\alpha},\delta),\dhell,K_j')\\
\label{Eq:bn3}
&\hspace{0.5cm}+\textstyle\sum_{j=a_+}^{\,L} H_{[\,]}(\varepsilon_j,\mathcal{G}(f_{K,\alpha},\delta),\dhell,K_j'),
\end{align}
and we now address each of the terms~\eqref{Eq:bn1},~\eqref{Eq:bn2} and~\eqref{Eq:bn3} in turn. Note that while there are infinitely many summands in~\eqref{Eq:bn3} when $M_{K,\alpha}=L=\infty$, it will follow from the bounds we obtain that only finitely many of these are non-zero.

For~\eqref{Eq:bn1}, let $A_C:=c_{K,\alpha}/\mu_d(K_C')$, which by Lemma~\ref{Lem:log1affbds} satisfies $e^{-C}\leq A_C\leq\inv{\gamma}$. For $f\in\mathcal{G}(f_{K,\alpha},\delta)$, define $\tilde{f}_C\colon\R^d\to [0,\infty)$ by $\tilde{f}_C(x):=A_C\exp(\tm{\alpha}x)f(x)\Ind_{\{x\in K_C'\}}$ and observe that
\[\delta^2\geq\int_{K_C'}\bigl(f_C^{1/2}-f_{K,\alpha}^{1/2}\bigr)^2 =\int_{K_C'}\frac{e^{-\tm{\alpha}x}}{A_C}\bigl\{\tilde{f}_C^{1/2}(x)-f_{K_C'}^{1/2}(x)\bigr\}^2\,dx\geq\frac{e^{-C}}{A_C}\int_{K_C'}\bigl(\tilde{f}_C^{1/2}-f_{K_C'}^{1/2}\bigr)^2,\]
which shows that $\tilde{f}_C\in\mathcal{G}\bigl(f_{K_C'},A_C^{1/2}e^{C/2}\delta\bigr)$. Since $\delta<\varrho<\upsilon e^{-C/2}\,\gamma^{1/2}$, it follows from the above bounds on $A_C$ that 
\begin{equation}
\label{Eq:ACbds}
\delta\leq A_C^{1/2}e^{C/2}\delta<\upsilon<2^{-3/2}\quad\text{and}\quad A_C^{-1/2}\inv{\varepsilon_C}\lesssim\inv{\varepsilon}.
\end{equation}
Recalling that $K_C'\in\mathcal{P}^{m+1}$, we can now apply Proposition~\ref{Prop:PolytopeEntropy} to deduce that there exists an $\bigl(A_C^{1/2}\varepsilon_C\bigr)$-Hellinger bracketing set $\{[\tilde{g}_\ell^L,\tilde{g}_\ell^U]:1\leq\ell\leq N_C\}$ for $\mathcal{G}\bigl(f_{K_C'},A_C^{1/2}e^{C/2}\delta\bigr)$ such that 
\begin{equation}
\label{Eq:bracketKC}
\log N_C\lesssim (m+1)H_d\bigl(A_C^{1/2}e^{C/2}\delta,A_C^{1/2}\varepsilon_C\bigr)\lesssim m H_d(\delta,\varepsilon).
\end{equation}
We see that $\{f\Ind_{K_C'}:f\in\mathcal{G}(f_{K,\alpha},\delta)\}$ is covered by the brackets $\{[g_\ell^L,g_\ell^U]:1\leq\ell\leq N_C\}$ defined by \[g_\ell^L(x):=\inv{A_C}\exp(-\tm{\alpha}x)\,\tilde{g}_\ell^L(x);\quad g_\ell^U(x):=\inv{A_C}\exp(-\tm{\alpha}x)\,\tilde{g}_\ell^U(x).\] 
Moreover, $\exp(-\tm{\alpha}x)\leq 1$ for all $x\in K_C'$, so \[\int_{K_C'}\,\biggl(\sqrt{g_\ell^U}-\sqrt{g_\ell^L}\biggr)^2=\inv{A_C}\int_{K_C'}\,\biggl(\sqrt{\tilde{g}_\ell^U(x)}-\sqrt{\tilde{g}_\ell^L(x)}\biggr)^2\exp(-\tm{\alpha}x)\,dx \leq\varepsilon_C^2\]
for all $1\leq\ell\leq N_C$. Together with~\eqref{Eq:bracketKC}, this implies that
\begin{equation}
\label{Eq:bn1bd}
H_{[\,]}(\varepsilon_C,\mathcal{G}(f_{K,\alpha},\delta),\dhell,K_C')\leq\log N_C\lesssim m H_d(\delta,\varepsilon).
\end{equation}
For~\eqref{Eq:bn2}, fix an integer $C+1\leq j\leq a_+-1$ (if such a $j$ exists) and let $A_j:=c_{K,\alpha}/\mu_d(K_j')$. For $f\in\mathcal{G}(f_{K,\alpha},\delta)$, define $\tilde{f}_j\colon\R^d\to [0,\infty)$ by $\tilde{f}_j(x):=A_j\exp(\tm{\alpha}x)f(x)\Ind_{\{x\in K_j'\}}$. Now
\[\delta^2\geq\int_{K_j'}\bigl(f^{1/2}-f_{K,\alpha}^{1/2}\bigr)^2 =\!\int_{K_j'}\frac{e^{-\tm{\alpha}x}}{A_j}\bigl\{\tilde{f}_j^{1/2}(x)-f_{K_j'}^{1/2}(x)\bigr\}^2 \,dx
\geq\frac{e^{-j}}{A_j}\int_{K_j'}\bigl(\tilde{f}_j^{1/2}-f_{K_j'}^{1/2}\bigr)^2,\]
so $\tilde{f}_j\in\mathcal{G}\bigl(f_{K_j'},A_j^{1/2}\,e^{j/2}\,\delta\bigr)$. Since $j\leq a_+-1$, it follows from the definition of $a_+$ that $A_j<\delta^{-2}\upsilon^2\,e^{-(j+1)}$. In addition, since $K_j'\subseteq K_{\alpha,j}^+$, we can apply Lemma~\ref{Lem:log1affbds} to deduce that $A_j\geq c_{K,\alpha}/\mu_d(K_{\alpha,j}^+)\geq e^{-j}$. Therefore, \[\delta\leq A_j^{1/2}\,e^{j/2}\,\delta<\upsilon<2^{-3/2}\quad\text{and}\quad A_j^{-1/2}\,e^{-(j-1)/2}\lesssim 1.\]
Since $K_j'\in\mathcal{P}^{m+2}$, we can apply Proposition~\ref{Prop:PolytopeEntropy} to deduce that there exists an $(A_j^{1/2}\,e^{(j-1)/2}\,\varepsilon_j)$-Hellinger bracketing set
$\{[\tilde{g}_\ell^L,\tilde{g}_\ell^U]:1\leq\ell\leq N_j\}$ for $\mathcal{G}\bigl(f_{K_j'},A_j^{1/2}\,e^{j/2}\,\delta\bigr)$ such that
\begin{equation}
\label{Eq:bracketKj}
\log N_j\lesssim (m+2)H_d\bigl(A_j^{1/2}\,e^{j/2}\,\delta,A_j^{1/2}\,e^{(j-1)/2}\,\varepsilon_j\bigr)\lesssim m H_d(\delta,\varepsilon_j).
\end{equation}
We see that $\{f\Ind_{K_j'}:f\in\mathcal{G}(f_{K,\alpha},\delta)\}$ is covered by the brackets $\{[g_\ell^L,g_\ell^U]:1\leq\ell\leq N_j\}$ defined by \[g_\ell^L(x):=\inv{A_j}\exp(-\tm{\alpha}x)\,\tilde{g}_\ell^L(x);\quad g_\ell^U(x):=\inv{A_j}\exp(-\tm{\alpha}x)\,\tilde{g}_\ell^U(x).\]
Moreover, $\exp(-\tm{\alpha}x)\leq e^{-(j-1)}$ for all $x\in K_j'$, so
\[\int_{K_j'}\,\biggl(\sqrt{g_\ell^U}-\sqrt{g_\ell^L}\biggr)^2=\inv{A_j}\int_{K_j'}\,\biggl(\sqrt{\tilde{g}_\ell^U(x)}-\sqrt{\tilde{g}_\ell^L(x)}\biggr)^2\exp(-\tm{\alpha}x)\,dx\leq\varepsilon_j^2\]
for all $1\leq\ell\leq N_j$. Together with~\eqref{Eq:bracketKj} and the fact that $a_+\lesssim\log(1/\delta)$, this implies that 
\[\sum_{j=C+1}^{a_+-1}H_{[\,]}(\varepsilon_j,\mathcal{G}(f_{K,\alpha},\delta),\dhell,K_j')\lesssim\log(1/\delta)\,m H_d\bigl(\delta,\,\varepsilon/\log(1/\delta)^{1/2}\bigr),\]
which is bounded above up to a universal constant by
\begin{equation}
\label{Eq:bn2bd2}
m\rbr{\frac{\delta}{\varepsilon}}\log^3\rbr{\frac{1}{\delta}}\log^{3/2}\rbr{\frac{\log(1/\delta)}{\varepsilon}}
\end{equation}
when $d=2$ and
\begin{equation}
\label{Eq:bn2bd3}
m\cbr{\rbr{\frac{\delta}{\varepsilon}}^2\log^6\rbr{\frac{1}{\delta}}+\rbr{\frac{\delta}{\varepsilon}}^{3/2}\log^7\rbr{\frac{1}{\delta}}}
\end{equation}
when $d=3$.

For~\eqref{Eq:bn3}, if $L\geq C+1$, consider $f=e^\phi\in\mathcal{G}(f_{K,\alpha},\delta)$ and define $\psi\equiv\tilde{\phi}_{K,\alpha}\colon\R^d\to [-\infty,\infty)$ by $\psi(x):=\phi(x)+\tm{\alpha}x+\log c_{K,\alpha}$, as in the statement of Lemma~\ref{Lem:hellnunbd}. First, we claim that
\begin{equation}
\label{Eq:logenv}
\psi(x)\leq\frac{4d+2}{a_+-2}\,\tm{\alpha}x
\end{equation}
for all $x\in K\setminus K_{\alpha,\,a_+-1}^+$. To see this, first set $\tilde{K}:=K_{\alpha,\,a_+-1}^+$ and $\tilde{A}:=c_{K,\alpha}/\mu_d(\tilde{K})$, and define $\tilde{f}\colon\R^d\to [0,\infty)$ by $\tilde{f}(x):=\tilde{A}\exp(\tm{\alpha}x)f(x)\Ind_{\{x\in\tilde{K}\}}$. Observe that
\[\log\tilde{f}(x)=\log f(x)+\tm{\alpha}x+\log c_{K,\alpha}-\log\mu_d(\tilde{K})=\psi(x)-\log\mu_d(\tilde{K}).\]
By similar arguments to those given above, we deduce that $\tilde{f}\in\mathcal{G}\bigl(f_{\tilde{K}},\tilde{A}^{1/2}\,e^{(a_+-1)/2}\,\delta\bigr)$. Moreover, if $a_+\geq C+2$, then it follows from the definitions of $a_+$ and $\upsilon$ that \[\tilde{A}e^{a_+-1}\delta^2\leq\inv{\mu_d(\breve{K}_{\alpha,\,a_+-1})}c_{K,\alpha}\,e^{a_+-1}\delta^2<\inv{e}\upsilon^2<2^{-3}.\]
Otherwise, if $a_+=C+1$, then recall from~\eqref{Eq:ACbds} that
\[\tilde{A}e^{a_+-1}\delta^2=A_C\,e^C\delta^2<\upsilon^2<2^{-3}.\]
Therefore, in all cases, Lemma~\ref{Lem:hellunbds}(ii) implies that $\log\tilde{f}(x)\leq 2^{7/2}d\,(\tilde{A}^{1/2}\,e^{(a_+-1)/2}\,\delta)-\log\mu_d(\tilde{K})$ for all $x\in\tilde{K}$, and hence that $\psi\leq 4d$ on $K_{\alpha,\,a_+-1}$. On the other hand, we know from Lemma~\ref{Lem:hellnunbd} that there exists some $x_-\in K_{\alpha,1}^+$ such that $\psi(x_-)>-2$. Now if $x\in K$ and $\tm{\alpha}x>a_+-1$, then $s:=(a_+-1-\tm{\alpha}x_-)/(\tm{\alpha}x-\tm{\alpha}x_-)$ satisfies $1\geq s\geq (a_+-2)/(\tm{\alpha}x-1)>0$, and $w:=sx+(1-s)x_-$ lies in $K_{\alpha,\,a_+-1}$. It then follows from the concavity of $\psi$ that
\[\psi(x)\leq\frac{1}{s}\,\psi(w)-\frac{1-s}{s}\,\psi(x_-)\leq\frac{4d}{s}+\frac{2\,(1-s)}{s}=\frac{4d+2}{s}-2<\frac{4d+2}{a_+-2}\,\tm{\alpha}x,\]
which yields~\eqref{Eq:logenv}, as required.

Now fix an integer $a_+\leq j\leq L$ (if such a $j$ exists). First, recalling the definition of $a_+$, we deduce from the bound~\eqref{Eq:ratioslices} in Lemma~\ref{Lem:log1affbds} that 
\begin{equation}
\label{Eq:bn3ratio}
\frac{\mu_d(K_j')}{c_{K,\alpha}\,e^{a_+}}\leq \rbr{\frac{j}{a_+-1}}^{d-1}\frac{\mu_d(\breve{K}_{\alpha,\,a_+})}{c_{K,\alpha}\,e^{a_+}}\leq e\upsilon^{-2}\delta^2\rbr{\frac{j}{a_+-1}}^{d-1}.
\end{equation}
Also, it follows from~\eqref{Eq:logenv} that if $f\in\mathcal{G}(f_{K,\alpha},\delta)$, then the function $\tilde{f}_j\colon\R^d\to [0,\infty)$ defined by $\tilde{f}_j(x):=c_{K,\alpha}\exp(\tm{\alpha}x)f(x)\Ind_{\{x\in K_j'\}}$ belongs to $\mathcal{G}_{-\infty,B_j}(K_j'):=\bigl\{g\Ind_{K_j'}:g\in\mathcal{G},\,g\Ind_{K_j'}\leq e^{B_j}\bigr\}$, where $B_j:=(4d+2)j/(a_+-2)$. Now if $\{[\tilde{g}_\ell^L,\tilde{g}_\ell^U]:1\leq\ell\leq N\}$ is a $(c_{K,\alpha}^{1/2}e^{(j-1)/2}\,\varepsilon_j)$-Hellinger bracketing set for $\mathcal{G}_{-\infty,B_j}(K_j')$, then $\{f\Ind_{K_j'}:f\in\mathcal{G}(f_{K,\alpha},\delta)\}$ is covered by the brackets $\{[g_\ell^L,g_\ell^U]:1\leq\ell\leq N\}$ defined by \[g_\ell^L(x):=\inv{c_{K,\alpha}}\exp(-\tm{\alpha}x)\,\tilde{g}_\ell^L(x);\quad g_\ell^U(x):=\inv{c_{K,\alpha}}\exp(-\tm{\alpha}x)\,\tilde{g}_\ell^U(x).\]
Moreover, $\exp(-\tm{\alpha}x)\leq e^{-(j-1)}$ for all $x\in K_j'$, so \[\int_{K_j'}\,\biggl(\sqrt{g_\ell^U}-\sqrt{g_\ell^L}\biggr)^2=\inv{c_{K,\alpha}}\int_{K_j'}\,\biggl(\sqrt{\tilde{g}_\ell^U(x)}-\sqrt{\tilde{g}_\ell^L(x)}\biggr)^2\exp(-\tm{\alpha}x)\,dx \leq\varepsilon_j^2\]
for all $1\leq\ell\leq N$. Recalling that $a_+\geq C=8d+7$ and that $h_d$ is a decreasing function for $d=2,3$, we now apply~\eqref{Eq:bn3ratio} and the bound~\eqref{Eq:beconvubd} from Proposition~\ref{Prop:bebds} to deduce that
\begin{align*}
H_{[\,]}\bigl(\varepsilon_j,\mathcal{G}&(f_{K,\alpha},\delta),\dhell,K_j')\leq
H_{[\,]}\bigl(c_{K,\alpha}^{1/2}e^{(j-1)/2}\,\varepsilon_j,\mathcal{G}_{-\infty,B_j}(K_j'),\dhell\bigr)\\
&\lesssim h_d\rbr{\frac{c_{K,\alpha}^{1/2}e^{(j-1)/2}\,\varepsilon_j}{\mu_d(K_j')^{1/2}\exp\cbr{\frac{(4d+2)\,j}{2(a_+-2)}}}}\\
&\lesssim h_d\rbr{\cbr{\frac{c_{K,\alpha}}{\mu_d(K_j')}}^{1/2}\frac{e^{(j-a_+)/2}e^{a_+/2}e^{-(j-a_+)/4}\,\varepsilon}{\exp\cbr{\frac{(4d+2)(j-a_+)}{2(a_+-2)}+\frac{(4d+2)\,a_+}{2(a_+-2)}}}}\\
&\lesssim h_d\rbr{\cbr{\frac{c_{K,\alpha}\,e^{a_+}}{\mu_d(K_j')}}^{1/2}\varepsilon\exp\cbr{-\rbr{\frac{4d+2}{2(a_+-2)}-\frac{1}{4}}(j-a_+)}}\\
&\lesssim h_d\rbr{\frac{\varepsilon}{\delta}\rbr{\frac{a_+-1}{j}}^{\frac{d-1}{2}}\exp\cbr{-\rbr{\frac{4d+2}{2(a_+-2)}-\frac{1}{4}}(j-a_+)}}\\
&\lesssim h_d\rbr{\frac{\varepsilon}{\delta}\rbr{\frac{a_+-1}{j}}^{\frac{d-1}{2}}\exp\cbr{-\rbr{\frac{4d+2}{2(8d+5)}-\frac{1}{4}}(j-a_+)}},
\end{align*}
and we note that $\frac{4d+2}{2(8d+5)}-\frac{1}{4}<0$. Thus, when $d=2$, the final expression above is bounded above by a constant multiple of 
\[\frac{\delta}{\varepsilon}\log^{3/2}\rbr{\frac{\delta}{\varepsilon}}j^{1/2}\,(\log^{3/2} j)\exp\cbr{-\rbr{\frac{1}{4} - \frac{4d+2}{2(8d+5)}}(j-a_+)},\]
where we have used the fact that $\log_+(ax)\leq (1+\log a)\log_+x$ for all $x>0$ and $a\geq 1$. It follows that 
\begin{equation}
\label{Eq:bn3bd-2}
\sum_{j=a_+}^{\,L} H_{[\,]}(\varepsilon_j,\mathcal{G}(f_{K,\alpha},\delta),\dhell,K_j')\lesssim\dfrac{\delta}{\varepsilon}\log^{3/2}\rbr{\dfrac{\delta}{\varepsilon}}
\end{equation}
when $d=2$. Similarly, when $d=3$, we conclude that 
\begin{equation}
\label{Eq:bn3bd-3}
\sum_{j=a_+}^{\,L} H_{[\,]}(\varepsilon_j,\mathcal{G}(f_{K,\alpha},\delta),\dhell,K_j')\lesssim\rbr{\dfrac{\delta}{\varepsilon}}^2.
\end{equation}
The result follows upon combining the bounds~\eqref{Eq:bn1},~\eqref{Eq:bn2},~\eqref{Eq:bn3},~\eqref{Eq:bn1bd},~\eqref{Eq:bn2bd2}, \eqref{Eq:bn2bd3},~\eqref{Eq:bn3bd-2} and~\eqref{Eq:bn3bd-3}.
\end{proof}
\vspace{-0.3cm}
We are now in a position to give the proof of Theorem~\ref{Thm:k1}.
\begin{proof}[Proof of Theorem~\ref{Thm:k1}]
By the affine equivariance of the log-concave maximum likelihood estimator \citep[Remark~2.4]{DSS11} and the affine invariance of $\dhell$, we may assume without loss of generality that $f_0\in\mathcal{F}_d^{0,I}$. In addition, by~\citet[Lemma~6]{KS16}, we have
\begin{equation}
\label{Eq:KSmv}
\sup_{f_0\in\mathcal{F}_d^{0,I}}\Pr(\hat{f}_n\notin\tilde{\mathcal{F}}_d^{1,\eta_d})=O(\inv{n}),
\end{equation}
where $\tilde{\mathcal{F}}_d^{1,\eta_d}$ is the class of `near-isotropic' log-concave densities defined at the start of Section~\ref{Sec:LogkaffineProofs}. For fixed $f_0\in\mathcal{F}_d$ and $m\geq d$, let \[\Delta:=\inf_{\substack{f\in\mathcal{F}^1(\mathcal{P}^m)\\\supp f_0\subseteq\supp f}}\dhell^2(f_0,f).\]
First we consider the case $d=2$ and assume for the time being that $\Delta\leq\varrho_2/2$, where $\varrho_2$ is taken from Proposition~\ref{Prop:Log1affEntropy}. If $\delta\in (0,\varrho_2-\Delta)$, then for all $\eta'\in (0,\varrho_2-\Delta-\delta)$, there exists $f\in\mathcal{F}^1(\mathcal{P}^m)$ with $\supp f_0\subseteq\supp f$ such that $\dhell(f_0,f)\leq\Delta+\eta'$. It follows from the triangle inequality that $\mathcal{F}(f_0,\delta)\subseteq\mathcal{F}(f,\delta+\Delta+\eta')\subseteq\mathcal{F}(f,\varrho_2)$, and we deduce from the first bound~\eqref{Eq:Log1affEntropy2} in Proposition~\ref{Prop:Log1affEntropy} that
\[H_{[\,]}(2^{1/2}\varepsilon,\mathcal{F}(f_0,\delta),\dhell)\lesssim m\rbr{\frac{\delta+\Delta+\eta'}{\varepsilon}}\log^3\rbr{\frac{1}{\delta}}\log^{3/2}\rbr{\frac{\log(1/\delta)}{\varepsilon}}.\]
But since $\eta'\in (0,\varrho_2-\Delta-\delta)$ was arbitrary, it follows that
\begin{equation}
\label{Eq:Log1affEntropy2G}
H_{[\,]}(2^{1/2}\varepsilon,\mathcal{F}(f_0,\delta),\dhell)\lesssim m\rbr{\frac{\delta+\Delta}{\varepsilon}}\log^3\rbr{\frac{1}{\delta}}\log^{3/2}\rbr{\frac{\log(1/\delta)}{\varepsilon}}
\end{equation}
and hence that
\begin{equation}
\label{Eq:Log1affEntInt1}
\int_{\delta^2/2^{13}}^\delta \!H_{[\,]}^{1/2}(\varepsilon,\mathcal{F}(f_0,\delta),\dhell)\,d\varepsilon\lesssim m^{1/2}(\delta+\Delta)^{1/2}\log^{3/2}\rbr{\frac{1}{\delta}}\int_0^\delta\varepsilon^{-1/2}\log^{3/4}\rbr{\frac{\log(1/\delta)}{\varepsilon}}d\varepsilon.
\end{equation}
Now for any $a>e\delta$, we can integrate by parts to establish that
\begin{align}
\int_0^\delta\!\varepsilon^{-1/2}\log^{3/4}\rbr{\frac{a}{\varepsilon}}d\varepsilon&=a^{1/2}\!\int_{\,\log(a/\delta)}^{\,\infty}\!u^{3/4}e^{-u/2}\,du=2\delta^{1/2}\log^{3/4}\rbr{\frac{a}{\delta}}+\frac{3a^{1/2}}{2}\!\!\int_{\,\log(a/\delta)}^{\,\infty}\!\frac{e^{-u/2}}{u^{1/4}}\,du\notag\\
\label{Eq:Log1affEntInt2}
&\leq 5\delta^{1/2}\log^{3/4}(a/\delta).
\end{align}
Thus, setting $a:=\log(1/\delta)$ and combining the bounds in~\eqref{Eq:Log1affEntInt1} and~\eqref{Eq:Log1affEntInt2}, we see that
\[\frac{1}{\delta^2}\!\int_{\delta^2/2^{13}}^\delta H_{[\,]}^{1/2}(\varepsilon,\mathcal{F}(f_0,\delta)\cap\tilde{\mathcal{F}}^{1,\eta_2},\dhell)\,d\varepsilon\lesssim m^{1/2}\rbr{\frac{\delta+\Delta}{\delta^3}}^{1/2}\log^{9/4}\rbr{\frac{1}{\delta}},\]
where the right-hand side is a decreasing function of $\delta\in (0,\varrho_2-\Delta)$. On the other hand, if $\delta\geq\varrho_2-\Delta$, which is at least $\varrho_2/2$, then it follows from~\citet[Theorem~4]{KS16} that \[H_{[\,]}(\varepsilon,\tilde{\mathcal{F}}^{1,\eta_2},\dhell)\lesssim h_2(\varepsilon)\lesssim\frac{1}{\varepsilon}\log_+^{3/2}\rbr{\frac{1}{\varepsilon}}\lesssim\frac{\delta}{\varepsilon}\log_+^{3/2}\rbr{\frac{1}{\varepsilon}}\]
and hence that
\[\frac{1}{\delta^2}\!\int_{\delta^2/2^{13}}^\delta \!H_{[\,]}^{1/2}(\varepsilon,\mathcal{F}(f_0,\delta)\cap\tilde{\mathcal{F}}^{1,\eta_2},\dhell)\,d\varepsilon\lesssim\frac{1}{\delta}\log_+^{3/4}\rbr{\frac{1}{\delta}}.\]
Consequently, there exists a universal constant $C_2'>0$ such that the function $\Psi_2\colon (0,\infty)\to (0,\infty)$ defined by \[\Psi_2(\delta):=C_2'\,m^{1/2}\,\delta^{1/2}(\delta+\Delta)^{1/2}\,\log_+^{9/4}(1/\delta)\]
satisfies $\Psi_2(\delta)\geq\delta\vee\int_{\delta^2/2^{13}}^\delta H_{[\,]}^{1/2}(\varepsilon,\mathcal{F}(f_0,\delta)\cap\tilde{\mathcal{F}}^{1,\eta_2},\dhell)\,d\varepsilon$ for all $\delta>0$ and has the property that $\delta\mapsto\delta^{-2}\,\Psi_2(\delta)$ is decreasing. Setting $c_2:=2^{69/4}\,C_2'\vee 1$ and $\delta_n:=(c_2^2\,mn^{-1}\log^{9/2}n+\Delta^2)^{1/2}$, we have $\Delta\leq\delta_n$ and $\inv{\delta_n}\leq\inv{c_2}m^{-1/2}\,n^{1/2}\log^{-9/4}n\leq n^{1/2}$, so
\begin{equation}
\label{Eq:Psi2}
\delta_n^{-2}\,\Psi_2(\delta_n)\leq 2^{1/2}\,C_2'\,m^{1/2}\,\inv{\delta_n}\log^{9/4}(n^{1/2})\leq 2^{-19}n^{1/2}.
\end{equation}
We are now in a position to apply~\citet[Corollary~7.5]{vdG00}, which is restated as Theorem~10 in the online supplement to~\citet{KGS18}. It follows from this,~\eqref{Eq:KSmv} and the bound~\eqref{Eq:tailprdex} from Lemma~\ref{Lem:dexbd} that there are universal constants $\bar{C},c,c',c''>0$ such that
\begin{align}
\E\{\dex^2(\hat{f}_n,f_0)\}&\leq\int_0^{8d\log n}\Pr\bigl[\{\dex^2(\hat{f}_n,f_0)\geq t\}\cap\{\hat{f}_n\in\tilde{\mathcal{F}}^{1,\eta_2}\}\bigr]\,dt\notag\\
&\hspace{1.25cm}+(8d\log n)\,\Pr(\hat{f}_n\notin\tilde{\mathcal{F}}^{1,\eta_2})+\int_{8d\log n}^\infty\Pr\bigl\{\dex^2(\hat{f}_n,f_0)\geq t\bigr\}\,dt\notag\\
&\leq\delta_n^2+\int_{\delta_n^2}^\infty c\exp(-nt/c^2)\,dt+c'\inv{n}\log n+c''n^{-3}\leq\delta_n^2+2c'\inv{n}\log n\notag\\
\label{Eq:vdGBound}
&\leq\frac{\bar{C}m}{n}\log^{9/2}n+\Delta^2
\end{align}
for all $n\geq 3$, provided that $\Delta\leq\varrho_2/2$. On the other hand, when $\Delta>\varrho_2/2$, observe that by Theorem~\ref{Thm:WorstCaseRates}, which is a small modification of~\citet[Theorem~5]{KS16}, we have $\E\{\dex^2(\hat{f}_n,f_0)\}\lesssim n^{-2/3}\log n\lesssim (\varrho_2/2)^2\leq\Delta^2$. We have now established the $d=2$ case of the desired result.

The proof for the case $d=3$ is very similar in most respects, except that the first term in the local bracketing entropy bound~\eqref{Eq:Log1affEntropy3} from Proposition~\ref{Prop:Log1affEntropy} gives rise to a divergent entropy integral. If $\Delta\leq\varrho_3/2$, then
\[\frac{1}{\delta^2}\!\int_{\delta^2/2^{13}}^\delta H_{[\,]}^{1/2}(\varepsilon,\mathcal{F}(f_0,\delta)\cap\tilde{\mathcal{F}}^{1,\eta_3},\dhell)\,d\varepsilon
\lesssim m\rbr{\frac{\delta+\Delta}{\delta^2}}\log_+^4\rbr{\frac{1}{\delta}}\]
for all $\delta>0$, where we once again appeal to the global entropy bound 
\[H_{[\,]}(\varepsilon,\tilde{\mathcal{F}}^{1,\eta_3},\dhell)\lesssim h_3(\varepsilon)\lesssim\frac{1}{\varepsilon^2}\]
from~\citet[Theorem~4]{KS16} to handle the case $\delta\geq\varrho_3-\Delta$. We conclude as above that there exists $C_3'>0$ such that the function $\Psi_3\colon (0,\infty)\to (0,\infty)$ defined by
\[\Psi_3(\delta):=C_3'\,m^{1/2}(\delta+\Delta)\log_+^4(1/\delta)\]
has all the required properties. Also, if we set $c_3:=2^{16}\,C_3'\vee 1$, then $\delta_n:=\bigl(c_3^2\,m\inv{n}\log^8 n+\Delta^2\bigr)^{1/2}$ satisfies $\delta_n^{-2}\,\Psi_3(\delta_n)\leq 2^{-19}n^{1/2}$ for all $n\geq 4$. The rest of the argument above then goes through, and we once again use the worst-case bound $\E\{\dex^2(\hat{f}_n,f_0)\}\lesssim n^{-1/2}\log n$ from Theorem~\ref{Thm:WorstCaseRates} to handle the case where $\Delta>\varrho_3/2$.
\end{proof}
\begin{proof}[Proof of Proposition~\ref{Prop:PolytopeRisk}]
Observe that in Proposition~\ref{Prop:PolytopeEntropy}, the polylogarithmic exponents in the local bracketing entropy bounds for uniform densities on polytopes in $\mathcal{P}^m$ are smaller than those that appear in Proposition~\ref{Prop:Log1affEntropy}.  We can therefore exploit this and deduce Proposition~\ref{Prop:PolytopeRisk} from Proposition~\ref{Prop:PolytopeEntropy} in the same way as Theorem~\ref{Thm:k1} is derived from Proposition~\ref{Prop:Log1affEntropy}.  We omit the details for brevity.
\end{proof}
Now that we have established our main novel results of this section, the proof of Theorem~\ref{Thm:MainLogkaffine} is broadly similar to that of the univariate oracle inequality stated as Theorem~3 in~\citet{KGS18}, so our exposition will be brief, and we will seek to emphasise the main points of difference.
\begin{proof}[Proof of Theorem~\ref{Thm:MainLogkaffine}]
Fix $f_0\in\mathcal{F}$ and an arbitrary $f\in \bigcup_{m \in \mathbb{N}} \mathcal{F}^{k}(\mathcal{P}^m)$ such that $\KL(f_0,f)<\infty$. Note that we must have $\supp f_0\subseteq\supp f$. Proposition~\ref{Prop:Logkaff} yields a polyhedral subdivision $E_1,\dotsc,E_\ell$ of $\supp f\in\mathcal{P}$ with $\ell:=\kappa(f)\leq k$ such that $\log f$ is affine on each $E_j$, and recall that $\Gamma(f)=\sum_{j=1}^\ell\,d_j$, where $d_j := |\mathscr{F}(E_j)|$. Setting $p_j:=\int_{E_j}f_0$ and $q_j:=\int_{E_j}f$ for each $j\in \{1,\ldots,\ell\}$, we see that $\sum_{j=1}^\ell\,p_j=\sum_{j=1}^\ell\,q_j=1$. Moreover, let $N_j:=\sum_{i=1}^n\Ind_{\{X_i\in E_j\}}$ for each $j\in \{1,\ldots,\ell\}$, and partition the set of indices $\{1,\dotsc,\ell\}$ into the subsets $J_1:=\{j:N_j\geq d+1\}$ and $J_2:=\{j:N_j\leq d\}$.  Then $\abs{J_2}\leq d\ell$ and
\begin{equation}
\label{Eq:dexbd1}
\dex^2(\hat{f}_n,f_0)\leq
\frac{1}{n}\sum_{j\in J_1}\sum_{i:X_i\in E_j}\log\frac{\hat{f}_n(X_i)}{f_0(X_i)}+
\frac{d\ell}{n}\max_{1\leq i\leq n}\log\frac{\hat{f}_n(X_i)}{f_0(X_i)}.
\end{equation}
The bound~\eqref{Eq:exdex} from Lemma~\ref{Lem:dexbd} controls the expectation of the second term on the right-hand side of~\eqref{Eq:dexbd1}, so it remains to handle the first term. For each $j\in J_1$, let $f_0^{(j)},f^{(j)}\in\mathcal{F}$ be the functions defined by $f_0^{(j)}(x):=\inv{p_j}f_0(x)\Ind_{\{x\in E_j\}}$ and $f^{(j)}(x):=\inv{q_j}f(x)\Ind_{\{x\in E_j\}}$. We also denote by $\hat{f}^{(j)}$ the maximum likelihood estimator based on $\{X_1,\dotsc,X_n\}\cap E_j$, which exists and is unique with probability 1 for each $j\in J_1$ \citep[][Theorem~2.2]{DSS11}. Writing $M_1:=\sum_{j\in J_1} N_j$ and arguing as in~\citet{KGS18}, we find that
\[
\sum_{j \in J_1} \sum_{i:X_i \in E_j} \hat{f}_n(X_i) \leq \sum_{j \in J_1} \sum_{i:X_i \in E_j} \frac{N_j}{M_1}\hat{f}^{(j)}(X_i).
\]
It follows that 
\begin{align}
&\frac{1}{n}\,\E\cbr{\sum_{j\in J_1}\sum_{i:X_i\in E_j}\log\frac{\hat{f}_n(X_i)}{f_0(X_i)}}\leq\frac{1}{n}\,\E\cbr{\sum_{j\in J_1}\sum_{i:X_i\in E_j}\log\frac{N_j\hat{f}^{(j)}(X_i)/M_1}{p_jf_0^{(j)}(X_i)}}\notag\\
\label{Eq:dexbd3terms}
&=\frac{1}{n}\,\E\cbr{\sum_{j\in J_1}\sum_{i:X_i\in E_j}\log\frac{\hat{f}^{(j)}(X_i)}{f_0^{(j)}(X_i)}}+\E\rbr{\sum_{j\in J_1}\frac{N_j}{n}\log\frac{N_j}{np_j}}+\E\rbr{\frac{M_1}{n}\log\frac{n}{M_1}}\\
&=:r_1+r_2+r_3.\notag
\end{align}
To bound $r_1$, we observe that $f^{(j)}\in\mathcal{F}^1(\mathcal{P}^{d_j})$ and $\supp f_0^{(j)}\subseteq\supp f^{(j)}$ for each $j\in J_1$. Consequently, after conditioning on the set of random variables $\{N_j:j=1,\dotsc,\ell\}$, we can apply the risk bound in Theorem~\ref{Thm:k1} to each $f_0^{(j)}$ and the corresponding $\hat{f}^{(j)}$ to deduce that
\begin{align}
r_1&\leq\frac{1}{n}\,\E\,\Biggl(\,\sum_{j\in J_1}N_j\,\Biggl\{\frac{\bar{C}d_j}{N_j}\log^{\gamma_d}N_j\,+\inf_{\substack{f_1\in\mathcal{F}^1(\mathcal{P}^{d_j})\\\supp f_0^{(j)}\subseteq\supp f_1}}\dhell^2\bigl(f_0^{(j)},f_1\bigr)\Biggr\}\Biggl)\notag\\
&\leq\frac{\bar{C}\,\Gamma(f)}{n}\log^{\gamma_d}n+\sum_{j=1}^\ell\,p_j\,\dhell^2\bigl(f_0^{(j)},f^{(j)}\bigr)\notag\\
\label{Eq:r1}
&\leq\frac{\bar{C}\,\Gamma(f)}{n}\log^{\gamma_d}n+\KL(f_0,f),
\end{align}
where the penultimate inequality follows as in the proof of~\citet[Theorem~3]{KGS18}.  
Moreover,
\begin{align}
\label{Eq:r2}
r_2&\leq\sum_{j=1}^\ell\E\cbr{\frac{N_j}{n}\rbr{\frac{N_j}{np_j}-1}}-\E\,\Biggl(\,\sum_{j\in J_2}\frac{N_j}{n}\log\frac{N_j}{np_j}\Biggr)\leq\frac{\ell}{n}+\frac{d\ell}{n}\log n.
\end{align}
Finally, for $r_3$, we first suppose that $d\ell<n/2$, in which case $M_1/n\geq 1-(d\ell)/n>1/2$. Thus, arguing as in~\citet{KGS18}, we deduce that $r_3\leq (2\ell d)/n$. Together with~\eqref{Eq:dexbd3terms},~\eqref{Eq:r1},~\eqref{Eq:r2} and the fact that $\ell\leq\Gamma(f)$, this implies that the desired bound~\eqref{Eq:OracleIneq} holds whenever $d\ell<n/2$. On the other hand, if $d\ell\geq n/2$, then $\Gamma(f)/n\gtrsim 1$ and we can apply Lemma~\ref{Lem:dexbd} again to conclude that
\[\E\{\dex^2(\hat{f}_n,f_0)\}\leq\E\cbr{\max_{1\leq i\leq n}\log\frac{\hat{f}_n(X_i)}{f_0(X_i)}}\lesssim\log n\lesssim\frac{\Gamma(f)}{n}\log^{\gamma_d}n.\]
This completes the proof of~\eqref{Eq:OracleIneq}. The final assertion of Theorem~\ref{Thm:MainLogkaffine} follows from Lemma~\ref{Lem:euler} in the case $d=2$ and from the final assertion of Proposition~\ref{Prop:Logkaff} in the case $d=3$.
\end{proof}

\subsection{Proofs of results in Section~\ref{Sec:ThetaPolytope}}
\label{Subsec:ThetaProofs}
\begin{proof}[Proof of Theorem~\ref{Thm:ThetaRisk}]
As in the proof of Theorem~\ref{Thm:k1}, we apply some empirical process theory to convert the local bracketing entropy bound in Proposition~\ref{Prop:ThetaEntropy} into a statistical risk bound. Fix $f_0\in\mathcal{F}_3$, $m\geq d+1=4$ and $\theta\in (1,\infty)$, and let \[\Delta:=\inf_{\substack{f\in\mathcal{F}^{[\theta]}(\mathcal{P}^m)\\\supp f_0\subseteq\supp f}}\dhell^2(f_0,f).\]
Suppose first that $\Delta<(8\theta)^{-1/2}/2=:\theta'/2$. If $\delta\in (0,\theta'-\Delta)$, then by analogy with the derivation of~\eqref{Eq:Log1affEntropy2G} in the proof of Theorem~\ref{Thm:k1}, we deduce from Proposition~\ref{Prop:ThetaEntropy} and the triangle inequality that
\begin{alignat}{3}
&\hspace{-1.2cm}H_{[\,]}(2^{1/2}\varepsilon,\mathcal{F}(f_0,\delta),\dhell)&&\notag\\
&\lesssim m\,\Biggl\{\frac{\log^{3/2}\theta+(\delta+\Delta)^{3/5}}{\varepsilon^{3/2}}\log^{17/4}\rbr{\frac{1}{\theta\delta^2}}&&+\theta^{3/4}\rbr{\frac{\delta+\Delta}{\varepsilon}}^{3/2}\log^{21/4}\rbr{\frac{1}{\theta\delta^2}}\notag\\
\label{Eq:ThetaEntropyG1}
&&&+\theta\log^3(e\theta)\rbr{\frac{\delta+\Delta}{\varepsilon}}^2\log^4\rbr{\frac{1}{\theta\delta^2}}\Biggr\}.
\end{alignat}
On the other hand, if $\delta\geq\theta'-\Delta$, then an application of the global entropy bound in \citet[Theorem~4]{KS16} yields
\begin{equation}
\label{Eq:ThetaEntropyG2}
H_{[\,]}(2^{1/2}\varepsilon,\mathcal{F}(f_0,\delta)\cap\tilde{\mathcal{F}}^{1,\eta},\dhell)\leq H_{[\,]}(2^{1/2}\varepsilon,\tilde{\mathcal{F}}^{1,\eta},\dhell)\lesssim\frac{1}{\varepsilon^2}\lesssim\theta\rbr{\frac{\delta+\Delta}{\varepsilon}}^2,
\end{equation}
where the final inequality follows since $\theta^{-1/2}\lesssim\theta'\leq\delta+\Delta$. We deduce from~\eqref{Eq:ThetaEntropyG1} and~\eqref{Eq:ThetaEntropyG2} that
\begin{alignat*}{3}
&\hspace{-1cm}\frac{1}{\delta^2}\int_{\delta^2/2^{13}}^\delta H_{[\,]}^{1/2 }(2^{1/2}\varepsilon,\mathcal{F}(f_0,\delta)\cap\tilde{\mathcal{F}}^{1,\eta},\dhell)\,d\varepsilon&&\\
&\lesssim m^{1/2}\,\Biggl\{\frac{\log^{3/4}\theta+(\delta+\Delta)^{3/10}}{\delta^{7/4}}\log_+^{17/8}\rbr{\frac{1}{\theta\delta^2}}&&+\theta^{3/8}\,\frac{(\delta+\Delta)^{3/4}}{\delta^{7/4}}\log_+^{21/8}\rbr{\frac{1}{\theta\delta^2}}\\
&&&+\theta^{1/2}\log^{3/2}(e\theta)\,\frac{\delta+\Delta}{\delta^2}\log_+^3\rbr{\frac{1}{\delta}}\Biggr\}\\
&\lesssim m^{1/2}\,\Biggl\{\frac{\log^{3/4}\theta+(\delta+\Delta)^{3/10}}{\delta^{7/4}}\log_+^{17/8}\rbr{\frac{1}{\theta\delta^2}}&&+\theta^{1/2}\log^{3/2}(e\theta)\,\frac{\delta+\Delta}{\delta^2}\log_+^3\rbr{\frac{1}{\delta}}\Biggr\}
\end{alignat*}
for all $\delta>0$. Therefore, setting
\begin{align*}
\Phi_1(\delta)&:=m^{1/2}\,(\log^{3/4}\theta)\,\delta^{-7/4}\,\log_+^{17/8}\bigl(1/(\theta\delta^2)\bigr)\\
\Phi_2(\delta)&:=m^{1/2}\,(\delta+\Delta)^{3/10}\,\delta^{-7/4}\,\log_+^{17/8}\bigl(1/(\theta\delta^2)\bigr)\\
\Phi_3(\delta)&:=m^{1/2}\,\theta^{1/2}\log^{3/2}(e\theta)(\delta+\Delta)\,\delta^{-2}\,\log_+^3(1/\delta)\\
\Phi(\delta)&:=\Phi_1(\delta)\vee\Phi_2(\delta)\vee\Phi_3(\delta)
\end{align*}
for all $\delta>0$, we conclude that there exists a universal constant $C>0$ such that $\Psi(\delta):=C\delta^2\,\Phi(\delta)\geq\delta\vee\int_{\delta^2/2^{13}}^\delta H_{[\,]}^{1/2}(\varepsilon,\tilde{\mathcal{F}}(f_0,\delta)\cap\tilde{\mathcal{F}}^{1,\eta},\dhell)\,d\varepsilon$ for all such $\delta$. Note also that $\Phi_1,\Phi_2,\Phi_3$ and hence $\Phi$ are decreasing on $(0,\infty)$. Now for some universal constant $\tilde{C}\geq 1$, define 
\begin{align*}
\delta_1&:=\{\tilde{C}(\log^{6/7}\theta)\,(m/n)^{4/7}\log_+^{17/7}(n/\log^{3/2}\theta)\}^{1/2}\\
\delta_2&:=\{\tilde{C}(m/n)^{20/29}\log^{85/29}n+\Delta^2\}^{1/2}\\
\delta_3&:=\{\tilde{C}m\theta\log^3(e\theta)\,(m/n)\log^6 n+\Delta^2\}^{1/2}.
\end{align*}
Since $\delta_1^{-7/4}\leq\tilde{C}^{-7/8} (\log^{-3/4}\theta)\,(m/n)^{-1/2}\log_+^{-17/8}(n/\log^{3/2}\theta)\leq (n/\log^{3/2}\theta)^{1/2}$, it follows that $\log_+(1/\delta_1^2)\lesssim\log_+(n/\log^{3/2}\theta)$, so if $\tilde{C}\geq 1$ is chosen to be sufficiently large, then
\[C\Phi_1(\delta_1)\leq Cm^{1/2}\,(\log^{3/4}\theta)\,\delta_1^{-7/4}\log_+^{17/8}(1/\delta_1^2)\leq 2^{-19}n^{1/2}.\]
Similarly, since $\delta_k+\Delta\leq 2\delta_k$ for $k=2,3$, it can be verified that $C\Phi_k(\delta_k)\leq 2^{-19}n^{1/2}$ for $k=2,3$ so long as $\tilde{C}\geq 1$ is taken to be sufficiently large; see~\eqref{Eq:Psi2} in the proof of Theorem~\ref{Thm:k1} for details of a similar calculation. Since $\Phi_1,\Phi_2,\Phi_3$ are decreasing, we conclude that if $\delta_1,\delta_2,\delta_3$ are defined as above for some suitably large universal constant $\tilde{C}\geq 1$, then every $\delta\geq\delta_*:=\delta_1\vee\delta_2\vee\delta_3$
satisfies \[\delta^{-2}\,\Psi(\delta)=C\Phi(\delta)\leq C\{\Phi_1(\delta_1)\vee\Phi_2(\delta_2)\vee\Phi_3(\delta_3)\}\leq 2^{-19}n^{1/2}.\]
Thus, arguing as in the proof of Theorem~\ref{Thm:k1} and recalling the derivation of~\eqref{Eq:vdGBound} in particular, we can now apply~\citet[Corollary~7.5]{vdG00} and Lemma~\ref{Lem:dexbd} to conclude that there exists a universal constant $c'>0$ such that
\[\E\{\dex^2(\hat{f}_n,f_0)\}\leq\delta_*^2+C'n^{-1}\log n=(\delta_1^2\vee\delta_2^2\vee\delta_3^2)+c'n^{-1}\log n,\]
which implies the bound~\eqref{Eq:OracleTheta}. This completes the proof of the theorem in the case where $\Delta<\theta'/2$. Finally, suppose on the other hand that $\Delta\geq\theta'/2=(32\theta)^{-1/2}$. By Theorem~\ref{Thm:WorstCaseRates}, a small modification of~\citet[Theorem~5]{KS16}, there exists a universal constant $C'>0$ such that $\E\{\dex^2(\hat{f}_n,f_0)\}\leq C'n^{-1/2}\log n$. Observe that there exists a universal constant $\bar{c}>0$ such that if $n\geq\bar{c}\,\theta^2\log^2\theta$, then $C'n^{-1/2}\log n\leq 1/(32\theta)\leq\Delta^2$, in which case the desired bound~\eqref{Eq:OracleTheta} follows. Otherwise, if $4\leq n<\bar{c}\,\theta^2\log^2\theta$, then $n^{1/2}\log^{-5}n\leq (4^{1/2}\log^{-5}4)\vee \{(\theta\log\theta)\log^{-5}\theta\}\lesssim\theta$,
so again by Theorem~\ref{Thm:WorstCaseRates}, we conclude that \[\E\{\dex^2(\hat{f}_n,f_0)\}\lesssim n^{-1/2}\log n\lesssim\theta n^{-1}\log^6 n\lesssim\theta\log^3(e\theta)(m/n)\log^6 n.\]
This completes the proof of~\eqref{Eq:OracleTheta} in the remaining cases.
\end{proof}
Next, we study the map $f_0\mapsto\E\{\dex^2(\hat{f}_n,f_0)\}$ and prove the lower semi-continuity result stated as Proposition~\ref{Prop:LSC}, for which we require the following additional definitions. We write $d_W(Q_1,Q_2):=\inf_{(X,Y)}\,\E(\norm{X-Y})$ for the \emph{1-Wasserstein distance} between probability measures $Q_1,Q_2$ on $\R^d$, where the infimum is taken over all pairs of random variables $X,Y$ that are defined on a common probability space and have marginal distributions $Q_1,Q_2$ respectively. For probability measures $Q,Q_1,Q_2,\dotsc$ on $\R^d$, recall that $d_W(Q_n,Q)\to 0$ if and only if $Q_n\to Q$ weakly and $\int\norm{x}\,dQ_n(x)\to\int\norm{x}\,dQ(x)$. For $\phi\in\Phi_d$ and a probability measure $Q$ on $\R^d$, we also define $L(\phi,Q):=1+\int\phi\,dQ-\int e^\phi$ and $L(Q):=\sup_{\phi\in\Phi_d}L(\phi,Q)$, as in~\citet{DSS11}.
\begin{proof}[Proof of Proposition~\ref{Prop:LSC}]
Since $P^{(\ell)}\to P^{(0)}$ weakly, Skorokhod's representation theorem~\citep[e.g.][Theorem~2.19]{vdV98} implies that there exist random variables $\{X^{(\ell)}:\ell\in\N_0\}$ defined on a common probability space $(\tilde{\Omega},\tilde{\mathcal{A}},\tilde{\Pr})$, with the property that $X^{(\ell)}\sim P^{(\ell)}$ for all $\ell\in\N_0$ and $X^{(\ell)}\to X^{(0)}$ almost surely. Now consider the $n$-fold product space $(\Omega,\mathcal{A},\Pr):=(\tilde{\Omega}^n,\tilde{\mathcal{A}}^{\otimes n},\tilde{\Pr}^{\otimes n})$, where $\tilde{\Pr}^{\otimes n}$ denotes the $n$-fold product measure, and for $i\in\{1,\dotsc,n\}$ and $\ell\in\N_0$, define $X_i^{(\ell)}\colon\tilde{\Omega}^n\to\R^d$ by $X_i^{(\ell)}(\omega_1,\dotsc,\omega_n):=X^{(\ell)}(\omega_i)$. Then for each such $i$, we certainly have $X_i^{(\ell)}\to X_i^{(0)}$ almost surely as $\ell\to\infty$. Moreover, if $A_1,\dotsc,A_n\in\tilde{\mathcal{A}}$ and $\ell\in\N_0$ is fixed, then \[\textstyle\Pr\bigl(\bigcap_{\,i=1}^{\,n}\bigl\{X_i^{(\ell)}\in A_i\bigr\}\bigr)=\tilde{\Pr}^{\otimes n}\bigl(\prod_{i=1}^n (X^{(\ell)})^{-1}(A_i)\bigr)=\prod_{i=1}^n\tilde{\Pr}\bigl((X^{(\ell)})^{-1}(A_i)\bigr)=\prod_{i=1}^n P^{(\ell)}(A_i),\]
which shows that $X_1^{(\ell)},\dotsc,X_n^{(\ell)}\iid P^{(\ell)}$.

Next, for each $\ell\in\N_0$ (and $\omega\in\Omega$), denote by $\Pr_n^{(\ell)}\equiv\Pr_n^{(\ell)}(\omega):=n^{-1}\sum_{i=1}^n\delta_{X_i^{(\ell)}(\omega)}$ the empirical measure of $X_1^{(\ell)}(\omega),\dotsc,X_n^{(\ell)}(\omega)$, where we write $\delta_x$ for a Dirac (point) mass at $x\in\R^d$. Since $X_i^{(\ell)}\to X_i^{(0)}$ almost surely for each $i\in\{1,\dotsc,n\}$, it follows that there exists $\Omega_0\in\mathcal{A}$ with $\Pr(\Omega_0)=1$ such that whenever $\omega\in\Omega_0$, we have \[\int g(x)\,\Pr_n^{(\ell)}(\omega)(dx)=\frac{1}{n}\,\sum_{i=1}^n g\bigl(X_i^{(\ell)}(\omega)\bigr)\to \frac{1}{n}\,\sum_{i=1}^n g\bigl(X_i^{(0)}(\omega)\bigr)=\int g(x)\,\Pr_n^{(0)}(\omega)(dx)\]
as $\ell\to\infty$ for all continuous $g\colon\R^d\to\R$. In particular, this implies that $\Pr_n^{(\ell)}(\omega)\to\Pr_n^{(0)}(\omega)$ weakly and $\int\norm{x}\,\Pr_n^{(\ell)}(\omega)(dx)\to\int\norm{x}\,\Pr_n^{(0)}(\omega)(dx)$ for all $\omega\in\Omega_0$. Therefore, $d_W(\Pr_n^{(\ell)},\Pr_n^{(0)})\to 0$ almost surely as $\ell\to\infty$. 

By assumption, $f^{(\ell)}\in\mathcal{F}_d$ for all $\ell\in\N$ and $P^{(0)}$ has a Lebesgue density. Since $P^{(\ell)}\to P^{(0)}$ weakly, it follows from~\citet[Proposition~2]{CS10} that $P^{(0)}$ has a log-concave density $f^{(0)}$ with $f^{(\ell)}\to f^{(0)}$ almost everywhere. Since replacing $f^{(0)}$ by an equivalent density alters $\dex^2(\hat{f}_n^{(0)},f^{(0)})$ only up to almost sure equivalence, we may assume without loss of generality that $f^{(0)}$ is upper semi-continuous, i.e.\ that $f^{(0)}\in\mathcal{F}_d$. Therefore, the random variables $\dex^2(\hat{f}_n^{(\ell)},f^{(\ell)})\colon\Omega\to\R$ satisfy $\dex^2(\hat{f}_n^{(\ell)},f^{(\ell)})\geq\KL(\hat{f}_n^{(\ell)},f^{(\ell)})\geq 0$ for all $\ell\in\N_0$, so in view of Fatou's lemma, the desired conclusion will follow if we can show that 
\begin{equation}
\label{Eq:dex2conv}
\dex^2(\hat{f}_n^{(\ell)},f^{(\ell)})=\int\log\bigl(\hat{f}_n^{(\ell)}/f^{(\ell)}\bigr)\,d\Pr_n^{(\ell)}\to\int\log\bigl(\hat{f}_n^{(0)}/f^{(0)}\bigr)\,d\Pr_n^{(0)}=\dex^2(\hat{f}_n^{(0)},f^{(0)})
\end{equation}
almost surely as $\ell\to\infty$.

To this end, note that since $d_W(\Pr_n^{(\ell)},\Pr_n^{(0)})\to 0$ almost surely as $\ell\to\infty$, we deduce from the definition of $\hat{f}_n^{(\ell)}$ and~\citet[Theorem~2.15]{DSS11} that 
\begin{equation}
\label{Eq:dex2conv1}
\int\log\hat{f}_n^{(\ell)}\,d\Pr_n^{(\ell)}=L\bigl(\Pr_n^{(\ell)}\bigr)\to L\bigl(\Pr_n^{(0)}\bigr)=\int\log\hat{f}_n^{(0)}\,d\Pr_n^{(0)}
\end{equation}
almost surely as $\ell\to\infty$. To establish~\eqref{Eq:dex2conv}, it will suffice to show that $\log f^{(\ell)}\bigl(X_i^{(\ell)}\bigr)\to\log f^{(0)}\bigl(X_i^{(0)}\bigr)$ almost surely for each $i\in\{1,\dotsc,n\}$, since this will imply that
\begin{equation}
\label{Eq:dex2conv2}
\int\log f^{(\ell)}\,d\Pr_n^{(\ell)}=\frac{1}{n}\,\sum_{i=1}^n\log f^{(\ell)}\bigl(X_i^{(\ell)}\bigr)\to \frac{1}{n}\,\sum_{i=1}^n\log f^{(0)}\bigl(X_i^{(0)}\bigr)=\int\log f^{(0)}\,d\Pr_n^{(0)}
\end{equation}
almost surely as $\ell\to\infty$. Recall that $\phi_\ell:=\log f^{(\ell)}$ is concave for each $\ell\in\N_0$ and that $\phi_\ell\to\phi_0$ almost everywhere. First fix $x\in\Int\dom\phi_0$ and note that we can find $\delta>0$ and $w_1,\dotsc,w_{d+1}\in\Int\dom\phi_0$ such that $\phi_\ell(w_k)\to\phi_0(w_k)$ for all $1\leq k\leq d+1$ and $B(x,\delta)\subseteq\conv\{w_1,\dotsc,w_{d+1}\}\subseteq\Int\dom\phi_0$. Then $\inf_{B(x,\delta)}\phi_\ell\geq\inf_{1\leq k\leq d+1}\phi_\ell(w_k)$ by the concavity of $\phi_\ell$ for each $\ell\in\N$, and since the latter quantity converges to $\inf_{1\leq k\leq d+1}\phi_0(w_k)>-\infty$, we deduce that $B(x,\delta)\subseteq\Int\dom\phi_\ell$ for all sufficiently large $\ell\in\N$. Since $\phi_\ell\to\phi_0$ (almost everywhere) on $B(x,\delta)\subseteq\Int\dom\phi_0$,~\citet[Theorem~10.8]{Rock97} implies that
$\phi_\ell\to\phi_0$ uniformly on compact subsets of $B(x,\delta)$. In view of this and the continuity of $\phi_0$ on $B(x,\delta)\subseteq\Int\dom\phi_0$~\citep[Theorem~1.5.3]{Sch14}, it follows that if $x_\ell\to x$, then
\begin{equation}
\label{Eq:contconv}
\abs{\phi_\ell(x_\ell)-\phi_0(x)}\leq\abs{\phi_\ell(x_\ell)-\phi_0(x_\ell)}+\abs{\phi_0(x_\ell)-\phi_0(x)}\to 0
\end{equation}
since $x_\ell\in\bar{B}(x,\delta/2)$ for all sufficiently large $\ell$. Thus, writing $\Omega_0'\in\mathcal{A}$ for the event on which $X_i^{(\ell)}\to X_i^{(0)}$ and $X_i^{(0)}\in\Int\dom\phi_0=\Int\supp f^{(0)}$ for all $i\in\{1,\dotsc,n\}$, we see that $\Pr(\Omega_0')=1$ and deduce from~\eqref{Eq:contconv} that $\phi_\ell\bigl(X_i^{(\ell)}(\omega)\bigr)\to\phi_0\bigl(X_i^{(0)}(\omega)\bigr)$ for all $i\in\{1,\dotsc,n\}$ whenever $\omega\in\Omega_0'$. This yields~\eqref{Eq:dex2conv2}, which together with~\eqref{Eq:dex2conv1} implies~\eqref{Eq:dex2conv}, as required.
\end{proof}
\subsection{Proofs of results in Section~\ref{Sec:Smoothness}}
\label{Subsec:SmoothnessProofs}
Our first task in this section is to give the proofs of Propositions~\ref{Prop:SepProperties} and~\ref{Prop:Differentiable}, which are fairly routine.
\begin{proof}[Proof of Proposition~\ref{Prop:SepProperties}]
For (i), fix a density $f\in\mathcal{F}^{(\beta,\Lambda,\tau)}$ and define $g\in\mathcal{F}_d$ by $g(x):=\inv{\abs{\det A}}f(\inv{A}(x-b))$, where $b\in\R^d$ and $A\in\R^{d\times d}$ is invertible. Since $\Sigma_g=A\Sigma_f\tm{A}$, we have $g(x)\det^{1/2}\Sigma_g=f(\inv{A}(x-b))\det^{1/2}\Sigma_f$ and $\norm{\inv{A}x}_{\Sigma_f}=\norm{x}_{\Sigma_g}$ for all $x\in\R^d$. Now fix $x,y\in\R^d$ satisfying $g(y)<g(x)<\tau\det^{-1/2}\Sigma_g$, and let $x':=\inv{A}(x-b)$ and $y':=\inv{A}(y-b)$. Since $f(y')<f(x')<\tau\det^{-1/2}\Sigma_f$, it follows from~\eqref{Eq:Separation} that
\begin{align*}
\norm{x-y}_{\Sigma_g}=\norm{x'-y'}_{\Sigma_f}\geq\frac{\{f(x')-f(y')\}\det^{1/2} \Sigma_f}{\Lambda\,\bigl\{f(x')\det^{1/2}\Sigma_f\}^{1-1/\beta}}=\frac{\{g(x)-g(y)\}\det^{1/2} \Sigma_g}{\Lambda\,\bigl\{g(x)\det^{1/2}\Sigma_g\}^{1-1/\beta}}.
\end{align*}
This shows that $g\in\mathcal{F}^{(\beta,\Lambda,\tau)}$, as required.

For (ii), fix $f\in\mathcal{F}^{(\beta,\Lambda,\tau)}\cap\mathcal{F}^{0,I}$ and consider $x,y\in\R^d$ such that $f(x)\geq\tau$ and $f(y)<f(x)$. Now let $z_t:=x+t(y-x)$ and define $h(t):=-\log f(z_t)$ for $t\geq 0$, so that $h\colon [0,\infty)\to\R$ is continuous and convex. 
Then there exist unique $t_2>t_1\geq 0$ such that $h(t_1)=-\log\tau$, $h(t_2)=-\log\tau+\log f(x)-\log f(y)$ and $h$ is strictly increasing on $[t_1,\infty)$. It follows from the convexity of $h$ that $h(t_1+1)-h(t_1)\geq h(1)-h(0)=\log f(x)-\log f(y)=h(t_2)-h(t_1)$. Thus, $h(t_1+1)\geq h(t_2)>h(t_1)$ and so $0<t_2-t_1\leq 1$. Since $\tau\,f(y)/f(x)=f(z_{t_2})<f(z_{t_1})=\tau$ and $f\leq B_d$ on $\R^d$, we deduce from~\eqref{Eq:Separation} that 
\begin{align*}
\norm{x-y}&\geq\norm{z_{t_1}-z_{t_2}}\\
&\geq\inv{\Lambda}\,\frac{f(z_{t_1})-f(z_{t_2})}{f(z_{t_1})^{1-1/\beta}}=\inv{\Lambda}\,\frac{\tau\{1-f(y)/f(x)\}}{\tau^{1-1/\beta}}=\inv{\Lambda}\rbr{\frac{\tau}{f(x)}}^{1/\beta}\,\frac{f(x)-f(y)}{f(x)^{1-1/\beta}}\\[5pt]
&\geq\inv{\Lambda}\rbr{\frac{\tau}{B_d}}^{1/\beta}\,\frac{f(x)-f(y)}{f(x)^{1-1/\beta}}.
\end{align*}
This shows that $\mathcal{F}^{(\beta,\Lambda,\tau)}\cap\mathcal{F}^{0,I}\subseteq\mathcal{F}^{(\beta,\Lambda^*)}\cap\mathcal{F}^{0,I}$ for all $\Lambda^*\geq\Lambda(B_d/\tau)^{1/\beta}$. Since we established in (i) that the classes $\mathcal{F}^{(\beta,\Lambda,\tau)}$ and $\mathcal{F}^{(\beta,\Lambda^*)}$ are affine invariant, the desired conclusion (ii) follows. 

As for (iii), we rely on affine invariance once again in order to reduce to the isotropic case, and the result is an immediate consequence of the fact that $f\leq B_d$ on $\R^d$ for all $f\in\mathcal{F}^{0,I}$; indeed, for each such $f$, we have $\Lambda'f(x)^{1-1/\alpha}\geq\Lambda f(x)^{1-1/\beta}$ for all $x\in\R^d$ whenever $\Lambda'\geq B_d^{1/\alpha-1/\beta}\Lambda$.

To establish (iv), we appeal to~\citet[Theorem~5.14(c)]{LV06}, which asserts that $\max_{x\in\R^d}h(x)>(4e\pi)^{-d/2}=:t_d$ for all $h\in\mathcal{F}^{0,I}$. We also recall that there exist $\tilde{A}_d>0$ and $\tilde{B}_d\in\R$, which depend only on $d$, such that $h(x)\leq\exp(-\tilde{A}_d\norm{x}+\tilde{B}_d)$ for all $h\in\mathcal{F}^{0,I}$ and $x\in\R^d$ \citep[e.g.][Theorem~2(a)]{KS16}. This implies that there exists $R_d>0$, which depends only on $d$, such that $h(x)<t_d/2$ whenever $h\in\mathcal{F}^{0,I}$ and $\norm{x}>R_d$. Now if $\beta\geq 1$ and $\Lambda>0$ are such that $\mathcal{F}^{(\beta,\Lambda)}$ is non-empty, then by affine invariance, there must exist $f\in\mathcal{F}^{(\beta,\Lambda)}\cap\mathcal{F}^{0,I}$. By the facts above and the continuity of $f$, there exist $x,y\in\R^d$ such that $f(x)=t_d$ and $f(y)=t_d/2$. It follows from the defining condition~\eqref{Eq:Separation} that
\[\inv{\Lambda}\,t_d/2\leq\inv{\Lambda}\,t_d^{1/\beta}/2=\frac{f(x)-f(y)}{\Lambda f(x)^{1-1/\beta}}\leq\norm{x-y}\leq\norm{x}+\norm{y}\leq 2R_d\]
and hence that $\Lambda\geq\inv{R}_d\,t_d/4=:\Lambda_{0,d}$, as required.
\end{proof}
\begin{proof}[Proof of Proposition~\ref{Prop:Differentiable}]
Throughout, we write $\Sigma\equiv\Sigma_f$ for convenience. First suppose that \eqref{Eq:GradBd} holds for all $x\in\R^d$ satisfying $f(x)<\tau\det^{-1/2}\Sigma$, where $\tau\leq\tau^*\det^{1/2}\Sigma$ is fixed. Fix $x,y\in\R^d$ such that $f(y)<f(x)<\tau\det^{-1/2}\Sigma$, and let $t':=\inf\,\{t\in (0,1]:f(y+t(x-y))=f(x)\}$. It follows from the continuity of $f$ that $t'>0$ and that $x':=y+t'(x-y)$ satisfies $f(x')=f(x)>0$. Moreover, since $[y,x']\subseteq\{w:f(w)<\tau\det^{-1/2}\Sigma_f\}$, an open set on which $f$ is differentiable, the mean value theorem guarantees the existence of $z\in [y,x']$ such that $\tm{\nabla f(z)}\,\Sigma\,\bigl\{\inv{\Sigma}(x'-y)\bigr\}=\tm{\nabla f(z)}(x'-y)=f(x')-f(y)=f(x)-f(y)$. By considering the inner product $\ipr{v}{w}':=(\det\Sigma)\,(\tm{v}\Sigma w)$ on $\R^d$ that gives rise to the norm $\norm{{\cdot}}_{\inv{\Sigma}}'$, we can apply the Cauchy--Schwarz inequality together with~\eqref{Eq:GradBd} to deduce that
\[f(x)-f(y)\leq\frac{\norm{\nabla f(z)}_{\inv{\Sigma}}'\,\norm{\inv{\Sigma}(x'-y)}_{\inv{\Sigma}}'}{\det\Sigma}\leq\frac{\Lambda\,\bigl\{f(z)\det^{1/2}\Sigma\bigr\}^{1-1/\beta}\,\norm{x'-y}_\Sigma}{\det^{1/2}\Sigma}.\]
By the choice of $t'$, we have $f(z)\leq f(x)$, so we obtain the desired conclusion that
\[\norm{x-y}_\Sigma\geq\norm{x'-y}_\Sigma\geq\inv{\Lambda}\,\frac{\{f(x)-f(y)\}\det^{1/2}\Sigma}{\bigl\{f(x)\det^{1/2}\Sigma\bigr\}^{1-1/\beta}}.\]
Turning to the reverse implication, suppose that~\eqref{Eq:Separation} holds whenever $x,y\in\R^d$ satisfy $f(y)<f(x)<\tau\det^{-1/2}\Sigma_f$. Now fix $x\in\R^d$ such that $f(x)<\tau\det^{-1/2}\Sigma$, which by assumption is a point at which $f$ is differentiable, and let $u:=-\Sigma\,\nabla f(x)$. It can be assumed without loss of generality that $u\neq 0$, since otherwise the desired conclusion follows trivially.
Setting $h(t):=f(x+tu)$ for $t\geq 0$, we now apply the chain rule to deduce that
\begin{align*}
\norm{\nabla f(x)}_{\inv{\Sigma}}'&=-\frac{h'(0)\det^{1/2}\Sigma}{\norm{\Sigma\,\nabla f(x)}_{\Sigma}}=\lim_{t\searrow 0}\frac{\{h(0)-h(t)\}\det^{1/2}\Sigma}{\norm{tu}_\Sigma}\\[5pt]
&=\lim_{t\searrow 0}\frac{\{f(x)-f(x+tu)\}\det^{1/2}\Sigma}{\norm{tu}_\Sigma}\leq\Lambda\,\bigl\{f(x)\det^{1/2}\Sigma\bigr\}^{1-1/\beta},
\end{align*}
as required, where we have used~\eqref{Eq:Separation} to obtain the final bound.
\end{proof}
Next, we establish a local bracketing entropy bound from which we will subsequently deduce our main result, Theorem~\ref{Thm:Smoothness}.
\begin{proposition}
\label{Prop:SmoothEntropy}
Let $d=3$ and let $\Lambda_0\equiv\Lambda_{0,3}>0$ be the universal constant from Proposition~\ref{Prop:SepProperties}(iv). Then there exists a universal constant $\bar{c}\in (0,1)$ such that whenever $0<\varepsilon<\delta<\inv{e}\wedge(\bar{c}\Lambda^{-3/2}\log_+^{-1}\Lambda)$ and $f_0\in\mathcal{F}^{(\beta,\Lambda)}$ for some $\beta\geq 1$ and $\Lambda\geq\Lambda_0$, we have
\begin{align}
H_{[\,]}(2^{1/2}\varepsilon,\tilde{\mathcal{F}}(f_0,\delta)\cap\tilde{\mathcal{F}}^{1,\eta},\dhell)&\lesssim \Lambda\,\Biggl\{\frac{(\Lambda^3\delta^2)^{(\beta-1)/(\beta+3)}}{\varepsilon^2}\log^{6(\beta+2)/(\beta+3)}(\Lambda^{-3}\delta^{-2})\notag\\
\label{Eq:SmoothEntropy}
&\hspace{1.2cm}+\frac{(\Lambda^3\delta^2)^{-\frac{(4-3\beta)^+}{4(\beta+3)}}}{\varepsilon^{3/2}}\,\log^{(16\beta+39)/\{2(\beta+3)\}}(\Lambda^{-3}\delta^{-2})\Biggr\}.
\end{align}
\end{proposition}
\begin{proof}
Let $\bar{c}:=e^{-1/2}\wedge\tilde{c}\wedge\tilde{c}'$, where $\tilde{c},\tilde{c}'$ are the universal constants from Lemmas~\ref{Lem:smoothlb} and~\ref{Lem:smoothub} respectively, and fix $0<\varepsilon<\delta<\inv{e}\wedge(\bar{c}\Lambda^{-3/2}\log_+^{-1}\Lambda)$. Since $\dhell$ is affine invariant, we may assume without loss of generality that $f_0\in\mathcal{F}^{(\beta,\Lambda)}\cap\mathcal{F}^{0,I}$. First, recall from~\citet[Corollary~3(a)]{KS16} that there exist universal constants $\tilde{a}_3>0$ and $\tilde{b}_3\in\R$ such that
\begin{equation}
\label{Eq:nisoenv}
\sup_{\,h\in\tilde{\mathcal{F}}^{1,\eta}}h(x)\leq\exp\bigl(-\tilde{a}_3\norm{x}+\tilde{b}_3\bigr)=:M(x).
\end{equation}
for all $x\in\R^3$. Thus, there exists a universal constant $C_3^*>0$ such that $f_0(x)\leq M(x)\leq\Lambda^3\delta^2$ whenever $\norm{x}>C_3^*\log(\Lambda^{-3}\delta^{-2})$. 
Now let $r:=\ceil{C_3^*\log(\Lambda^{-3}\delta^{-2})\vee\inv{(2\Lambda_0\bar{\eta})}}$
and $D:=[-r,r]^3$, where the universal constants $\Lambda_0>0$ and $\bar{\eta}\equiv\bar{\eta}_3>0$ are taken from Proposition~\ref{Prop:SepProperties} and Lemma~\ref{Lem:BI75} respectively, and observe that
\begin{equation}
\label{Eq:befirst}
H_{[\,]}(2^{1/2}\varepsilon,\tilde{\mathcal{F}}(f_0,\delta)\cap\tilde{\mathcal{F}}^{1,\eta},\dhell)\leq H_{[\,]}(\varepsilon,\tilde{\mathcal{F}}^{1,\eta},\dhell,\cm{D})+H_{[\,]}(\varepsilon,\tilde{\mathcal{F}}(f_0,\delta)\cap\tilde{\mathcal{F}}^{1,\eta},\dhell,D).
\end{equation}

We begin by considering the first quantity on the right-hand side. For each $z=(z_1,z_2,z_3)\in\Z^3$, let $R_z:=\prod_{j=1}^3\,[z_j,z_j+1)$ and note that $m_z:=\max_{y\in R_z}M(y)\leq e^{\sqrt{3}\tilde{a}_3}M(w)$ for all $w\in R_z$. Writing $A\,\triangle\,B$ for the symmetric difference of sets $A,B$, we set $J:=\{z\in\Z^3:R_z\not\subseteq D\}$ and observe that $\mu_3\bigl(\cm{D}\,\triangle\,\bigcup_{z\in J}R_z\bigr)=0$, from which we deduce that
\begin{align}
b:=\sum_{z\in J}m_z^{1/2}\leq e^{\sqrt{3}\tilde{a}_3/2}\int_{\cm{D}}M^{1/2}\leq e^{\sqrt{3}\tilde{a}_3/2}\int_{\cm{\bar{B}(0,r)}}M^{1/2}&=4\pi e^{\sqrt{3}\tilde{a}_3/2}\int_r^\infty t^2e^{-(\tilde{a}_3t-\tilde{b}_3)/2}\,dt\notag\\
\label{Eq:sumboxes}
&\lesssim\Lambda^{3/2}\delta\log^2(\Lambda^{-3}\delta^{-2}).
\end{align}
We now apply the bound~\eqref{Eq:beconvubd} from Proposition~\ref{Prop:bebds} to establish that 
\begin{align}
\label{Eq:be1}
H_{[\,]}(\varepsilon,\tilde{\mathcal{F}}^{1,\eta},\dhell,\cm{D})&\leq\sum_{z\in J}H_{[\,]}(m_z^{1/4}\,b^{-1/2}\varepsilon,\tilde{\mathcal{F}}^{1,\eta},\dhell,R_z)\notag\\
&\lesssim\sum_{z\in J}\,\frac{m_z^{1/2}}{b^{-1}\varepsilon^2}\lesssim\frac{\Lambda^3\delta^2\log^4(\Lambda^{-3}\delta^{-2})}{\varepsilon^2}.
\end{align}

To handle the second term on the right-hand side of~\eqref{Eq:befirst}, we subdivide $D$ further into regions that are derived from polytopal approximations to the closed, convex sets defined by $U_{f_0,t}:=\{x\in\R^d:f_0(x)\geq t\}$ for $t\geq 0$. We start by making some further definitions.
Let $\ell:=(\tilde{c}\Lambda^3\delta^2)^{\beta/(\beta+3)}$, where $\tilde{c}>1$ is the universal constant defined in Lemma~\ref{Lem:smoothlb}. Since $\delta<\bar{c}\Lambda^{-3/2}\log_+^{-1}\Lambda$ and $\bar{c}<1$, we have $\Lambda^3\delta^2<1$, so $\ell\geq\tilde{c}^{\beta/(\beta+3)}\Lambda^3\delta^2>\Lambda^3\delta^2$. Also, by~\citet[Theorem~5.14(c)]{LV06} and the proof of~\citet[Corollary~3(b)]{KS16}, we have $\inf_{h\in\tilde{\mathcal{F}}^{1,\eta}}\sup_{x\in\R^d}h(x)>(1+\eta)^{-3/2}\,(4e\pi)^{-3/2}=:t_0$, and note that $k_0:=\floor{\log_2(t_0\,\inv{\ell})}\lesssim\log(\Lambda^{-3}\delta^{-2})$.

Now for $k\in\{0,\dotsc,k_0\}$, define $U_k:=U_{f_0,\,2^k\ell}$ and $r_k:=\inv{(2\Lambda)}(2^k\ell)^{1/\beta}$. Then $r_k\leq r_{k_0}\leq\inv{(2\Lambda)}t_0^{1/\beta}\leq\inv{(2\Lambda)}$ for all such $k$. Recalling the definition of $r$ and the fact that $\ell>\Lambda^3\delta^2$, we see that $D\supseteq\bar{B}(0,r)\supseteq U_{f_0,\,\Lambda^3\delta^2}\supseteq U_0\supseteq U_1\supseteq\dotsm\supseteq U_{k_0}\supseteq U_{f_0,\,t_0}\neq\emptyset$.
Also, since $f$ satisfies \eqref{Eq:Separation}, it follows that $U_{k-1}\supseteq U_k+\bar{B}(0,r_k)$ for every $k\in\{1,\dotsc,k_0\}$. 

Next, we obtain suitable approximating polytopes $P_0\supseteq\dotsm\supseteq P_{k_0-1}$. For each fixed $k\in\{1,\dotsc,k_0\}$, note that $r_k/r\leq\inv{(2\Lambda)}/\inv{(2\Lambda_0\bar{\eta})}\leq\bar{\eta}$ in view of the definition of $r$, and that $U_k\subseteq\bar{B}(0,r)$ is compact and convex. Thus, by Lemma~\ref{Lem:BI75}, there exists a polytope $P_{k-1}$ with at most $\bar{C}^*\inv{(r_k/r)}\lesssim (2^k\ell)^{-1/\beta}\Lambda\log(\Lambda^{-3}\delta^{-2})$ vertices such that $U_k\subseteq P_{k-1}\subseteq U_k+\bar{B}(0,r_k)\subseteq U_{k-1}$. 
We emphasise that the hidden constant here does not depend on $k$. In addition, let $P_{k_0}:=\emptyset$.

In the argument below, we also use the fact that if $P\subseteq Q\subseteq\R^3$ are polytopes with $p$ and $q$ vertices respectively, then there is a triangulation of $Q\setminus\Int P$ containing $\lesssim p+q$ simplices~\citep{WY00}.
We will apply this to the nested pairs $P_0\subseteq D$ and $P_k\subseteq P_{k-1}$ for $1\leq k\leq k_0$. 

First, we consider the region $D\setminus P_0$. By the facts above, $D\setminus\Int P_0$ can be triangulated into $\lesssim\ell^{-1/\beta}\Lambda\log(\Lambda^{-3}\delta^{-2})$ simplices. Also, since $U_1\subseteq P_0$, we have $f_0(x)\leq 2\ell$ for all $x\in D\setminus P_0$. Thus, by Lemma~\ref{Lem:smoothub} and the fact that $\ell\asymp (\Lambda^3\delta^2)^{\beta/(\beta+3)}$, each $f\in\tilde{\mathcal{F}}(f_0,\delta)\cap\tilde{\mathcal{F}}^{1,\eta}$ satisfies $f\lesssim\ell\log^{2\beta/(\beta+3)}(\Lambda^{-3}\delta^{-2})$ on $D\setminus P_0$. We now apply the final assertion of Proposition~\ref{Prop:bebds} together with the bound $\mu_3(D\setminus P_0)\leq\mu_3(D)\lesssim\log^3(\Lambda^{-3}\delta^{-2})$ to deduce that
\begin{align}
H_{[\,]}(2^{-1/2}\varepsilon,\tilde{\mathcal{F}}(f_0,\delta)\cap\tilde{\mathcal{F}}^{1,\eta},\dhell,D\setminus P_0)\notag\\
&\hspace{-4cm}\lesssim\ell^{-1/\beta}\Lambda\log(\Lambda^{-3}\delta^{-2})\,\frac{\ell\log^{2\beta/(\beta+3)}(\Lambda^{-3}\delta^{-2})\log^3(\Lambda^{-3}\delta^{-2})}{\varepsilon^2}\notag\\
\label{Eq:beouter}
&\hspace{-4cm}\lesssim\Lambda\,\frac{(\Lambda^3\delta^2)^{(\beta-1)/(\beta+3)}}{\varepsilon^2}\log^{6(\beta+2)/(\beta+3)}(\Lambda^{-3}\delta^{-2}).
\end{align}

Next, fix $1\leq k\leq k_0$ and consider $P_{k-1}\setminus P_k$. By the facts above, $P_{k-1}\setminus\Int P_k$ can be triangulated into $\lesssim (2^k\ell)^{-1/\beta}\Lambda\log(\Lambda^{-3}\delta^{-2})$ simplices. Moreover, since $P_{k-1}\setminus P_k\subseteq U_{k-1}\setminus U_{k+1}$, we have $2^{k-1}\ell\leq f_0(x)<2^{k+1}\ell$ for all $x\in P_{k-1}\setminus P_k$. Thus, by the choice of $\bar{c}$ at the start of the proof and the fact that $\ell\asymp (\Lambda^3\delta^2)^{\beta/(\beta+3)}$, it follows from Lemmas~\ref{Lem:smoothlb} and~\ref{Lem:smoothub} that \[2^{k-2}\ell\leq f(x)\lesssim 2^k\ell\log^{2\beta/(\beta+3)}(\Lambda^{-3}\delta^{-2})\]
whenever $f\in\tilde{\mathcal{F}}(f_0,\delta)\cap\tilde{\mathcal{F}}^{1,\eta}$ and $x\in P_{k-1}\setminus P_k$. In addition, $\mu_3(P_{k-1}\setminus P_k)\leq\mu_3(D)\lesssim\log^3(\Lambda^{-3}\delta^{-2})$, so by applying the first bound~\eqref{Eq:besimp} from Proposition~\ref{Prop:bebds}, we conclude that
\begin{align}
H_{[\,]}(\varepsilon/\sqrt{2k_0},\tilde{\mathcal{F}}(f_0,\delta)\cap\tilde{\mathcal{F}}^{1,\eta},\dhell,P_{k-1}\setminus P_k)\notag\\
&\hspace{-6cm}\lesssim (2^k\ell)^{-1/\beta}\Lambda\log(\Lambda^{-3}\delta^{-2})\rbr{\frac{(2^k\ell)^{1/2}\log^{\beta/(\beta+3)}(\Lambda^{-3}\delta^{-2})\log^{5/2}(\Lambda^{-3}\delta^{-2})}{k_0^{-1/2}\,\varepsilon}}^{3/2}\notag\\
\label{Eq:be2}
&\hspace{-6cm}\lesssim\Lambda\,\frac{(2^k\ell)^{3/4-1/\beta}}{\varepsilon^{3/2}}\,\log^{(14\beta+33)/\{2(\beta+3)\}}(\Lambda^{-3}\delta^{-2}),
\end{align}
where we emphasise again that the hidden constant here does not depend on $k$. Now for any $\alpha\in\R$, we have the simple bound $\sum_{k=0}^{k_0}\,(2^k\ell)^{-\alpha}\leq\ell^{-\alpha^+}\,k_0\lesssim\ell^{-\alpha^+}\log(\Lambda^{-3}\delta^{-2})$. Since $P_0=\bigcup_{k=1}^{\,k_0} \,(P_{k-1}\setminus P_k)$, it follows from the above that
\begin{align}
H_{[\,]}(2^{-1/2}\varepsilon,\tilde{\mathcal{F}}(f_0,\delta)\cap\tilde{\mathcal{F}}^{1,\eta},\dhell,P_0)&\leq\sum_{k=1}^{k_0}\,H_{[\,]}(\varepsilon/\sqrt{2k_0},\tilde{\mathcal{F}}(f_0,\delta)\cap\tilde{\mathcal{F}}^{1,\eta},\dhell,P_{k-1}\setminus P_k)\notag\\
\label{Eq:be3}
&\lesssim\Lambda\,\frac{(\Lambda^3\delta^2)^{-\frac{(4-3\beta)^+}{4(\beta+3)}}}{\varepsilon^{3/2}}\,\log^{(16\beta+39)/\{2(\beta+3)\}}(\Lambda^{-3}\delta^{-2}).
\end{align}
Finally, since 
\begin{align*}
H_{[\,]}(\varepsilon,\tilde{\mathcal{F}}(f_0,\delta)\cap\tilde{\mathcal{F}}^{1,\eta},\dhell,D)&\leq H_{[\,]}(2^{-1/2}\varepsilon,\tilde{\mathcal{F}}(f_0,\delta)\cap\tilde{\mathcal{F}}^{1,\eta},\dhell,D\setminus P_0)\\
&\hspace{1.5cm}+H_{[\,]}(2^{-1/2}\varepsilon,\tilde{\mathcal{F}}(f_0,\delta)\cap\tilde{\mathcal{F}}^{1,\eta},\dhell,P_0),
\end{align*}
we can combine the bounds~\eqref{Eq:befirst},~\eqref{Eq:be1},~\eqref{Eq:beouter} and~\eqref{Eq:be3} to obtain the desired local bracketing entropy bound~\eqref{Eq:SmoothEntropy}.
\end{proof}
Theorem~\ref{Thm:Smoothness} now follows from Proposition~\ref{Prop:SmoothEntropy} and standard empirical process theory in much the same way that Theorem~\ref{Thm:k1} follows from Proposition~\ref{Prop:Log1affEntropy}.
\begin{proof}[Proof of Theorem~\ref{Thm:Smoothness}]
Fix $f_0\in\mathcal{F}_3$, $\beta\geq 1$ and $\Lambda\geq\Lambda_0$, and let
$\tilde{\Delta}:=\inf_{f \in \mathcal{F}_3^{(\beta,\Lambda)}}\dhell^2(f_0,f)$.
Defining the universal constant $\bar{c}\in (0,1)$ as in Proposition~\ref{Prop:SmoothEntropy}, we first suppose that $\tilde{\Delta}<\inv{2}\{\inv{e}\wedge(\bar{c}\Lambda^{-3/2}\log_+^{-1}\Lambda)\}=:2^{-1}\tilde{\Lambda}$. If $\delta\in (0,\tilde{\Lambda}-\tilde{\Delta})$, then by analogy with the derivation of~\eqref{Eq:Log1affEntropy2G} in the proof of Theorem~\ref{Thm:k1}, we deduce from Proposition~\ref{Prop:SmoothEntropy} and the triangle inequality that
\begin{align}
H_{[\,]}(2^{1/2}\varepsilon,\tilde{\mathcal{F}}(f_0,\delta)\cap\tilde{\mathcal{F}}^{1,\eta},\dhell)&\lesssim \Lambda\,\Biggl\{\frac{\bigl\{\Lambda^3(\delta+\tilde{\Delta})^2\bigr\}^{\alpha}}{\varepsilon^2}\log^\gamma(\Lambda^{-3}\delta^{-2})\notag\\
\label{Eq:SmoothEntropyG}
&\hspace{3.2cm}+\frac{\bigl\{\Lambda^3(\delta+\tilde{\Delta})^2\bigr\}^{-\tilde{\alpha}}}{\varepsilon^{3/2}}\,\log^{\tilde{\gamma}}(\Lambda^{-3}\delta^{-2})\Biggr\},
\end{align}
where we set $\alpha:=(\beta-1)/(\beta+3)$, $\gamma:=6(\beta+2)/(\beta+3)$, $\tilde{\alpha}:=-(4-3\beta)^+/\{4(\beta+3)\}$ and $\tilde{\gamma}:=(16\beta+39)/\{2(\beta+3)\}$. Since $\Lambda^{-3}\leq\Lambda_0^{-3}\lesssim 1$ and $1+\gamma/2\leq\tilde{\gamma}/2$, it follows from~\eqref{Eq:SmoothEntropyG} that
\begin{align*}
&\hspace{-3pt}\frac{1}{\delta^2}\int_{\delta^2/2^{13}}^\delta H_{[\,]}^{1/2 }(2^{1/2}\varepsilon,\tilde{\mathcal{F}}(f_0,\delta)\cap\tilde{\mathcal{F}}^{1,\eta},\dhell)\,d\varepsilon\\
&\hspace{0.4cm}\lesssim\frac{\Lambda^{1/2}}{\delta^2}\bigl\{\Lambda^3(\delta+\tilde{\Delta})^2\bigr\}^{\alpha/2}\log^{\gamma/2}(\Lambda^{-3}\delta^{-2})\log(1/\delta)+\frac{\Lambda^{1/2}}{\delta^{7/4}}\bigl\{\Lambda^3(\delta+\tilde{\Delta})^2\bigr\}^{-\tilde{\alpha}/2}\,\log^{\tilde{\gamma}/2}(\Lambda^{-3}\delta^{-2})\\
&\hspace{0.4cm}\lesssim\log^{\tilde{\gamma}/2}(1/\delta)\rbr{\frac{\Lambda^{1/2}}{\delta^2}\bigl\{\Lambda^3(\delta+\tilde{\Delta})^2\bigr\}^{\alpha/2}+\frac{\Lambda^{1/2}}{\delta^{7/4}}\bigl\{\Lambda^3(\delta+\tilde{\Delta})^2\bigr\}^{-\tilde{\alpha}/2}}=:\Phi_1(\delta)
\end{align*}
for all $\delta\in (0,\tilde{\Lambda}-\tilde{\Delta})$. Setting $\tilde{\delta}:=\Lambda^3\delta^2$, observe that since  \[r_\beta^{-1}=(2-\alpha)\vee(\tilde{\alpha}+7/4)=\left\{\begin{array}{ll}\!2-\alpha=(\beta+7)/(\beta+3)&\quad\mbox{if $\alpha<1/4$}\\ \!7/4&\quad\mbox{if $\alpha\geq 1/4$}\end{array}\right.\]
and $\Lambda^{25/8}\leq\Lambda_0^{-3/8}\Lambda^{7/2}$, we have
\begin{align}
\Phi_1(\delta)&\leq 2\log^{\tilde{\gamma}/2}(\Lambda^{3/2}\tilde{\delta}^{-1})\bigl(\Lambda^{7/2}\,\tilde{\delta}^{\alpha-2}+\Lambda^{25/8}\,\tilde{\delta}^{-(\tilde{\alpha}+7/4)}\bigr)\notag\\
\label{Eq:Phi1}
&\leq 2\bigl(1+\Lambda_0^{-3/8}\bigr)\Lambda^{7/2}\,\tilde{\delta}^{-1/r_\beta}\log^{\tilde{\gamma}/2}(\Lambda^{3/2}\tilde{\delta}^{-1})
\end{align}
for all $\delta\in (\tilde{\Delta},\tilde{\Lambda}-\tilde{\Delta})$. On the other hand, if $\delta\geq\tilde{\Lambda}-\tilde{\Delta}>\tilde{\Lambda}/2$, then by~\citet[Theorem~4]{KS16}, we have
\[H_{[\,]}(2^{1/2}\varepsilon,\tilde{\mathcal{F}}(f_0,\delta)\cap\tilde{\mathcal{F}}^{1,\eta},\dhell)\leq H_{[\,]}(2^{1/2}\varepsilon,\tilde{\mathcal{F}}^{1,\eta},\dhell)\lesssim\frac{1}{\varepsilon^2}\lesssim\frac{(\delta+\tilde{\Delta})^2}{\tilde{\Lambda}^2\,\varepsilon^2}\lesssim\frac{(\delta+\tilde{\Delta})^2\,\Lambda^3\log_+^2\Lambda}{\varepsilon^2},\]
so
\[\frac{1}{\delta^2}\int_{\delta^2/2^{13}}^\delta H_{[\,]}^{1/2 }(2^{1/2}\varepsilon,\tilde{\mathcal{F}}(f_0,\delta)\cap\tilde{\mathcal{F}}^{1,\eta},\dhell)\,d\varepsilon\lesssim\frac{\delta+\tilde{\Delta}}{\delta^2}\log_+\!\rbr{\frac{1}{\delta}}\Lambda^{3/2}\log_+^2\Lambda=:\Phi_2(\delta)\]
for all $\delta\geq\tilde{\Lambda}-\tilde{\Delta}$. It is straightforward to verify that $\Phi_1$ and $\Phi_2$ are decreasing functions of $\delta$. Moreover, since $\tilde{\Lambda}/2<\tilde{\Lambda}-\tilde{\Delta}\leq\tilde{\Lambda}$, we see that $\Phi_1(\delta)\gtrsim\Lambda^{7/2}\log_+^{2+\tilde{\gamma}/2}\Lambda$ for all $\delta\in (0,\tilde{\Lambda}-\tilde{\Delta})$ and $\Phi_2(\delta)\lesssim\Lambda^3\log_+^3\Lambda$ for all $\delta\geq\tilde{\Lambda}-\tilde{\Delta}$. Consequently, there exist universal constants $\tilde{C}',\tilde{C}''>0$ such that the function $\Psi\colon (0,\infty)\to (0,\infty)$ defined by \[\Psi(\delta):=\left\{\begin{array}{ll} \!\!\tilde{C}'\delta^2\,\Phi_1(\delta) & \mbox{\quad if $\delta\in (0,\tilde{\Lambda}-\tilde{\Delta})$}\\
\!\!\tilde{C}''\delta^2\,\Phi_2(\delta) & \mbox{\quad if $\delta\geq\tilde{\Lambda}-\tilde{\Delta}$}\end{array}\right.\]
satisfies $\Psi(\delta)\geq\delta\vee\int_{\delta^2/2^{13}}^\delta H_{[\,]}^{1/2}(\varepsilon,\tilde{\mathcal{F}}(f_0,\delta)\cap\tilde{\mathcal{F}}^{1,\eta},\dhell)\,d\varepsilon$ for all $\delta>0$ and has the property that $\delta\mapsto\delta^{-2}\,\Psi(\delta)$ is decreasing. Next, let $\tilde{C}:=\Lambda_0^{-1}(1\vee\Lambda_0^{1/2})\vee\bigl\{2^{20}\tilde{C}'\bigl(1+\Lambda_0^{-3/8}\bigr)\bigr\}^{2r_\beta}$, which satisfies $\tilde{C}\Lambda^{7r_\beta-3}\geq\tilde{C}\Lambda_0^{7r_\beta-3}\geq 1$ in view of the fact that $r_\beta\in (1/2,4/7]$, and define \[\delta_n:=\bigl(\tilde{C}\Lambda^{7r_\beta-3} n^{-r_\beta}\log^{\tilde{\gamma}r_\beta}n+\tilde{\Delta}^2\bigr)^{1/2}.\] 
It is straightforward to verify that there exists a universal constant $\tilde{K}>1$ such that $\bar{\delta}_n:=(\tilde{C}\Lambda^{7r_\beta-3}n^{-r_\beta}\log^{\tilde{\gamma}r_\beta}n)^{1/2}\leq \tilde{\Lambda}/2<\tilde{\Lambda}-\tilde{\Delta}$ for all $n\geq\ceil{\tilde{K}\Lambda^8}$. Since $\delta_n>\tilde{\Delta}$ and $\log(1/\bar{\delta}_n)\leq\log n$, it follows from~\eqref{Eq:Phi1} that for all $n\geq\tilde{K}\Lambda^8$, we have 
\begin{align*}
\delta_n^{-2}\,\Psi(\delta_n)&\leq\bar{\delta}_n^{-2}\,\Psi(\bar{\delta}_n)=\tilde{C}'\,\Phi_1(\bar{\delta}_n)\\
&\leq 2\tilde{C}'\bigl(1+\Lambda_0^{-3/8}\bigr)\Lambda^{7/2}(\tilde{C}^{-1}\Lambda^{-7r_\beta} n^{r_\beta}\log^{-\tilde{\gamma}r_\beta}n)^{1/(2r_\beta)}\log^{\tilde{\gamma}/2}n\\
&\leq 2^{-19}n^{1/2}.
\end{align*}
Thus, arguing as in the proof of Theorem~\ref{Thm:k1} and recalling the derivation of~\eqref{Eq:vdGBound} in particular, we can now apply~\citet[Corollary~7.5]{vdG00} and Lemma~\ref{Lem:dexbd} to conclude that there exists a universal constant $\tilde{C}^*>0$ such that \[\E\{\dex^2(\hat{f}_n,f_0)\}\leq\tilde{C}^*\Lambda^{7r_\beta-3} n^{-r_\beta}\log^{\tilde{\gamma}r_\beta}n+\tilde{\Delta}^2\]
for all $n\geq\tilde{K}\Lambda^8$, provided that $\tilde{\Delta}<\tilde{\Lambda}/2$. 

Suppose on the other hand that $\tilde{\Delta}\geq\tilde{\Lambda}/2$. By Theorem~\ref{Thm:WorstCaseRates}, a small modification of~\citet[Theorem~5]{KS16}, there exists a universal constant $C'>0$ such that $\E\{\dex^2(\hat{f}_n,f_0)\}\leq C'n^{-1/2}\log n$, and observe that there exists a universal constant $K'\geq\tilde{K}$ such that if  $n\geq K'\Lambda^8$ then 
\[\E\{\dex^2(\hat{f}_n,f_0)\}\leq C'n^{-1/2}\log n\leq\{ e^{-2}\wedge(\bar{c}^2\Lambda^{-3}\log_+^{-2}\Lambda)\}/4=\tilde{\Lambda}^2/4\leq\tilde{\Delta}^2.\]
Otherwise, if $4\leq n<K'\Lambda^8$ and $\tilde{\Delta}\geq 0$, then since $8(r_\beta-1/2)\leq 7r_\beta-3$, it again follows from Theorem~\ref{Thm:WorstCaseRates} that 
\[\E\{\dex^2(\hat{f}_n,f_0)\}\lesssim n^{-1/2}\log n\lesssim\Lambda^{7r_\beta-3}\,n^{-r_\beta}\log^{\tilde{\gamma}r_\beta}n.\]
This completes the proof of the oracle inequality~\eqref{Eq:Smoothness}.
\end{proof}

\section*{Acknowledgements}

The authors would like to thank the Isaac Newton Institute for Mathematical Sciences for support and hospitality during the programme `Statistical Scalability' when work on this paper was undertaken. This work was supported by EPSRC grant number EP/R014604/1.  The first author is grateful to Adam P. Goucher for helpful conversations. Finally, we thank the anonymous reviewers for their constructive comments, which led to several improvements in the paper.

\clearpage

\setcounter{section}{0}
\setcounter{equation}{0}
\setcounter{theorem}{0}
\def\theequation{S\arabic{equation}}
\def\thesection{S\arabic{section}}
\def\thetheorem{S\arabic{theorem}}
\def\thefigure{S\arabic{figure}}

\begin{center}
\Large{Supplementary material for `Adaptation in multivariate log-concave density estimation'} \\ \vspace{0.2in}
\large{Oliver Y. Feng, Adityanand Guntuboyina, Arlene K. H. Kim \\
  and Richard J. Samworth} 
\end{center}

This is the supplementary material for~\citet{FGKS18}, hereafter referred to as the main text.
\addtocontents{toc}{\protect\setcounter{tocdepth}{2}}
\tableofcontents
\vspace{0.5cm}
Section~\ref{Subsec:LocalBE} contains the supporting results that are most directly relevant to the proofs of the main theorems in Sections~\ref{Sec:Logkaffine} and~\ref{Sec:ThetaPolytope}. Much of the groundwork for Section~\ref{Subsec:LocalBE} is laid in Section~\ref{Subsec:EntropyCalcs}, where many of the technical lemmas have a strong geometric flavour. The structural results in Section~\ref{Subsec:LogkaffineDensities} are rooted in convex analysis and underpin many of the key definitions and calculations in Sections~\ref{Sec:Logkaffine} and~\ref{Sec:LogkaffineProofs}. Other supplementary results of a more statistical nature may be found in Sections~\ref{Subsec:MaxBound} and~\ref{Subsec:Envelope} (such the envelope result stated as Proposition~\ref{Prop:Envelope}, which may be of interest in its own right). 

The technical preparation for the proof of the main result in Section~\ref{Sec:Smoothness} is carried out in Section~\ref{Subsec:SmoothnessSupp}, in which the key auxiliary results play a similar role to those in Section~\ref{Subsec:EntropyCalcs}. Two of the examples in Section~\ref{Sec:Smoothness} draw on the background material in Section~\ref{Subsec:Holder}, which develops a notion of affine invariant smoothness and reviews some existing results on nonparametric density estimation over H\"older classes.
\section{Supplementary proofs for Sections~\ref{Sec:LogkaffineProofs} and~\ref{Subsec:ThetaProofs}}
\label{Sec:StatSupp}

\subsection{Tail bounds for \texorpdfstring{$\dex^2$}{squared d\_X} divergence and their consequences}
\label{Subsec:MaxBound}

\begin{lemma}
\label{Lem:dexbd}	
Fix $d\in\N$. Let $X_1,\dotsc,X_n\iid f_0\in\mathcal{F}_d$ with $n\geq d+1$ and let $\hat{f}_n$ denote the corresponding log-concave maximum likelihood estimator. Then
\begin{equation}
\label{Eq:exdex}
\E\cbr{\sup_{x\in\R^d}\log\hat{f}_n(x)+\max_{i=1,\dotsc,n}\log\frac{1}{f_0(X_i)}}\lesssim_d\log n,
\end{equation}
and
\begin{equation}
\label{Eq:tailprdex}
\int_{8d\log n}^\infty\Pr\bigl\{\dex^2(\hat{f}_n,f_0)\geq t\bigr\}\,dt\,\lesssim_d n^{-3}.
\end{equation}
\end{lemma}
\begin{proof}
The case $d=1$ of this result was proved in Lemma~2 in the online supplement to~\citet{KGS18}, so suppose now that $d\geq 2$. Let $T\colon\R^d\to\R^d$ be the invertible affine transformation defined by $T(x):=\Sigma_{f_0}^{1/2}x+\mu_{f_0}$ for all $x\in\R^d$, and let $Y_i:=\inv{T}(X_i)$ for all $i=1,\dotsc,n$. Setting $g_0(x):=f_0(T(x))\det^{1/2}\Sigma_{f_0}$ for all $x\in\R^d$, we have $Y_1,\dotsc,Y_n\iid g_0\in\mathcal{F}_d^{0,I}$, and
\[\max_{i=1,\dotsc,n}\log\frac{1}{g_0(Y_i)}=\max_{i=1,\dotsc,n}\log\frac{1}{f_0(X_i)}-\log\det^{1/2}\Sigma_{f_0}.\]
Also, by the affine equivariance of the log-concave maximum likelihood estimator \citep[Remark~2.4]{DSS11}, the corresponding $\hat{g}_n$ based on $Y_1,\dotsc,Y_n$ is given by $\hat{g}_n(x):=\hat{f}_n(T(x))\det^{1/2}\Sigma_{f_0}$ for all $x\in\R^d$, so
\[\sup_{x\in\R^d}\log\hat{g}_n(x)=\sup_{x\in\R^d}\log\hat{f}_n(x)+\log\det^{1/2}\Sigma_{f_0}.\]
It follows that \[\sup_{x\in\R^d}\log\hat{g}_n(x)+\max_{i=1,\dotsc,n}\log\frac{1}{g_0(Y_i)}=\sup_{x\in\R^d}\log\hat{f}_n(x)+\max_{i=1,\dotsc,n}\log\frac{1}{f_0(X_i)},\]
and a similar argument shows that $\dex^2(\hat{f}_n,f_0)=\dex^2(\hat{g}_n,g_0)$, i.e.\ that $\dex^2$ is affine invariant. Therefore, for the purposes of establishing~\eqref{Eq:exdex} and~\eqref{Eq:tailprdex}, there is no loss of generality in assuming henceforth that $f_0\in\mathcal{F}_d^{0,I}$.

%
To begin with, we control the first term on the left-hand side of~\eqref{Eq:exdex}. Write $\hat{\Sigma}_n:= n^{-1}\sum_{i=1}^n (X_i-\bar{X})(X_i-\bar{X})^\top$ for the sample covariance matrix, where $\bar{X}:=n^{-1}\sum_{i=1}^n X_i$, and let $\lambda_{\min}(\hat{\Sigma}_n)$ and $\lambda_{\max}(\hat{\Sigma}_n)$ denote the smallest and largest eigenvalues of $\hat{\Sigma}_n$ respectively. By the affine equivariance of the log-concave maximum likelihood estimator, together with straightforward modifications of arguments in the proof of Lemma~2 in the online supplement to~\citet{KGS18}, there exists $\tilde{C}_d>0$, depending only on $d$, such that for every $t>0$,
\[\Pr\rbr{\sup_{x \in\R^d} \log\hat{f}_n(x)>\frac{t}{2} \log n}\leq\Pr\bigl(\det\hat{\Sigma}_n\leq \tilde{C}_dn^{-t/2}\bigr)\leq \Pr\bigl(\lambda_{\min}(\hat{\Sigma}_n) \leq \tilde{C}_d^{1/d}n^{-t/(2d)}\bigr).\]
To handle the final expression above, we now seek an upper bound on $\Pr\bigl(\lambda_{\min}(\hat{\Sigma}_n)\leq s\bigr)$ for each $s>0$. To this end, let $\mathcal{N}\equiv\mathcal{N}(s^2/2) \subseteq S^{d-1}$ be an $(s^2/2)$-net of $S^{d-1}:=\{x\in\R^d:\norm{x}\leq 1\}$ of cardinality $K_d\,s^{-2(d-1)}$, say, where $K_d>0$ depends only on $d$. Then for each $u\in S^{d-1}$, there exists $\tilde{u}\in\mathcal{N}$ such that $\norm{u-\tilde{u}}\leq s^2/2$, whence
\begin{equation}
\label{Eq:net}
-s^2\lambda_{\max}(\hat{\Sigma}_n)\leq u^\top \hat{\Sigma}_nu-\tilde{u}^\top \hat{\Sigma}_n\tilde{u}=\tilde{u}^\top \hat{\Sigma}_n (u-\tilde{u}) + (u - \tilde{u})^\top \hat{\Sigma}_n u\leq s^2\lambda_{\max}(\hat{\Sigma}_n).
\end{equation}
In particular, if $\mathcal{N}'$ is a (1/4)-net of $S^{d-1}$ of cardinality $\tilde{K_d}:=2^{d-1}K_d$, then setting $s=1/\sqrt{2}$ and taking $u\in S^{d-1}$ to be a unit eigenvector of $\hat{\Sigma}_n$ with corresponding eigenvalue $\lambda_{\max}(\hat{\Sigma}_n)$, we deduce from~\eqref{Eq:net} that
\vspace{-0.15cm} 
\begin{equation}
\label{Eq:maxeval}
\max_{u\in\mathcal{N}'}\tm{u}\hat{\Sigma}_nu\geq\frac{1}{2}\max_{u\in S^{d-1}}\tm{u}\hat{\Sigma}_nu=\frac{1}{2}\,\lambda_{\max}(\hat{\Sigma}_n).
\end{equation}
Next, fix $\tilde{u}\in S^{d-1}$ and let $Y:=(Y_1,\dotsc,Y_n)$, where $Y_i:=\tilde{u}^\top X_i$ for $i=1,\dotsc,n$. Also, let $Q\in\R^{n\times n}$ be an orthogonal matrix such that $Q_{nj}=n^{-1/2}$ for all $j=1,\dotsc,n$, and define $Z:=QY$ and $W:=(Z_1,\dotsc,Z_{n-1})$. Then $Y$ has an isotropic log-concave density, so the same is true of $Z$ and $W$. Writing $f_W$ for the density of $W$, we deduce from~\citet[Theorem~5.14(e)]{LV06} that $f_W\leq \{2^{16}(n-1)\}^{(n-1)/2}$. Moreover, setting $\bar{Y}:=\tilde{u}^\top\bar{X}$, we have
\[n(\tm{\tilde{u}}\hat{\Sigma}_n\tilde{u})=\sum_{i=1}^n (Y_i-\bar{Y})^2=\norm{Y}^2-n\bar{Y}^2=\norm{Z}^2-Z_n^2=\norm{W}^2,\]
Thus, for all $a>0$, it follows that
\begin{align}
\Pr\bigl(\tilde{u}^\top\hat{\Sigma}_n\tilde{u}\leq a\bigr)&= \Pr\bigl(\norm{W}^2\leq na\bigr)=\int_{\bar{B}(0,n^{1/2}a^{1/2})}\!f_W(w)\,dw\notag\\
\label{Eq:weval}
&\leq\{2^{16}(n-1)\}^{(n-1)/2}\,\mu_{n-1}(\bar{B}(0,n^{1/2}a^{1/2}))\leq\frac{(2^{16}\pi n^2a)^{(n-1)/2}}{\Gamma\bigl((n+1)/2\bigr)},
\end{align}
where we have used the fact that $\mu_{n-1}(\bar{B}(0,r))=(\pi^{1/2}r)^{n-1}/\Gamma\bigl((n+1)/2\bigr)$ for all $r>0$. Furthermore, by~\citet[Theorem~1.1]{GM11}, there exist universal constants $C,c>0$ such that for all $b>0$, we have 
\begin{align}
\Pr\bigl(\tilde{u}^\top\hat{\Sigma}_n\tilde{u}>b\bigr)&\leq 
\Pr\rbr{\sum_{i=1}^n\,(\tilde{u}^\top X_i)^2/n > b}\notag\\
\label{Eq:aeval}
&\leq C \exp\rbr{-cn^{1/2}\min\cbr{n^{3/2}(b^{1/2}-1)^3,n^{1/2}(b^{1/2}-1)}},
\end{align}
which is at most $C\exp(-cn(b^{1/2}-1))\leq C\exp(-cnb^{1/2}/2)$ when $b\geq 4$. Now let $s:=\tilde{C}_d^{1/d}n^{-t/(2d)}$, and for this value of $s$, let $\mathcal{N}$ be an $(s^2/2)$-net of $S^{d-1}$. Then taking $a=2s$ and $b=(2s)^{-1}$ in~\eqref{Eq:weval} and~\eqref{Eq:aeval} respectively, we deduce that
\begin{align}
\hspace{-0.2cm}\Pr\rbr{\sup_{x\in\R^d} \log\hat{f}_n(x)>\frac{t}{2} \log n}\notag\\
&\hspace{-4.7cm}\leq\Pr\bigl(\lambda_{\min}(\hat{\Sigma}_n)\leq s\bigr)=\Pr\rbr{\inf_{u\in S^{d-1}} u^\top\hat{\Sigma}_n u \leq s}\notag\\
&\hspace{-4.7cm}\leq\Pr\rbr{\min_{u\in\mathcal{N}} u^\top\hat{\Sigma}_n u - s^2\lambda_{\max}(\hat{\Sigma}_n) \leq s}\leq\Pr\rbr{\min_{u\in\mathcal{N}} u^\top\hat{\Sigma}_n u \leq 2s}
+\Pr\bigl(\lambda_{\max}(\hat{\Sigma}_n)>s^{-1}\bigr)\notag\\
&\hspace{-4.7cm}\leq\Pr\rbr{\min_{u\in\mathcal{N}} u^\top\hat{\Sigma}_n u \leq 2s}+\Pr\rbr{\max_{u\in\mathcal{N}'} u^\top\hat{\Sigma}_nu>(2s)^{-1}}\notag\\
\label{Eq:errterm1}
&\hspace{-4.7cm}\leq K_d\,\tilde{C}_d^{-2(d-1)/d}n^{t(d-1)/d}\,\frac{(2^{16}\pi)^{(n-1)/2}}{\Gamma\bigl((n+1)/2\bigr)}n^{n-1}\bigl(2\tilde{C}_d^{1/d}n^{-t/(2d)}\bigr)^{(n-1)/2}\\
\label{Eq:errterm2}
&\hspace{-4.2cm}+\tilde{K}_d\,C \exp\rbr{-cn^{1/2}\min\cbr{n^{3/2}(\tilde{C}_d^{-1/(2d)}n^{t/(4d)}-1)^3,n^{1/2}(\tilde{C}_d^{-1/(2d)}n^{t/(4d)}-1)}},\\
&\hspace{-4.7cm}=:R_1+R_2\notag,
\end{align}
where the second, fourth and fifth inequalities follow from~\eqref{Eq:net},~\eqref{Eq:maxeval} and a union bound respectively. We now consider $R_1$ and $R_2$ separately. Setting $\zeta_d:=(d+1)\vee 2^{17}\tilde{C}_d^{1/d}\pi$, note that we can find $n_d>\zeta_d$ depending only on $d$ such that 
\begin{equation}
\label{Eq:errtermint1}
\zeta_d^{(n-1)/2}\,n^{n-1}\int_{8d}^\infty n^{t\,(\frac{d-1}{d}-\frac{n-1}{4d})}\,dt\leq \zeta_d^{(n-1)/2}\,n^{n-1}\,n^{8d\,(\frac{d-1}{d}-\frac{n-1}{4d})}\leq n^{-n/2}
\end{equation}
for all $n\geq n_d$. This takes care of $R_1$.
Also, $\tilde{C}_d^{-1/d}n^{t/(2d)}\geq 4$ whenever $t\geq 4d$ and $n\geq n_d>2\tilde{C}_d^{1/(2d)}$, so $R_2\lesssim_d\exp(-c_d'n^{t/(4d)})$ for all such $t$ and $n$, where $c_d'>0$ depends only on $d$. But since 
\begin{equation}
\label{Eq:errtermint2}
\int_{8d}^\infty\exp(-c_d'n^{t/(4d)})\,dt=\frac{4d}{\log n}\int_{n^2}^\infty\inv{s}\,e^{-c_d's}\,ds\lesssim_d e^{-c_d'n^2},
\end{equation}
it follows from~\eqref{Eq:errterm1},~\eqref{Eq:errterm2},~\eqref{Eq:errtermint1} and~\eqref{Eq:errtermint2} that there exists $n_d'>0$ depending only on $d$ such that
\begin{equation}
\label{Eq:tailprbd1}
\int_{8d}^\infty\Pr\rbr{\sup_{x\in\R^d}\log\hat{f}_n(x)>\frac{t}{2}\log n}dt\lesssim_d n^{-n/2}
\end{equation}
for all $n\geq n_d'$. We deduce that 
\begin{equation}
\label{Eq:exdexbd1}
\E\cbr{\sup_{x\in\R^d}\log\hat{f}_n(x)}\leq\frac{\log n}{2}\cbr{8d+\int_{8d}^\infty\Pr\rbr{\,\sup_{x\in C_n}\log\hat{f}_n(x)\geq\frac{t}{2}\log n}dt}\lesssim_d\log n
\end{equation}
for all $n\geq n_d'$. By increasing the multiplicative constants to deal with smaller values of $n$ if necessary, we can ensure that~\eqref{Eq:tailprbd1} and~\eqref{Eq:exdexbd1} hold for all $n\geq d+1$.

Next, we address the second term on the left-hand side of~\eqref{Eq:exdex}. Let $X\equiv (X^1,\dotsc,X^d)\sim f_0$, and for $j=1,\dotsc,d$, let $f_{j|1:(j-1)}(\cdot\,|\,x^1,\dotsc,x^{j-1})$ denote the conditional density of $X^j$ given $(X^1,\dotsc,X^{j-1}) = (x^1,\dotsc,x^{j-1})$, where we adopt the convention that the $j=1$ case refers to the marginal density of $X^1$. By~\citet[Proposition~1(a)]{CSS10}, each of these densities is then log-concave. For $j=1,\dotsc,d$, let $F_{j|1:(j-1)}(\cdot\,|\,x^1,\dotsc,x^{j-1})$ denote the corresponding distribution function, and define $U^j := F_{j|1:(j-1)}(X^j\,|\,X^1,\dotsc,X^{j-1})$. Then $U^j\,|\,(X^1,\dotsc,X^{j-1})\sim U(0,1)$ for all $j$, so in particular each $U^j$ has a marginal $U(0,1)$ distribution. In addition, for $j=1,\dotsc,d$, let
\[I_{j|1:(j-1)}(\cdot\,|\,X^1,\dotsc,X^{j-1}) := f_{j |1:(j-1)}\bigl(F_{j |1:(j-1)}^{-1} (\cdot\,|\,X^1,\dotsc,X^{j-1})\,|\,X^1,\dotsc,X^{j-1}\bigr),\]
which by~\citet[Proposition~A.1(c)]{Bob96} is positive and concave on $(0,1)$. We now need to understand how the functions $I_{j|1:(j-1)}$ transform under affine maps. To this end, consider a real-valued random variable $Y$ with density $f$ and distribution function $F$, and let $I:=f\circ F^{-1}$. Now for $\mu\in\R$ and $\sigma>0$, let $Z:=\sigma Y+\mu$ and let $f_{\mu,\sigma}$, $F_{\mu,\sigma}$ and $I_{\mu,\sigma}$ denote the corresponding functions for $Z$. Then
\[f_{\mu,\sigma}(z) = \frac{1}{\sigma}\,f\biggl(\frac{z-\mu}{\sigma}\biggr), \quad F_{\mu,\sigma}(z) = F\biggl(\frac{z-\mu}{\sigma}\biggr), \quad F_{\mu,\sigma}^{-1}(u) = \sigma F^{-1}(u) + \mu,\]
so that
\[I_{\mu,\sigma}(u) = f_{\mu,\sigma}\bigl(F_{\mu,\sigma}^{-1}(u)\bigr) = \frac{1}{\sigma}f\bigl(F^{-1}(u)\bigr).\]
It follows from this and Lemma~3 in the online supplement to~\citet{KGS18} that there exists a universal constant $\alpha>0$ such that for every $u\in (0,1)$ and $j\in\{1,\dotsc,d\}$, we have
\[I_{j|1:(j-1)}(u\,|\,X^1,\dotsc,X^{j-1})\geq\frac{\alpha}{\Var^{1/2}(X^j\,|\,X^1,\dotsc,X^{j-1})}\min(u,1-u).\]
Therefore, if $X\sim f_0$, then for all $t>0$, we have
\begin{align*}
\Pr\biggl(\log\frac{1}{f_0(X)} \geq t\biggr)
&=\Pr\big(f_0(X) \leq e^{-t}\bigr)
=\Pr\rbr{\textstyle\prod_{j=1}^d\, f_{j|1:(j-1)}(X^j\,|\,X^1,\dotsc,X^{j-1})\leq e^{-t}}\\
&=\Pr\rbr{\textstyle\prod_{j=1}^d\,I_{j|1:(j-1)}(U^j\,|\,X^1,\dotsc,X^{j-1})\leq e^{-t}}\\
&\leq\Pr\rbr{\min_{j=1,\dotsc,d}\,I_{j|1:(j-1)}(U^j\,|\,X^1,\dotsc,X^{j-1}) \leq e^{-t/d}}\\
&\leq\Pr\rbr{\min_{j=1,\dotsc,d}\,\frac{\alpha}{\Var^{1/2}(X^j\,|\,X^1,\dotsc,X^{j-1})}\min(U^j,1-U^j) \leq e^{-t/d}}.
\end{align*}
Also, the function $h(x):=e\vee\exp(\sqrt{x}/2)$ is convex and increasing on $[0,\infty)$. By applying Proposition~\ref{Prop:Envelope}(iii) to the density of $X^j$, which lies in $\mathcal{F}_1^{0,I}$, we find that
\begin{align*}
\E\bigl\{e^{\Var^{1/2}(X^j\,|\,X^1,\dotsc,X^{j-1})/2}\bigr\}
&\leq\E\bigl\{e^{\E^{1/2}((X^j)^2 \,|\,X^1,\dotsc,X^{j-1})/2}\}\\
&\leq\E\bigl\{h\bigl(\E\bigl\{\!(X^j)^2\,|\,X^1,\dotsc,X^{j-1}\bigr\}\bigr)\bigr\}\\
&\leq\E\bigl\{e\vee\exp\,(|X^j|/2)\bigr\}
\leq e+2\int_0^\infty e^{-x/2+1}\,dx=5e,
\end{align*}
where we have used Jensen's inequality to obtain the penultimate bound. Defining the event $B:=\bigl\{\max_{j=1,\dotsc,d}\Var^{1/2}(X^j\,|\,X^1,\dotsc,X^{j-1})\leq t\bigr\}$, we deduce that if $n\geq d+1$ and $X_1,\dotsc,X_n\iid f_0$, then
\begin{align*}
&\Pr\rbr{\min_{j=1,\dotsc,d}\,\frac{\alpha}{\mathrm{Var}^{1/2}(X^j\,|\,X^1,\dotsc,X^{j-1})}\min(U^j,1-U^j) \leq e^{-t/d}}\\
&\leq\Pr\rbr{\cbr{\min_{j=1,\dotsc,d} \frac{\alpha\min(U^j,1-U^j)}{\Var^{1/2}(X^j\,|\,X^1,\dotsc,X^{j-1})}\!\leq e^{-t/d}}\cap B}+\Pr(\cm{B}) \\
&\leq 2\inv{\alpha}dte^{-t/d}+5de^{1-t/2}.
\end{align*}
It follows that for all $t>0$, we have
\[\Pr\rbr{\max_{i=1,\dotsc,n}\log\frac{1}{f_0(X_i)}>\frac{t}{2}\log n}
\leq n\,\Pr\rbr{\log \frac{1}{f_0(X)}>\frac{t}{2}\log n}
\leq\frac{dt\log n}{\alpha n^{t/(2d)-1}} + \frac{5ed}{n^{t/4-1}}.\]
Since $\int_{8d}^\infty\,(t\log n)\,n^{1-t/(2d)}\,dt\lesssim_d n\int_{4\log n}^\infty \,\inv{(\log n)}\,se^{-s}\,ds\lesssim_d n^{-3}$, we conclude that
\begin{equation}
\label{Eq:tailprbd2}
\int_{8d}^\infty\Pr\rbr{\max_{i=1,\dotsc,n}\log\frac{1}{f_0(X_i)}>\frac{t}{2}\log n}dt\lesssim_d n^{-3}
\end{equation}
and hence that
\begin{equation}
\label{Eq:exdexbd2}
\E\cbr{\max_{i=1,\dotsc,n}\log\frac{1}{f_0(X_i)}}\leq\frac{\log n}{2}\cbr{8d+\int_{8d}^\infty\Pr\rbr{\max_{i=1,\dotsc,n}\log\frac{1}{f_0(X_i)}>\frac{t}{2}\log n}dt}\lesssim_d\log n
\end{equation}
for all $n\geq d+1$. The required bounds~\eqref{Eq:exdex} and~\eqref{Eq:tailprdex} follow by combining~\eqref{Eq:exdexbd1} and~\eqref{Eq:exdexbd2}, and~\eqref{Eq:tailprbd1} and~\eqref{Eq:tailprbd2} respectively.
\end{proof}
Recalling that $\dhell^2(\hat{f}_n,f_0)\leq\KL(\hat{f}_n,f_0)\leq\dex^2(\hat{f}_n,f_0)$, as mentioned in the introduction, we record here that in~\citet[Theorem~5]{KS16}, the worst-case $\dhell^2$ risk bounds for $\hat{f}_n$ in dimensions $d=1,2,3$ can be strengthened to $\dex^2$ risk bounds of the same form.
This requires only a small modification to the original proof, as we now explain.
\begin{theorem}
\label{Thm:WorstCaseRates}
Let $X_1,\dotsc,X_n\iid f_0\in\mathcal{F}_d$ with $n\geq d+1$, and let $\hat{f}_n$ denote the corresponding log-concave maximum likelihood estimator. Then
\begin{equation}
\label{Eq:WorstCaseRates}
\sup_{f_0\in\mathcal{F}_d}\E\{\dex^2(\hat{f}_n,f_0)\}=
\left\{
\begin{array}{ll} O(n^{-4/5}) &\quad\mbox{if $d=1$} \\
O(n^{-2/3}\log n) &\quad\mbox{if $d=2$} \\
O(n^{-1/2}\log n) &\quad\mbox{if $d=3$.}
\end{array} 
\right.
\end{equation}
\end{theorem}
\begin{proof}
In the original proof of~\citet[Theorem~5]{KS16}, the key bracketing entropy bounds from~\citet[Theorem~4]{KS16} are converted into $\dhell^2$ risk bounds by appealing to~\citet[Theorem~7.4]{vdG00}, a result from empirical process theory that is restated as Theorem~5 in the online supplement to~\citet{KS16}. Note that  $\dex^2(\hat{f}_n,f_0)=n^{-1}\sum_{i=1}^n\log\frac{\hat{f}_n(X_i)}{f_0(X_i)}=\int\log(\hat{f}_n/f_0)\,d\Pr_n$, where $\Pr_n$ denotes the empirical measure of $X_1,\dotsc,X_n$. In view of this, we can derive~\eqref{Eq:WorstCaseRates} from~\citet[Theorem~4]{KS16} by carrying out identical calculations to those in the proof of~\citet[Theorem~5]{KS16} but instead appealing to~\citet[Corollary~7.5]{vdG00}. The latter is restated as Theorem~10 in the online supplement to~\citet{KGS18} and holds under the same conditions as those required for~\citet[Theorem~7.4]{vdG00}.
\end{proof}
\subsection{The envelope function for the class of isotropic log-concave densities on \texorpdfstring{$\R$}{R}}
\label{Subsec:Envelope}
In the proof of Lemma~\ref{Lem:dexbd}, we make use of Proposition~\ref{Prop:Envelope}, a result of independent interest that characterises the envelope function for the class $\mathcal{F}_1^{0,1}\equiv\mathcal{F}_1^{0,I}$ of all real-valued isotropic log-concave densities. Previously, it was known that $F(x):=\sup_{f\in\mathcal{F}_1^{0,1}}f(x)\leq 1$ for all $x\in\R$ \citep[Lemma~5.5(a)]{LV06} and that there exist $A>0$ and $B\in\R$ such that $F(x)\leq e^{-A\abs{x}+B}$ for all $x\in\R$ \citep[e.g.][Theorem~2(a)]{KS16}. The proposition below shows that we can in fact take $A=B=1$ and moreover that this is the optimal choice of $A$ and $B$, in the sense that the bound above does not hold for all $x\in\R$ if either $A>1$ or $A=1$ and $B<1$. Furthermore, there is a simple closed form expression for $F(x)$ when $\abs{x}\leq 1$, and it is somewhat surprising that $F(x)$ increases as $\abs{x}$ increases from 0 to 1.
\begin{proposition}
\label{Prop:Envelope}
The envelope function $F$ for $\mathcal{F}_1^{0,1}$ is even and piecewise smooth, and has the following properties:
\begin{enumerate}[label=(\roman*)]
\item $F(x)=(2-x^2)^{-1/2}$ for all $x\in (-1,1)$;
\item $F(x)\geq e^{-(x+1)}$ for all $x\geq -1$ and $e^{x+1}F(x)\to 1$ as $x\to\infty$;
\item $F(x)\leq 1\wedge e^{-\abs{x}+1}$ for all $x\in\R$.
\end{enumerate}
\end{proposition}
As a by-product of the proof, we obtain an explicit expression for $F$: see~\eqref{Eq:Envelope} below. Moreover, we will see that for each $x\in\R$, there exists a piecewise log-affine density $f_x$ that achieves the supremum in the definition of $F(x)$. This extremal distribution $f_x$ can take one of two forms depending on whether $\abs{x}<1$ or $\abs{x}\geq 1$, and we treat these cases separately. The proof relies heavily on stochastic domination arguments based on the following lemma.
\begin{lemma}
\label{Lem:stocdom}
Let $X,Y$ be real-valued random variables with densities $f,g$ and corresponding distribution functions $F,G$ respectively. Then we have the following:
\begin{enumerate}[label=(\roman*)]
\item If there exists $a\in\R$ such that $f\leq g$ on $(-\infty,a)$ and $f\geq g$ on $(a,\infty)$, then $F\leq G$, i.e.\ $X$ stochastically dominates $Y$. If in addition $X$ and $Y$ are integrable and $f,g$ differ on a set of positive Lebesgue measure, then $\E(X)>\E(Y)$.
\item Suppose that there exist $a<b$ such that $f\geq g$ on $(a,b)$ and $f\leq g$ on $(-\infty,a)\cup(b,\infty)$, and moreover that $f$ and $g$ are not equal almost everywhere when restricted to either $(-\infty,a)$ or $(b,\infty)$. Then there exists a unique $c\in\R$ with $0<F(c)=G(c)<1$ such that $F\leq G$ on $(-\infty,c)$ and $F\geq G$ on $(c,\infty)$. If in addition $X$ and $Y$ are square-integrable and $\E(X)=\E(Y)$, then $\Var(X)<\Var(Y)$.
\end{enumerate}
\end{lemma}
\begin{proof}[Proof of Lemma~\ref{Lem:stocdom}]
Note that $\Pr(X\geq t)=\int_t^\infty f(s)\,ds\geq\int_t^\infty g(s)\,ds=\Pr(Y\geq t)$ when $t\geq a$. Similarly, $\Pr(X\geq t)=1-\int_{-\infty}^t f(s)\,ds\geq 1-\int_{-\infty}^t g(s)\,ds=\Pr(Y\geq t)$ when $t\leq a$. Part (i) now follows immediately from the identity \[\E(X)=\E(X^+)-\E(X^-)=\int_0^\infty\bigl\{\Pr(X\geq s)-\Pr(X<-s)\bigr\}\,ds.\]
For (ii), $G-F$ is an (absolutely) continuous function that is increasing on the intervals $(-\infty,a)$ and $(b,\infty)$, so the first assertion is an immediate consequence of the fact that $(G-F)(a)>0>(G-F)(b)$. Also, if $W$ is a square-integrable random variable, then Fubini's theorem implies that
\begin{equation}
\label{Eq:stocdom1}
\int_0^\infty\int_s^\infty\Pr(W\geq t)\,dt\,ds=\E\rbr{\int_0^\infty\int_0^\infty\Ind_{\{W\geq t\geq s\}}\,dt\,ds}=\E\rbr{\int_0^\infty(W-s)\Ind_{\{W\geq s\}}\,ds}.
\end{equation}
Now since $\E\bigl(W\int_0^\infty\Ind_{\{W\geq s\}}\,ds\bigr)=\E\{(W^+)^2\}$ and 
\[\E\rbr{\int_0^\infty s\,\Ind_{\{W\geq s\}}\,ds}=\int_0^\infty s\,\Pr(W\geq s)\,ds=\frac{1}{2}\int_0^\infty\Pr\bigl\{(W^+)^2\geq s\bigr\}\,ds=\frac{1}{2}\,\E\{(W^+)^2\},\]
the right-hand side of~\eqref{Eq:stocdom1} is equal to $2^{-1}\,\E\{(W^+)^2\}$. By applying similar reasoning to $W^-$, we conclude that  \[\E(W^2)=\E\{(W^+)^2\}+\E\{(W^-)^2\}=2\rbr{\int_0^\infty\int_s^\infty\bigl\{\Pr(W\geq t)+\Pr(W\leq -t)\bigr\}\,dt\,ds}.\]
Since $\Pr(X\geq c+t)\leq\Pr(Y\geq c+t)$ and $\Pr(X\leq c-t)\leq\Pr(Y\leq c-t)$ for all $t\geq 0$, and since $F$ and $G$ do not agree almost everywhere by hypothesis, it follows from this that $\E\{(X-c)^2\}<\E\{(Y-c)^2\}$. This implies the second assertion in view of our assumption that $\E(X)=\E(Y)$.
\end{proof}
The proofs of parts (ii) and (iii) of Proposition~\ref{Prop:Envelope} require some additional probabilistic input. For $K\in (0,\infty]$ and a non-degenerate random variable $W$ that takes non-negative values, let $W_K$ be a random variable whose distribution is that of $W$ conditioned to lie in $[0,K)$. Letting $Y$ be a random variable with an $\Exp(1)$ distribution, we now set
\begin{align}
h(K)&:=\E(Y_K)=\frac{1-(K+1)\,e^{-K}}{1-e^{-K}}\\
V(K)&:=\Var(Y_K)=1-\frac{K^2}{2\,(\cosh K-1)}.
\end{align}
It is easily verified that $V$ is positive, strictly increasing and tends to 1 as $K\to\infty$. Also, $h$ and $V$ are smooth, so in particular, $V$ has a smooth inverse $\inv{V}\colon (0,1)\to (0,\infty)$. Moreover, for each $\lambda\in(0,1]$, let $W^\lambda$ be a random variable distributed as $\Exp(\lambda)$. We now make crucial use of the scaling property $W^{\lambda}\overset{d}{=}Y/\lambda$ of exponential random variables, a consequence of which is that we need only work with the functions $h$ and $V$. Indeed, we have $\E(W_K^\lambda)=h(\lambda K)/\lambda$ and $\Var(W_K^\lambda)=V(\lambda K)/\lambda^2$, so we can find a unique
\begin{equation}
\label{Eq:K}
K=K(\lambda):=\inv{V}(\lambda^2)/\lambda\in (0,\infty)
\end{equation}
such that $\Var(W_K^\lambda)=1$. Thus, the density $f_{\lambda}$ of $X_{\lambda}:=W_{K(\lambda)}^\lambda-\E(W_{K(\lambda)}^\lambda)$ lies in $\mathcal{F}_1^{0,1}$ and is log-affine on its support $[-m(\lambda),a(\lambda)]$, where
\begin{align}
\label{Eq:m}
m(\lambda):=\E(W_{K(\lambda)}^\lambda)&=\frac{1}{\lambda}\rbr{1-\frac{\lambda K}{e^{\lambda K}-1}},\\
\label{Eq:a}
a(\lambda):=K(\lambda)-m(\lambda)&=\frac{1}{\lambda}\rbr{\frac{\lambda Ke^{\lambda K}}{e^{\lambda K}-1}-1}
\end{align}
are smooth, non-negative functions of $\lambda$. We now show that:
\begin{lemma}
\label{Lem:maK}
The functions $m, a, K$ are smooth bijections from $(0,1)$ to $(1,\sqrt{3})$, $(\sqrt{3},\infty)$ and $(2\sqrt{3},\infty)$ respectively. Moreover, $m$ is strictly decreasing and $a,K$ are strictly increasing. We also have $(1-\lambda)K(\lambda)\to 0$ as $\lambda\nearrow 1$.
\end{lemma}
While much of the following argument relies only on elementary analysis, a little probabilistic reasoning based on Lemma~\ref{Lem:stocdom} helps to simplify the proof. 
\begin{proof}[Proof of Lemma~\ref{Lem:maK}]
The functions $K,m$ and $a$ defined in~\eqref{Eq:K},~\eqref{Eq:m} and~\eqref{Eq:a} are certainly smooth, and since $\inv{V}(y)\to\infty$ as $y\!\nearrow1$, we see that $m(\lambda)\to 1$ and $K(\lambda),\,a(\lambda)\to\infty$ when $\lambda\!\nearrow 1$. Furthermore, $s:=\inv{V}(\lambda^2)\to 0$ as $\lambda\to 0$, and since \[V(s)=1-\frac{1}{1+s^2/12+o(s^2)}=\frac{s^2}{12}+o(s^2)\]
as $s\to 0$, we have $K(\lambda)=s\,V(s)^{-1/2}\to\sqrt{12}=2\sqrt{3}$ as $\lambda\to 0$. Consequently, $m(\lambda)=K(\lambda)\,(1-s/(e^s-1))/s\to 2\sqrt{3}\times 1/2=\sqrt{3}$ as $\lambda\to 0$.

Finally, to show that $K(\lambda)(1-\lambda)\to 0$ as $\lambda\to\infty$, note that $\cosh y\,/e^y\to 1/2$ as $y\to\infty$ and moreover that for each $\delta>0$, we have $y^2e^{-y}<e^{-(1-\delta)y}$ for all sufficiently large $y$. This then implies that $V(y)>1-\exp(-y/2)$ for all sufficiently large $y$. Thus, $\inv{V}(\lambda^2)<2\log(1-\lambda^2)$ for all $\lambda$ sufficiently close to 1, so $K(\lambda)(1-\lambda)=\inv{V}(\lambda^2)(1-\lambda)/\lambda\to 0$ as $\lambda\!\nearrow 1$, as required.

It remains to show that $m$ and $a$ are strictly monotone. Fix $0<\lambda_1<\lambda_2<1$ and for $i=1,2$, let $\phi_{\lambda_i}:=\log f_{\lambda_i}$ and recall that the density $f_{\lambda_i}\in\mathcal{F}_1^{0,1}$ is supported on $[-m(\lambda_i),a(\lambda_i)]$. In addition, observe that an (infinite) straight line in $\R^2$ with slope $-\lambda_2$ intersects the graph of $\phi_{\lambda_1}$ in exactly two points in $\R^2$, one of which has $x$-coordinate $a(\lambda_1)$. Now if the graphs of $\phi_{\lambda_1},\phi_{\lambda_2}$ intersect in at most two points, then the densities $f_{\lambda_1},f_{\lambda_2}\in\mathcal{F}_1^{0,1}$ satisfy one of the two sets of hypotheses in Lemma~\ref{Lem:stocdom}. However this then implies that either the means or the variances of $f_{\lambda_1},f_{\lambda_2}$ do not match, which yields a contradiction. Therefore, it follows that $m(\lambda_2)<m(\lambda_1)$ and $a(\lambda_2)>a(\lambda_1)$.

The fact that $K$ is strictly increasing follows readily from the observation that the function $s\mapsto s^{-2}-2^{-1}\inv{(\cosh s-1)}$ is strictly decreasing on $(0,\infty)$, as can be seen by applying the simple fact Lemma~\ref{Lem:radconv} below. 
\end{proof}
\begin{lemma}
\label{Lem:radconv}
If $\sum_{n\geq 0}a_nz^n$ and $\sum_{n\geq 0} b_nz^n$ are power series with infinite radii of convergence such that $a_n,b_n>0$ for all $n\geq 0$ and $(a_n/b_n)_{n\geq 0}$ is a strictly decreasing sequence, then $z\mapsto\bigl(\sum a_nz^n\bigr)/\bigl(\sum b_nz^n\bigr)$ is a strictly decreasing function of $z\in (0,\infty)$.
\end{lemma}
\begin{proof}
Fix $0<w<z$ and let $c_n:=a_n/b_n$ for each $n$. Then for fixed $n>m\geq 0$, we have $c_n<c_m$, which implies that $c_nz^nw^m+c_mz^mw^n<c_nz^mw^n+c_mz^nw^m$. Thus, summing over all $m,n\geq 0$, we obtain the desired conclusion that 
\begin{align*}
\hspace{-0.5cm}\textstyle\bigl(\sum_{n\geq 0} a_nz^n\bigr)\bigl(\sum_{m\geq 0} b_mw^m\bigr)&=\textstyle\sum_{n,m\geq 0}c_nb_nb_mz^nw^m\\
&<\textstyle\sum_{n,m\geq 0}c_nb_nb_mz^mw^n=\bigl(\sum_{n\geq 0} a_nw^n\bigr)\bigl(\sum_{m\geq 0} b_mz^m\bigr).
\end{align*}

\vspace{-0.8cm}
\end{proof}
\begin{proof}[Proof of Proposition~\ref{Prop:Envelope}]
For a fixed $x\in (-1,1)$, we make the ansatz
\begin{equation}
\label{Eq:ans1}
f_x(w):=C\cbr{\exp\rbr{-\frac{w-x}{a_1}}\Ind_{\{w\geq x\}}+\exp\rbr{\frac{w-x}{a_2}}\Ind_{\{w<x\}}},
\end{equation}
where $C,a_1,a_2>0$ are to be determined. Note that $\log f_x$ is continuous on $\R$ and affine on the intervals $(-\infty,x]$ and $[x,\infty)$. If we are to ensure that $f_x\in\mathcal{F}_1^{0,1}$, then the parameters $C,a_1,a_2>0$ must satisfy the constraints
\begin{align}
&C(a_1+a_2)=1\\
&C(a_2^2-a_1^2)=x\\
&C\bigl\{2(a_1^3+a_2^3)+2x(a_1^2-a_2^2)+x^2(a_1+a_2)\bigr\}=1,
\end{align}
which respectively guarantee that $f_x$ integrates to 1 and has mean 0 and variance 1. The first two equations yield $a_1=(\inv{C}-x)/2$ and $a_2=(\inv{C}+x)/2$, so in particular we require $Cx<1$. After substituting these expressions into the final equation, we conclude that $1+x^2=2C(a_1^3+a_2^3)=2x^2+(C^{-2}-x^2)/2$, so $C=(2-x^2)^{-1/2}$. Since $\abs{x}<1$, it is indeed the case that $Cx<1$, so these equations uniquely determine the form of $f_x$ in terms of $x$.

Next, to show that $g(x)\leq C=f_x(x)$ for all $g\in\mathcal{F}_1^{0,1}$, we work on the logarithmic scale and suppose for a contradiction that there exists $g\in\mathcal{F}_1^{0,1}$ such that $g(x)>C$. Now since the functions $\phi:=\log g$ and $\phi_x:=\log f_x$ are concave and upper semi-continuous, it follows from the assumption $\phi(x)>\phi_x(x)$ that the graphs of $\phi$ and $\phi_x$ (viewed as subsets of $\R^2$) intersect in at most one point in each of the regions $(-\infty,x)\times\R$ and $(x,\infty)\times\R$. Note that here, we also take into account those $x'$ which satisfy $\phi(x')\geq\phi_x(x')$ and which correspond to intersection points on the boundary of the support of $g$. To obtain the required contradiction, observe that the densities $g,f_x\in\mathcal{F}_1^{0,1}$ must therefore satisfy one of the two sets of hypotheses in Lemma~\ref{Lem:stocdom}. It follows from this that either the means or the variances of $g,f_x$ do not match. This concludes the proof of part (i) of the proposition.

If instead our fixed $x\in\R$ satisfies $\abs{x}\geq 1$, then the previous system of equations does not admit suitable solutions, so we take a different approach. Suppose first that there exists a compactly supported density in $\mathcal{F}_1^{0,1}$ of the form
\begin{equation}
\label{Eq:ans2}
f_x(w):=C\exp\{-\lambda(w-x)\}\Ind_{\{w\in [-a,x]\}},
\end{equation}
for some $C>0$, $a\in(0,\infty]$ and $\lambda\in\R$. Then by appealing to Lemma~\ref{Lem:stocdom} and arguing as in the previous paragraph, it follows that $F(x)=f_x(x)$; the key observation is that if there did exist $g\in\mathcal{F}_1^{0,1}$ satisfying $g(x)>f_x(x)$, then the graphs of $g$ and $f_x$ would intersect in either one or two points (in the sense described above), and these would necessarily lie in the region $[-a,x)\times\R$. 

It therefore remains to show that, for each $x\in\R$ with $\abs{x}\geq 1$, the class $\mathcal{F}_1^{0,1}$ does indeed contain a log-affine density $f_x$ of this type (which in view of Lemma~\ref{Lem:stocdom} is clearly unique if it exists). To see this, we reparametrise the densities of the form~\eqref{Eq:ans2} in terms of $\lambda$, and then appeal to~\eqref{Eq:m},~\eqref{Eq:a}, Lemma~\ref{Lem:maK} and the probabilistic setup on which these are based.
When $x\in (1,\sqrt{3})$, we can take $f_x$ to be the density of $-X_{\lambda}$ with $\lambda=\inv{m}(x)$, and when $x\in (\sqrt{3},\infty)$, we can take $f_x$ to be the density of $X_\lambda$ with $\lambda=\inv{a}(x)$. We can then argue by symmetry to handle negative values of $x$.

Finally, we consider the borderline cases $\abs{x}=1,\sqrt{3}$. If $\abs{x}=1$, then the extremal density $f_x$ is the density of $\pm(Y-1)$ where $Y\sim\Exp(1)$: this can be realised either as the limit of~\eqref{Eq:ans1} as $\abs{x}\!\nearrow 1$ or as the limit of~\eqref{Eq:ans2} as $\abs{x}\!\searrow 1$, so (unsurprisingly) $F$ is continuous at 1. The case $\abs{x}=\sqrt{3}$ corresponds to $\lambda=0$; indeed, $f_x$ is the density of $U[-\sqrt{3},\sqrt{3}]$ in this instance.

In summary, using the fact that $f_x$ integrates to 1 for each $x$, we deduce that
\begin{equation}
\label{Eq:Envelope}
F(x)=f_x(x)=
\begin{cases}
\,\lambda\bigl(e^{\lambda K(\lambda)}-1\bigr)^{-1}&\text{with }\lambda=\inv{a}(\abs{x})\text{ when }\abs{x}>\sqrt{3}\\
\,(2\sqrt{3})^{-1}&\text{when }\abs{x}=\sqrt{3}\\
\,\lambda\bigl(1-e^{-\lambda K(\lambda)}\bigr)^{-1}&\text{with }\lambda=\inv{m}(\abs{x})\text{ when }1<\abs{x}<\sqrt{3}\\
\,(2-x^2)^{-1/2}&\text{when }\abs{x}\leq 1,\\
\end{cases}
\end{equation}
which is smooth on $(-1,1)$ and $(-\infty,-1)\cup (1,\infty)$. Also, $a\circ\inv{m}$ is a smooth and strictly decreasing function from $(1,\sqrt{3})$ to $(\sqrt{3},\infty)$, and the values of $F$ on these intervals are related by the identity
\begin{equation}
\label{Eq:duality}
F(a(\lambda))=F(m(\lambda))e^{-\lambda K(\lambda)},
\end{equation}
which holds for all $\lambda\in (0,1)$. We now return to the assertions in part (ii) of the proposition. The first of these follows immediately from the fact that the density of an $\Exp(1)-1$ random variable lies in $\mathcal{F}_1^{0,1}$. For the second, we write $\lambda=\inv{a}(\abs{x})$, and in view of~\eqref{Eq:Envelope}, it suffices to compute the ratio of $\lambda/(e^{\lambda K(\lambda)}-1)$ and $e^{-a(\lambda)-1}=e^{-K(\lambda)+m(\lambda)-1}$ as $\lambda\!\nearrow 1$. Note that $\log\lambda\to 0$ and $m(\lambda)\to 1$ as $\lambda\!\nearrow 1$. Thus, after taking logarithms, it is enough to consider the difference of $\lambda K(\lambda)$ and $K(\lambda)$, which does indeed tend to 0 by the final part of Lemma~\ref{Lem:maK}. This completes the proof of (ii), and also shows that there exists  a constant $B\in\R$ such that $F(x)\leq\exp(-\abs{x}+B)$ for all $x\in\R$.

To establish part (iii) of the result, it remains to show that we can take $B=1$. The expressions above can be made more analytically tractable by reparametrising everything in terms of $s=\inv{V}(\lambda^2)=\lambda K(\lambda)$, which is strictly increasing in $\lambda\in (0,1)$, so in a slight abuse of notation, we start by redefining
\begin{align*}
&V(s)=1-2^{-1}s^2/(\cosh s-1),\\
&K(s)=s\,{V(s)}^{-1/2},\\
&m(s)=(1-s/(e^s-1))\,{V(s)}^{-1/2},\\
&a(s)=(-1+se^s/(e^s-1))\,{V(s)}^{-1/2}
\end{align*}
as functions of $s\in (0,\infty)$. Lemma~\ref{Lem:maK} implies that all of these are strictly monotone. In view of~\eqref{Eq:Envelope}, we need to show that
\begin{align}
\label{Eq:smallx}
&{V(s)}^{1/2}\,e^{m(s)-1}\,\inv{(1-e^{-s})}\leq 1;\\
\label{Eq:largex}
&{V(s)}^{1/2}\,e^{a(s)-1}\,\inv{(e^s-1)}\leq 1
\end{align}
for all $s\in(0,\infty)$. First we address~\eqref{Eq:smallx}, which corresponds to values of $x\in (1,\sqrt{3})$. This can be verified directly by numerical calculation for $s\in (0,5]$. Indeed, for $s\in (0,3/2]$, the left-hand side can be rewritten as \[\inv{K(s)}\,e^{m(s)-1}\,\frac{s}{1-e^{-s}}.\]
For $s\in (0,1]$, this is bounded above by $\inv{(2\sqrt{3})}\,e^{\sqrt{3}-1}\,s/(1-e^{-s})\leq 1$ in view of Lemma~\ref{Lem:maK}, and for $s\in [1,3/2]$, this is at most $\inv{K(1)}\,e^{m(1)-1}\,s/(1-e^{-s})\leq 1$. Similarly, for each $k\in\{3,\ldots,9\}$, the left-hand side of~\eqref{Eq:smallx} is at most \[{V\rbr{\frac{k+1}{2}}}^{1/2}e^{m(k/2)-1}\,\inv{(1-e^{-k/2})}\]
for all $s\in [k/2,(k+1)/2]$, and all these values can be checked to be less than 1, as required.

Now we present a general argument that handles the case $s\geq 5$. We certainly have \[1-s^2\inv{(e^s-2)}\leq V(s)\leq 1-s^2\inv{(e^s-1)}\]
for all $s\geq 0$, and we claim that
\begin{equation}
\label{Eq:sqrtVbds}
1-s^2e^{-s}/2\geq\sqrt{V(s)}\geq\rbr{1-\frac{s^2}{e^s-2}}^{1/2}\geq 1-s^2e^{-s}/2-s^4e^{-2s}/2
\end{equation}
for all $s\geq 5/2$. To obtain the final inequality, observe that $(1-z)^{1/2}\geq 1-z/2-z^2/4$ for all $0\leq z\leq 2(\sqrt{2}-1)$ and that $u:=s^2\inv{(e^s-2)}\leq 2^2/(e^2-2)<2(\sqrt{2}-1)$ for $s\geq 2$, so it is enough to prove that the right-hand side above is at most $1-u/2-u^2/4$. This reduces to showing that $s^2(1-2e^{-s})-s^2(e^{-s}+\inv{(e^s-2)})-2\geq 0$ for all $s\geq 5/2$, but as the left-hand side of this final inequality is an increasing function of $s$ which is non-negative at $s=5/2$, the claim in~\eqref{Eq:sqrtVbds} follows.

The original inequality~\eqref{Eq:smallx} can be rewritten in the form \[\frac{1}{\sqrt{V(s)}}\rbr{1-\frac{s}{e^s-1}}-1\leq\log\rbr{\frac{1-e^{-s}}{\sqrt{V(s)}}},\]
and since $s^2\leq e^s$ when $s\geq 0$, we have $e^{-s}(s^2/2-1)\inv{(1-s^2e^{-s}/2)}\leq 1$. So by the bounds in~\eqref{Eq:sqrtVbds} and the fact that $\log(1+z)\geq z-z^2/2$ for $\abs{z}<1$, it is enough to show that \[\frac{1-se^{-s}}{1-s^2e^{-s}/2-s^4e^{-2s}/2}-1\leq e^{-s}\frac{s^2/2-1}{1-s^2e^{-s}/2}\rbr{1-\frac{e^{-s}(s^2/2-1)}{2\rbr{1-s^2e^{-s}/2}}}.\]
This is equivalent to showing that \[\frac{(s^2-2s)+s^4e^{-s}}{s^2-2}\leq\rbr{1-\frac{s^4e^{-2s}}{2\rbr{1-s^2e^{-s}/2}}}\rbr{1-\frac{e^{-s}(s^2/2-1)}{2\rbr{1-s^2e^{-s}/2}}}\]
for all $s\geq 5$. In fact, we will establish the slightly stronger bound \[\frac{(s^2-2s)+s^4e^{-s}}{s^2-2}\leq 1-\frac{s^4e^{-2s}+e^{-s}(s^2/2-1)}{2\rbr{1-s^2e^{-s}/2}}\]
for all $s\geq 5$. After clearing denominators and simplifying, we arrive at the equivalent inequality \[4e^s(s-1)+2s^4e^{-s}+4s^2\geq 5s^4/2+2s^3+2,\]
which certainly holds whenever $s\geq 5$. Indeed, $16e^s-5s^4/2-2s^3\geq 0$ for all $s\geq 0$, since $5s^4e^{-s}+4s^3e^{-s}\leq 5(4/e)^4+4(3/e)^3\leq 32$, so we are done.

Now that we have established~\eqref{Eq:smallx}, it is relatively straightforward to obtain~\eqref{Eq:largex}. For $s\leq 3$, we again proceed by direct calculation: here, the left-hand side equals \[\inv{K(s)}\,e^{a(s)-1}\,l(s),\]
where $l(s):=s/(e^s-1)$ for $s>0$ and $l(0)=1$. Each of the terms in this product is a monotone in $s$ by Lemma~\ref{Lem:maK}, so for each $k=0,1,2$ and for all $s\in[k,k+1]$, their product is at most $\inv{(2\sqrt{3})}\,e^{a(k+1)-1}\,l(k)\leq 1$. On the other hand, for $s\geq 3$, the desired result will follow if we can establish that $a(s)\leq m(s)+s$. This is equivalent to the inequality
\begin{equation}
\label{eq:sqrtVlb}
\sqrt{V(s)}\geq 1-2/s+2/(e^s-1),
\end{equation}
so it suffices to prove that the lower bound in~\eqref{Eq:sqrtVbds} is at least $1-2/s+2/(e^s-1)$, which amounts to showing that \[s^3e^{-s}/2+s^5e^{-2s}/2+2s\,\inv{(e^s-1)}\leq 2\]
for all $s\geq 3$. To establish this, we simply bound each summand on the left-hand side by its global maximum in $[3,\infty)$. This completes the proof of (iii).
\end{proof}

\subsection{Local bracketing entropy bounds} 
\label{Subsec:LocalBE}
The aim of this section is to prove some local bracketing entropy results that form the backbone of the proofs of Theorems~\ref{Thm:k1} and~\ref{Thm:ThetaRisk} in Sections~\ref{Sec:Logkaffine} and~\ref{Sec:ThetaPolytope} respectively. Theorem~\ref{Thm:k1} is a consequence of the key local bracketing entropy bound stated as Proposition~\ref{Prop:Log1affEntropy} in Section~\ref{Sec:LogkaffineProofs}, which in turn builds on two intermediate results that we establish below, namely Propositions~\ref{Prop:SimplexEntropy} and~\ref{Prop:PolytopeEntropy}. By modifying the proofs of these two results, we obtain the analogous bounds in Propositions~\ref{Prop:ThetaSimplexEntropy} and~\ref{Prop:ThetaEntropy}, which constitute the crux of the proof of Theorem~\ref{Thm:ThetaRisk}. 
Throughout, we rely heavily on the technical tools developed in Section~\ref{Subsec:EntropyCalcs}. We will use the notation introduced at the start of Sections~\ref{Subsec:Notation},~\ref{Sec:Proofs} and~\ref{Sec:LogkaffineSupp}, as well as the key definitions from Sections~\ref{Sec:Logkaffine} and~\ref{Sec:ThetaPolytope}. 

We start by collecting together some global bracketing entropy bounds which are minor modifications of those that appear in~\citet{GW15},~\citet{KS16}, and~\citet{KGS18}. Recall that, as in~\citet{KS16}, we define $h_2,\,h_3\colon (0,\infty)\to  (0,\infty)$ by $h_2(x):=\inv{x}\log_+^{3/2}(\inv{x})$ and $h_3(x):=x^{-2}$ respectively. For measurable $f,g\colon\R^d\to\R$, let $L_2(f,g):=\bigl\{\int_{\R^d}\,(f-g)^2\bigr\}^{1/2}$. In addition, denote by $\mathcal{K}^b\equiv\mathcal{K}_d^b$ the collection of all compact, convex sets $K\subseteq\R^d$ with non-empty interior. For $K\in\mathcal{K}^b$, let $\Phi(K):=\{\restr{\phi}{K}:\phi\in\Phi\}$, and for $-\infty\leq B_1<B_2<\infty$ and $K_1,\dotsc,K_m\in\mathcal{K}^b$, define $\Phi_{B_1,B_2}(K_1,\dotsc,K_m):=\{\phi\colon\bigcup_{j=1}^{\,m} K_j\to [B_1,B_2]:\restr{\phi}{K_j}\in\Phi(K_j)\text{ for all }1\leq j\leq m\}$ and $\mathcal{G}_{B_1,B_2}(K_1,\dotsc,K_m):=\bigl\{e^\phi\,\Ind_{\bigcup_{j=1}^{\,m} K_j}:\phi\in\Phi,\,\restr{\phi}{\,\bigcup_{j=1}^{\,m} K_j}\in\Phi_{B_1,B_2}(K_1,\dotsc,K_m)\bigr\}$.
\begin{proposition}
\label{Prop:bebds}
For $d\in\N$, let $S_1,\dotsc,S_m\subseteq\R^d$ be $d$-simplices with pairwise disjoint interiors. If $\varepsilon>0$ and $-\infty<B_1<B_2<\infty$, then
\begin{equation}
\label{Eq:besimp}
H_{[\,]}\bigl(\varepsilon,\textstyle\mathcal{G}_{B_1,B_2}(S_1,\dotsc,S_m),\dhell\bigr)\displaystyle\lesssim_d m\rbr{\frac{e^{B_2/2}\rbr{B_2-B_1}\mu_d^{1/2}\bigl(\bigcup_{j=1}^{\,m}S_j\bigr)}{\varepsilon}}^{d/2}.
\end{equation}
Suppose henceforth that $d\in\{2,3\}$. If $K\in\mathcal{K}^b\equiv\mathcal{K}_d^b$ and $\varepsilon>0$, then
\begin{equation}
\label{Eq:beconvbd}
H_{[\,]}(\varepsilon,\mathcal{G}_{B_1,B_2}(K),\dhell)\lesssim h_d\rbr{\frac{\varepsilon}{e^{B_2/2}\rbr{B_2-B_1}\mu_d^{1/2}(K)}}
\end{equation}
whenever $-\infty<B_1<B_2<\infty$, and
\begin{equation}
\label{Eq:beconvubd}
H_{[\,]}(\varepsilon,\mathcal{G}_{-\infty,B}(K),\dhell)\lesssim h_d\rbr{\frac{\varepsilon}{e^{B/2}\mu_d^{1/2}(K)}}
\end{equation}
for all $B\in\R$. Finally, for any family of sets $K_1,\dotsc,K_m\in\mathcal{K}^b$ with pairwise disjoint interiors, we can obtain
bounds for $H_{[\,]}\bigl(\varepsilon,\mathcal{G}_{B_1,B_2}(K_1,\dotsc,K_m),\dhell\bigr)$ and $H_{[\,]}\bigl(\varepsilon,\mathcal{G}_{-\infty,B}(K_1,\dotsc,K_m),\dhell\bigr)$ by multiplying the right-hand sides of~\eqref{Eq:beconvbd} and~\eqref{Eq:beconvubd} respectively by $m$, and replacing $\mu_d(K)$ with $\mu_d\bigl(\bigcup_{j=1}^{\,m}K_j\bigr)$ throughout.
\end{proposition}
\begin{proof}
We first address~\eqref{Eq:besimp}, which is a $d$-dimensional version of Proposition~7 in the online supplement to~\citet{KGS18}. First, for a fixed $\varepsilon>0$, set $D:=\bigcup_{j=1}^{\,m} S_j$ and $\varepsilon_j:=\{\mu_d(S_j)/\mu_d(D)\}^{1/2}\,\varepsilon$ for each $1\leq j\leq m$, and observe that by~\citet[Theorem~1.1(ii)]{GW15}, we have
\begin{align}
H_{[\,]}\bigl(\varepsilon,\Phi_{B_1,B_2}(S_1,\dotsc,S_m),L_2\bigr)&\leq\sum_{j=1}^m H_{[\,]}(\varepsilon_j,\Phi_{B_1,B_2}(S_j),L_2)\lesssim_d\sum_{j=1}^m\rbr{\frac{(B_2-B_1)\,\mu_d^{1/2}(S_j)}{\varepsilon_j}}^{d/2}\notag\\
\label{Eq:besimpgw}
&\lesssim_d m\rbr{\frac{(B_2-B_1)\,\mu_d^{1/2}(D)}{\varepsilon}}^{d/2}.
\end{align}
To obtain~\eqref{Eq:besimp}, fix $\varepsilon>0$ and set $\zeta:=2\varepsilon e^{-B_2/2}$. We deduce from~\eqref{Eq:besimpgw} that there exists a bracketing set $\{[\phi_j^L,\phi_j^U]:1\leq j\leq M\}$ for $\Phi_{B_1,B_2}(S_1,\dotsc,S_m)$ such that $L_2(\phi_j^L,\phi_j^U)\leq\zeta$ and $\phi_j^U\leq B_2$, where $\log M$ is bounded above by the right hand side of~\eqref{Eq:besimp} up to a multiplicative factor that depends only on $d$. Since
\[\int_D\,\bigl(e^{\phi_j^U/2}-e^{\phi_j^L/2}\bigr)^2\leq\frac{e^{B_2}}{4}\int_D\,\bigl(\phi_j^U-\phi_j^L\bigr)^2\leq\varepsilon^2,\]
it follows that $\{[e^{\phi_j^L},e^{\phi_j^U}]:1\leq j\leq M\}$ is an $\varepsilon$-Hellinger bracketing set for $\{\restr{f}{D}:f\in\mathcal{G}_{B_1,B_2}(S_1,\dotsc,S_m)\}$, as required.

In view of Proposition~4 in the online supplement to~\citet{KS16}, a similar proof to that given above yields~\eqref{Eq:beconvbd}. As for~\eqref{Eq:beconvubd}, we fix $d\in\{2,3\}$ and begin by outlining a simple scaling argument that allows us to deduce the general result from the special case where $B=-2$ and $\mu_d(K)=1$. For $B'\in\R$ and $K'\in\mathcal{K}_d^b$, let $K:=\inv{\lambda}K'$ and $\lambda:=\mu_d(K')^{1/d}$, and suppose that we have already shown that 
\begin{equation}
\label{Eq:beconvuspec}
H_{[\,]}(\varepsilon,\mathcal{G}_{-\infty,-2}(K),\dhell)\lesssim h_d(\varepsilon)
\end{equation}
for all $\varepsilon>0$. Then for each $\varepsilon>0$, we can find a bracketing set $\{[f_j^L,f_j^U]:1\leq j\leq M\}$ for $\mathcal{G}_{-\infty,-2}(K)$ such that \[\int_K\rbr{\sqrt{f_j^U}-\sqrt{f_j^L}}^2\leq\varepsilon^2e^{-(B'+2)}\lambda^{-d}\]
for each $1\leq j\leq M$ and $\log M \lesssim h_d\bigl(\varepsilon e^{-(B'+2)/2}\lambda^{-d/2}\bigr)\lesssim h_d\bigl(\varepsilon e^{-B'/2}\mu_d^{-1/2}(K')\bigr)$. Since every $g\in\mathcal{G}_{-\infty,B'}(K')$ takes the form $x\mapsto e^{B'+2}f(\inv{\lambda}x)$ for some $f\in\mathcal{G}_{-\infty,-2}(K)$, it follows that $\mathcal{G}_{-\infty,B'}(K')$ is covered by the brackets $\{[g_j^L,g_j^U]:1\leq j\leq M\}$ defined by \[g_j^L(x):=e^{B'+2}f_j^L(\inv{\lambda}x),\quad g_j^U(x):=e^{B'+2} f_j^U(\inv{\lambda}x).\]
To see that this constitutes a valid $\varepsilon$-bracketing set and thereby implies the desired conclusion, observe that
\begin{align*}
\int_{K'}\rbr{\sqrt{g_j^U}-\sqrt{g_j^L}}^2&=e^{B'+2}\int_{K'}\cbr{f_j^U\bigl(\inv{\lambda}x\bigr)^{1/2}-f_j^L\bigl(\inv{\lambda}x\bigr)^{1/2}}^2\,dx\\
&=e^{B'+2}\,\lambda^d\int_K\rbr{\sqrt{f_j^U}-\sqrt{f_j^L}}^2\leq\varepsilon^2
\end{align*}
for all $1\leq j\leq M$, as required. Therefore, it remains to establish~\eqref{Eq:beconvuspec}. This will require only a few small adjustments to the arguments in steps 2 and 3 of the proof of Theorem~4 in~\citet{KS16}, where it was shown that 
\begin{equation}
\label{Eq:ksinterm}
H_{[\,]}((4+e)\varepsilon,\mathcal{G}_{-\infty,-2\,}(K),\dhell)=H_{[\,]}((4+e)\varepsilon,\mathcal{G}_{-\infty,-1\,}(K),L_2)\lesssim h_d(\varepsilon)
\end{equation}
for all $0<\varepsilon<\inv{e}$ when $K=[0,1]^d$. For $\varepsilon>\inv{e}$, we may use a single bracketing pair $[f^L,f^U]$ with $f^L\equiv 0$ and $f^U\equiv 1$, so the left-hand side is 0 in this case. Therefore~\eqref{Eq:ksinterm} holds for all $\varepsilon>0$ when $K=[0,1]^d$. Furthermore, since the bounds in Propositions~2 and 4 in the online supplement to~\citet{KS16} depend on the convex domain $K$ only through $\mu_d(K)$, all the intermediate steps in the proof of~\eqref{Eq:ksinterm} in~\citet{KS16} remain valid when $[0,1]^d$ is replaced with an arbitrary $K\in\mathcal{K}_d^b$ with $\mu_d(K)=1$. This crucial observation completes the proof of~\eqref{Eq:beconvuspec} and hence that of~\eqref{Eq:beconvubd}.

The final assertion of Proposition~\ref{Prop:bebds} follows from~\eqref{Eq:beconvbd} and~\eqref{Eq:beconvubd} in much the same way that~\eqref{Eq:besimp} follows from the special case $m=1$.
\end{proof}
\begin{figure}[htb]
\centering
\includegraphics[trim=3cm 13.7cm 0 1.9cm, width=0.8\textwidth]{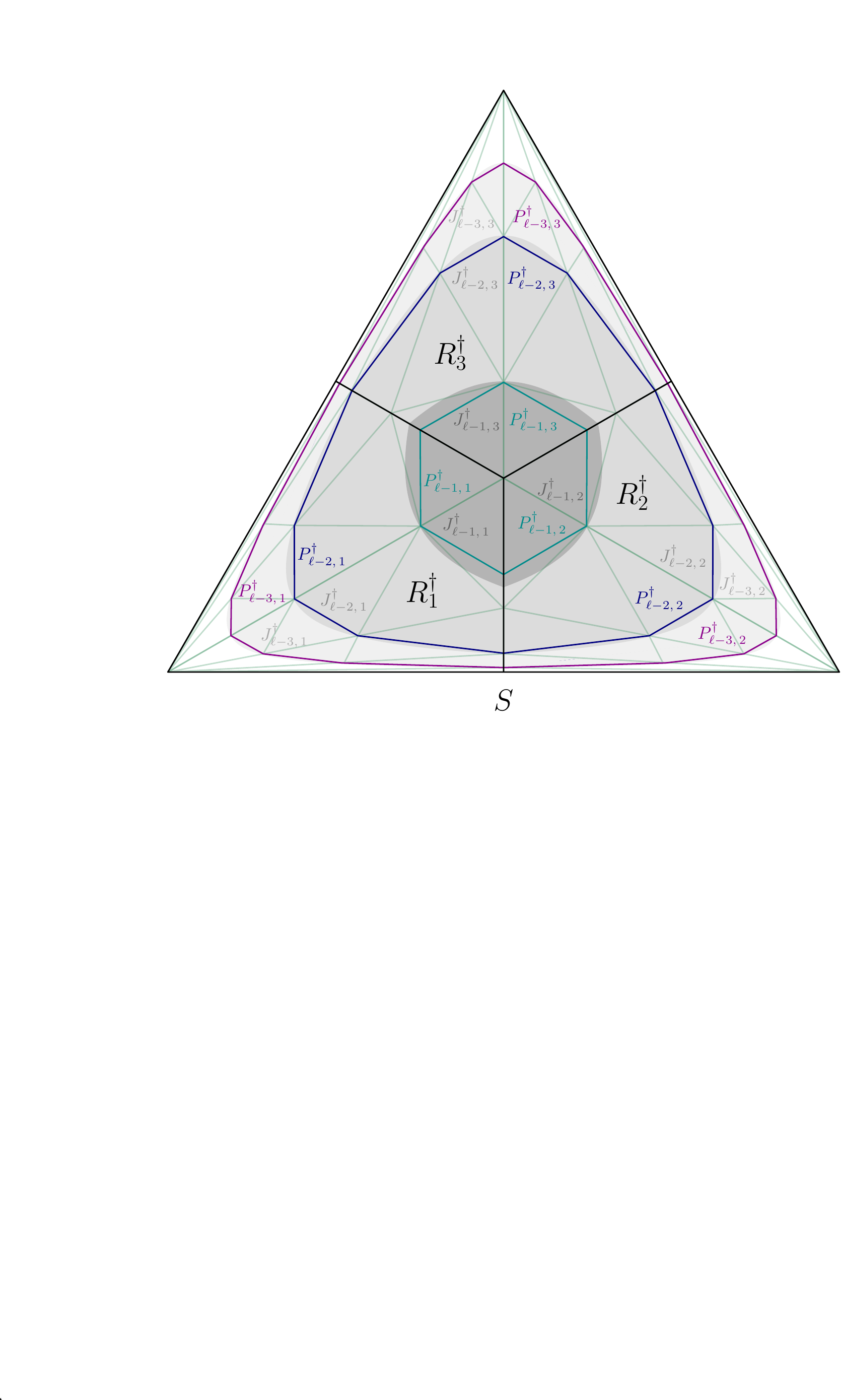}
\caption{Illustration of the proof of \protect{Proposition~\ref{Prop:SimplexEntropy}} when $d=2$ (and $\delta<1/24$, so that $\ell\geq 4$). The polytopes $R_1^\dagger,R_2^\dagger,R_3^\dagger\subseteq S$ are demarcated by the black lines. For $i\in\{\ell-3,\ell-2,\ell-1\}$ and $j\in\{1,2,3\}$, the `invelopes' $J_{i,j}^\dagger\subseteq R_j^\dagger$ are represented by the grey shaded regions (see Lemma~\ref{Lem:invelopesimp}) and the boundaries of the approximating polytopes $P_{i,j}^\dagger\subseteq J_{i,j}^\dagger$ are outlined in colour (see Corollary~\ref{Cor:polyapproxsimp}). Moreover, the regions between the nested polytopes $R_j^\dagger\supseteq P_{1,j}^\dagger\supseteq P_{2,j}^\dagger\supseteq\cdots\supseteq P_{\ell,j}^\dagger$ may be triangulated, as is indicated by the green line segments. The main reason for considering these sets is that by Lemma~\ref{Lem:invelopesimp}(iii), every $f\in\mathcal{G}(f_S,\delta)$ satisfies a pointwise lower bound $\log f+\log\mu_d(S)\geq -2^{-i+2}$ on $J_{i,j}^\dagger\supseteq P_{i,j}^\dagger$.}
\label{Fig:SimplexEntropy}
\end{figure}
As a first step towards proving Proposition~\ref{Prop:Log1affEntropy} for general $f_0\in\mathcal{F}^1(\mathcal{P}^m)$, we consider here the special case where $K$ is a $d$-simplex and $f_0=f_K=\mu_d(K)^{-1}\Ind_K$ is the uniform density on $K$. For further discussion of the proof techniques we employ, see the discussion after the statement of Theorem~\ref{Thm:k1} in Section~\ref{Sec:Logkaffine} and page~\getpagerefnumber{Disc:SimplexEntropy} of Section~\ref{Subsec:EntropyCalcs}. See Figure~\ref{Fig:SimplexEntropy} for an illustration of the proof. For $d\in\N$, we define a `canonical' regular $d$-simplex $\triangle\equiv\triangle_d:=\conv\{e_1,\dotsc,e_{d+1}\}\subseteq\R^{d+1}$ of side length $\sqrt{2}$, which will be viewed as a subset of its affine hull, namely $\aff\triangle=\{x = (x_1,\dotsc,x_{d+1})\in\R^{d+1}:\sum_{j=1}^{d+1}x_j=1\}$.
\begin{proposition}
\label{Prop:SimplexEntropy}
Let $d\in\{2,3\}$. If $0<\varepsilon<\delta<(d+1)^{-d/2}$ and $S\subseteq\R^d$ is a $d$-simplex, then
\begin{equation}
\label{Eq:lbebd2}
H_{[\,]}(2^{1/2}\varepsilon,\mathcal{G}(f_S,\delta),\dhell)\lesssim\rbr{\frac{\delta}{\varepsilon}}\log^{3/2}\rbr{\frac{1}{\delta}}\log^{3/2}\rbr{\frac{\log(1/\delta)}{\varepsilon}}=:H_2(\delta,\varepsilon)
\end{equation}
when $d=2$ and 
\begin{equation}
\label{Eq:lbebd3}
H_{[\,]}(2^{1/2}\varepsilon,\mathcal{G}(f_S,\delta),\dhell)\lesssim\rbr{\frac{\delta}{\varepsilon}}^2\log^4\rbr{\frac{1}{\delta}}+\rbr{\frac{\delta}{\varepsilon}}^{3/2}\log^{21/4}\rbr{\frac{1}{\delta}}=:H_3(\delta,\varepsilon)
\end{equation}
when $d=3$.
\end{proposition}
\begin{proof}
Fix $d\in\{2,3\}$ and suppose that $0<\varepsilon<\delta<(d+1)^{-d/2}$. In addition, define $\varepsilon':=\varepsilon/\sqrt{d+1}$ and $\ell:=\ceil{\log_2((d+1)^{-d/2}\inv{\delta})}$, so that $\ell$ is the smallest integer $i$ such that $4^i\delta^2\geq (d+1)^{-d}$, and note that $1\leq\ell\lesssim\log(1/\delta)$. Since $\dhell$ is affine invariant, we may assume without loss of generality that $S$ is a regular $d$-simplex with side length $\sqrt{2}$. Then since $\triangle\equiv\triangle_d=\conv\{e_1,\dotsc,e_{d+1}\}\subseteq\R^{d+1}$ is also a $d$-simplex with side length $\sqrt{2}$, there is an (affine) isometry $T\colon\aff\triangle\to\R^d$ such that $T(\triangle)=S$. For each $j\in\{1,\dotsc,d+1\}$, define $R_j\subseteq\triangle$ as in~\eqref{Eq:Rj} and let $R_j^\dagger:=T(R_j)\subseteq S$. Then $S$ is the union of the polytopes $R_1^\dagger,\dotsc,R_{d+1}^\dagger$, whose interiors are pairwise disjoint. 

Fix $j\in\{1,\dotsc,d+1\}$, and for each $i\in\{1,\dotsc,\ell-1\}$, define $J_{i,j}^\dagger:=T\bigl(R_j\cap J_{4^i\delta^2}^\triangle\bigr)$ and $P_{i,j}^\dagger:=T\bigl(P_{4^i\delta^2,\,j}^\triangle\bigr)$,
where $J_{4^i\delta^2}^\triangle\subseteq\triangle$ and $P_{4^i\delta^2,\,j}^\triangle\subseteq R_j$ are taken from Lemma~\ref{Lem:invelopebox}, Lemma~\ref{Lem:invelopesimp} and Corollary~\ref{Cor:polyapproxsimp} respectively. In addition, set $J_{\ell,j}^\dagger:=\emptyset$ and $P_{\ell,j}^\dagger:=\emptyset$. It follows from Corollary~\ref{Cor:polyapproxsimp} that $P_{i+1,\,j}^\dagger\subseteq P_{i,j}^\dagger\subseteq J_{i,j}^\dagger\subseteq R_j^\dagger$ for all $1\leq i\leq \ell-1$, so 
\begin{equation}
\label{Eq:Rjshells}
\textstyle R_j^\dagger=\bigl(R_j^\dagger\setminus\Int P_{1,j}^\dagger\bigr)\cup\,\bigcup_{\,i=1}^{\,\ell-1}\, \bigl(P_{i,j}^\dagger\setminus\Int P_{i+1,\,j}^\dagger\bigr).
\end{equation}
By our choice of $\ell$, the interiors of $(R_j^\dagger\setminus\Int P_{1,j}^\dagger),(P_{1,j}^\dagger\setminus\Int P_{2,j}^\dagger),\dotsc,(P_{\ell-1,\,j}^\dagger\setminus\Int P_{\ell,j}^\dagger)$ are non-empty and pairwise disjoint, and by Corollary~\ref{Cor:polyapproxsimp}(ii), each of these $\ell$ sets can be expressed as the union of $\lesssim\log^{d-1}(1/\delta)$ $d$-simplices with pairwise disjoint interiors. Moreover, defining $Q^\triangle$ as in Lemma~\ref{Lem:invelopesimp} and $J_\eta$ as in Lemma~\ref{Lem:invelopebox} for $\eta>0$, we can apply Corollary~\ref{Cor:polyapproxsimp}(i), Lemma~\ref{Lem:invelopesimp} and Lemma~\ref{Lem:invelopebox}(iii) in that order to deduce that
\begin{equation}
\label{Eq:Rjvols}
\mu_d(R_j^\dagger\setminus P_{i,j}^\dagger)\lesssim\mu_d(R_j^\dagger\setminus J_{i,j}^\dagger)\lesssim\mu_d(Q^\triangle\setminus J_{4^i\delta^2})\lesssim\mu_d([0,1/2]^d\setminus J_{4^i\delta^2})\lesssim 4^i\delta^2\log^{d-1}(1/\delta)
\end{equation}
for all $1\leq i\leq\ell$.
It follows that $\mu_d(P_{i,j}^\dagger\setminus P_{i+1,\,j}^\dagger)\leq\mu_d(R_j^\dagger\setminus P_{i+1,\,j}^\dagger)\lesssim 4^i\delta^2\log^{d-1}(1/\delta)$ for all $1\leq i\leq\ell-1$. We emphasise here that the hidden multiplicative constants in these bounds do not depend on $i$. 

Now if $f\in\mathcal{G}(f_S,\delta)$, then Lemma~\ref{Lem:hellunbds}(ii) implies that $\log f\leq 2^{7/2}d\delta-\log\mu_d(S)\leq 2^{7/2}d\,(d+1)^{-d/2}-\log\mu_d(S)$ on $S$. 
Also, for each $1\leq i\leq\ell-1$, we deduce from Lemma~\ref{Lem:invelopesimp}(iii) that $\log f(x)\geq -2^{-i+2}(d!)^{-1/2}-\log\mu_d(S)\geq -2^{-i+2}-\log\mu_d(S)$ for all $x\in P_{i,j}^\dagger\setminus\Int P_{i+1,\,j}^\dagger\subseteq J_{i,j}^\dagger$. Thus, for each $1\leq i\leq\ell-1$, it follows from the observations above and~\eqref{Eq:besimp} from Proposition~\ref{Prop:bebds} that
\begin{align*}
H_{[\,]}(\varepsilon'/\sqrt{\ell},\mathcal{G}(f_S,\delta),\dhell,P_{i,j}^\dagger\setminus P_{i+1,\,j}^\dagger)&\lesssim\log^{d-1}\!\rbr{\frac{1}{\delta}}\!\rbr{\frac{(2^{7/2}d\delta+2^{-i+2})\{4^i\delta^2\log^{d-1}(1/\delta)\}^{1/2}}{\ell^{-1/2}\,\varepsilon}}^{d/2}\\
&\lesssim\rbr{\frac{\delta}{\varepsilon}}^{d/2}\log^{\frac{d(d+4)}{4}-1}\rbr{\frac{1}{\delta}}\bigl(2^{7/2}d\,(2^i\delta)+4\bigr)^{d/2}\\
&\lesssim\rbr{\frac{\delta}{\varepsilon}}^{d/2}\log^{\frac{d(d+4)}{4}-1}\rbr{\frac{1}{\delta}},
\end{align*}
where the final bound follows from the fact that $2^i\delta\leq 2^{\ell-1}\delta\leq (d+1)^{-d/2}\leq 2^{-d/2}$. Since $1\leq\ell\lesssim\log(1/\delta)$ and the hidden multiplicative constants in the bounds above do not depend on $i$, we conclude that
\vspace{-0.3cm}
\begin{align}
H_{[\,]}(\varepsilon',\mathcal{G}(f_S,\delta),\dhell,P_{1,j}^\dagger)&\leq\sum_{i=1}^{\ell-1}H_{[\,]}(\varepsilon'/\sqrt{\ell},\mathcal{G}(f_S,\delta),\dhell,P_{i,j}^\dagger\setminus P_{i+1,\,j}^\dagger)\notag\\
\label{Eq:beint}
&\lesssim\rbr{\frac{\delta}{\varepsilon}}^{d/2}\log^{\frac{d(d+4)}{4}}\rbr{\frac{1}{\delta}}.
\end{align}
Furthermore, recalling that every $f\in\mathcal{G}(f_S,\delta)$ satisfies $f\leq e^{2^{7/2}d\delta-\log\mu_d(S)}\lesssim 1$ on $R_j^\dagger\setminus\Int P_{1,j}^\dagger$, we may apply the final assertion of Proposition~\ref{Prop:bebds} together with~\eqref{Eq:Rjvols} to deduce that
\begin{equation}
\label{Eq:beboundary}
H_{[\,]}(\varepsilon',\mathcal{G}(f_S,\delta),\dhell,R_j^\dagger\setminus P_{1,j}^\dagger)\lesssim\log^{d-1}\rbr{\frac{1}{\delta}}h_d\rbr{\frac{\varepsilon}{\delta\log^{(d-1)/2}(1/\delta)}}.
\end{equation}
Having now established~\eqref{Eq:beint} and~\eqref{Eq:beboundary} for each fixed $1\leq j\leq d+1$, we finally note that 
\begin{equation}
\label{Eq:bepolysum}
H_{[\,]}(2^{1/2}\varepsilon,\mathcal{G}(f_S,\delta),\dhell)\leq \sum_{j=1}^{d+1}\,\bigl\{H_{[\,]}(\varepsilon',\mathcal{G}(f_S,\delta),\dhell,R_j^\dagger\setminus P_{1,j}^\dagger)+H_{[\,]}(\varepsilon',\mathcal{G}(f_S,\delta),\dhell,P_{1,j}^\dagger)\bigr\}.
\end{equation}
Thus, when $d=2$, we conclude that
\begin{equation}
\label{Eq:bepoly2}
H_{[\,]}(2^{1/2}\varepsilon,\mathcal{G}(f_S,\delta),\dhell)\lesssim\rbr{\frac{\delta}{\varepsilon}}\log^{3/2}\rbr{\frac{1}{\delta}}\cbr{\log^{3/2}\rbr{\frac{1}{\delta}}+\log^{3/2}\rbr{\frac{\delta\log^{1/2}(1/\delta)}{\varepsilon}}},
\end{equation}
which is bounded above by the quantity $H_2(\delta,\varepsilon)$ in~\eqref{Eq:lbebd2} up to a universal constant. 
Similarly, when $d=3$, the bound~\eqref{Eq:lbebd3} follows immediately on combining~\eqref{Eq:beint},~\eqref{Eq:beboundary} and~\eqref{Eq:bepolysum}.
\end{proof}
We now extend Proposition~\ref{Prop:SimplexEntropy} to the case where $f_0$ is the uniform density $f_K$ on a polytope $K\in\mathcal{P}^m$. By Lemma~\ref{Lem:minfacets}, every polytope in $\mathcal{P}_d$ has at least as many facets as a $d$-simplex, namely $d+1$.
\begin{proposition}
\label{Prop:PolytopeEntropy}
Let $d\in\{2,3\}$ and fix $m\in\N$ with $m\geq d+1$. If $0<\varepsilon<\delta<2^{-3/2}$ and $K\in\mathcal{P}^m$ is a polytope, then
\begin{equation}
\label{Eq:lbebd-dm1}
H_{[\,]}(2^{1/2}\varepsilon,\mathcal{G}(f_K,\delta),\dhell)\lesssim m\rbr{\frac{\delta}{\varepsilon}}\log^{3/2}\rbr{\frac{1}{\delta}}\log^{3/2}\rbr{\frac{\log(1/\delta)}{\varepsilon}}=m H_2(\delta,\varepsilon)
\end{equation}
when $d=2$ and 
\begin{equation}
\label{Eq:lbebd-dm2}
H_{[\,]}(2^{1/2}\varepsilon,\mathcal{G}(f_K,\delta),\dhell)\lesssim m\cbr{\rbr{\frac{\delta}{\varepsilon}}^2\log^4\rbr{\frac{1}{\delta}}+\rbr{\frac{\delta}{\varepsilon}}^{3/2}\log^{21/4}\rbr{\frac{1}{\delta}}}=m H_3(\delta,\varepsilon)
\end{equation}
when $d=3$.
\end{proposition}
\begin{proof}
Fix $d\in\{2,3\}$ and suppose that $0<\varepsilon<\delta<2^{-3/2}$. By Proposition~\ref{Lem:euler}, we can find $M\leq 6m$ $d$-simplices $S_1,\dotsc,S_M$ with pairwise disjoint interiors whose union is $K$. Set $\alpha_j:=\{\mu_d(S_j)/\mu_d(K)\}^{1/2}$ for each $1\leq j\leq M$, so that $\sum_{j=1}^M\alpha_j^2=1$. For each $f\in\mathcal{G}(f_K,\delta)$ and $1\leq j\leq M$, let $n_j(f)$ be the smallest $n_j\in\N$ for which $\int_{S_j}\,\bigl(\sqrt{f}-\sqrt{f_K}\bigr)^2\leq\alpha_j^2\,n_j\delta^2$.
By the minimality of $n_j(f)$, we have $\alpha_j^2\,(n_j(f)-1)\,\delta^2\leq\int_{S_j}\,\bigl(\sqrt{f}-\sqrt{f_K}\bigr)^2$
for each $j$, so
\begin{align}
\sum_{j=1}^M\alpha_j^2\,n_j(f)=1+\sum_{j=1}^M\alpha_j^2\,(n_j(f)-1)&\leq 1+\delta^{-2}\,\sum_{j=1}^M\int_{S_j}\,\bigl(\sqrt{f}-\sqrt{f_K}\bigr)^2\notag\\
\label{Eq:njfsum}
&=1+\delta^{-2}\,\int_K\,\bigl(\sqrt{f}-\sqrt{f_K}\bigr)^2\leq 2,
\end{align}
where the final inequality follows because $f\in\mathcal{G}(f_K,\delta)$. We also claim that $n_j(f)\lesssim\delta^{-2}$ for all $1\leq j\leq M$. To see this, note that since $f\in\mathcal{G}(f_K,\delta)$ and $\delta<2^{-3/2}$, it follows from Lemma~\ref{Lem:hellunbds}(ii) that
\begin{equation}
\label{Eq:unifbdK}
0\leq f\leq e^{8\sqrt{2}d\delta}f_K\leq e^{4d}f_K=e^{4d}\mu_d(K)^{-1}\:\;\text{on $K$}.
\end{equation}
Thus, we have $\bigl(\sqrt{f}-\sqrt{f_K}\bigr)^2\leq f\vee f_K\lesssim f_K=\inv{\mu_d(K)}$ on $K$, so $\int_{S_j}\,\bigl(\sqrt{f}-\sqrt{f_K}\bigr)^2\lesssim\mu_d(S_j)/\mu_d(K)=\alpha_j^2$ for all $j$. Recalling the definition of $n_j(f)$, we deduce that $n_j(f)\lesssim\delta^{-2}$ for all $j$, as required.

Now let $U:=\{(n_1(f),\dotsc,n_M(f)):f\in\mathcal{G}(f_K,\delta)\}$, and for each $(n_1,\dotsc,n_M)\in U$, define \[\mathcal{G}(f_K,\delta;n_1,\dotsc,n_M):=\{f\in\mathcal{G}(f_K,\delta):n_j(f)=n_j\text{ for all }1\leq j\leq M\}.\]
Since $\mathcal{G}(f_K,\delta)$ is the union of these subclasses, it follows that
\[N_{[\,]}(2^{1/2}\varepsilon,\mathcal{G}(f_K,\delta),\dhell)\leq\sum_{(n_1,\dotsc,n_M)\in U}N_{[\,]}(2^{1/2}\varepsilon,\mathcal{G}(f_K,\delta;n_1,\dotsc,n_M),\dhell),\]
so
\begin{equation}
\label{Eq:lbesum}
H_{[\,]}(2^{1/2}\varepsilon,\mathcal{G}(f_K,\delta),\dhell)\leq\log\,\abs{U}+\max_{(n_1,\dotsc,n_M)\in U}H_{[\,]}(2^{1/2}\varepsilon,\mathcal{G}(f_K,\delta;n_1,\dotsc,n_M),\dhell).
\end{equation}
Since $n_j(f)\lesssim\delta^{-2}$ for all $1\leq j\leq M$, we deduce that $\abs{U}\lesssim\delta^{-2M}$ and hence that
\begin{equation}
\label{Eq:ucardbd}
\log\,\abs{U}\lesssim M\log(1/\delta)\lesssim m\log(1/\delta).
\end{equation}
Next, we bound the second term on the right-hand side of~\eqref{Eq:lbesum}. Fix $j\in\{1,\dotsc,M\}$ and $(n_1,\dotsc,n_M)\in U$. If $f\in\mathcal{G}(f_K,\delta;n_1,\dotsc,n_M)$, then $f_K\Ind_{S_j}=\alpha_j^2\,f_{S_j}$ and $\int_{S_j}\,\bigl(\sqrt{f}-\sqrt{f_K}\bigr)^2\leq\alpha_j^2\,n_j\delta^2$, so $\int_{S_j}\,\bigl\{(\alpha_j^{-2}f)^{1/2}-f_{S_j}^{1/2}\bigr\}^2\leq n_j\delta^2$. This shows that
$\alpha_j^{-2}f\Ind_{S_j}\in\mathcal{G}(f_{S_j},\sqrt{n_j}\,\delta)$. In addition, it follows from~\eqref{Eq:unifbdK} that 
\begin{equation}
\label{Eq:unifbdSj}
0\leq\alpha_j^{-2}f\Ind_{S_j}=\mu_d(K)f\Ind_{S_j}/\mu_d(S_j)\leq e^{4d}\mu_d(S_j)^{-1}\Ind_{S_j}=e^{4d}f_{S_j}\:\;\text{on $S_j$}.
\end{equation}

Suppose first that $\sqrt{n_j}\,\delta<(d+1)^{-d/2}$. Since $0<\sqrt{n_j}\,\varepsilon<\sqrt{n_j}\,\delta<(d+1)^{-d/2}$, we can apply Proposition~\ref{Prop:SimplexEntropy} to deduce that there exists a $\sqrt{n_j}\,\varepsilon$-Hellinger bracketing set $\{[g_\ell^L,g_\ell^U]:1\leq\ell\leq N_j\}$ for $\mathcal{G}(f_{S_j},\sqrt{n_j}\,\delta)$ such that \[\log N_j\lesssim H_d(\sqrt{n_j}\,\delta,\sqrt{n_j}\,\varepsilon)\lesssim H_d(\delta,\varepsilon).\]
Note that we can find $1\leq\ell\leq N_j$ such that $g_\ell^L(x)\leq\alpha_j^{-2}f(x)\leq g_\ell^U(x)$ for all $x\in S_j$. Therefore, $\{f\Ind_{S_j}:f\in\mathcal{G}(f_K,\delta;n_1,\dotsc,n_M)\}$ is covered by the brackets $\{[\alpha_j^2\,g_\ell^L,\alpha_j^2\,g_\ell^U]:1\leq\ell\leq N_j\}$. Since
\begin{equation}
\label{Eq:bracketsize}
\int_{S_j}\,\biggl(\sqrt{\alpha_j^2\,g_\ell^U}-\sqrt{\alpha_j^2\,g_\ell^L}\biggr)^2=\alpha_j^2\int_{S_j}\,\biggl(\sqrt{g_\ell^U}-\sqrt{g_\ell^L}\biggr)^2\leq\alpha_j^2\,n_j\,\varepsilon^2
\end{equation}
for all $1\leq\ell\leq N_j$, we conclude that 
\begin{equation}
\label{Eq:beSj}
H_{[\,]}\bigl(\alpha_j\sqrt{n_j}\,\varepsilon,\mathcal{G}(f_K,\delta;n_1,\dotsc,n_M),\dhell,S_j\bigr)\lesssim H_d(\delta,\varepsilon),
\end{equation}
provided that $\sqrt{n_j}\,\delta<(d+1)^{-d/2}$. 

We now verify that~\eqref{Eq:beSj} remains valid when $\sqrt{n_j}\,\delta\geq (d+1)^{-d/2}$. In this case, we define $B_j:=4d\log(\mu_d(S_j)^{-1})$ and deduce from~\eqref{Eq:unifbdSj} that $\{\alpha_j^{-2}f\Ind_{S_j}:f\in\mathcal{G}(f_K,\delta;n_1,\dotsc,n_M)\}\subseteq\mathcal{G}_{-\infty,B_j}(S_j)$. By the final bound~\eqref{Eq:beconvubd} from Proposition~\ref{Prop:bebds}, we can find a $\sqrt{n_j}\,\varepsilon$-Hellinger bracketing set $\{[\tilde{g}_\ell^L,\tilde{g}_\ell^U]:1\leq\ell\leq\tilde{N}_j\}$ for $\mathcal{G}_{-\infty,B_j}(S_j)$ such that 
\[\log\tilde{N}_j\lesssim h_d\rbr{\frac{\sqrt{n_j}\,\varepsilon}{e^{B_j/2}\mu_d(S_j)^{1/2}}}=h_d\rbr{\frac{\sqrt{n_j}\,\varepsilon}{e^{2d}}}\lesssim h_d\rbr{\frac{\varepsilon}{\delta}}\lesssim H_d(\delta,\varepsilon).\]
Indeed, the penultimate inequality above follows since $\sqrt{n_j}\,\delta\geq(d+1)^{-d/2}\gtrsim 1$ and $h_d$ is decreasing, and the final inequality can be verified separately for $d=2,3$; see~\eqref{Eq:bepoly2} for example to obtain the bound when $d=2$. As above, we see that $\{f\Ind_{S_j}:f\in\mathcal{G}(f_K,\delta;n_1,\dotsc,n_M)\}$ is covered by the brackets $\{[\alpha_j^2\,\tilde{g}_\ell^L,\alpha_j^2\,\tilde{g}_\ell^U]:1\leq\ell\leq\tilde{N}_j\}$, and as in~\eqref{Eq:bracketsize}, we have \[\int_{S_j}\,\biggl(\sqrt{\alpha_j^2\,\tilde{g}_\ell^U}-\sqrt{\alpha_j^2\,\tilde{g}_\ell^L}\biggr)^2\leq\alpha_j^2\,n_j\,\varepsilon^2\]
for all $j$. It follows that 
\[H_{[\,]}(\alpha_j\sqrt{n_j}\,\varepsilon,\mathcal{G}(f_K,\delta;n_1,\dotsc,n_M),\dhell,S_j)\leq\log\tilde{N}_j\lesssim H_d(\delta,\varepsilon),\]
so~\eqref{Eq:beSj} holds when $\sqrt{n_j}\,\delta\geq (d+1)^{-d/2}$, as required.

Finally, since $(n_1,\dotsc,n_M)\in U$, it follows from~\eqref{Eq:njfsum} and the definition of $U$ above that $\sum_{j=1}^M\alpha_j^2\,n_j\,\varepsilon^2\leq 2\varepsilon^2$. Having established~\eqref{Eq:beSj} for all $1\leq j\leq M$, we conclude that
\begin{align*}
H_{[\,]}\bigl(2^{1/2}\varepsilon,\mathcal{G}(f_K,\delta;n_1,\dotsc,n_M),\dhell\bigr)&\leq\sum_{j=1}^M H_{[\,]}\bigl(\alpha_j\sqrt{n_j}\,\varepsilon,\mathcal{G}(f_K,\delta;n_1,\dotsc,n_M),\dhell,S_j\bigr)\\
&\lesssim M H_d(\delta,\varepsilon)\lesssim m H_d(\delta,\varepsilon)
\end{align*}
whenever $0<\varepsilon<\delta<2^{-3/2}$ and $(n_1,\dotsc,n_M)\in U$. Together with~\eqref{Eq:lbesum} and~\eqref{Eq:ucardbd}, this implies the desired conclusion.
\end{proof}

Turning now to the subclasses $\mathcal{F}^{[\theta]}(\mathcal{P}^m)$ defined in Section~\ref{Sec:ThetaPolytope}, we first establish an analogue of Proposition~\ref{Prop:SimplexEntropy}.
\begin{proposition}
\label{Prop:ThetaSimplexEntropy}
Let $d=3$ and let $S\subseteq\R^3$ be a 3-simplex. If $0<\varepsilon<\delta<2^{-3}\,\theta^{-1/2}$ and $f_0 \in\mathcal{F}^{[\theta]}(S)\equiv\mathcal{F}_3^{[\theta]}(S)$ for some $\theta\in (1,\infty)$, then
\begin{alignat}{3}
H_{[\,]}(2^{1/2}\varepsilon,\mathcal{G}(f_0,\delta),\dhell)&\lesssim\frac{\log^{3/2}\theta+\delta^{3/5}}{\varepsilon^{3/2}}\log^{17/4}\rbr{\frac{1}{\theta\delta^2}}&&+\theta^{3/4}\rbr{\frac{\delta}{\varepsilon}}^{3/2}\log^{21/4}\rbr{\frac{1}{\theta\delta^2}}\notag\\
\label{Eq:ThetaSimplexEntropy}
&&&+\theta\log^3(e\theta)\rbr{\frac{\delta}{\varepsilon}}^2\log^4\rbr{\frac{1}{\theta\delta^2}}\\
&=:H_{3,\theta}(\delta,\varepsilon).&&\notag
\end{alignat}
\end{proposition}
The following proof is similar in structure and content to that of Proposition~\ref{Prop:SimplexEntropy}, although alterations must be made to the arguments that previously relied on the pointwise upper bound from Lemma~\ref{Lem:hellunbds}(ii), which applies only when $f_0$ is uniform. For general $\theta\in (1,\infty)$ and $f_0\in\mathcal{F}^{[\theta]}(S)$, we instead turn to Lemma~\ref{Lem:hellunbds}(iii) for a pointwise upper bound on functions $f\in\mathcal{G}(f_0,\delta)$. 
While the bound in Lemma~\ref{Lem:hellunbds}(ii) features a term of order $\delta$, the corresponding term in Lemma~\ref{Lem:hellunbds}(iii) is of order $\delta^{2/(d+2)}=\delta^{2/5}$ when $d=3$. The latter manifests itself in the overall bound~\eqref{Eq:ThetaSimplexEntropy} through the contribution of order $(\delta^{2/5}/\varepsilon)^{3/2}\log^{17/4}(1/(\theta\delta^2))$, which in turn is ultimately responsible for the term of order $(m/n)^{20/29}\log^{85/29}n$ in the adaptive risk bound~\eqref{Eq:OracleTheta} in Theorem~\ref{Thm:ThetaRisk}. This explains why we do not recover the local bracketing entropy bounds~\eqref{Eq:lbebd3} and~\eqref{Eq:lbebd-dm2} in Propositions~\ref{Prop:SimplexEntropy} and~\ref{Prop:PolytopeEntropy} or the risk bound~\eqref{Eq:OraclePm} in Proposition~\ref{Prop:PolytopeRisk} when we take the limit $\theta\searrow 1$ in~\eqref{Eq:ThetaSimplexEntropy},~\eqref{Eq:ThetaEntropy} and~\eqref{Eq:OracleTheta} respectively.

On the other hand, since $f_0\in\mathcal{F}^{[\theta]}(S)$ is bounded below by $\theta^{-1}f_S$ (and hence bounded away from 0) on $S$, the pointwise lower bound on $f\in\mathcal{G}(f_0,\delta)$ from Lemma~\ref{Lem:hellunbds}(i) can still be applied in this context. By extension, the same is true of the constructions and reasoning based on Corollary~\ref{Cor:polyapproxsimp} and Lemmas~\ref{Lem:invelopebox} and~\ref{Lem:invelopesimp}. As such, we will reuse much of the notation from the proof of Proposition~\ref{Prop:SimplexEntropy}, and we will also repeat many of the key definitions and intermediate assertions without further justification or explanation.
\begin{proof}
Suppose that $0<\varepsilon<\delta<2^{-3}\,\theta^{-1/2}$. Define $\varepsilon':=\varepsilon/\sqrt{d+1}=\varepsilon/2$ and $\ell:=\ceil{\log_2(\theta^{-1/2}\inv{\delta}/8)}$, so that $\ell$ is the smallest integer $i$ such that $4^i\theta\delta^2\geq 4^{-3}=(d+1)^{-d}$, and note that $1\leq\ell\lesssim\log(1/(\theta\delta^2))$. As in the proof of Proposition~\ref{Prop:SimplexEntropy}, we may assume without loss of generality that $S$ is a regular $3$-simplex with side length $\sqrt{2}$. Once again, let $T\colon\aff\triangle\to\R^3$ be an affine isometry such that $T(\triangle)=S$, and define $R_j^\dagger:=T(R_j)$ for $1\leq j\leq d+1=4$, so that $R_1^\dagger,\dotsc,R_4^\dagger$ are polytopes with disjoint interiors whose union is $S$. 

For a fixed $j\in\{1,\dotsc,4\}$, we now follow closely the second paragraph of the proof of Proposition~\ref{Prop:SimplexEntropy} and construct a family of nested polytopes within $R_j^\dagger$. For each $i\in\{1,\dotsc,\ell-1\}$, define $J_{i,j}^\dagger:=T\bigl(R_j\cap J_{4^i\theta\delta^2}^\triangle\bigr)$ and $P_{i,j}^\dagger:=T\bigl(P_{4^i\theta\delta^2,\,j}^\triangle\bigr)$. In addition, let $J_{\ell,j}^\dagger:=\emptyset$ and $P_{\ell,j}^\dagger:=\emptyset$. Then $P_{i+1,\,j}^\dagger\subseteq P_{i,j}^\dagger\subseteq J_{i,j}^\dagger\subseteq R_j^\dagger$ for all $1\leq i\leq \ell-1$, and as in~\eqref{Eq:Rjshells}, we can write $R_j^\dagger$ as the union of $(R_j^\dagger\setminus\Int P_{1,j}^\dagger),(P_{1,j}^\dagger\setminus\Int P_{2,j}^\dagger),\dotsc,(P_{\ell-1,\,j}^\dagger\setminus\Int P_{\ell,j}^\dagger)$, whose interiors are non-empty and pairwise disjoint. Each of these $\ell$ sets may be expressed as the union of $\lesssim\log^2(1/(\theta\delta^2))$ $3$-simplices with pairwise disjoint interiors, and similarly to~\eqref{Eq:Rjvols}, we have
\begin{equation}
\label{Eq:thetaRjvols}
\mu_3(R_j^\dagger\setminus P_{i,j}^\dagger)\lesssim\mu_3(R_j^\dagger\setminus J_{i,j}^\dagger)\lesssim\mu_3(Q^\triangle\setminus J_{4^i\theta\delta^2})\lesssim\mu_3([0,1/2]^3\setminus J_{4^i\theta\delta^2})\lesssim 4^i\theta\delta^2\log^2\bigl(1/(\theta\delta^2)\bigr)
\end{equation}
for all $1\leq i\leq\ell$. Thus, $\mu_3(P_{i,j}^\dagger\setminus P_{i+1,\,j}^\dagger)\leq\mu_3(R_j^\dagger\setminus P_{i+1,\,j}^\dagger)\lesssim 4^i\theta\delta^2\log^2(1/(\theta\delta^2))$ for all $1\leq i\leq\ell-1$. We emphasise again that the hidden multiplicative constants in these bounds do not depend on $i$. 

Now let the universal constants $s_3\geq 1$ and $s'>0$ be as in Lemma~\ref{Lem:hellunbds}(iii). For $\tilde{\theta}\in [1,\infty)$, define $t(\tilde{\theta})\equiv t_3(\tilde{\theta})=\log(s_3\log^3(e\tilde{\theta})-s_3+1)$ as in the proof of this result, and note that $t(\tilde{\theta})\lesssim\log\tilde{\theta}$ and $e^{t(\tilde{\theta})}\lesssim\log^3(e\tilde{\theta})$. If $f\in\mathcal{G}(f_0,\delta)$, then it follows from Lemma~\ref{Lem:hellunbds}(iii) that $\log f\leq t(\theta)+s'(3^4\delta)^{2/5}-\log\mu_3(S)$ on $S$. 
Also, for each $1\leq i\leq\ell-1$, we deduce from Lemma~\ref{Lem:invelopesimp}(iii) that $\log f(x)\geq -2^{-i+2}(3!)^{-1/2}-\log\theta-\log\mu_3(S)\geq -2^{-i+2}-\log\theta-\log\mu_3(S)$ for all $x\in P_{i,j}^\dagger\setminus\Int P_{i+1,\,j}^\dagger\subseteq J_{i,j}^\dagger$. Thus, for each $1\leq i\leq\ell-1$, it follows from the observations above and~\eqref{Eq:besimp} from Proposition~\ref{Prop:bebds} that
\begin{align*}
&H_{[\,]}(\varepsilon'/\sqrt{\ell},\mathcal{G}(f_0,\delta),\dhell,P_{i,j}^\dagger\setminus P_{i+1,\,j}^\dagger)\\
&\hspace{1cm}\lesssim\log^2\!\rbr{\frac{1}{\theta\delta^2}}\!\rbr{\frac{\bigl\{\log\theta+t(\theta)+s'(3^4\delta)^{2/5}+2^{-i+2}\bigr\}\bigl\{4^i\theta\delta^2\log^2\bigl(1/(\theta\delta^2)\bigr)\!\bigr\}^{1/2}}{\ell^{-1/2}\,\varepsilon}}^{3/2}\\
&\hspace{1cm}\lesssim\rbr{\frac{\delta}{\varepsilon}}^{3/2}\log^{17/4}\rbr{\frac{1}{\theta\delta^2}}\theta^{3/4}\,\bigl(2^i\log\theta+2^i\delta^{2/5}+4\bigr)^{3/2}\\
&\hspace{1cm}\lesssim\rbr{\frac{\delta}{\varepsilon}}^{3/2}\log^{17/4}\rbr{\frac{1}{\theta\delta^2}}\theta^{3/4}\,\bigl(2^{3i/2}\log^{3/2}\theta+2^{3i/2}\delta^{3/5}+1\bigr).
\end{align*}
By our choice of $\ell$, we have $\ell\lesssim\log(1/(\theta\delta^2))$ and $\sum_{i=1}^{\ell-1}\,2^{3i/2}\lesssim 2^{3(\ell-1)/2}-1\lesssim (\theta\delta^2)^{-3/4}$, and since $\theta>1$ and $\theta\delta^2<2^{-3}$, we conclude that
\begin{align}
H_{[\,]}(\varepsilon',\mathcal{G}(f_0,\delta),\dhell,P_{1,j}^\dagger)&\leq\sum_{i=1}^{\ell-1}H_{[\,]}(\varepsilon'/\sqrt{\ell},\mathcal{G}(f_0,\delta),\dhell,P_{i,j}^\dagger\setminus P_{i+1,\,j}^\dagger)\notag\\
&\lesssim\frac{1}{\varepsilon^{3/2}}\log^{17/4}\rbr{\frac{1}{\theta\delta^2}}\bigl(\log^{3/2}\theta+\delta^{3/5}+(\theta\delta^2)^{3/4}\,\ell\bigr)\notag\\
\label{Eq:thetabeint}
&\lesssim\frac{1}{\varepsilon^{3/2}}\log^{17/4}\rbr{\frac{1}{\theta\delta^2}}\bigl(\log^{3/2}\theta+\delta^{3/5}\bigr)+\theta^{3/4}\rbr{\frac{\delta}{\varepsilon}}^{3/2}\log^{21/4}\rbr{\frac{1}{\theta\delta^2}}.
\end{align}
Furthermore, recalling that every $f\in\mathcal{G}(f_0,\delta)$ satisfies $f\leq e^{t(\theta)+s'(3^4\delta)^{2/5}-\log\mu_3(S)}\lesssim e^{t(\theta)}\lesssim\log^3(e\theta)$ on $R_j^\dagger\setminus\Int P_{1,j}^\dagger$, we may apply the final assertion of Proposition~\ref{Prop:bebds} together with~\eqref{Eq:thetaRjvols} to deduce that
\begin{align}
H_{[\,]}(\varepsilon',\mathcal{G}(f_0,\delta),\dhell,R_j^\dagger\setminus P_{1,j}^\dagger)&\lesssim\log^2\rbr{\frac{1}{\theta\delta^2}}\frac{\theta\delta^2\log^2(1/(\theta\delta^2))\log^3(e\theta)}{\varepsilon^2}\notag\\
\label{Eq:thetabeboundary}
&\lesssim\theta\log^3(e\theta)\rbr{\frac{\delta}{\varepsilon}}^2\log^4\rbr{\frac{1}{\theta\delta^2}}.
\end{align}
Now that we have established~\eqref{Eq:thetabeint} and~\eqref{Eq:thetabeboundary} for each $1\leq j\leq d+1=4$, the result follows on noting that
\[H_{[\,]}(2^{1/2}\varepsilon,\mathcal{G}(f_0,\delta),\dhell)\leq \sum_{j=1}^4\,\bigl\{H_{[\,]}(\varepsilon',\mathcal{G}(f_0,\delta),\dhell,R_j^\dagger\setminus P_{1,j}^\dagger)+H_{[\,]}(\varepsilon',\mathcal{G}(f_0,\delta),\dhell,P_{1,j}^\dagger)\bigr\}.\]

\vspace{-0.6cm}
\end{proof}
By imitating the proof of Proposition~\ref{Prop:PolytopeEntropy}, we obtain the key local bracketing entropy result that enables us to prove Theorem~\ref{Thm:ThetaRisk}.
\begin{proposition}
\label{Prop:ThetaEntropy}
Let $d=3$ and fix $\theta\in (1,\infty)$. If $0<\varepsilon<\delta<(8\theta)^{-1/2}$ and $f_0\in\mathcal{F}^{[\theta]}(\mathcal{P}^m)$ for some $m\geq 4$, then
\begin{alignat}{3}
H_{[\,]}(2^{1/2}\varepsilon,\mathcal{G}(f_0,\delta),\dhell)&\lesssim m\,\Biggl\{\frac{\log^{3/2}\theta+\delta^{3/5}}{\varepsilon^{3/2}}\log^{17/4}\rbr{\frac{1}{\theta\delta^2}}&&+\theta^{3/4}\rbr{\frac{\delta}{\varepsilon}}^{3/2}\log^{21/4}\rbr{\frac{1}{\theta\delta^2}}\notag\\
&&&+\theta\log^3(e\theta)\rbr{\frac{\delta}{\varepsilon}}^2\log^4\rbr{\frac{1}{\theta\delta^2}}\Biggr\}\notag\\
\label{Eq:ThetaEntropy}
&=m H_{3,\theta}(\delta,\varepsilon).&&
\end{alignat}
\end{proposition}
\begin{proof}
Suppose that $0<\varepsilon<\delta<(8\theta)^{-1/2}$. By Proposition~\ref{Lem:euler}, we can find $M\leq 6m$ 3-simplices $S_1,\dotsc,S_M$ with pairwise disjoint interiors whose union is $P:=\supp f_0$. Set $\alpha_j:=\{\mu_3(S_j)/\mu_3(P)\}^{1/2}$ for each $1\leq j\leq M$, so that $\sum_{j=1}^M\alpha_j^2=1$. For each $f\in\mathcal{G}(f_0,\delta)$ and $1\leq j\leq M$, let $n_j(f)$ be the smallest $n_j\in\N$ for which $\int_{S_j}\,\bigl(\sqrt{f}-\sqrt{f_0}\bigr)^2\leq\alpha_j^2\,n_j\delta^2$, so that $\sum_{j=1}^M\alpha_j^2\,n_j(f)\leq 2$, as in~\eqref{Eq:njfsum}. We also claim that $n_j(f)\lesssim\log^3(e\theta)\,\delta^{-2}$ for all $1\leq j\leq M$. To see this, let $t(\theta)\equiv t_3(\theta)$ be as in the proof of Lemma~\ref{Lem:hellunbds} and note that since $f_0\in\mathcal{F}^{[\theta]}(P)$, $\delta<(8\theta)^{-1/2}$ and $f\in\mathcal{G}(f_0,\delta)$, it follows from Lemma~\ref{Lem:hellunbds}(iii) that
\begin{equation}
\label{Eq:thetaunifbd}
0\leq f\vee f_0\lesssim e^{t(\theta)-\log\mu_3(P)}\lesssim\log^3(e\theta)\,\mu_3(P)^{-1}=\log^3(e\theta)\,f_P\:\;\text{on $P$}.
\end{equation}
Thus, we have $\bigl(\sqrt{f}-\sqrt{f_0}\bigr)^2\leq f\vee f_0\lesssim\log^3(e\theta)\,f_P=\log^3(e\theta)\,\inv{\mu_3(P)}$ on $P$, so $\int_{S_j}\,\bigl(\sqrt{f}-\sqrt{f_0}\bigr)^2\lesssim\log^3(e\theta)\,\mu_3(S_j)/\mu_3(P)=\log^3(e\theta)\,\alpha_j^2$ for all $j$. Recalling the definition of $n_j(f)$, we deduce that $n_j(f)\lesssim\log^3(e\theta)\,\delta^{-2}$ for all $j$, as required.

Now let $U:=\{(n_1(f),\dotsc,n_M(f)):f\in\mathcal{G}(f_0,\delta)\}$, and for each $(n_1,\dotsc,n_M)\in U$, define \[\mathcal{G}(f_0,\delta;n_1,\dotsc,n_M):=\{f\in\mathcal{G}(f_0,\delta):n_j(f)=n_j\text{ for all }1\leq j\leq M\}.\]
Since $\mathcal{G}(f_0,\delta)$ is the union of these subclasses, it follows that
\[N_{[\,]}(2^{1/2}\varepsilon,\mathcal{G}(f_0,\delta),\dhell)\leq\sum_{(n_1,\dotsc,n_M)\in U}N_{[\,]}(2^{1/2}\varepsilon,\mathcal{G}(f_0,\delta;n_1,\dotsc,n_M),\dhell),\]
so
\begin{equation}
\label{Eq:thetalbesum}
H_{[\,]}(2^{1/2}\varepsilon,\mathcal{G}(f_0,\delta),\dhell)\leq\log\,\abs{U}+\max_{(n_1,\dotsc,n_M)\in U}H_{[\,]}(2^{1/2}\varepsilon,\mathcal{G}(f_0,\delta;n_1,\dotsc,n_M),\dhell).
\end{equation}
Since $n_j(f)\lesssim\log^3(e\theta)\,\delta^{-2}$ for all $1\leq j\leq M$, we deduce that $\abs{U}\lesssim\log^{3M}(e\theta)\,\delta^{-2M}$ and hence that
\begin{equation}
\label{Eq:thetaucardbd}
\log\,\abs{U}\lesssim M\bigl(\log\log(e\theta)+\log(1/\delta)\bigr)\lesssim m\bigl(\log\log(e\theta)+\log(1/\delta)\bigr).
\end{equation}
Next, we bound the second term on the right-hand side of~\eqref{Eq:thetalbesum}. Fix $j\in\{1,\dotsc,M\}$ and $(n_1,\dotsc,n_M)\in U$. If $f\in\mathcal{G}(f_0,\delta;n_1,\dotsc,n_M)$, then $\int_{S_j}\,\bigl(\sqrt{f}-\sqrt{f_0}\bigr)^2\leq\alpha_j^2\,n_j\delta^2$, so $\int_{S_j}\,\bigl\{(\alpha_j^{-2}f)^{1/2}-(\alpha_j^{-2}f_0)^{1/2}\bigr\}^2\leq n_j\delta^2$. This shows that
$\alpha_j^{-2}f\Ind_{S_j}\in\mathcal{G}(\alpha_j^{-2}f_0\Ind_{S_j},\sqrt{n_j}\,\delta)$. Observe also that $\alpha_j^{-2}f_0\Ind_{S_j}\geq\inv{\theta}f_{S_j}$, whence $\alpha_j^{-2}f_0\Ind_{S_j}\in\mathcal{F}^{[\theta]}(S_j)$. Furthermore, it follows from~\eqref{Eq:thetaunifbd} that 
\begin{equation}
\label{Eq:thetaunifbdSj}
0\leq\alpha_j^{-2}f\Ind_{S_j}=\mu_3(P)f\Ind_{S_j}/\mu_3(S_j)\lesssim\log^3(e\theta)\,\mu_3(S_j)^{-1}\Ind_{S_j}=\log^3(e\theta)\,f_{S_j}\:\;\text{on $S_j$}.
\end{equation}
Suppose first that $\sqrt{n_j}\,\delta<2^{-3}\,\theta^{-1/2}$. Since $\alpha_j^{-2}f_0\Ind_{S_j}\in\mathcal{F}^{[\theta]}(S_j)$ and $0<\sqrt{n_j}\,\varepsilon<\sqrt{n_j}\,\delta<2^{-3}\,\theta^{-1/2}$, we can apply Proposition~\ref{Prop:ThetaSimplexEntropy} to deduce that there exists a $\sqrt{n_j}\,\varepsilon$-Hellinger bracketing set $\{[g_\ell^L,g_\ell^U]:1\leq\ell\leq N_j\}$ for $\mathcal{G}(\alpha_j^{-2}f_0\Ind_{S_j},\sqrt{n_j}\,\delta)$ such that \[\log N_j\lesssim H_{3,\theta}(\sqrt{n_j}\,\delta,\sqrt{n_j}\,\varepsilon)\lesssim H_{3,\theta}(\delta,\varepsilon).\]
Arguing as in the proof of Proposition~\ref{Prop:PolytopeEntropy}, we see that $\{[\alpha_j^2\,g_\ell^L,\alpha_j^2\,g_\ell^U]:1\leq\ell\leq N_j\}$ is an $(\alpha_j\sqrt{n_j}\,\varepsilon)$-Hellinger bracketing set for $\{f\Ind_{S_j}:f\in\mathcal{G}(f_0,\delta;n_1,\dotsc,n_M)\}$, which implies that
\begin{equation}
\label{Eq:thetabeSj}
H_{[\,]}\bigl(\alpha_j\sqrt{n_j}\,\varepsilon,\mathcal{G}(f_0,\delta;n_1,\dotsc,n_M),\dhell,S_j\bigr)\lesssim H_{3,\theta}(\delta,\varepsilon),
\end{equation}
provided that $\sqrt{n_j}\,\delta<2^{-3}\,\theta^{-1/2}$. 

We now verify that~\eqref{Eq:thetabeSj} remains valid even when $\sqrt{n_j}\,\delta\geq 2^{-3}\,\theta^{-1/2}$. In this case, we define $B_j:=\log(\log^3(e\theta)\mu_3(S_j)^{-1})$ and deduce from~\eqref{Eq:thetaunifbdSj} that for a sufficiently large universal constant $C>0$, we have  $\{\alpha_j^{-2}f\Ind_{S_j}:f\in\mathcal{G}(f_0,\delta;n_1,\dotsc,n_M)\}\subseteq\mathcal{G}_{-\infty,\,CB_j}(S_j)$. By the final bound~\eqref{Eq:beconvubd} from Proposition~\ref{Prop:bebds}, we can find a $\sqrt{n_j}\,\varepsilon$-Hellinger bracketing set $\{[\tilde{g}_\ell^L,\tilde{g}_\ell^U]:1\leq\ell\leq\tilde{N}_j\}$ for $\mathcal{G}_{-\infty,\,CB_j}(S_j)$ such that 
\[\log\tilde{N}_j\lesssim h_3\rbr{\frac{\sqrt{n_j}\,\varepsilon}{e^{B_j/2}\mu_3(S_j)^{1/2}}}\lesssim h_3\rbr{\frac{\sqrt{n_j}\,\varepsilon}{\log^{3/2}(e\theta)}}\lesssim h_3\rbr{\frac{\varepsilon}{\{\theta\log^3(e\theta)\}^{1/2}\,\delta}}\lesssim H_{3,\theta}(\delta,\varepsilon).\]
Indeed, the penultimate inequality above follows since $\sqrt{n_j}\,\delta\gtrsim\theta^{-1/2}$ and $h_3\colon\eta\mapsto\eta^{-2}$ is decreasing, and the final inequality is easily verified. As above, we see that $\{[\alpha_j^2\,\tilde{g}_\ell^L,\alpha_j^2\,\tilde{g}_\ell^U]:1\leq\ell\leq\tilde{N}_j\}$ an $(\alpha_j\sqrt{n_j}\,\varepsilon)$-Hellinger bracketing set for $\{f\Ind_{S_j}:f\in\mathcal{G}(f_0,\delta;n_1,\dotsc,n_M)\}$, which implies that
\[H_{[\,]}\bigl(\alpha_j\sqrt{n_j}\,\varepsilon,\mathcal{G}(f_0,\delta;n_1,\dotsc,n_M),\dhell,S_j\bigr)\leq\log\tilde{N}_j\lesssim H_{3,\theta}(\delta,\varepsilon)\]
and hence that~\eqref{Eq:thetabeSj} holds when $\sqrt{n_j}\,\delta\geq 2^{-3}\,\theta^{-1/2}$, as required.

Finally, since $(n_1,\dotsc,n_M)\in U$ and $\sum_{j=1}^M\alpha_j^2\,n_j(f)\leq 2$ for all $f\in\mathcal{G}(f_0,\delta)$, it follows from the definition of $U$ that $\sum_{j=1}^M\alpha_j^2\,n_j\,\varepsilon^2\leq 2\varepsilon^2$. Having established~\eqref{Eq:thetabeSj} for all $1\leq j\leq M$, we conclude that
\begin{align*}
H_{[\,]}\bigl(2^{1/2}\varepsilon,\mathcal{G}(f_0,\delta;n_1,\dotsc,n_M),\dhell\bigr)&\leq\sum_{j=1}^M H_{[\,]}\bigl(\alpha_j\sqrt{n_j}\,\varepsilon,\mathcal{G}(f_0,\delta;n_1,\dotsc,n_M),\dhell,S_j\bigr)\\
&\lesssim M H_{3,\theta}(\delta,\varepsilon)\lesssim m H_{3,\theta}(\delta,\varepsilon)
\end{align*}
whenever $0<\varepsilon<\delta<(8\theta)^{-1/2}$ and $(n_1,\dotsc,n_M)\in U$. Together with~\eqref{Eq:thetalbesum} and~\eqref{Eq:thetaucardbd}, this implies the desired conclusion.
\end{proof}

\section{Technical preparation for Sections~\ref{Sec:Logkaffine} and~\ref{Sec:StatSupp}}
\label{Sec:LogkaffineSupp}
In this section, we require some further concepts and definitions that are of a more technical nature compared with those already outlined in Section~\ref{Subsec:Notation} in the main text. Accessible introductions to this material can be found in~\citet{Sch14} and~\citet{Rock97}.

For $x\in\R^d$ and $r>0$, let $B(x,r):=\{w\in\R^d:\norm{w-x}<r\}$ and $\bar{B}(x,r):=\{w\in\R^d:\norm{w-x}\leq r\}$.
Recall that a \textit{line} is a set of the form $\{x+\lambda u:\lambda\in\R\}$ and that a \textit{ray} is a set of the form $\{x+\lambda u:\lambda\geq 0\}$, where $x\in\R^d$ and $u\in\R^d\setminus\{0\}$. For $A\subseteq\R^d$, let $\conv A$, $\aff A$, $\Span A$ respectively denote the convex hull, affine hull and linear span of $A$. 

A \textit{cone} is a set $C\subseteq\R^d$ with the property that $\lambda C\subseteq C$ for all $\lambda>0$. We say that $C$ is \textit{pointed} if $C\cap(-C)=\{0\}$. If $C$ is a non-empty, closed, convex cone, then the \textit{dual cone} $\xst{C}:=\{\alpha\in\R^d:\tm{\alpha}x\geq 0\text{ for all } x\in C\}$ is also closed and convex, and we have $C^{**}=C$ \citep[Theorem~1.6.1]{Sch14}. 

If $E\subseteq\R^d$ is non-empty and convex, then its \textit{relative interior} $\relint E$ is defined as the interior of $E$ within the ambient space $\aff E$, and we write $\partial E:=(\Cl E)\setminus(\relint E)$ for the \textit{relative boundary} of $E$. It is always the case that $\partial E=\partial(\Cl E)$ and $\mu_d(\partial E)=0$; see~\citet[Theorem~1.1.15(c)]{Sch14} and~\citet{Lang86} for example. If in addition $E$ is closed, then the \textit{recession cone} $\rec(E):=\{u\in\R^d:E+u\subseteq E\}$ is closed and convex, and we have $\rec(E)=\{0\}$ if and only if $E$ is compact \citep[Theorem 8.4]{Rock97}. 

Let $\mathcal{K}\equiv\mathcal{K}_d$ denote the collection of all closed, convex sets $K\subseteq\R^d$ with non-empty interior, and let $\mathcal{K}^b\equiv\mathcal{K}_d^b$ be the collection of all bounded $K\in\mathcal{K}_d$. 
We say that $K\in\mathcal{K}$ is \textit{line-free} if $K$ does not contain a line; i.e.\ for all $x\in K$ and $u\in\R^d\setminus \{0\}$, there exists some $\lambda\in\R$ such that $x+\lambda u\notin K$. Also, if $K\in\mathcal{K}$, then $K=\Cl\Int K$ \citep[Theorem~1.1.15(b)]{Sch14} and $\Exp K\subseteq\Ext K$, where $\Ext K$ and $\Exp K$ respectively denote the sets of extreme points and exposed points of $K$; see Section~\ref{Subsec:Notation} for the relevant definitions. For $K\in\mathcal{K}$, Straszewicz's theorem \citep[Theorem~1.4.7]{Sch14} asserts that $\Ext K \subseteq \Cl\Exp K$. Moreover, for each $K\in\mathcal{K}$ with $0\in\Int K$, the \textit{Minkowski functional} of $K$ is the function $\rho_K\colon\R^d\to [0,\infty)$ defined by $\rho_K(x):=\inf\{\lambda>0:x\in\lambda K\}$, which is easily seen to be positively homogeneous
(i.e.\ $\rho_K(\lambda x)=\lambda\rho_K(x)$ for all $\lambda>0$ and $x\in\R^d$) and subadditive (i.e.\ $\rho_K(x+y)\leq\rho_K(x)+\rho_K(y)$ for all $x,y\in\R^d$), and therefore convex; see~\citet[Section~1.7]{Sch14} for example.

Finally, if $f\colon S\to\R\cup\{\infty\}$ is a function whose domain $S$ is a subset of $\R^d$, then the \emph{epigraph} of $f$ is the set $\{(x,t)\in S\times\R:f(x)\leq t\}$.

\subsection{Properties of log-concave, log-\texorpdfstring{$k$}{k}-affine densities}
\label{Subsec:LogkaffineDensities}

The results in this section provide a basis for the definition of the complexity parameter $\Gamma(f)$ in Section~\ref{Sec:Logkaffine}, as well as for some key calculations 
in the derivation of the key local bracketing entropy bound (Proposition~\ref{Prop:Log1affEntropy}) in Section~\ref{Sec:LogkaffineProofs}. Some of the propositions below are of independent interest; in particular, we obtain an explicit parametrisation of the subclass $\mathcal{F}^1$ of log-1-affine densities in $\mathcal{F}_d$ (Proposition~\ref{Prop:log1aff}) and also provide a proof of Proposition~\ref{Prop:Logkaff} in the main text, which elucidates the geometric structure of log-concave, log-$k$-affine functions with polyhedral support. Much of the requisite convex analysis background and notation is set out above.
The subclass $\mathcal{F}^1$ is not to be confused with the subclass $\mathcal{F}_1^{0,1}$ studied in Section~\ref{Subsec:Envelope}.

To begin with, we state and prove two results from convex analysis, the second of which (Proposition~\ref{Prop:kslices}) plays a crucial role in the subsequent theoretical development. A key ingredient in the proof of Proposition~\ref{Prop:kslices} is the powerful Brunn--Minkowski inequality~\citep[Theorem~7.1.1]{Sch14}. 

\begin{lemma}
\label{Lem:coneprops}
Let $C\subseteq\R^d$ be a non-empty, closed, convex cone. Then we have the following:
\begin{enumerate}[label=(\roman*)]
\item $\Int\xst{C}=\{\alpha\in\R^d:\tm{\alpha}x>0\text{ for all }x\in C\setminus\{0\}\}$.
\item $C$ is pointed if and only if $\Int(\xst{C})$ is non-empty.
\end{enumerate}
\end{lemma}
This appears as Exercise~B.16 in~\citet{BTN15}, and we provide a proof here for convenience.

\begin{proof}
Let $h_C\colon\R^d\to\R$ be the function defined by $h_C(\alpha):=\inf\{\tm{\alpha}x:x\in C\cap S^{d-1}\}$, and observe that since $h_C(\alpha)\geq h_C(\alpha')+h_C(\alpha-\alpha')$ for all $\alpha,\alpha'\in\R^d$, it follows that $h_C$ is in fact 1-Lipschitz with respect to the Euclidean norm. Indeed, we have
\[\abs{h_C(\alpha)-h_C(\alpha')}\leq\max\{-h_C(\alpha-\alpha'),-h_C(\alpha'-\alpha)\}\leq\norm{\alpha-\alpha'}\]
for all $\alpha,\alpha'\in\R^d$. Since $h_C$ is positively homogeneous (i.e.\ $h_C(\lambda\alpha)=\lambda h_C(\alpha)$ for all $\lambda>0$ and $\alpha\in\R^d$), we have $\alpha\in\xst{C}$ if and only if $h_C(\alpha)\geq 0$. Now fix $\alpha\in\R^d$. If $\tm{\alpha}x>0$ for all $x\in C\setminus\{0\}$, then since $h_C$ is continuous and $C\cap S^{d-1}$ is compact, it follows that $h_C(\alpha)>0$ and hence that $\alpha\in\Int(\xst{C})$. Conversely, if there exists $x\in C\setminus\{0\}$ such that $\tm{\alpha}x\leq 0$, then fix $v\in\R^d$ such that $\tm{\alpha}v<0$ and note that $\tm{\alpha}(x+\varepsilon v)<0$ for all $\varepsilon>0$. This implies that $\alpha\notin\Int(\xst{C})$, so the proof of (i) is complete.

For (ii), observe that $\xst{C}$ has empty interior if and only if $\Span(\xst{C})$ has dimension at most $d-1$, which is equivalent to saying that there exists $x\in\R^d\setminus\{0\}$ such that $\tm{\alpha}x=0$ for all $\alpha\in\xst{C}$. If this latter condition holds, then $x$ and $-x$ both belong to $C^{**}=C$, so $C$ is not pointed. On the other hand, if there exists $x\in\R^d\setminus\{0\}$ such that $x$ and $-x$ lie in $C$, it follows from the definition of $\xst{C}$ that $\tm{\alpha}x=0$ for all $\alpha\in\xst{C}$, so the converse is also true.
\end{proof}
For $K\in\mathcal{K}$ and $\alpha\in\R^d$, let $m_{K,\alpha}:=\inf_{x\in K}\tm{\alpha}x$ and $M_{K,\alpha}:=\sup_{x\in K}\tm{\alpha}x$, and for each $t\in\R$, define the closed, convex sets
\begin{equation}
\label{Eq:kalphat}
K_{\alpha,t}^+:=K\cap\{x\in\R^d:\tm{\alpha}x\leq t\}\quad\text{and}\quad K_{\alpha,t}:=K\cap\{x\in\R^d:\tm{\alpha}x=t\}.
\end{equation}
\begin{proposition}
\label{Prop:kslices}
Let $K\in\mathcal{K}$ and $\alpha\in\R^d$. Then we have the following:
\begin{enumerate}[label=(\roman*)]
\item $K_{\alpha,t}^+$ is compact for all $t\in\R$ if and only if $\alpha\in\Int(\xst{\rec(K)})$. 
\item If $\alpha\in\Int(\xst{\rec(K)})\setminus\{0\}$, then $m_{K,\alpha}$ is finite and $K_{\alpha,\,m_{K,\alpha}}$ is a non-empty exposed face of $K$. Moreover, if $d\geq 2$, then the function $t\mapsto\mu_{d-1}(K_{\alpha,t})^{1/(d-1)}$ is concave, finite-valued and strictly positive on $(m_{K,\alpha},M_{K,\alpha})$. 
\end{enumerate}
\end{proposition}
\begin{proof}
(i) By taking $C:=\rec(K)$ in Lemma~\ref{Lem:coneprops}(i), we see that
\begin{equation}
\label{Eq:intreccone}
\Int(\xst{\rec(K)})=\bigl\{\alpha\in\R^d:\tm{\alpha}u>0\text{ for all }u\in\rec(K)\setminus\{0\}\bigr\}.
\end{equation}
If $\alpha\notin\Int(\xst{\rec(K)})$, then there exists $u\in\rec(K)\setminus\{0\}$ such that $\tm{\alpha}u\leq 0$. Thus, if $K_{\alpha,t}^+$ is non-empty, then $x+\lambda u\in K$ and $\tm{\alpha}(x+\lambda u)\leq\tm{\alpha}x\leq t$ for all $x\in K_{\alpha,t}^+$ and $\lambda>0$, so $K_{\alpha,t}^+$ is unbounded.

If $\alpha\in\Int(\xst{\rec(K)})$ and $t\in\R$ are such that $K_{\alpha,t}^+$ is non-empty, let $H^+:=\{u\in\R^d:\tm{\alpha}u>0\}$ and $H^-:=\{u\in\R^d:\tm{\alpha}u\leq 0\}$. Note that $\rec(K_{\alpha,t}^+)\setminus\{0\}\subseteq\rec(K)\setminus \{0\}$, which by~\eqref{Eq:intreccone} is disjoint from $H^-$. Moreover, if $x\in K_{\alpha,t}^+$ and $u\in H^+$, then there exists $\lambda>0$ such that $\tm{\alpha}(x+\lambda u)>t$. Thus, since $x+\lambda u\notin K_{\alpha,t}^+$, it follows that $u\notin\rec(K_{\alpha,t}^+)$. We conclude that $\rec(K_{\alpha,t}^+)=\{0\}$ and therefore that $K_{\alpha,t}^+$ is compact \citep[Theorem~8.4]{Rock97}.

(ii) For $\alpha\in\Int(\xst{\rec(K)})\setminus\{0\}$, if we fix $y\in K_{\alpha,t}^+$ and set $s:=\tm{\alpha}y$, then it follows from the compactness of $K_{\alpha,s}^+$ that
\[
s\geq m_{K,\alpha}=\inf_{x \in K} \tm{\alpha}x = \inf_{x\in K_{\alpha,s}^+}\tm{\alpha}x > -\infty,
\]
and that the infimum is attained at some $z\in K_{\alpha,s}^+$ with $\tm{\alpha}z=m_{K,\alpha}$. It is now clear that $K_{\alpha,\,m_{K,\alpha}}$ is a non-empty exposed face of $K$. Finally, if $d\geq 2$, fix $\lambda\in (0,1)$ and $t_1,t_2\in\R$ with $m_{K,\alpha}\leq t_1<t_2\leq M_{K,\alpha}$, and set $t:=\lambda t_1+(1-\lambda)t_2$. Also, for $j=1,2$, fix $a_j\in K_{\alpha,t_j}$ and let $K_j':=K_{\alpha,t_j}-a_j$. Setting $a:=\lambda a_1+(1-\lambda)a_2\in K_{\alpha,t}$, we see that $K':=K_{\alpha,t}-a$, $K_1'$ and $K_2'$ are contained in the $(d-1)$-dimensional subspace $H:=\{u\in\R^d:\tm{\alpha}u=0\}$. Since $K_{\alpha,t_2}^+$ is compact and convex, the sets $K_{\alpha,t}$, $K_{\alpha,t_1}$ and $K_{\alpha,t_2}$ are all non-empty and compact, and we have $K_{\alpha,t}\supseteq\lambda K_{\alpha,t_1}+(1-\lambda)K_{\alpha,t_2}$. This implies that $K'\supseteq\lambda K_1'+(1-\lambda)K_2'$, so taking the ambient space to be $H$, we can apply the Brunn--Minkowski inequality \citep[Theorem~7.1.1]{Sch14} to deduce that
\begin{align*}
\mu_{d-1}(K_{\alpha,t})^{1/(d-1)}=\mu_{d-1}(K')^{1/(d-1)}&\geq\lambda\mu_{d-1}(K_1')^{1/(d-1)}+(1-\lambda)\mu_{d-1}(K_2')^{1/(d-1)}\\
&=\lambda\mu_{d-1}(K_{\alpha,t_1})^{1/(d-1)}+(1-\lambda)\mu_{d-1}(K_{\alpha,t_2})^{1/(d-1)}.
\end{align*}
Thus, $t\mapsto\mu_{d-1}(K_{\alpha,t})^{1/(d-1)}$ is indeed concave and finite-valued on $(m_{K,\alpha},M_{K,\alpha})$. Since 
\begin{equation}
\label{Eq:volident}
0<\mu_d(K)=\inv{\norm{\alpha}}\int_{m_{K,\alpha}}^{M_{K,\alpha}}\mu_{d-1}(K_{\alpha,t})\,dt,
\end{equation}
we deduce from this that $\mu_{d-1}(K_{\alpha,t})>0$ for all $t\in (m_{K,\alpha},M_{K,\alpha})$, as required. To verify the identity~\eqref{Eq:volident}, one can proceed as follows: let $\{u_1,\dotsc,u_d\}$ be an orthonormal basis for $\R^d$ such that $u_d=\alpha/\norm{\alpha}$, and let $Q\colon\R^d\to\R^d$ be the invertible linear map defined by setting $Qu_j=e_j$ for $j = 1,\dotsc,d-1$ and $Qu_d=\norm{\alpha}e_d$. Since $\det Q=\norm{\alpha}$ and $\tm{\alpha}\inv{Q}w=\tm{\alpha}(w_d\,u_d/\norm{\alpha})=w_d$ for all $w=(w_1,\dotsc,w_d)\in\R^d$, it follows that 
\begin{align*}
\mu_d(K)&=\int_{Q(K)}\!\inv{\norm{\alpha}}\,dw=\inv{\norm{\alpha}}\!\int_{m_{K,\alpha}}^{M_{K,\alpha}}\!\!\mu_{d-1}(\{w\in\inv{Q}(K):w_d=t\})\,dw_d\\
&=\inv{\norm{\alpha}}\!\int_{\,m_{K,\alpha}}^{\,M_{K,\alpha}}\!\!\mu_{d-1}(K_{\alpha,t})\,dt,
\end{align*}
as claimed, where we have used Fubini's theorem to obtain the first equality.
\end{proof}

Next, we obtain a useful geometric characterisation of the sets $K\in\mathcal{K}$ for which $\Int(\xst{\rec(K)})$ is non-empty.
\begin{proposition}
\label{Prop:linefree}
For a fixed $K\in\mathcal{K}$, the following are equivalent:

\begin{minipage}{0.45\textwidth}
\vspace{0.2cm}
\begin{enumerate}[label=(\roman*),leftmargin=0.9cm]
\item $K$ is line-free;
\setcounter{enumi}{2}
\item $\Int(\xst{\rec(K)})$ is non-empty;
\end{enumerate}
\end{minipage}
\begin{minipage}{0.5\textwidth}
\vspace{0.2cm}
\begin{enumerate}[label=(\roman*),leftmargin=0.9cm]
\setcounter{enumi}{1}
\item $\rec(K)$ is a pointed cone;	
\setcounter{enumi}{3}
\item $K$ has at least one exposed point.
\end{enumerate}
\end{minipage}
\end{proposition}
\begin{proof}
(i) $\Leftrightarrow$ (ii): If $K$ contains the line $L := \{y+\lambda u:\lambda \in \R\}$ and $x \in K$, then $x+\lambda u\in\Cl\conv\bigl(\{x\}\cup L\bigr)\subseteq K$ for every $\lambda\in\R$, so $\{\lambda u:\lambda\in\R\}\subseteq\rec(K)$.  Therefore, $\rec(K)$ is not pointed.  Conversely, if $\rec(K)$ is not pointed, then there exists $u\in\R^d\setminus\{0\}$ such that $u\in\rec(K)\cap (-\rec(K))$, so $K+\lambda u\subseteq K$ for all $\lambda\in\R$.  Therefore, $K$ is not line-free.

(iv) $\Rightarrow$ (i): As above, if there exist $y\in K$ and $u\in\R^d\setminus \{0\}$ such that $y+\lambda u\in K$ for all $\lambda\in\R$, then $x+\lambda u\in K$ for all $x\in K$ and $\lambda\in\R$. In particular, $K$ has no extreme points, so it has no exposed points.

(ii) $\Leftrightarrow$ (iii): This follows directly from Lemma~\ref{Lem:coneprops}(ii).

(iii) $\Rightarrow$ (iv): If (iii) holds, then there exists $\alpha\in\Int(\xst{\rec(K)})\setminus\{0\}$. For a fixed $t\in (m_{K,\alpha},\infty)$, we know from Proposition~\ref{Prop:kslices} that $K_{\alpha,\,m_{K,\alpha}}$ is a non-empty exposed face of $K_{\alpha,t}^+$, which is $d$-dimensional, compact and convex. Thus, $K_{\alpha,\,m_{K,\alpha}}$ must itself be compact and convex, so it has at least one extreme point $z$ \citep[Corollary~1.4.4]{Sch14}. Now $z$ is necessarily an extreme point of $K_{\alpha,t}^+$, so it follows from Straszewicz's theorem \citep[Theorem~1.4.7]{Sch14} that $z$ is the limit of a sequence of exposed points of $K_{\alpha,t}^+$. But by the convexity of $K$, every exposed point of $K_{\alpha,t}^+$ must be an exposed point of $K$, so (iv) holds, as required.
\end{proof}
Using the above results, we now derive necessary and sufficient conditions for a density $f\colon\R^d\to [0,\infty)$ to belong to the subclass $\mathcal{F}^1$ of log-1-affine densities in $\mathcal{F}_d$. 
\begin{proposition}
\label{Prop:log1aff}
For $K\in\mathcal{K}$ and $\alpha\in\R^d$, the function $g_{K,\alpha}\colon\R^d\to[0,\infty)$ defined by $g_{K,\alpha}(x):=\exp(-\tm{\alpha}x)\Ind_{\{x\in K\}}$ is integrable if and only if $K$ is line-free and $\alpha\in\Int(\xst{\rec(K)})$. It follows from this that 
\begin{equation}
\label{Eq:F1}
\mathcal{F}^1=\bigl\{f_{K,\alpha}:=\inv{c_{K,\alpha}}\,g_{K,\alpha}:K\in\mathcal{K},\, K\text{ is line-free and }\alpha\in\Int(\xst{\rec(K)})\bigr\},
\end{equation}
where $c_{K,\alpha}:=\int_K\,\exp(-\tm{\alpha}x)\,dx$.
\end{proposition}
\begin{proof}
The result is clear if $d=1$, so suppose now that $d\geq 2$. First we consider the case where $\alpha = (\alpha_1,\dotsc,\alpha_d)\notin\Int(\xst{\rec(K)})$, which by~Proposition~\ref{Prop:linefree} covers all instances where $K$ is not line-free. By~\eqref{Eq:intreccone}, there exists $u\in\rec(K)\setminus \{0\}$ such that $\tm{\alpha}u\leq 0$, and we may assume without loss of generality that $u=e_d$. Then for a fixed $x\in\Int K$, we can find $\varepsilon>0$ such that $R_{x,\varepsilon}:=\bigl(\prod_{j=1}^{d-1}\,[x_j-\varepsilon,x_j+\varepsilon]\bigr)\times [x_d-\varepsilon,\infty)\subseteq K$. But since $\alpha_d=\tm{\alpha}e_d\leq 0$, we have
\[c_{K,\alpha}\geq\int_{R_{x,\varepsilon}}\exp(-\tm{\alpha}w)\,dw=\int_{x_d-\varepsilon}^\infty\,\exp(-\alpha_dw_d)\,dw_d\times\prod_{j=1}^{d-1}\,\int_{x_j-\varepsilon}^{x_j+\varepsilon}\,\exp(-\alpha_jw_j)\,dw_j=\infty,\]
so $f_{K,\alpha}$ is not integrable.

Now suppose that $K$ is line-free and $\alpha\in\Int(\xst{\rec(K)})$. By~\eqref{Eq:intreccone}, the case $\alpha=0$ is possible if and only if $\rec(K) = \{0\}$. But by~\citet[Theorem~8.4]{Rock97}, this is equivalent to requiring that $K$ be compact, in which case the result is clear. We can therefore assume that $\alpha\neq 0$. By the final assertion of Proposition~\ref{Prop:kslices}(ii), the function $t\mapsto\mu_{d-1}(K_{\alpha,t})^{1/(d-1)}$ is concave and takes strictly positive values on $(m_{K,\alpha},M_{K,\alpha})$, so there exist $a,b\in\R$ such that $\mu_{d-1}(K_{\alpha,t})\leq\abs{at+b}^{d-1}$ for all $t$ in this range. By analogy with~\eqref{Eq:volident} and its derivation, we have
\[c_{K,\alpha}=\int_K e^{-\tm{\alpha}x}\,dx=\inv{\norm{\alpha}}\int_{\,m_{K,\alpha}}^{\,M_{K,\alpha}}\mu_{d-1}(K_{\alpha,t})\,e^{-t}\,dt\leq\inv{\norm{\alpha}}\int_{\,m_{K,\alpha}}^{\,M_{K,\alpha}}\abs{at+b}^{d-1}e^{-t}\,dt<\infty,\]
as required. The final assertion of the proposition now follows immediately.
\end{proof}
\begin{figure}[htb]
\vspace{-0.2cm}
\centering
\includegraphics[width=0.8\textwidth]{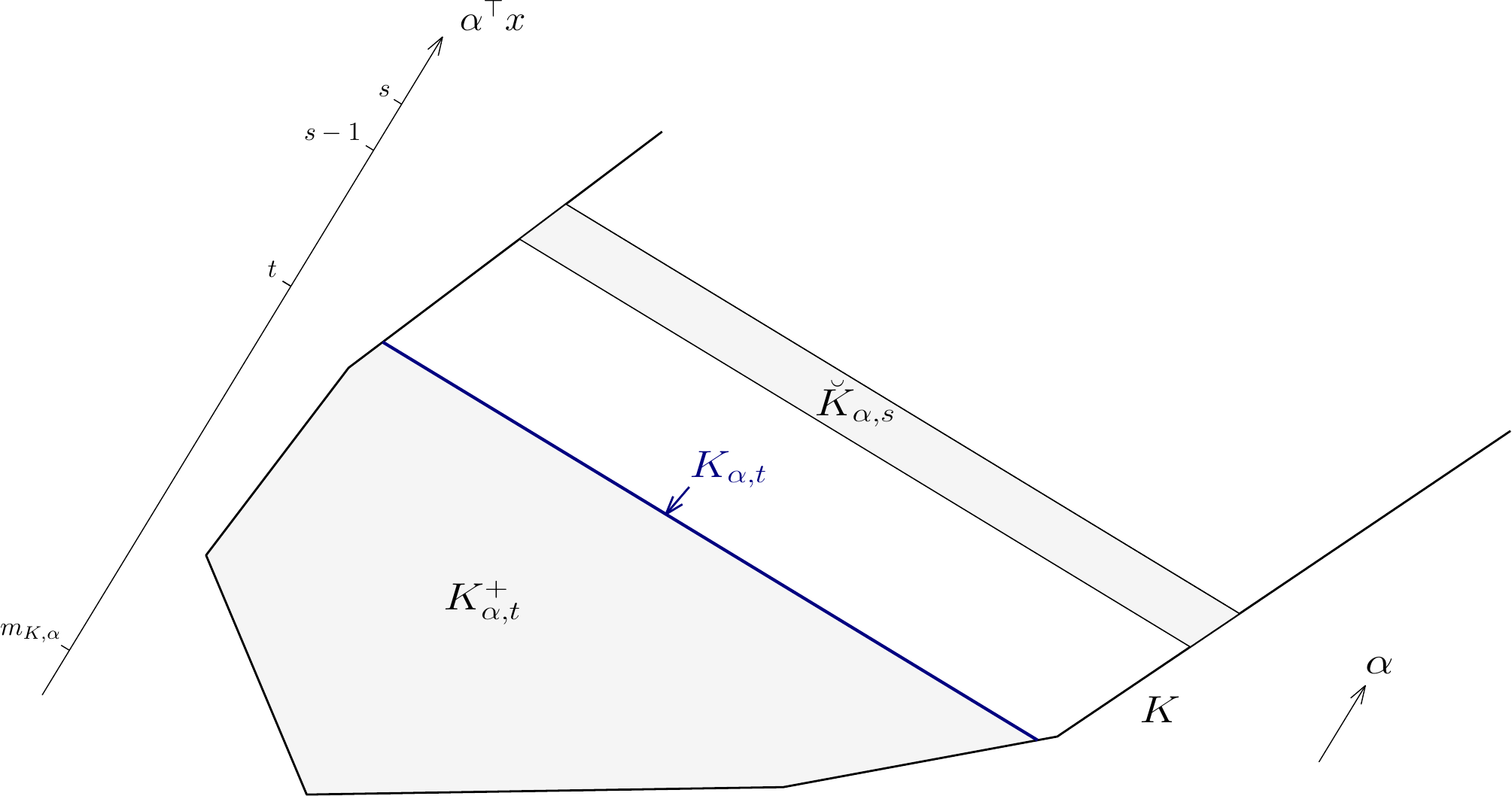}

\vspace{0.3cm}
\caption{Illustration of the sets $K_{\alpha,t}^+$, $K_{\alpha,t}$ and $\breve{K}_{\alpha,s}$.}
\label{Fig:Kalphat}
\end{figure}
Now let $\mathcal{F}_\star^1\equiv\mathcal{F}_{d,\star}^1$ denote the collection of all $f_{K,\alpha}\in\mathcal{F}^1$ for which $m_{K,\alpha}=0$. An immediate consequence of Proposition~\ref{Prop:kslices}(ii) and Proposition~\ref{Prop:log1aff} is the following: 
\begin{corollary}
\label{Cor:log1affcanon}
If $X\sim f\in\mathcal{F}^1$, there exists $x\in\R^d$ such that the density of $X-x$ lies in $\mathcal{F}_\star^1$. 
\end{corollary}
If $f_{K,\alpha}\in\mathcal{F}^1$, then by Proposition~\ref{Prop:kslices}(i) and Proposition~\ref{Prop:log1aff}, the sets $K_{\alpha,t}^+$ and 
\begin{equation}
\label{Eq:kbaralphat}
\breve{K}_{\alpha,t}:=K\cap\{x\in\R^d:t-1\leq\tm{\alpha}x\leq t\}
\end{equation}
are compact and convex for all $t\in\R$. See Figure~\ref{Fig:Kalphat} for an illustration of the sets $K_{\alpha,t}^+$ and $\breve{K}_{\alpha,s}$. We now derive simple bounds on $\mu_d(K_{\alpha,t}^+)$ and $\mu_d(\breve{K}_{\alpha,t})$ that apply to all $f_{K,\alpha}\in\mathcal{F}_\star^1$ with $\alpha\neq 0$.
\begin{lemma}
\label{Lem:log1affbds}
If $f_{K,\alpha}\in\mathcal{F}_\star^1$ and $\alpha\neq 0$, then 
\begin{equation}
\label{Eq:ratiobds}
\gamma(d,t):=1-e^{-t}\,\sum_{\ell=0}^{d-1}\,\frac{t^\ell}{\ell!} \leq \frac{\mu_d(K_{\alpha,t}^+)}{c_{K,\alpha}} \leq e^t
\end{equation}
for all $t>0$. Moreover, if $1\leq s\leq t-1$, then
\begin{equation}
\label{Eq:ratioslices}
\mu_d(\breve{K}_{\alpha,t})\leq\frac{t^d-(t-1)^d}{s^d-(s-1)^d}\,\mu_d(\breve{K}_{\alpha,s}). 
\end{equation}
\end{lemma}
\begin{proof}
The result is clear if $d=1$, so suppose now that $d\geq 2$. Fix $t>0$ and observe that, by analogy with~\eqref{Eq:volident}, we can write 
\begin{equation}
\label{Eq:ratioprod}
\dfrac{\mu_d(K_{\alpha,t}^+)}{c_{K,\alpha}}=
\dfrac{\textcolor{white}{\bigl)}\int_0^t\mu_{d-1}(K_{\alpha,u})\,du\textcolor{white}{\bigr)}}{\textcolor{white}{\Bigl(}\int_0^t\mu_{d-1}(K_{\alpha,u})\,e^{-u}\,du\textcolor{white}{\Bigr)}}\times
\frac{\textcolor{white}{\bigl(}\int_0^t\mu_{d-1}(K_{\alpha,u})\,e^{-u}\,du\textcolor{white}{\bigr)}}{\textcolor{white}{\Bigl)}\int_0^\infty\mu_{d-1}(K_{\alpha,u})\,e^{-u}\,du\textcolor{white}{\Bigr)}}.
\end{equation}
The first term on the right-hand side is bounded above and below by $e^t$ and $1$ respectively. 
The second term is clearly at most 1, and we now show that it is bounded below by $\gamma(d,t)$. This is certainly the case when $t\geq M_{K,\alpha}$,
so suppose henceforth that $t<M_{K,\alpha}$. We know from Proposition~\ref{Prop:kslices}(ii) that $u\mapsto\mu_{d-1}(K_{\alpha,u})^{1/(d-1)}$ is concave and strictly positive on $(0,M_{K,\alpha})$, so $\mu_{d-1}(K_{\alpha,u})\geq (u/t)^{d-1}\mu_{d-1}(K_{\alpha,t})$ for all $0\leq u\leq t$ and $\mu_{d-1}(K_{\alpha,u})\leq (u/t)^{d-1}\mu_{d-1}(K_{\alpha,t})$ for all $u\geq t$.
Therefore, introducing random variables $Y\sim\Gamma(d,1)$ and $W\sim\Po(t)$, we conclude that the second term in~\eqref{Eq:ratioprod} is bounded below by
\begin{equation}
\label{Eq:poigam}
\dfrac{\textcolor{white}{\bigl)}\int_0^t u^{d-1}\mu_{d-1}(K_{\alpha,t})\,e^{-u}\,du\textcolor{white}{\bigr)}}{\textcolor{white}{\Bigl(}\int_0^\infty u^{d-1}\mu_{d-1}(K_{\alpha,t})\,e^{-u}\,du\textcolor{white}{\Bigr)}}=\Pr(Y\leq t)=\Pr(W\geq d)=\gamma(d,t).
\end{equation}
This establishes~\eqref{Eq:ratiobds}. Similarly, if $1\leq s\leq t-1$ and $s<M_{K,\alpha}$, then
\[\frac{\mu_d(\breve{K}_{\alpha,t})}{\textcolor{white}{\bigl)}\mu_d(\breve{K}_{\alpha,s})\textcolor{white}{\Bigr)}}=
\dfrac{\textcolor{white}{\bigl)}\int_{t-1}^t\mu_{d-1}(K_{\alpha,u})\,du\textcolor{white}{\bigr)}}{\textcolor{white}{\Bigl(}\int_{s-1}^s\mu_{d-1}(K_{\alpha,u})\,du\textcolor{white}{\Bigr)}}\leq
\dfrac{\textcolor{white}{\bigl)}\int_{t-1}^t (u/s)^{d-1}\mu_{d-1}(K_{\alpha,s})\,du\textcolor{white}{\bigr)}}{\textcolor{white}{\Bigl(}\int_{s-1}^s(u/s)^{d-1}\mu_{d-1}(K_{\alpha,s})\,du\textcolor{white}{\Bigr)}}=
\frac{t^d-(t-1)^d}{s^d-(s-1)^d}.
\]
The bound~\eqref{Eq:ratioslices} holds trivially when $s\geq M_{K,\alpha}$, so we are done.
\end{proof}
\begin{remark*}
By appealing to Proposition~\ref{Prop:kslices}(ii) and stochastic domination arguments, one can in fact show that the reciprocal of the first term on the right-hand side of~\eqref{Eq:ratioprod} is bounded below by $dt^{-d}\gamma(d,t)$ and above by $d!\,\sum_{\ell=1}^{d-1}\,(-1)^{\ell-1}t^{-\ell}/(d-\ell)!$ (and also that these bounds are tight). 
\end{remark*}
We now turn to the proof of Proposition~\ref{Prop:Logkaff} in the main text, which makes use of the following two facts from general topology and convex analysis. Recall that $\mathcal{P}\equiv\mathcal{P}_d$ denotes the collection of all polyhedral subsets of $\R^d$, namely those that can be expressed as the intersection of finitely many closed half-spaces (and $\R^d$ itself).
\begin{lemma}
\label{Lem:clint}
If $K$ is a subset of a topological space $E$ and if $K_1,\dotsc,K_\ell\subseteq E$ are closed sets such that $\Cl\Int K=\bigcup_{j=1}^{\,\ell} K_j$, then $\Cl\Int K$ is in fact the union of those $K_j$ for which $\Int K_j\neq\emptyset$. Moreover, $\bigcup_{j=1}^{\,\ell}\Int K_j$ is dense in $\Cl\Int K$. 
\end{lemma}
\begin{proof}[Proof of \protect{Lemma~\ref{Lem:clint}}]
We first verify that if $A,B\subseteq E$ and $A$ is closed, then
\begin{equation}
\label{Eq:clintcup}
\Cl\Int(A\cup B)=(\Cl\Int A)\cup(\Cl\Int B)=\Cl((\Int A)\cup(\Int B)).
\end{equation}
Indeed, recalling that $(\Cl P)\cup (\Cl Q)=\Cl(P\cup Q)$ and $(\Int P)\cup(\Int Q)\subseteq\Int(P\cup Q)$ for all $P,Q\subseteq E$, we immediately obtain the second equality and the inclusion $\Cl\Int(A\cup B)\supseteq(\Cl\Int A)\cup(\Cl\Int B)$. It now remains to show that $\Cl\Int(A\cup B)\subseteq(\Cl\Int A)\cup(\Cl\Int B)$. To this end, fix $x\in\Cl\Int(A\cup B)$ and suppose that $x\notin\Cl\Int A$, in which case there exists an open neighbourhood $U$ of $x$ that is disjoint from $\Int A$. Now let $V$ be an arbitrary open neighbourhood of $x$ and let $W:=U\cap V\cap\Int(A\cup B)$. Then $W\subseteq A\cup B$ is a non-empty open set that is disjoint from $\Int A$, so $W\not\subseteq A$. Thus, since $A$ is closed by assumption, it follows that $W\cap\cm{A}$ is a non-empty open set contained within $B$, and hence that $W\cap\cm{A}\subseteq\Int B$. We conclude that $V\cap\Int B\neq\emptyset$ for all open neighbourhoods $V$ of $x$, whence $x\in\Cl\Int B$. This completes the proof of the first equality in~\eqref{Eq:clintcup}.

Moreover, if $K\subseteq E$, then $\Cl\Int\Cl\Int K=\Cl\Int K$; indeed, note that $\Int\Cl\Int K\supseteq\Int\Int\Int K=\Int K$ and $\Cl\Int\Cl\Int K\subseteq\Cl\Cl\Cl\Int K=\Cl\Int K$. We deduce from this and~\eqref{Eq:clintcup} that if $K,K_1,\dotsc,K_\ell$ are as in the statement of the lemma, then
\[
\textstyle\Cl\Int K=\Cl\Int\Bigl(\bigcup_{j=1}^{\,\ell}K_j\Bigr)=\Cl\,\Bigl(\bigcup_{j=1}^{\,\ell}\Int K_j\Bigr)=\bigcup_{j=1}^{\,\ell}\Cl\Int K_j\subseteq\bigcup_{\,j\,:\,\Int K_j\neq\emptyset\,}K_j.
\]
Since $\bigcup_{\,j\,:\,\Int K_j\neq\emptyset\,}K_j\subseteq\bigcup_{j=1}^{\,\ell}K_j=\Cl\Int K$, this yields the first assertion of the lemma. The second assertion follows from the first two equalities in the display above.
\end{proof}
\begin{lemma}
\label{Lem:polycomplex}
Suppose that $E_1,\dotsc,E_\ell\in\mathcal{P}$ have pairwise disjoint interiors and that any intersection $E_r\cap E_s$ with affine dimension $d-1$ is a common face of $E_r$ and $E_s$. If $P:=\bigcup_{s=1}^{\,\ell} E_s\in\mathcal{P}$, then $E_1,\dotsc,E_\ell$ constitutes a polyhedral subdivision of $P$.
\end{lemma}
The proof of Lemma~\ref{Lem:polycomplex} relies on the auxiliary result below.
\begin{lemma}
\label{Lem:rdpcon}
Let $d\geq 2$ and let $U\subseteq\R^d$ be a path-connected open set. If we have a finite collection of sets $E_1,\dotsc,E_\ell\subseteq\R^d$, each of affine dimension at most $d-2$, then $A:=U\setminus\bigcup_{s=1}^{\,\ell}E_s$ is path-connected. In fact, for any $x,y\in A$, there is a piecewise linear path $\gamma\colon [0,1]\to A$ with $x=\gamma(0)$ and $y=\gamma(1)$.
\end{lemma}
\begin{proof}[Proof of \protect{Lemma~\ref{Lem:rdpcon}}]
Before proving the result in full generality, we first specialise to the case where $U$ is an open ball and proceed by induction on $d\geq 2$. The base case $d=2$ is trivial, 
so now consider a general $d>2$ and fix $x,y\in A$. For each $1\leq s\leq\ell$, define a linear subspace $W_s:=\{z-w:z,w\in\aff E_s\}$ and suppose for the time being that $x-y\notin\bigcup_{s=1}^{\,\ell} W_s$. Thus, we cannot have $[x,y]\subseteq\aff E_s$ for any $s$. Consequently, if $\dim(E_s)=d-2$, then $\dim(E_s\cup\{x,y\})=d-1$, so $H_s:=\aff(E_s\cup\{x,y\})$ is the unique affine hyperplane $H$ for which $x,y\in H$ and $E_s\subseteq H$. In other words, if $H\neq H_s$ is any other affine hyperplane through $x,y$, then $E_s\setminus H\neq\emptyset$, in which case $\aff(E_s\cup H)=\R^d$ and $\dim(H\cap E_s)=\dim(H)+\dim(E_s)-\dim(\R^d)=d-3$. On the other hand, if $\dim(E_s)\leq d-3$, then clearly $\dim(H\cap E_s)\leq\dim(E_s)\leq d-3$ for any affine hyperplane $H$ through $x,y$. We therefore conclude that there is an affine hyperplane $H$ (of dimension $d-1$) such that $x,y\in H$ and $\dim(H\cap E_s)\leq d-3$ for all $s$. Since $H\cap U$ is an open ball inside $H$, it follows by induction that there is a piecewise linear path $\gamma\colon [0,1]\to H$ from $x$ to $y$, as required.

In general, if $x,y$ are arbitrary points of $A$, then since $A$ and 
$\R^d\setminus\bigcup_{s=1}^{\,\ell}W_s$ have non-empty interior,
we can find $z\in A$ such that neither $x-z$ nor $z-y$ lie in $\bigcup_{s=1}^{\,\ell}W_s$. 
We have already established that there exist piecewise linear paths in $A$ from $x$ to $z$ and from $z$ to $y$, so we can concatenate these to obtain a suitable path from $x$ to $y$. This shows that $U\setminus\bigcup_{s=1}^{\,\ell}E_s$ is path-connected whenever $U$ is an open ball.

Now let $U$ be an arbitrary path-connected open set. Then for fixed $x,y\in A=U\setminus\bigcup_{s=1}^{\,\ell}E_s$, there exists a path $\alpha\colon [0,1]\to U$ with $\alpha(0)=x$ and $\alpha(1)=y$. Since $\cm{U}$ is closed and $\Img\alpha$ is compact, there exists $\delta>0$ such that $B(\alpha(t),\delta)\subseteq U$ for all $t\in [0,1]$. We now claim that there exist $K\in\N$ and $0\leq t_0,\dotsc,t_K\leq 1$ such that $x\in B(\alpha(t_0),\delta)$, $y\in B(\alpha(t_K),\delta)$ and $B(\alpha(t_{j-1}),\delta)\cap B(\alpha(t_j),\delta)\neq\emptyset$ for all $1\leq j\leq K$. Indeed, 
by the compactness of $\Img\alpha$, we can extract a finite subset $T\subseteq [0,1]$ such that $\Img\alpha\subseteq\bigcup_{t\in T}B(\alpha(t),\delta)$. Now consider the graph with vertex set $T$ in which $r,s\in T$ are joined by an edge if and only if $B(\alpha(r),\delta)\cap B(\alpha(s),\delta)\neq\emptyset$. If $\emptyset\neq S\subseteq T$ constitutes a connected component of this graph, then $\bigcup_{t\in S}B(\alpha(t),\delta)$ and $\bigcup_{t\in T\setminus S}B(\alpha(t),\delta)$ are disjoint open sets that cover $\Img\alpha$. But since $\Img\alpha$ is connected, it follows that $S=T$ and hence that the graph is connected. Choosing $r,s\in T$ such that $x\in B(\alpha(r),\delta)$ and $y\in B(\alpha(s),\delta)$, we deduce that there is a path in the graph from $r$ to $s$, as claimed.

Thus, since $\bigcup_{s=1}^{\,\ell}E_s$ has empty interior, we can find $x_1,x_2,\dotsc,x_K\in A$ such that $x_j\in B(\alpha(t_{j-1}),\delta)\cap B(\alpha(t_j),\delta)$ for all $1\leq j\leq K$. Setting $x_0:=x$ and $x_{K+1}:=y$, we have $x_j,x_{j+1}\in B(\alpha(t_j),\delta)\setminus\bigcup_{s=1}^{\,\ell}E_s$ for $0\leq j\leq K$, so it follows from the previous argument that there exists a piecewise linear path in $A$ from $x_{j-1}$ to $x_j$ for each $j$. As before, we can concatenate these to obtain a suitable path from $x$ to $y$.
\end{proof}
\begin{proof}[Proof of Lemma~\ref{Lem:polycomplex}]
Let $d\geq 2$ and fix $1\leq r,s\leq\ell$. First we claim that for any fixed $x\in\relint (E_r\cap E_s)$, there exist $r=r_0,r_1,\dotsc,r_L=s$ such that $x\in\bigcap_{\,j=0}^{\,L}E_{r_j}$ and $E_{r_{j-1}}\cap E_{r_j}$ has affine dimension $d-1$ for all $1\leq j\leq L$. Indeed, let $J$ be the set of indices $t\in\{1,\dotsc,m\}$ such that $x\in E_t$. Since each $E_t$ is closed, we can find $\delta>0$ such that $B(x,\delta)\cap E_t\neq\emptyset$ if and only if $t\in J$. Now fix $y\in B(x,\delta)\cap\Int E_r$ and $z\in B(x,\delta)\cap\Int E_s$,
and let $E'$ be the union of all sets of the form $E_j\cap E_k$ with affine dimension at most $d-2$, where $j,k\in J$.

By Lemma~\ref{Lem:rdpcon}, there is a piecewise linear path $\gamma\colon [0,1]\to \Int P\cap B(x,\delta)\setminus E'$ with $\gamma(0)=y$ and $\gamma(1)=z$.
Let $J'$ be the set of indices $t\in J$ for which $\Img\gamma\cap E_t\neq\emptyset$, and define $\theta(t):=\inf\{u\in [0,1]:\gamma(u)\in E_t\}$ for each $t\in J'$.
Now enumerate the elements of $J'$ as $t_0',t_1',\dotsc,t_K'$ in such a way that $0=\theta(t_0')<\theta(t_1')\leq\theta(t_2')\leq\dotsc\leq\theta(t_K')<1$. 
Then for each $2\leq J\leq K$, there exists $1\leq I<J$ such that $\gamma(\theta(t_J'))\in E_{t_I'}\cap E_{t_J'}$, so by the definition of $E'$, it follows that $E_{t_I'}\cap E_{t_J'}$ must have affine dimension $d-1$. Consequently, we can extract a subsequence $r=r_0,r_1,\dotsc,r_L=s$ of $t_0',\dotsc,t_K'$ with the required properties.

Setting $E_j':=E_{r_j}$ for $0\leq j\leq L$, we now show that $E_r\cap E_s=E_0'\cap E_L'=\bigcap_{\,j=0}^{\,L}E_j'$. Indeed, first note that since $\relint(E_0'\cap E_L')$ is a relatively open convex subset of $E_0'$, it follows from~\citet[Theorem~2.1.2]{Sch14} that there is a unique face $F_0'$ of $E_0'$ with $\relint(E_0'\cap E_L')\subseteq\relint F_0'$. Moreover, since $E_0'\cap E_1'$  has affine dimension $d-1$ by construction, the conditions of the lemma imply that $E_0'\cap E_1'$ is a face of $E_0'$ that contains $x$. Thus, $x\in (E_0'\cap E_1')\cap\relint F_0'$, and we now appeal to the following consequence of the proof of~\citet[Theorem~2.1.2]{Sch14}, which applies to any closed, convex and non-empty $K\subseteq\R^d$: if $G,G'$ are faces of $K$ such that $G\cap\relint G'\neq\emptyset$, then $G\supseteq G'$. Applying this with $K=E_0'$, $G=E_0'\cap E_1'$ and $G'=F_0'$, and invoking~\citet[Theorem~1.1.15(b)]{Sch14},
we deduce that $E_0'\cap E_1'\supseteq F_0'\supseteq E_0'\cap E_L'$ and 
hence that $E_0'\cap E_L'=E_0'\cap E_1'\cap E_L'$. 

Next, it follows from this and~\citet[Theorem~2.1.2]{Sch14} that there is a face $F_1'$ of $E_1'$ with $x\in\relint(E_0'\cap E_L')=\relint(E_0'\cap E_1'\cap E_L')\subseteq\relint F_1'$. Since $E_1'\cap E_2'$ is a face of $E_1'$ that contains $x$, we can apply the fact above with $K=E_1'$, $G=E_1'\cap E_2'$ and $G'=F_1'$ to deduce that $E_1'\cap E_2'\supseteq F_1'\supseteq E_0'\cap E_1'\cap E_L'$ and hence that 
$E_0'\cap E_L'=E_0'\cap E_1'\cap E_L'=E_0'\cap E_1'\cap E_2'\cap E_L'$. Continuing inductively in this vein, we obtain the conclusion that $E_r\cap E_s=E_0'\cap E_L'=\bigcap_{\,j=0}^{\,L}E_j'$, as claimed.

Now suppose that we have $z\in\bigcap_{\,j=0}^{\,L}E_j'$ and $x,y\in E_0'$ such that $z\in\relint [x,y]$. Then $E_0'\cap E_1'$ is a face of $E_0'$ that contains $z$, so $x,y\in E_1'$. In view of this and the fact that $E_1'\cap E_2'$ is a face of $E_1'$ that contains $z$, it then follows that $x,y\in E_2'$. By repeating this argument, we conclude that $x,y\in \bigcap_{\,j=0}^{\,L}E_j'$. This shows that $E_r\cap E_s$ is a face of $E_r=E_0'$, and it follows by symmetry that $E_r\cap E_s$ is a face of $E_s$.
\end{proof}
We are now ready to assemble the proof of Proposition~\ref{Prop:Logkaff}. This proceeds via a series of intermediate claims which together imply the result. In Section~\ref{Subsec:Notation}, we provide only a `bare-bones' definition of a log-$k$-affine function $f\in\mathcal{G}_d$ in the sense that the subdomains on which $f$ is log-affine are assumed only to be closed. Starting from this, we show in Claim~\ref{Claim:logkaff1} that $f$ has a `minimal' representation in which the subdomains are closed, convex sets and the restrictions of $f$  to these sets are distinct log-affine functions. In the remainder of the proof, we investigate the boundary structure of these subdomains more closely and establish that, under the hypotheses of Proposition~\ref{Prop:Logkaff}, these are in fact polyhedral sets that form a subdivision of $\supp f$.
\begin{proof}[Proof of Proposition~\ref{Prop:Logkaff}]
For convenience, we set $g:=\log f$. The case $k=1$ is trivial, so we assume throughout that $f$ is not log-1-affine. By Lemma~\ref{Lem:clint} and the fact that $K=\Cl\Int K$ for all $K\in\mathcal{K}$~\citep[Theorem~1.1.15(b)]{Sch14}, we may assume without loss of generality that there exist $k\geq 2$ closed sets $E_1',\dotsc,E_k'$ with non-empty interiors such that $P:=\supp f=\bigcup_{j=1}^{\,k}E_j'=\Cl\,\bigl(\bigcup_{j=1}^{\,k}\Int E_j'\bigr)$, and $\theta_1,\dotsc,\theta_k\in\R^d$, $\zeta_1,\dotsc,\zeta_k\in\R$ such that $g(x)=\tm{\theta_j}x+\zeta_j$ for all $x\in E_j'$. Moreover, we may suppose that there exist pairwise distinct $\alpha_1,\dotsc,\alpha_\ell\in\R^d$ and a subsequence $0=k_0<k_1<\dotsc<k_\ell=k$ of the indices $0,1,\dotsc,k$ such that $\theta_j=\alpha_s$ whenever $k_{s-1}<j\leq k_s$. Let $\kappa(f):=\ell$, and for each $1\leq s\leq\ell$, let $E_s:=\bigcup_{j=k_{s-1}+1}^{\,k_s}E_j'$.
\begin{claim}
\label{Claim:logkaff1}
$E_1,\dotsc,E_\ell$ are closed, convex sets with pairwise disjoint and non-empty interiors. Moreover, if $\theta_i=\theta_j$, then $\zeta_i=\zeta_j$, i.e.\ there exist $\beta_1,\dotsc,\beta_\ell\in\R$ such that $g(s)=g_s(x):=\tm{\alpha_s}x+\beta_s$ for all $x\in E_s$. Thus, the restrictions of $g$ to the sets $E_1,\dotsc,E_\ell$ are distinct affine functions $g_1,\dotsc,g_\ell$.
\end{claim}
\begin{proof}[Proof of Claim~\ref{Claim:logkaff1}]
If $E\subseteq\R^d$ has non-empty interior and $g_1,g_2\colon E\to\R$ are affine functions that agree on a non-empty open subset of $E$, then $g_1,g_2$ must in fact agree everywhere on $E$, so $(\Int E_r)\cap(\Int E_s)=\emptyset$ whenever $r\neq s$. Moreover, for distinct $i,j\in\{k_{s-1}+1,\dotsc,k_s\}$, fix $x_i'\in\Int E_i'$ and $x_j'\in\Int E_j'$, so that $g(x)=\tm{\alpha_s}x+\zeta_j$ for all $x\in [x_i',x_j']$ sufficiently close to $x_j'$ and $g(x)=\tm{\alpha_s}x+\zeta_i$ for all $x\in [x_i',x_j']$ sufficiently close to $x_i'$. But since $g$ is concave on $[x_i',x_j']$, it follows that $\zeta_i=\zeta_j$, as required. Thus, there exist $\beta_1,\dotsc,\beta_\ell\in\R$ such that $g(x)=g_s(x):=\tm{\alpha_s}x+\beta_s$ for all $x\in E_s$.

It remains to show that each $E_s$ is convex. Fix $y\in\Int E_s$ and suppose for a contradiction that $(\conv E_s)\setminus E_s$ is non-empty. Since $\conv E_s\subseteq\Cl(\conv E_s)=\Cl\Int(\conv E_s)$, it follows that $\Int(\conv E_s)\setminus E_s\subseteq P$ is a non-empty open set.
We know from Lemma~\ref{Lem:clint} that $\bigcup_{r=1}^{\,\ell}\Int E_r$ is dense in $P$, i.e.\ that it has non-empty intersection with any non-empty open subset of $P$. Thus, there exist $r\neq s$ such that $W:=(\Int E_r)\cap\Int(\conv E_s)\setminus E_s$ is a non-empty open set. Now for each $x\in W$, there exist $x_1,x_2\in E_s$ such that $x\in [x_1,x_2]$. Since $g$ is concave on $[x_1,x_2]$ and $g(x_j)=\tm{\alpha_s}x_j+\beta_s$ for $j=1,2$, we have $g(x)\geq\tm{\alpha_s}x+\beta_s$. On the other hand, since $g$ is concave on $[x,y]$ and $g(z)=\tm{\alpha_s}z+\beta_s$ for all $z\in [x,y]$ sufficiently close to $y$, it follows that $g(x)\leq\tm{\alpha_s}x+\beta_s$. Thus, $g(x)=\tm{\alpha_s}x+\beta_s=\tm{\alpha_r}x+\beta_r$ for all $x\in W\subseteq\Int E_r$, so $\alpha_r=\alpha_s$ and $\beta_r=\beta_s$. This contradicts the fact that $\alpha_1,\dotsc,\alpha_\ell$ are pairwise distinct, so the proof of the claim is complete.
\end{proof}
\begin{claim}
\label{Claim:logkaff2}
If $E_r\cap E_s\neq\emptyset$ and $r\neq s$, then $E_r\cap E_s$ is a closed, convex subset of $\partial E_r$. Moreover, if $E_r\cap E_s$ has affine dimension $d-1$, then there exists a unique closed half-space $H_{rs}^+$ containing $E_r$ such that $E_r\cap E_s=E_r\cap H_{rs}$, where $H_{rs}:=\partial H_{rs}^+$.
Thus, in this case, $E_r\cap E_s$ is a common (exposed) facet of $E_r$ and $E_s$, and we must have $H_{sr}^+=\cm{(\Int H_{rs}^+)}$ and $H_{rs}=H_{sr}$.
\end{claim}
\begin{proof}[Proof of \protect{Claim~\ref{Claim:logkaff2}}]
Since $E_s$ is convex and $\Int E_s\subseteq\cm{(\Int E_r)}$, we have $E_s=\Cl\Int E_s\subseteq\cm{(\Int E_r)}$, so $E_r\cap E_s$ is a closed, convex subset of $\partial E_r$. Then by~\citet[Theorem~2.1.2]{Sch14}, there exists a unique proper face $F$ of $E_r$ (whose affine dimension is at most $d-1$) such that $\relint(E_r\cap E_s)\subseteq\relint F$.
Now suppose that $E_r\cap E_s$ has affine dimension $d-1$. Then $F$ is a facet of $E_r$, so by~\citet[Theorem~2.1.2]{Sch14} and the final observation in the paragraph after the proof of this result~\citep[page~75]{Sch14}, there exists a closed half-space $H_{rs}^+\supseteq E_r$ such that $H_{rs}:=\partial H_{rs}^+$ is a supporting hyperplane to $E_r$ with $F=E_r\cap H_{rs}$. 
Note that a closed half-space $H_{rs}^+$ with these properties must be unique.
Furthermore, since the affine functions $g_r$ and $g_s$ agree on a relatively open subset of $H_{rs}$, namely $\relint(E_r\cap E_s)$, they must agree everywhere on $H_{rs}$.

We now show that $E_r\cap E_s=F$. If this is not the case, then there exist $y\in F\setminus (E_r\cap E_s)$ and $z\in\relint(E_r\cap E_s)\subseteq\relint F$. Thus, there is some $x\in (y,z]$ that belongs to $\partial(E_r\cap E_s)\cap (\relint F)$
and there exists $\eta>0$ such that $B(x,\eta)\cap H_{rs}\subseteq F\subseteq E_r$. 
Now fix $w\in\Int E_r$ and
observe that we can find $\delta\in (0,\eta)$ such that $E_s\not\subseteq B(x,\delta)$ and $B(x,\delta)\cap H_{rs}^+\subseteq\conv(\{w\}\cup B(x,\eta)\cap H_{rs})\subseteq E_r$. Since $E_r\subseteq H_{rs}^+$, it follows that $B(x,\delta)\cap H_{rs}^+=B(x,\delta)\cap E_r$. Writing $H_{sr}^+:=\cm{(\Int H_{rs}^+)}$ for the other closed half-space bounded by $H_{rs}$, we note in addition that $E_s\subseteq H_{sr}^+$; indeed, if there did exist $\tilde{x}\in E_s\setminus H_{sr}^+=E_s\cap\Int H_{rs}^+$, then $[x,\tilde{x}]$ would be contained within $E_s$ and also have non-empty intersection with $B(x,\delta)\cap\Int H_{rs}^+=\Int(B(x,\delta)\cap H_{rs}^+)\subseteq\Int E_r$, 
which would contradict the fact that $E_s=\Cl\Int E_s\subseteq\cm{(\Int E_r)}$.

Next, fix $x'\in E_s\setminus B(x,\delta)\subseteq H_{sr}^+\setminus B(x,\delta)$ and note that there exists $\delta'\in (0,\delta]$ such that for every $y'\in B(x,\delta')\cap\Int H_{sr}^+$, we can find $z'\in B(x,\eta)\cap H_{rs}\subseteq F\subseteq H_{rs}$ with $y'\in [x',z']$.
Thus, if $y',z'$ are as above, then since $g(x')=g_s(x')$ and $g(z')=g_r(z')=g_s(z')$, it follows from the concavity of $g$ on $[x',z']$ that $g(y')\geq g_s(y')$. On the other hand, there exists $w'\in [x,x']\subseteq E_s$ sufficiently close to $x$ such that $w'\in [y',z'']$ for some $z''\in B(x,\eta)\cap H_{rs}\subseteq F\subseteq H_{rs}$. As before, we have $g(w')=g_s(w')$ and $g(z'')=g_r(z'')=g_s(z'')$, so $g(y')\leq g_s(y')$ by the concavity of $g$ on $[w',z'']$. We therefore conclude that $g(y')=g_s(y')$ for all $y'\in B(x,\delta')\cap\Int H_{sr}^+$.
Note that we cannot have $B(x,\delta')\cap\Int H_{sr}^+\subseteq\Int E_s$; indeed, it would follow that
$B(x,\delta')\cap H_{rs}\subseteq E_r\cap (\Cl\Int E_s)=E_r\cap E_s$, and since $\aff(E_r\cap E_s)=H_{rs}$, this would contradict the fact that $x\in\partial(E_r\cap E_s)$. Thus, there exists $t\neq r,s$ such that the intersection of $\Int E_t$ with $B(x,\delta')\cap\Int H_{sr}^+$ is a non-empty open set, which we denote by $U$. We see that the affine functions $g_s$ and $g_t$ both agree with $g$ on $U$, so in fact $g_s=g_t$ on $\R^d$. This contradicts Claim~\ref{Claim:logkaff1}, so it must therefore be the case that $E_r\cap E_s=F=E_r\cap H_{rs}$, as required.

By interchanging $E_r$ and $E_s$ in the argument above, we deduce that there exists a closed half-space $H^+$ containing $E_s$ such that $E_r\cap E_s=E_s\cap\partial H^+$. Then $E_r\cap E_s\subseteq H_{rs}\cap\partial H^+\subseteq H_{rs}$, so $\dim(E_r\cap E_s)=\dim(H_{rs}\cap\partial H^+)=\dim(H_{rs})=d-1$.
It follows that $\partial H^+=H_{rs}$ and hence that $H^+=H_{sr}^+$, which yields the final assertion of the claim. 
\end{proof}
Since $P\in\mathcal{P}$ by hypothesis, there exist closed half-spaces $H_1^+,\dotsc,H_M^+$ such that $P=\bigcap_{\,j=1}^{\,M} H_j^+$. For each $1\leq j\leq M$, let $H_j:=\partial H_j^+$, and for each $1\leq r\leq\ell$, let $I_r$ be the set of indices $s\in\{1,\dotsc,\ell\}$ for which $E_r\cap E_s$ has affine dimension $d-1$.
\begin{claim}
\label{Claim:logkaff3}
If $1\leq r\leq\ell$ and $x\in\partial E_r$, then either $x\in E_r\cap H_j$ for some $1\leq j\leq M$ or $x\in E_r\cap E_s=E_r\cap H_{rs}$ for some $s\in I_r$.
\end{claim}
\begin{proof}[Proof of \protect{Claim~\ref{Claim:logkaff3}}]
Fix $1\leq\ell\leq r$, and note that $\partial P\subseteq\bigcup_{j=1}^{\,M} H_j$
and $\Cl\cm{E_r}\subseteq\Cl\,\bigl(\cm{P}\cup\bigcup_{s\neq r}E_s\bigr)=(\Cl\cm{P})\cup\bigl(\bigcup_{s\neq r}E_s\bigr)=\cm{P}\cup (\partial P)\cup\bigl(\bigcup_{s\neq r}E_s\bigr)$. Since $\Int E_r\cap\partial P\subseteq\Int P\cap\partial P=\emptyset$, we have $E_r\cap\partial P\subseteq\partial E_r$. Recalling from Claim~\ref{Claim:logkaff2} that $\bigcup_{s\neq r}\,(E_r\cap E_s)\subseteq\partial E_r$, we deduce that
$\partial E_r=E_r\cap\Cl\cm{E_r}$ is the union of the sets $E_r\cap H_1,\dotsc,E_r\cap H_M$ and $\bigcup_{s\neq r}\,(E_r\cap E_s)$, all of which are closed. In view of Claim~\ref{Claim:logkaff2} and Lemma~\ref{Lem:clint}, in which we set $E:=\partial E_r$ (equipped with the subspace topology), it suffices to show that $E_r\cap E_s$ has non-empty interior in $\partial E_r$ if and only if $s\in I_r$. 

Suppose first that $s\in I_r$, so that $\dim(E_r\cap E_s)=d-1$. Then $E_r\cap E_s=E_r\cap H_{rs}\subseteq\partial E_r$ and $\aff(E_r\cap E_s)=H_{rs}$ by Claim~\ref{Claim:logkaff2}, so for a fixed $x\in\relint(E_r\cap E_s)$, there exists $\delta>0$ such that $B(x,\delta)\cap H_{rs}\subseteq E_r\cap E_s$. Now fix $w\in\Int E_r$ and note that there exists $\delta'\in (0,\delta]$ such that $B(x,\delta')\cap H_{rs}^+\subseteq\conv(\{w\}\cup B(x,\delta)\cap H_{rs})\subseteq E_r$. Since $B(x,\delta')\cap\Int H_{sr}^+\subseteq\cm{E_r}$
and $B(x,\delta')\cap\Int H_{rs}^+=\Int(B(x,\delta)\cap H_{rs}^+)\subseteq\Int E_r$, we deduce that $B(x,\delta')\setminus H_{rs}=\emptyset$ and hence that $B(x,\delta')\cap\partial E_r=(B(x,\delta')\cap H_{rs})\cap\partial E_r\subseteq E_r\cap E_s$. Thus, $E_r\cap E_s$ has non-empty interior in $E=\partial E_r$, as required. 

On the other hand, if $s\notin I_r$, then $F:=E_r\cap E_s$ has affine dimension at most $d-2$. Fix $x\in F$ and $w\in\Int E_r$, and let $\eta>0$ be such that $B(w,\eta)\subseteq\Int E_r$. Then there exist $w_1,w_2\in B(w,\eta)$ such that $w_1\notin\aff F$ and $w_2\notin\aff(F\cup\{w_1\})$. Now let $A:=\aff\{x,w_1,w_2\}$. Then $A\cap F=\{x\}$ and $A\cap B(w,\eta)\neq\emptyset$ by construction, so $A\cap E_r$ is a closed, convex set with $\dim(A\cap E_r)=2$. 

We now verify that $\partial(A\cap E_r)=A \cap\partial E_r$. 
Indeed, since $A\cap\Int E_r$ and $\partial(A\cap E_r)$ are disjoint, 
we have $\partial(A\cap E_r)\subseteq A\cap\partial E_r$. For the reverse inclusion, note that if $x\in A\cap\partial E_r$, then there exists an open half-space $H^-\subseteq\R^d$ such that $H^-\cap E_r=\emptyset$ and $\partial H^-$ is a supporting hyperplane to $E_r$ at $x$. Since $A\cap\Int E_r\neq\emptyset$, we cannot have $A\subseteq\partial H^-$, so $\dim(A\cap\partial H^-)=1$ and therefore $x\in\Cl(A\cap H^-)\subseteq\Cl(A\setminus E_r)\subseteq\cm{(\relint(A\cap E_r))}$, as required. 

Since $x\in F\subseteq A\cap\partial E_r=\partial(A\cap E_r)$, it follows that $x\in\Cl\,(\partial(A\cap E_r)\setminus\{x\})$. 
By combining the observations above, we see that $\partial(A\cap E_r)\setminus\{x\}= (A\setminus\{x\})\cap\partial E_r\subseteq\cm{F}\cap\partial E_r$, so $x\in\Cl(\cm{F}\cap\partial E_r)$. Since $x\in F$ was arbitrary, we conclude that $F\subseteq\Cl(\cm{F}\cap\partial E_r)$, 
which implies that $F=E_r\cap E_s$ has non-empty interior in $\partial E_r$.
\end{proof}
\begin{claim}
\label{Claim:logkaff4}
For each $1\leq r\leq\ell$, we have $E_r=P\,\cap\,\bigcap_{\,s\in I_r\!}H_{rs}^+$, so in particular $E_r\in\mathcal{P}$.
\end{claim}
\begin{proof}[Proof of Claim~\ref{Claim:logkaff4}]
For a fixed $1\leq r\leq\ell$, we already know that $E_r\subseteq P\cap\bigcap_{\,s\in I_r\!}H_{rs}^+$. Now fix $x\in\cm{E_r}$ and $w\in\Int E_r$, and note that there exists $y\in\partial E_r\cap [x,w)$. By Claims~\ref{Claim:logkaff2} and~\ref{Claim:logkaff3}, there is a closed half-space $H^+\supseteq E_r$ with $y\in\partial H^+$ such that either $H^+=H_j^+$ for some $1\leq j\leq M$, or $H^+=H_{rs}^+$ for some $s\in I_r$. In all cases, we have $w\in H^+$, so it follows that $x\notin H^+$ and hence that $x\notin P\cap\bigcap_{\,s\in I_r\!}H_{rs}^+$.
\end{proof}
We have now established the first part of Proposition~\ref{Prop:Logkaff}, as well as the fact that $E_r\cap E_s$ is a common face of $E_r$ and $E_s$ whenever this intersection has affine dimension $d-1$. In view of Claim~\ref{Claim:logkaff1}, a direct application of Lemma~\ref{Lem:polycomplex} yields the conclusion that $E_1,\dotsc,E_\ell$ constitutes a polyhedral subdivision of $P$. Since $\abs{I_r}\leq k-1$ for all $1\leq r\leq\ell$, it follows from Claim~\ref{Claim:logkaff4} that each $E_r$ can be expressed as the intersection of at most $M+\abs{I_r}\leq M+k-1$ closed half-spaces. In view of \citet[Theorem~1.6]{BG09}, this implies the last assertion of Proposition~\ref{Prop:Logkaff}. 

Finally, to show that the triples $(\alpha_j,\beta_j,E_j)_{j=1}^{\kappa(f)}$ are unique up to reordering, we make the following observation: if $g$ is affine on some $\tilde{E}\in\mathcal{K}$ with $\tilde{E}\subseteq P$, then there exists a unique $1\leq r\leq\ell$ such that $\tilde{E}_j\subseteq E_r$. Indeed, it cannot happen that there exist distinct $1\leq r,s\leq\ell$ such that $\Int\tilde{E}$ intersects both $\Int E_r$ and $\Int E_s$, since $g_1,\dotsc,g_\ell$ are distinct affine functions by Claim~\ref{Claim:logkaff1}. Thus, there exists a unique $1\leq r\leq\ell$ such that $\Int\tilde{E}\subseteq\Int E_r$, so $\tilde{E}=\Cl\Int\tilde{E}\subseteq\Cl\Int E_r=E_r$, whereas for $s\neq r$, we cannot have $\tilde{E}\subseteq E_s$ since $\Int\tilde{E}\not\subseteq\Int E_s$.

Consequently, if $\tilde{E}_1,\dotsc,\tilde{E}_{\ell'}\in\mathcal{K}$ are such that $P=\bigcup_{j=1}^{\,\ell'}\tilde{E}_j$ and the restrictions of $g$ to these sets are distinct affine functions, then the observation above
implies that $\ell'=\ell$ and that there is some permutation $\pi\colon\{1,\dotsc,\ell\}\to\{1,\dotsc,\ell\}$ such that $\tilde{E}_{\pi(j)}\subseteq E_j$ for all $1\leq j\leq\ell$. In fact, we must have $\tilde{E}_{\pi(j)}=E_j$ for all $j$. Indeed, if $E_j\setminus\tilde{E}_{\pi(j)}\neq\emptyset$ for some $j$, then since $E_j=\Cl\Int E_j$, it would follow that $W:=(\Int E_j)\setminus\tilde{E}_{\pi(j)}$ is a non-empty open subset of $P$. Since $\tilde{E}_{\pi(j)}\cap W=\emptyset$ and $\tilde{E}_{\pi(j')}\cap W\subseteq E_{j'}\cap\Int E_j=\emptyset$ for all $j'\neq j$, this would imply that $\bigcup_{j=1}^{\,\ell'}\tilde{E}_j\subseteq P\setminus W\subsetneqq P$. This contradiction completes the proof.
\end{proof}
Now for $k\in\N$ and $P\in\mathcal{P}\equiv\mathcal{P}_d$, denote by $\mathcal{F}^k(P)\equiv\mathcal{F}_d^k(P)$ the collection of all $f\in\mathcal{F}_d$ for which $\kappa(f)\leq k$ and $\supp f=P$. For $m\in\N_0$, recall that $\mathcal{P}^m\equiv\mathcal{P}_d^m$ denotes the collection of all $P\in\mathcal{P}_d$ with at most $m$ facets (and that we view $\R^d$ as a polyhedral set with 0 facets). Finally, for $k\in\N$ and $m\in\N_0$, define $\mathcal{F}^k(\mathcal{P}^m)\equiv\mathcal{F}_d^k(\mathcal{P}_d^m):=\bigcup_{P\in\mathcal{P}^m}\mathcal{F}^k(P)$.
\begin{proposition}
\label{Prop:FkPmempty}
For $k\in\N$ and $m\in\N_0$, the subclass $\mathcal{F}^k(\mathcal{P}^m)$ is non-empty if and only if $k+m\geq d+1$.
\end{proposition}
For one direction of the proof, we require the following basic result:
\begin{lemma}
\label{Lem:minfacets}
Every line-free $P\in\mathcal{P}_d$ has at least $d$ facets. Moreover, every bounded $P\in\mathcal{P}_d$ (i.e.\ every $d$-dimensional polytope) has at least $d+1$ facets. 
\end{lemma}
\begin{proof}[Proof of Lemma~\ref{Lem:minfacets}]
Fix $P\in\mathcal{P}$ and consider any 
representation of $P$ as the intersection of finitely many closed half-spaces $H_1^+,\dotsc,H_m^+$, where $H_j^+=\{x\in\R^d:\tm{\alpha_j}x\leq b_j\}$ for some $\alpha_j\in\R^d\setminus\{0\}$ and $b_j\in\R$. Then $C:=\bigl\{\sum_{j=1}^m\lambda_j\alpha_j:\lambda_j\geq 0\text{ for all }j\bigr\}$ is a closed, convex cone, and note that $\rec(P)=\{u\in\R^d:\tm{\alpha_j}u\leq 0\text{ for all }j\}=-C^*$. Thus, if $P$ is line-free, then $\Int(\rec(P)^*)=-\Int C$ is non-empty by Proposition~\ref{Prop:linefree}, so $m\geq\dim(C)=d$. On the other hand, if $P$ is bounded, then $\rec(P)=-C^*=\{0\}$ by~\citet[Theorem~8.4]{Rock97}, so $C=\R^d$. As above, this implies that $m\geq\dim(C)=d$, and observe that we cannot have $m=\dim(C)=d$, since then $\alpha_1,\dotsc,\alpha_d$ would be linearly independent, in which case $C\neq\R^d$. This implies that $m\geq d+1$, as required. In both the line-free and bounded cases, the result follows on applying~\citet[Theorem~1.6]{BG09}.
\end{proof}
\begin{proof}[Proof of Proposition~\ref{Prop:FkPmempty}]
Suppose first that $k\in\N$ and $m\in\N_0$ are such that $\mathcal{F}^k(\mathcal{P}^m)$ is non-empty. Then for $f\in\mathcal{F}^k(\mathcal{P}^m)$, we deduce from the final assertion of Proposition~\ref{Prop:Logkaff} that there are polyhedral sets $E_1,\dotsc,E_{\kappa(f)}\in\mathcal{P}^{k+m-1}$ such that $\restr{f}{E_j}$ is log-1-affine and integrable for each $1\leq j\leq\kappa(f)$. Thus, by Proposition~\ref{Prop:log1aff} and Lemma~\ref{Lem:minfacets}, each $E_j$ is line-free and therefore has at least $d$ facets. It follows that $k+m-1\geq d$, as desired.

To establish the converse, note that since the classes $\mathcal{F}^k(\mathcal{P}^m)$ are nested in $k$ and $m$ by definition, it will suffice to consider each $k\in\{1,\dotsc,d+1\}$ in turn and exhibit a density $f_{k,\,d+1-k}$ on $\R^d$ that lies in $\mathcal{F}^{k}(\mathcal{P}^{d+1-k})\equiv\mathcal{F}_d^{k}(\mathcal{P}_d^{d+1-k})$. We proceed by induction on $d\in\N$. When $d=1$, the univariate densities $f_{1,1}\colon x\mapsto e^{-x}\,\Ind_{\{x\geq 0\}}$ and $f_{2,0}\colon x\mapsto e^{-\abs{x}}/2$ have the required properties. For a general $d\geq 2$, first fix $k\in\{1,\dotsc,d\}$ and define $f_{k,\,d+1-k}\colon\R^d\to [0,\infty)$ by \[f_{k,\,d+1-k}(x_1,\dotsc,x_d):=f_{k,\,d-k}(x_1,\dotsc,x_{d-1})\,e^{-x_d}\,\Ind_{\{x_d\geq 0\}},\]
where $f_{k,\,d-k}\colon\R^{d-1}\to [0,\infty)$ is an element of $\mathcal{F}_{d-1}^k(\mathcal{P}_{d-1}^{d-k})$ whose existence is guaranteed by the inductive hypothesis. Then $\int_{\R^d}\,f_{k,\,d+1-k}=\bigl(\int_{\R^{d-1}}f_{k,\,d-k}\bigr)\bigl(\int_0^\infty e^{-x_d}\,dx_d\bigr)=1$ by Fubini's theorem, so $f_{k,\,d+1-k}$ is a density, which is easily seen to lie in $\mathcal{F}_d$. Observe also that $\supp f_{k,\,d+1-k}=(\supp f_{k,\,d-k})\times [0,\infty)\in\mathcal{P}_d^{d+1-k}$; indeed, $P:=\supp f_{k,\,d-k}$ has (at most) $d-k$ facets by induction, and we see that $F$ is a facet of $\supp f_{k,\,d+1-k}=P\times [0,\infty)$ if and only if either $F=P\times\{0\}$ or $F=F'\times [0,\infty)$ for some facet $F'$ of $P$. Furthermore, since $f_{k,\,d-k}\in\mathcal{F}_{d-1}^k(\mathcal{P}_{d-1}^{d-k})$, there are closed sets $E_1',\dotsc,E_k'\subseteq\R^{d-1}$ such that $P=\supp f_{k,\,d-k}=\bigcup_{j=1}^{\,k}E_j'$ and $\log f_{k,\,d-k}$ is affine on each $E_j'$. It follows that $\log f_{k,\,d+1-k}$ is affine on each of the sets $E_1'\times [0,\infty),\dotsc,E_k'\times [0,\infty)$, whose union is $P\times [0,\infty)=\supp f_{k,\,d+1-k}$. This shows that $f_{k,\,d+1-k}\in\mathcal{F}_d^{k}(\mathcal{P}_d^{d+1-k})$.

Finally, to define $f_{d+1,\,0}$, we fix $u_1,\dotsc,u_{d+1}\in\R^d$ such that $S:=\conv\{u_1,\dotsc,u_{d+1}\}$ is a $d$-simplex with $0\in\Int S$. Then as remarked at the start of Section~\ref{Sec:LogkaffineSupp}, the Minkowski functional $\rho_S\colon w\mapsto\inf\{\lambda>0:w\in\lambda S\}\in [0,\infty)$ is a convex (and therefore continuous) function on $\R^d$, and by~\citet[Proposition~2]{XS19}, there exists $c>0$ such that $f_{d+1,\,0}\colon x\mapsto ce^{-\rho_S(x)}$ is a density in $\mathcal{F}_d$. Defining $F_j:=\conv\{u_1,\dotsc,u_{j-1},u_{j+1},\dotsc,u_{d+1}\}$ for $1\leq j\leq d+1$, we see that the facets of $S$ are precisely $F_1,\dotsc,F_{d+1}$, and hence that $\log f_{d+1,\,0}=\log c-\rho_S$ is affine on $C_j:=\bigcup_{\lambda\in [0,\infty)}\lambda F_j$ for each $1\leq j\leq d+1$.
Since $\supp f_{d+1,\,0}=\R^d=\bigcup_{\lambda\in [0,\infty)}\lambda S=\bigcup_{j=1}^{\,d+1}C_j$, it follows that $f_{d+1,\,0}\in\mathcal{F}_d^{d+1}(\mathcal{P}_d^0)$, as required.
\end{proof}
To conclude this subsection, we also record the fact that if $d\leq 3$, then every polytope in $\mathcal{P}^m\equiv\mathcal{P}_d^m$ can be triangulated into $O(m)$ simplices.
\begin{lemma}
\label{Lem:euler}
If $d\leq 3$ and $P\in\mathcal{P}^m$ is a polytope, then $P$ has at most $2m$ vertices and there is a triangulation of $P$ that contains at most $6m$ simplices.
\end{lemma}
\begin{proof}
The cases $d=1,2$ are trivial, so suppose now that $d=3$. If $P$ has $v$ vertices, $e$ edges and $f$ facets, then Euler's formula asserts that $v-e+f=2$~\citep[e.g.][Section~20.1]{K04}. The edges of $P$ induce a graph structure on the set of vertices of $P$, and it is easy to see that the degree of every vertex is at least 3. This implies that $2e\geq 3v$, and we deduce from Euler's formula that $v\leq 2(f-2)$. The result above follows from the fact that $P$ has a triangulation that contains at most $3v-11$ simplices \citep{EPW90}.
\end{proof}
\begin{remark*}
In the case $d=3$, we also have the bound $2e\geq 3f$, since every face has at least 3 edges and every edge belong to exactly 2 faces. It then follows from Euler's formula that $f\leq 2(v-2)$.
\end{remark*}
In addition, when $d=2$, we have the following useful result about polyhedral subdivisions.
\begin{lemma}
\label{Lem:polysub2d}
If $d=2$ and $E_1,\dotsc,E_k$ is a subdivision of a polyhedral set $P\in\mathcal{P}^m$, then $\sum_{j=1}^k \,\abs{\mathscr{F}(E_j)}\lesssim k+m$.
\end{lemma}
\begin{proof}
\begin{figure}[htb]
\vspace{-0.3cm}
\centering
\includegraphics[width=0.9\textwidth]{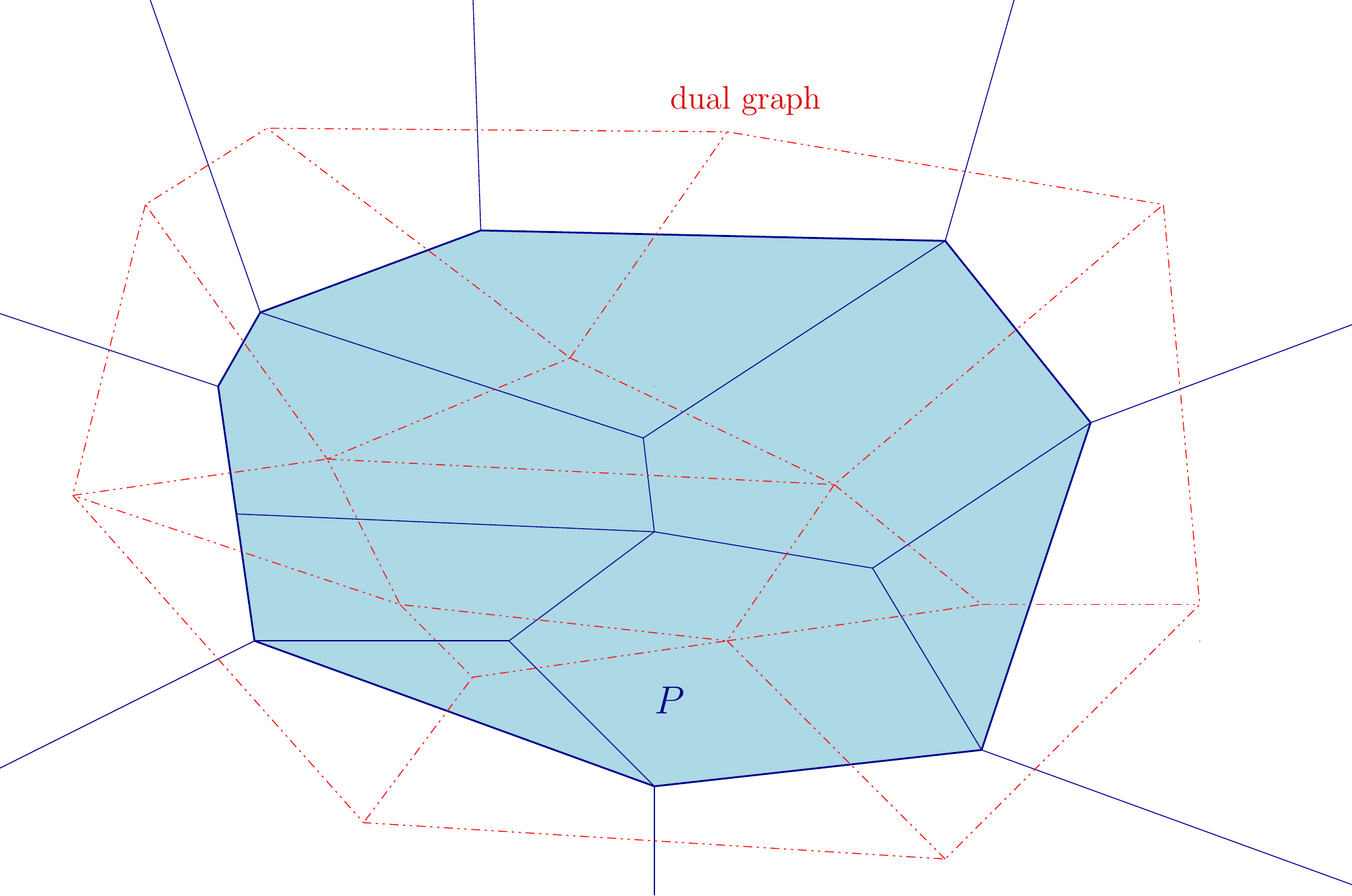}

\vspace{0.3cm}
\caption{Illustration of the configuration in \protect{Lemma~\ref{Lem:polysub2d}}.}
\label{Fig:polysub2d}
\end{figure}
Regardless of whether or not $P$ is bounded, observe that $\cm{P}$ can be subdivided into $m$ polyhedral sets in such a way that the intersection of each of these sets with $P$ is a facet of $P$. It may be helpful to refer to Figure~\ref{Fig:polysub2d}, in which $P$ is represented by the blue shaded region and the subdivisions of $P$ and $\cm{P}$ are indicated by the blue lines. 

Having dissected $\R^2=P\cup\cm{P}$ into $k+m$ polyhedral regions, we now form the \textit{dual graph} $G'$ of this configuration by fixing a point in each region and joining two points by an edge if the corresponding regions share a (non-trivial) line segment. This is highlighted in red in Figure~\ref{Fig:polysub2d}. In this graph, the degree of the vertex inside $E_j$ is simply the number of facets of $E_j$, and the sum of the degrees of all $k+m$ vertices is equal to twice the number of edges of $G'$. But $G'$ is \textit{planar} (i.e.\ it can be drawn in the plane in such a way that no two edges cross), so it has at most $3(k+m)-6$ edges \citep[Section~20.1]{K04}. Together with the previous observations, this implies the desired result.
\end{proof}

\subsection{Auxiliary results for bracketing entropy calculations}
\label{Subsec:EntropyCalcs}

The results in this subsection hold for any $d\in\N$. First, we consider log-concave functions $f_0,f$ whose restrictions to some $K\in\mathcal{K}^b$ are close in Hellinger distance, and obtain pointwise bounds on $f$ under the assumption that $f_0$ is bounded away from $0$ on $K$. Henceforth, we write $f_K:=f_{K,0}=\inv{\mu_d(K)}\Ind_K\in\mathcal{F}^1$ for the uniform density on $K$.

Recall that $\Phi\equiv\Phi_d$ denotes the set of all upper semi-continuous, concave functions $\phi\colon\R^d\to [-\infty,\infty)$, and that $\mathcal{G}\equiv\mathcal{G}_d=\{e^\phi:\phi\in\Phi\}$. For $\phi\in\Phi$ and $x\in\R^d$, let $D_{\phi,x}:=\{w\in\R^d:\phi(w)>\phi(x)\}$.
\begin{lemma}
\label{Lem:hellunbds}
Fix $K\in\mathcal{K}^b$ and $f_0\in\mathcal{F}$, and suppose that there exists $\theta\in [1,\infty)$ such that $f_0\geq\inv{\theta}f_K$ on $K$. Let $f\in\mathcal{G}$ and $\delta>0$ be such that $\int_K\,\bigl(\sqrt{f}-\sqrt{f_0}\bigr)^2\leq\delta^2$. Then setting $\phi:=\log f\in\Phi$, we have the following: 
\begin{enumerate}[label=(\roman*)]
\item If $x\in K$ satisfies $\mu_d(K\setminus D_{\phi,x})\geq 4\delta^2\theta\mu_d(K)$, then\\ $\phi(x)+\log\mu_d(K)\geq -4\delta\{\theta\mu_d(K)/\mu_d(K\setminus D_{\phi,x})\}^{1/2}-\log\theta$.
\item If $\theta=1$ and $\delta\in (0,2^{-3/2}]$, then $f_0=f_K$ and $\phi+\log\mu_d(K)\leq (8\sqrt{2}d)\delta$ on $K$.
\item There exist $s_d\geq 1$, depending only on $d$, and a universal constant $s'>0$ such that if $\theta>1$ and $\delta\in (0,(8\theta)^{-1/2}]$, then $\phi+\log\mu_d(K)\leq\log(s_d\log^d(e\theta)-s_d+1)+s'(d^{\,(d+1)}\delta)^{2/(d+2)}$ on $K$.
\end{enumerate}
\end{lemma}
\begin{proof}
For (i), we may assume that $\phi(x)+\log\mu_d(K)< -\log\theta$, for otherwise there is nothing to prove. Setting $c:=\theta\mu_d(K)$, we have $e^{\phi(w)/2}\leq e^{\phi(x)/2}<c^{-1/2}\leq\sqrt{f_0(w)}$ for all $w\in K\setminus D_{\phi,x}$, so
\[\delta^2\geq\int_{K\setminus D_{\phi,x}}\,\bigl(\sqrt{f_0}-e^{\phi/2}\bigr)^2\geq\int_{K\setminus D_{\phi,x}}\,\bigl(c^{-1/2}-e^{\phi/2}\bigr)^2\geq\mu_d(K\setminus D_{\phi,x})\bigl(c^{-1/2}-e^{\phi(x)/2}\bigr)^2.\]
Since $\log(1-t)\geq -2t$ for all $t\in [0,1/2]$, we deduce that if $\mu_d(K\setminus D_{\phi,x})\geq 4\delta^2 c$, then
\[\phi(x)\geq 2\log\rbr{1-\frac{\delta c^{1/2}}{\mu_d(K\setminus D_{\phi,x})^{1/2}}}-\log c\geq-\frac{4\delta c^{1/2}}{\mu_d(K\setminus D_{\phi,x})^{1/2}}-\log c,\]
as required. Turning to assertions (ii) and (iii), we suppose for now that $\mu_d(K)=1$ and begin by establishing that:
\begin{claim*}
There exists $s_d>0$, depending only on $d$, such that $\phi_0:=\log f_0\leq\log(s_d\log^d(e\theta)-s_d+1)=:t_d(\theta)$ on $K$.
\end{claim*}
\begin{proof}[Proof of Claim]
Since $\phi_0\in\Phi$ and $K\in\mathcal{K}$, we have $\sup_{x\in K}\phi_0(x)=\sup_{x\in\Int K}\phi_0(x)$. Indeed, for any $z\in K$, note that if $y\in\Int K$, then $[y,z)\subseteq\Int K$ \citep[Lemma~1.1.9]{Sch14}, so it follows from the concavity of $\phi_0$ that $\phi_0(z)\leq\sup_{x\in [y,z)}\phi_0(x)\leq\sup_{x\in\Int K}\phi_0(x)$. Thus, it will suffice to show that a bound of the above form holds on $\Int K$.

Fix $x\in\Int K$ and assume that $a:=\phi_0(x)+\log\theta>0$, for otherwise there is nothing to prove. Now $K-x\in\mathcal{K}$ and $0\in\Int(K-x)$, so as remarked at the start of Section~\ref{Sec:LogkaffineSupp}, the Minkowski functional $\rho_{K-x}\colon w\mapsto\inf\{\lambda>0:w\in\lambda(K-x)\}\in [0,\infty)$ is convex. Thus, the function $\psi_0\colon\R^d\to\R$ defined by
\[\psi_0(w):=a\bigl(1-\rho_{K-x}(w-x)\bigr)-\log\theta\]
is concave, and note that the restriction of $\psi_0$ to any ray with endpoint $x$ is an affine function. Also, $\psi_0(x)=\phi_0(x)$, and since $f_0\geq\inv{\theta}f_K=\inv{\theta}$ on $K$ by hypothesis, we have $\psi_0(w)=-\log\theta\leq\phi_0(w)$ for all $w\in\partial K$. Since $\phi_0$ is concave, this implies that $-\log\theta\leq\psi_0\leq\phi_0$ on $K$. Recalling that $f_0$ is a density and setting $g(r):=\theta^{-1}e^{a(1-r)}$ for $r\geq 0$ in~\citet[Lemma~S1]{XS19}, we deduce that
\begin{equation}
\label{Eq:thetaubd}
1\geq\int_K e^{\phi_0}\geq\int_K e^{\psi_0}=\theta^{-1}d\int_0^1 r^{d-1}e^{a(1-r)}\,dr=\theta^{-1}\,\frac{d!\,e^a}{a^d}\gamma(d,a)=\theta^{-1}\,\sum_{\ell=0}^\infty\frac{d!}{(\ell+d)!}\,a^\ell,
\end{equation}
where $\gamma(d,a)$ is defined as in~\eqref{Eq:ratiobds} and the second equality above follows by similar reasoning to that used to obtain~\eqref{Eq:poigam}. The final expression in~\eqref{Eq:thetaubd} is bounded below by $(1+a/(d+1))/\theta$, so
\begin{equation}
\label{Eq:thetaubd1}
\phi_0(x)=a-\log\theta\leq (d+1)(\theta-1)-\log\theta\leq (d+1)(\theta-1).
\end{equation}

Suppose now that $a\geq 1$. Since $s\mapsto\sum_{\ell=0}^\infty\,d!\,s^\ell/(\ell+d)!$ is increasing on $[0,\infty)$, it follows from~\eqref{Eq:thetaubd} that $\theta\geq d!\,e\gamma(d,1)=:\theta_d$. In addition, introducing random variables $W\sim\Po(1)$ and $W_a\sim\Po(a)$, we see that $\gamma(d,a)=\Pr(W_a\geq d)\geq\Pr(W\geq d)=\gamma(d,1)$. Thus, defining $l\colon (0,\infty)\to\R$ by $l(s):=e^s/s^d$, we deduce from~\eqref{Eq:thetaubd} that
\begin{equation}
\label{Eq:thetaubd2}
l(a)=e^a/a^d\leq\theta/(d!\,\gamma(d,1)).
\end{equation}
Observe now that there exists $C_d>1$, depending only on $d$, such that for all $\tilde{\theta}\geq\theta_d$, we have $\log(C_d\,\tilde{\theta}\log^d\tilde{\theta})\geq d$ and
$d!\,\gamma(d,1)C_d\log^d\tilde{\theta}\geq 2^{d-1}\bigl(\log^d C_d+\log^d(\tilde{\theta}\log^d\tilde{\theta})\bigr)\geq\log^d(C_d\,\tilde{\theta}\log^d\tilde{\theta})$,
so that
\begin{equation}
\label{Eq:thetaubd3}
l\bigl(\log(C_d\,\tilde{\theta}\log^d\tilde{\theta})\bigr)=\frac{C_d\,\tilde{\theta}\log^d\tilde{\theta}}{\log^d(C_d\,\tilde{\theta}\log^d\tilde{\theta})}\geq\frac{\tilde{\theta}}{d!\,\gamma(d,1)}.
\end{equation}
Since $l$ is increasing on $[d,\infty)$, it follows from~\eqref{Eq:thetaubd2} and~\eqref{Eq:thetaubd3} that
\begin{equation}
\label{Eq:thetaubd4}
\phi_0(x)=a-\log\theta\leq\log_+(C_d\log^d\theta),
\end{equation}
provided that $\phi_0(x)+\log\theta=a\geq 1$. In fact,~\eqref{Eq:thetaubd4} holds even when $a<1$, so in all cases, we deduce from this and~\eqref{Eq:thetaubd1} that $\phi_0(x)\leq\min\{(d+1)(\theta-1),\log_+(C_d\log^d\theta)\}$. Note that there exists $\tilde{\theta}_d>1$, depending only on $d$, such that $(d+1)(\tilde{\theta}-1)\geq\log_+(C_d\log^d\tilde{\theta})$ for all $\tilde{\theta}\geq\tilde{\theta}_d$. Also, since $\tilde{\theta}\mapsto\log^d(e\tilde{\theta})-1$ is increasing and has strictly positive derivative at $\tilde{\theta}=1$, we have $e^{(d+1)(\tilde{\theta}-1)}-1\lesssim_d\tilde{\theta}-1\lesssim_d\log^d(e\tilde{\theta})-1$ for all $\tilde{\theta}\in [1,\tilde{\theta}_d]$. We conclude that there exists $s_d\geq 1$, depending only on $d$, such that $\phi_0(x)\leq\log(s_d\log^d(e\theta)-s_d+1)$, as required.
\end{proof}
Proceeding with the proofs of (ii) and (iii), we may assume without loss of generality that $\supp f=\dom\phi\subseteq K$, since otherwise we can replace $f$ by $f\Ind_K$; indeed, the hypotheses and conclusions depend on $f$ only through $\restr{f}{K}$. Then $\sqrt{f_0}-e^{\phi/2}\geq\theta^{-1/2}$ on $K\setminus\dom\phi$ under our assumption that $\mu_d(K)=1$, so $\delta^2\geq\int_K\,(\sqrt{f_0}-e^{\phi/2})^2\geq\theta^{-1}\mu_d(K\setminus\dom\phi)=\theta^{-1}(1-\mu_d(\dom\phi))$. Since $\dom\phi$ is convex, this implies that $\Int\dom\phi$ is non-empty, and since $\phi$ is concave, it follows as in the first paragraph of the proof of the Claim above that $\sup_{\,x\in K}\phi(x)=\sup_{\,x\in\dom\phi}\phi(x)=\sup_{x\in\Int\dom\phi}\phi(x)$. Thus, it will suffice to show that the bounds in (ii) and (iii) hold on $\Int\dom\phi$. 

Fix $x\in\Int\dom\phi$ and assume that $\phi(x)>t\equiv t_d(\theta)$, for otherwise there is nothing to prove. Since $\phi$ is upper semi-continuous, the set $L:=\{u\in K:\phi(u)\geq t\}$ is compact and convex \citep[Theorem~7.1]{Rock97},
and since $\phi$ is continuous on $\Int\dom\phi$ \citep[Theorem~1.5.3]{Sch14}, we have $x\in\Int L$.
Thus, $L-x\in\mathcal{K}$ and $0\in\Int(L-x)$, so the Minkowski functional $\rho_{L-x}\colon\R^d\to [0,\infty)$ is convex.
Since $b:=\phi(x)-t>0$, the function $\psi\colon\R^d\to\R$ defined by
\[\psi(w):=b\bigl(1-\rho_{L-x}(w-x)\bigr)+t\]
is concave, and as was the case for the function $\psi_0$ defined in the proof of the Claim, the restriction of $\psi$ to any ray with endpoint $x$ is an affine function. Also, $\psi(x)=\phi(x)$ and $\psi(w)=t\leq\phi(w)$ for all $w\in\partial L$, so by the Claim above and the concavity of $\phi$, we deduce that
\begin{align}
\label{Eq:largealpha}
&\phi_0\leq t\leq\psi\leq\phi\;\:\text{on $L$,\;\,and}\\
\label{Eq:smallalpha}
&\phi_0\geq -\log\theta\geq t-\beta b\geq\psi\geq\phi\;\:\text{on }K\setminus\bigl(x+(1+\beta)(L-x)\bigr)
\end{align}
for all $\beta\geq (t+\log\theta)/b$. Now fix any $\alpha\in (0,1)$ and suppose first that $\mu_d(L)\geq\alpha$. By applying~\eqref{Eq:largealpha} and setting $g(r):=b^2(1-r)^2$ for $r\geq 0$ in~\citet[Lemma~S1]{XS19}, we find that
\begin{align}
\delta^2\geq\int_{L}\,(e^{\phi/2}-e^{\phi_0})^2\geq\int_{L}\,(e^{\psi/2}-e^{t/2})^2\geq e^t\int_{L}\,\frac{(\psi-t)^2}{4}&=\frac{d\mu_d(L)e^t}{4}\int_0^1 b^2(1-r)^2\,r^{d-1}\,dr\notag\\
\label{Eq:ubhell1}
&\geq\frac{\alpha e^t}{2(d+1)(d+2)}\,b^2.
\end{align}
On the other hand, suppose instead that $\mu_d(L)<\alpha$. Fix $\beta\geq (t+\log\theta)/b$, and let $t':=t-\beta b$ and $L_\beta:=x+(1+\beta)(L-x)$. Then $\mu_d(K\setminus L_\beta)\geq\mu_d(K)-\mu_d(L_\beta)=1-(1+\beta)^d\,\mu_d(L)>1-(1+\beta)^d\,\alpha$. Together with~\eqref{Eq:smallalpha}, this implies that
\begin{equation}
\label{Eq:ubhell2}
\delta^2\geq\int_{K\setminus L_\beta}\,\bigl(e^{\phi_0/2}-e^{\phi/2}\bigr)^2\geq\mu_d(K\setminus L_\beta)\bigl(\theta^{-1/2}-e^{t'/2}\bigr)^2\geq\cbr{1-\rbr{1+\frac{t-t'}{b}}^d\alpha}\bigl(\theta^{-1/2}-e^{t'/2}\bigr)^2.
\end{equation}

To obtain the bounds in (ii) and (iii), we now substitute suitably chosen values of $\alpha$ and $\beta$ into~\eqref{Eq:ubhell1} and~\eqref{Eq:ubhell2}. For (ii), let $\alpha=1/4$ and $\beta=2^{1/d}-1$. Since $t=0$, it follows from~\eqref{Eq:ubhell1} that if $\mu_d(L)\geq 1/4$, then $\phi(x)=b\leq\sqrt{8(d+1)(d+2)}\,\delta$. Otherwise, if $\mu_d(L)<1/4$, then since $\theta=1$ and $t'=-\beta\phi(x)$, we deduce from~\eqref{Eq:ubhell2} that
\begin{align*}
\phi(x)\leq-\frac{2}{\beta}\log(1-\sqrt{2}\delta)\leq\frac{2\sqrt{2}\delta}{\beta(1-\sqrt{2}\delta)}&=\frac{2\sqrt{2}\delta}{1-\sqrt{2}\delta}\rbr{1+2^{1/d}+\dotsc+2^{(d-1)/d}}\\
&\leq 4\sqrt{2}\delta\cdot 2d=(8\sqrt{2}d)\delta,
\end{align*}
where the second inequality above follows since $\delta\in (0,2^{-3/2}]$ by assumption. Therefore, (ii) holds when $\mu_d(K)=1$. As for (iii), suppose that $\theta>1$, and let $\alpha=(\delta/d)^{2d/(d+2)}e^{-t}/4\in (0,1/4)$ and $\beta=\{t+\log(4\theta)\}/b$, so that $t'=-\log(4\theta)$. If $\mu_d(L)\geq\alpha$, then $\phi(x)-t=b\leq\sqrt{8(d+1)(d+2)}\,d^{\,d/(d+2)}\delta^{2/(d+2)}\lesssim (d^{\,d+1}\delta)^{2/(d+2)}$ by~\eqref{Eq:ubhell1}. Otherwise, if $\mu_d(L)<\alpha$, then since $8\theta\delta^2\leq 1$ by assumption, we deduce from~\eqref{Eq:ubhell2} that \[\rbr{1+\frac{t+\log(4\theta)}{b}}^d\alpha\geq 1-\frac{\delta^2}{(\theta^{-1/2}/2)^2}=1-4\theta\delta^2\geq 1/2.\]
Now since $\alpha\leq 1/4$, we have $(2\alpha)^{-1/d}-1\geq (1-2^{-1/d})(2\alpha)^{-1/d}\geq\alpha^{-1/d}/(4d)>0$,
so by rearranging the inequality above, we conclude that
\begin{equation}
\label{Eq:ubhell3}
b\leq\frac{t+\log(4\theta)}{(2\alpha)^{-1/d}-1}\leq 4d\alpha^{1/d}\bigl(t+\log(4\theta)\bigr)\leq 4d^{\,d/(d+2)}\delta^{2/(d+2)}\,\frac{t+\log(4\theta)}{e^{t/d}}.
\end{equation}
Recalling that $e^{t_d(\tilde{\theta})}=s_d\log^d(e\tilde{\theta})-s_d+1\geq\log^d(e\tilde{\theta})$ for all $\tilde{\theta}\in [1,\infty)$ and that $se^{-s/d}\leq d/e$ for all $s\in [0,\infty)$, we see that there exists a universal constant $C'>0$ such that $\{t_d(\tilde{\theta})+\log(4\tilde{\theta})\}/e^{t_d(\tilde{\theta})/d}\leq C'd$ for all $\tilde{\theta}\in [1,\infty)$. 
Together with~\eqref{Eq:ubhell3}, this implies that $\phi(x)-t=b\lesssim (d^{\,d+1}\delta)^{2/(d+2)}$ when $\mu_d(L)<\alpha$. This completes the proof of (iii) in the special case where $\mu_d(K)=1$.

Having established (ii) and (iii) under the assumption that $\mu_d(K)=1$, we now extend these results to arbitrary $K\in\mathcal{K}^b$ by means of a simple scaling argument. For a general $K\in\mathcal{K}^b$, suppose that $K,\theta,f_0,f$ satisfy the conditions of the lemma and that $\delta\in (0,(8\theta)^{-1/2}]$. Let $\lambda:=\mu_d(K)^{1/d}$ and $K':=\inv{\lambda}K$, so that $\mu_d(K')=1$. Then defining $\tilde{f}_0,\tilde{f}\colon\R^d\to [0,\infty)$ by $\tilde{f}_0(x):=\lambda^df_0(\lambda x)$ and $\tilde{f}(x):=\lambda^df(\lambda x)$, we see that $\tilde{f}_0\in\mathcal{F}$ and $\tilde{f}\in\mathcal{G}$. Moreover, $\tilde{f}_0(x)\geq\lambda^d\,\inv{\theta}f_K(\lambda x)=\inv{\theta}f_{K'}(x)$ for all $x\in K'$ and $\int_{K'}\,\bigl(\tilde{f}^{1/2}-\tilde{f}_0^{1/2}\bigr)^2=\int_K\,\bigl(f^{1/2}-f_0^{1/2}\bigr)^2\leq\delta^2$.
This shows that $K',\theta,\tilde{f}_0,\tilde{f}$ satisfy the conditions of the lemma. Now for any $x\in K$, we have $\lambda^{-1}x\in K'$, and since $\mu_d(K')=1$, it follows from the bounds obtained hitherto that
\[\log f(x)+\log\mu_d(K)=\log f(x)+\log(\lambda^d)=\log\tilde{f}(\lambda^{-1}x)\leq
\begin{cases}
(8\sqrt{2}d)\delta\quad&\text{if }\theta=1\\
t_d(\theta)+s'(d^{\,(d+1)}\delta)^{2/(d+2)}\quad&\text{if }\theta>1,
\end{cases}
\]
as required.
\end{proof}
In addition, we derive a lower bound on $\sup_{x\in K_{\alpha,1}^+}\{\phi(x)+\tm{\alpha}x+\log c_{K,\alpha}\}$ that holds whenever $\phi\in\Phi$ and $e^\phi$ is close in Hellinger distance to some $f_{K,\alpha}\in\mathcal{F}_\star^1$ with $\alpha\neq 0$. For $d\in\N$, define $\nu_d:=\{2^{-3}d^{-d}e^{-1}\gamma(d,1)\}^{1/2}$, where $\gamma(d,1)$ is taken from Lemma~\ref{Lem:log1affbds}. Recall the definitions of $K_{\alpha,t}^+$ and $c_{K,\alpha}$ from~\eqref{Eq:kalphat} and Proposition~\ref{Prop:log1aff} respectively.
\begin{lemma}
\label{Lem:hellnunbd}
Fix $f_{K,\alpha}\in\mathcal{F}_\star^1$ with $K\in\mathcal{K}$ and $\alpha\neq 0$. For $\phi\in\Phi$, define $\tilde{\phi}_{K,\alpha}\colon\R^d\to [-\infty,\infty)$ by $\tilde{\phi}_{K,\alpha}:=\phi(x)+\tm{\alpha}x+\log c_{K,\alpha}$. If $\int_K\,\bigl(e^{\phi/2}-f_{K,\alpha}^{1/2}\bigr)^2\leq\delta^2$ for some $\delta\in (0,\nu_d]$, then there exists $x_-\in K_{\alpha,1}^+$ such that $\tilde{\phi}_{K,\alpha}(x_-)>-2$.
\end{lemma}
\begin{proof}
Let $\psi:=\tilde{\phi}_{K,\alpha}$. We first establish that there exists $x_-\in K':=K_{\alpha,1}^+$ with the property that $\mu_d(K'\cap H^+)\geq 2^{-1}d^{-d}\mu_d(K')$ whenever $H^+$ is a half-space whose boundary contains $x_-$. Then we show that any such $x_-$ necessarily satisfies $\psi(x_-)\geq -2$. To justify the first claim above, we apply Fritz John's theorem \citep{J48}, which asserts that there exists an invertible affine map $T\colon\R^d\to\R^d$ such that $\bar{B}(0,1/d)\subseteq\tilde{K}:=T(K')\subseteq\bar{B}(0,1)$. Now if $H^+$ is any hyperplane whose boundary contains 0, then $\mu_d(\tilde{K}\cap H^+)\geq 2^{-1}\mu_d(\bar{B}(0,1/d))=2^{-1}d^{-d}\mu_d(\bar{B}(0,1))\geq 2^{-1}d^{-d}\mu_d(\tilde{K})$, so $x_-:=\inv{T}(0)\in K'$ has the required property.

We may now assume that $\psi(x)\leq 0$ for every $x\in K'$, since otherwise the desired conclusion follows trivially. Then $\psi\in\Phi$ and $e^{\psi(u)/2}\leq e^{\psi(x_-)/2}\leq 1$ for all $u\in K'\setminus D_{\psi,x_-}\subseteq K$, so
\[\delta^2\geq\int_{K'\setminus D_{\psi,x_-}}\!\!\bigl(f_{K,\alpha}^{1/2}-e^{\phi/2}\bigr)^2
=\int_{K'\setminus D_{\psi,x_-}}\!\!\!\frac{e^{-\tm{\alpha}u}}{c_{K,\alpha}}\,\bigl(1-e^{\psi(u)/2}\bigr)^2\,du\geq (1-e^{\psi(x_-)/2})^2\!\int_{K'\setminus D_{\psi,x_-}}\!\!\!\frac{e^{-\tm{\alpha}u}}{c_{K,\alpha}}\,du.\]
Since $x_-\notin D_{\psi,x_-}$ and $D_{\psi,x_-}$ is convex, the separating hyperplane theorem \citep[Theorem~1.3.4]{Sch14} implies that there exists a open half-space $H^+$ such that $x_-\in\partial H^+$ and $K'\cap H^+\subseteq K'\setminus D_{\psi,x_-}$. In view of the defining property of $x_-$ and Lemma~\ref{Lem:log1affbds}, it follows that \[\int_{K'\setminus D_{\psi,x_-}}\!\frac{e^{-\tm{\alpha}u}}{c_{K,\alpha}}\,du\geq\int_{K'\cap H^+}\!\frac{e^{-\tm{\alpha}u}}{c_{K,\alpha}}\,du\geq\frac{e^{-1\,}\mu_d(K'\cap H^+)}{c_{K,\alpha}}\geq\frac{(2e)^{-1}d^{-d}\mu_d(K')}{c_{K,\alpha}}\geq 4\nu_d^2.\]
By combining the bounds in the two previous displays, we conclude that
\[\psi(x_-)\geq 2\log\rbr{1-\frac{\delta}{2\nu_d}}>-2,
\]
where we have used the fact that $\delta\in (0,\nu_d]$ to obtain the final inequality. This completes the proof of the lemma.
\end{proof}
\label{Disc:SimplexEntropy}
The remaining results in this subsection prepare the ground for the proof of Proposition~\ref{Prop:SimplexEntropy}, which establishes a local bracketing entropy bound for classes of log-concave functions $f$ that lie within small Hellinger neighbourhoods $\mathcal{G}(f_S,\delta)$ of the uniform density $f_S$ on a $d$-simplex $S$. As mentioned after the statement of Theorem~\ref{Thm:k1} in Section~\ref{Sec:Logkaffine} of the main text, we now develop further the pointwise lower bound from Lemma~\ref{Lem:hellunbds}(i) by identifying subsets (or `invelopes') $J_\eta^S$ of a $d$-simplex $S$ with the property that $\mu_d(S\setminus D_{\phi,x})\gtrsim\eta\mu_d(S)$ for all $x\in J_\eta^S$, which ensures that $\log f+\log\mu_d(S)\gtrsim -\delta/\eta^{1/2}$ for all $f\in\mathcal{G}(f_S,\delta)$. This is the purpose of Lemma~\ref{Lem:invelopesimp}(iii), which handles the case where $S$ is a regular $d$-simplex. However, in the proof of Proposition~\ref{Prop:SimplexEntropy}, it turns out that for technical reasons, we cannot work directly with the invelopes we obtain in this lemma; instead, the strategy we pursue involves constructing polytopal approximations that satisfy the conditions of Corollary~\ref{Cor:polyapproxsimp}. For technical convenience, we first derive analogous results in the case where the domain is $[0,1]^d$ (Lemmas~\ref{Lem:invelopebox} and~\ref{Lem:polyapprox}), before adapting the relevant geometric constructions to the simplicial setting described above. The volume bound in Lemma~\ref{Lem:invelopebox}(iii) and the properties in Corollary~\ref{Cor:polyapproxsimp} are exploited in the derivation of the local bracketing entropy bounds in Proposition~\ref{Prop:SimplexEntropy}, where they help to ensure that the exponent of $\delta$ matches that of $\varepsilon$; see the paragraph containing~\eqref{Eq:Rjvols}. This in turn is ultimately responsible for the essentially parametric adaptive rates that we are able to establish in Section~\ref{Sec:Logkaffine}.

Before proceeding, we make some further definitions. Fix $1\leq k\leq d$ and let $P\subseteq\R^d$ be a $k$-simplex or a $k$-parallelotope. Then for each vertex $v$ of $P$, there are exactly $k$ other vertices $v_1,\dotsc,v_k$ of $P$ for which $[v,v_1],\dotsc,[v,v_k]$ are edges of $P$.
Setting $w_j:=v_j-v$ for $1\leq j\leq k$, we note that
\[P=
\begin{cases}
\;\bigl\{v+\sum_{j=1}^k\lambda_jw_j:\lambda_j\geq 0,\,\sum_{j=1}^k\lambda_j\leq 1\text{ for all }j\bigr\}&\quad\text{if $P$ is a $k$-simplex}\\
\;\bigl\{v+\sum_{j=1}^k\lambda_jw_j:0\leq\lambda_j\leq 1\text{ for all }j\bigr\}&\quad\text{if $P$ is a $k$-parallelotope}.
\end{cases}
\]
For any fixed $x\in\relint P$, there exist unique $\tilde{x}_1,\dotsc,\tilde{x}_k\in (0,1)$ such that $x=v+\sum_{j=1}^k\tilde{x}_jw_j$. Then \[P^v(x):=\bigl\{v+\textstyle\sum_{j=1}^k\lambda_j\tilde{x}_jw_j:0\leq\lambda_j\leq 1\text{ for all }j\bigr\}\]
is a closed $k$-parallelotope, two of whose vertices are $v$ and $x$. Observe that $P^v(x)\subseteq P$; indeed, if $\lambda_1,\dotsc,\lambda_k\in [0,1]$, then $\lambda_j\tilde{x}_j\in [0,1]$ for all $1\leq j\leq k$ and $\sum_{j=1}^k\lambda_j\tilde{x}_j\leq\sum_{j=1}^k\tilde{x}_j$. Also, a simple calculation shows that
\begin{equation}
\label{Eq:pvx}
\mu_k(P^v(x))=
\begin{cases}
\,k!\,\mu_k(P)\,\prod_{j=1}^k\tilde{x}_j&\quad\text{if $P$ is a $k$-simplex}\\
\,\phantom{k!\,}\mu_k(P)\,\prod_{j=1}^k\tilde{x}_j&\quad\text{if $P$ is a $k$-parallelotope}.
\end{cases}
\end{equation}
The following elementary geometric result will be used in the proofs of Lemmas~\ref{Lem:invelopebox} and~\ref{Lem:invelopesimp}. 
\begin{lemma}
\label{Lem:simplevsets}
Let $P\subseteq\R^d$ be as above and fix $x\in\relint P$. If $H^+\subseteq\R^d$ is a closed half-space such that $x\in\partial H^+$, then there exists a vertex $v$ of $P$ for which $P^v(x)\subseteq P\cap H^+$.
\end{lemma}
\begin{proof}
Since $x\in\partial H^+$, there exists a unit vector $\theta\in\R^d$ such that $H^+=\{u\in\R^d:\tm{\theta}u\leq\tm{\theta}x\}$, and since $P$ is compact and convex, we can find a vertex $v$ of $P$ such that $v\in\argmin_{u\in P}\tm{\theta}u$. To see that $v$ has the required property, define $v_j,w_j,\tilde{x}_j$ as above for $1\leq j\leq k$, and observe that $\tm{\theta}w_j=\tm{\theta}(v_j-v)\geq 0$ for all $j$ by our choice of $v$. Therefore, since $\tilde{x}_j>0$ for all $j$, it follows from the definition of $P^v(x)$ that $\tm{\theta}u\leq\tm{\theta}x$ for all $u\in P^v(x)$, so $P^v(x)\subseteq H^+$, as required.
\end{proof}
We now consider $Q=Q_d:=[0,1]^d$. For each $\xi=(\xi_1,\dotsc,\xi_d)\in\{0,1\}^d$, let $g_{\xi}\colon Q\to Q$ be the function $(x_1,\dotsc,x_d)\mapsto (\xi_1+(-1)^{\xi_1}x_1,\dotsc,\xi_d+(-1)^{\xi_d}x_d)$. Then $G_R(Q)\equiv G_R(Q_d):=\{g_{\xi}:\xi\in\{0,1\}^d\}$ is the subgroup of (affine) isometries of $Q$ generated by reflections in the affine hyperplanes $\{(w_1,\dotsc,w_d)\in\R^d:w_j=1/2\}$, where $j=1,\dotsc,d$. We say that $D\subseteq [0,1]^d$ is \emph{$G_R(Q)$-invariant} if $g(D)=D$ for all $g\in G_R(Q)$.

Moreover, let $M\colon Q\to [0,1/2]^d$ be the function $(x_1,\dotsc,x_d)\mapsto (x_1\wedge (1-x_1),\dotsc,x_d\wedge (1-x_d))$. Then for each $x\in Q$, note that $M(x)$ is an element of the orbit of $x$ under $G_R(Q)$ that lies in $[0,1/2]^d$, and also that $M(g(x))=M(x)$ for all $g\in G_R(Q)$.

\begin{figure}[htb]
\centering
\includegraphics[width=0.7\textwidth]{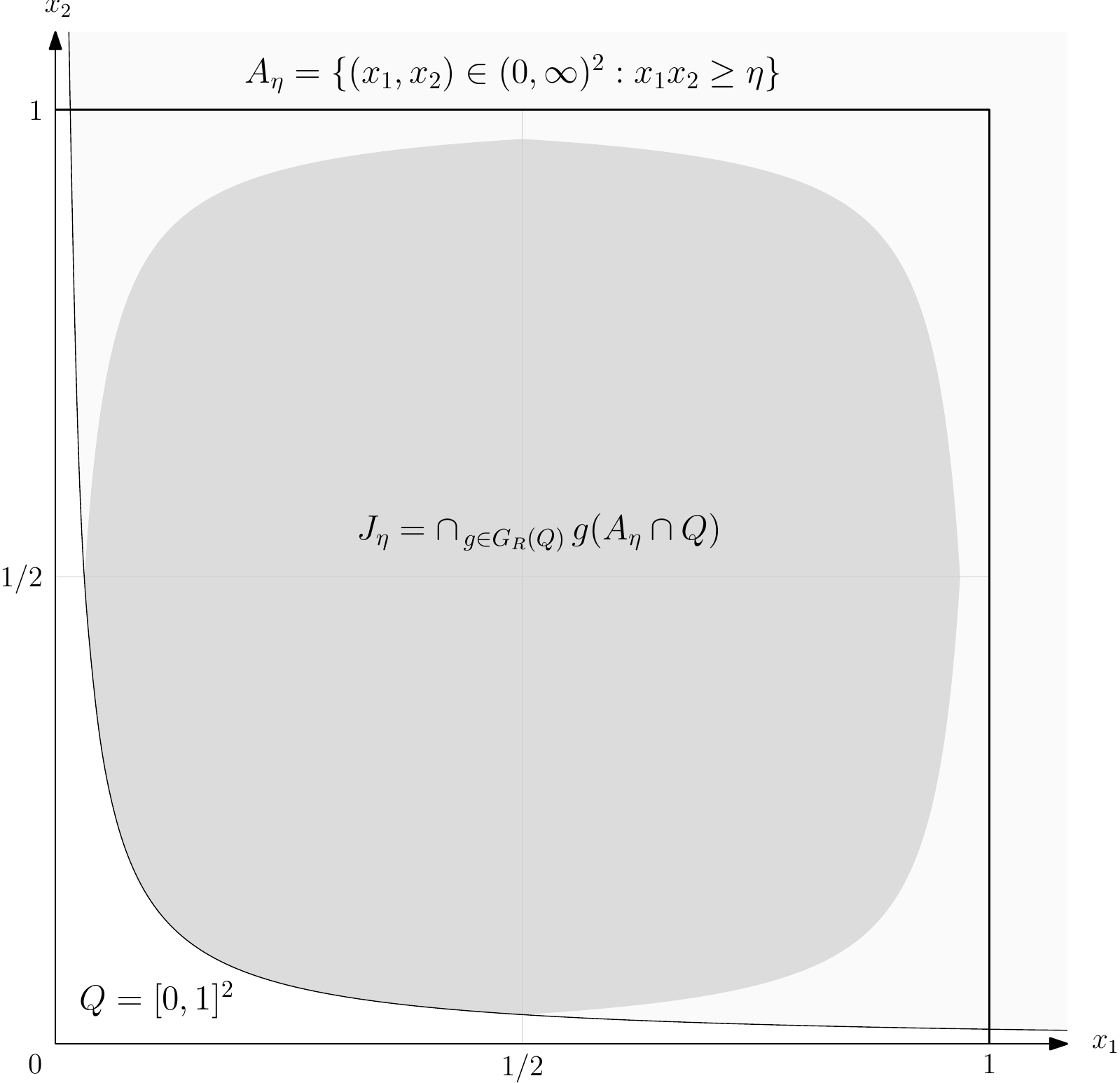}

\vspace{0.3cm}
\caption{Illustration of the sets $A_\eta$ (union of the lighter and darker regions) and $J_\eta$ (darker region) in \protect{Lemma~\ref{Lem:invelopebox}} when $d=2$.} 
\label{Fig:invelopebox}
\end{figure}
\begin{lemma}
\label{Lem:invelopebox}
For each $\eta>0$, the sets $A_\eta\equiv A_{d,\eta}:=\{(x_1,\dotsc,x_d)\in (0,\infty)^d:\prod_{j=1}^d x_j\geq\eta\}$ and $J_\eta\equiv J_{d,\eta}:=\{x\in Q:M(x)\in A_\eta\}$ are closed and convex, and have the following properties:

\begin{enumerate}[label=(\roman*)]
\item $A_\eta=A_{d,\eta}=\eta^{1/d}A_{d,1}$ and $[0,1/2]^d\cap J_\eta=[0,1/2]^d\cap A_\eta$.
\item 
$J_\eta=\bigcap_{\,g\in G_R(Q)}g(A_\eta\cap Q)$, so $J_\eta$ is $G_R(Q)$-invariant.
\item If $\eta\leq 2^{-d}$, then $\mu_d(Q\setminus J_\eta)=2^d\eta\,\textstyle\sum_{\ell=0}^{d-1}\,\log^\ell(2^{-d}\inv{\eta})/\ell!$ and therefore $\mu_d(Q\setminus J_{\alpha\eta})\leq\alpha\mu_d(Q\setminus J_{\eta})$ for all $\alpha\geq 1$.
\item If $C\subseteq Q$ is convex and $J_\eta\not\subseteq\Int C$, then $\mu_d(Q\setminus C)\geq\eta$.
\item If $\phi\in\Phi$ and $\delta,\eta>0$ are such that $4\delta^2\leq\eta\leq 2^{-d}$ and $\int_Q\,(e^{\phi/2}-1)^2\leq\delta^2$, then $\phi(x)\geq -4\delta\eta^{-1/2}$ for all $x\in J_\eta$.
\end{enumerate}
\end{lemma}
\begin{proof}
The assertions in (i) are immediate from the definitions above. In addition, the function $r\colon (0,\infty)^d\to (0,\infty)$ defined by $r(x_1,\dotsc,x_d):=\prod_{j=1}^d\inv{x_j}$ is convex since its Hessian matrix is positive definite everywhere, so $A_\eta=\{x\in (0,\infty)^d:r(x)\leq\inv{\eta}\}$ is convex for all $\eta>0$. Now for $x\in Q$, note that if $M(x)\in A_\eta$ (i.e.\ $x\in J_\eta$), then $g(x)\in A_\eta$ for all $g\in G_R(Q)$ by the definition of $M$, and since $M(x)=\tilde{g}(x)$ for some $\tilde{g}\in G_R(Q)$, the converse is also true. This shows that $J_\eta=\bigcap_{\,g\in G_R(Q)}g(A_\eta\cap Q)$, as claimed in (ii), and since $g(A_\eta\cap Q)$ is convex for all $g\in G_R(Q)$, we see that $J_\eta$ is convex for all $\eta>0$. 

Turning to (iii), the formula for $\mu_d(Q\setminus J_\eta)$ certainly holds when $d=1$, and we now extend this to all $d\geq 1$ by induction. For $d\geq 2$, we partition $[0,1/2]^d$ into the sets $D_1:=[0,1/2]^{d-1}\times [0,2^{d-1}\eta)\subseteq\cm{J_\eta}$ and $D_2:=[0,1/2]^{d-1}\times [2^{d-1}\eta,1/2]$, and write
\begin{align*}
\mu_d([0,1/2]^d\setminus J_\eta)=\mu_d(D_1)+\mu_d(D_2\setminus J_\eta)&=\eta+\int_{2^{d-1}\eta}^{1/2}\,\frac{\eta}{x_d}\,\sum_{\ell=0}^{d-2}\,\log^\ell\!\rbr{\frac{x_d}{2^{d-1}\eta}}\frac{1}{\ell!}\,dx_d \\
&= \eta + \eta\,\sum_{\ell=1}^{d-1}\log^{\ell}(2^{-d}\inv{\eta})\frac{1}{\ell!},
\end{align*}
where we have used the inductive hypothesis to obtain the integrand above. Since $\mu_d(J_\eta)=2^d\mu_d([0,1/2]^d\setminus J_\eta)$ by symmetry, this completes the inductive step. In particular, $\mu_d(Q\setminus J_\eta)\geq 2^d\eta$ when $\eta\leq 2^{-d}$. It follows from this and the formula for $\mu_d(Q\setminus J_\eta)$ that $\mu_d(Q\setminus J_{\alpha\eta})\leq\alpha\mu_d(Q\setminus J_\eta)$ for all $\alpha\geq 1$, including when $\alpha\eta\geq 2^{-d}$, in which case $\mu_d(Q\setminus J_{\alpha\eta})=1$. Alternatively, to obtain the final assertion of (iii), simply note that $[0,1/2]^d\setminus J_{\alpha\eta}\subseteq\alpha^{1/d}\,([0,1/2]^d\setminus J_\eta)$ for all $\alpha\geq 1$ and $\eta>0$.

To establish (iv), fix a convex set $C\subseteq Q$ and suppose that there exists $x\in J_\eta\setminus\Int C$. By the separating hyperplane theorem, there is a closed half-space $H^+$ such that $x\in\partial H^+$ and $C\cap\Int H^+=\emptyset$. 
Since $x\in\Int Q$, it follows from Lemma~\ref{Lem:simplevsets} that there exists a vertex $v$ of $Q$ such that $Q^v(x)\subseteq Q\cap H^+$. Therefore, $\mu_d(Q\setminus C)\geq\mu_d(Q\cap H^+)\geq\mu_d(Q^v(x))=\prod_{j=1}^d\,\abs{x_j-v_j}\geq\eta$ by~\eqref{Eq:pvx}, as desired.

Finally, for fixed $x\in J_\eta$ and $\eta\geq 4\delta^2$, let $C:=Q\cap D_{\phi,x}$. Since $x\in J_\eta\setminus \Int C$ by the definition of $D_{\phi,x}$, it follows from (iv) that $\mu_d(Q\setminus D_{\phi,x})\geq\eta$, and we deduce from Lemma~\ref{Lem:hellunbds}(i) that $\phi(x)\geq -4\delta\,\mu_d(Q\setminus D_{\phi,x})^{-1/2}\geq -4\delta\eta^{-1/2}$, as required.
\end{proof}
We now obtain an analogous result for $\triangle\equiv\triangle_d:=\conv\{e_1,\dots,e_{d+1}\}\subseteq\R^{d+1}$, a regular $d$-simplex of side length $\sqrt{2}$ that will be viewed as a subset of its affine hull $\aff\triangle=\{x = (x_1,\dotsc,x_{d+1})\in\R^{d+1}:\sum_{j=1}^{d+1}x_j=1\}$. Note that $\triangle$ can be subdivided into $d+1$ congruent polytopes $R_1,\dotsc,R_{d+1}$, where 
\begin{equation}
\label{Eq:Rj}
R_j:=\bigl\{(x_1,\dotsc,x_{d+1})\in\triangle:x_j=\textstyle\max_{1\leq\ell\leq d+1}x_\ell\bigr\}. 
\end{equation}
The proof of Lemma~\ref{Lem:invelopesimp} makes use of another elementary fact from linear algebra.
\begin{lemma}
\label{Lem:detproj}
Let $u_1,u_2\in\R^d$ be unit vectors and for $i=1,2$, let $H_i:=\{x\in\R^d:\tm{u_i}x=0\}$. For $x\in\R^d$, write $\Pi(x):=x-(\tm{u_2}x)u_2$ for the orthogonal projection of $x$ onto $H_2$. Then $\mu_{d-1}(\Pi(A))=\abs{\tm{u_1}u_2}\,\mu_{d-1}(A)$ for all Lebesgue-measurable $A\subseteq H_1$, where $\Pi(A)$ denotes the image of $A$ under $\Pi$.
\end{lemma}
\begin{proof}
Let $\alpha:=\tm{u_1}u_2$. The result holds trivially if $u_1=u_2$, so we now assume that $u_1\neq u_2$, in which case $H_1\cap H_2$ has dimension $d-2$. Fix an orthonormal basis $B=\{v_3,\dotsc,v_d\}$ of $H_1\cap H_2$, and let $v_1:=u_2-\alpha u_1$ and $v_2:=-u_1+\alpha u_2$. Observe that $\norm{v_i}^2=\tm{v_i}v_i=1-\alpha^2$ for $i=1,2$ and let $v_i':=v_i/\norm{v_i}$ for each $i$. Then $B_i:=\{v_i'\}\cup B$ is an orthonormal basis of $H_i$ for each $i$, and note that $\Pi(v_1)=u_2-(\tm{u_1}u_2)u_1-(1-\alpha^2)u_2=\alpha v_2$, whence $\Pi(v_1')=\alpha v_2'$. Thus, $\Pi$ is represented by the matrix $\diag(\alpha,1,\dotsc,1)$ with respect to the bases $B_1$ and $B_2$.

Now for $i=1,2$, let $T_i\colon\R^{d-1}\to H_i$ be the linear map defined by setting $T_i\,e_1=v_i'$ and $T_i\,e_j=v_{j+1}$ for $2\leq j\leq d-1$. Then $\mu_{d-1}(A)=\mu_{d-1}(T_i(A))$ for all $i$ and for all measurable $A\subseteq\R^{d-1}$. Defining $D\colon\R^{d-1}\to\R^{d-1}$ by $D(x_1,\dotsc,x_{d-1}):=(\alpha x_1,\dotsc,x_{d-1})$, we see that if $R$ is a hyperrectangle of the form $\prod_{j=1}^{d-1}\,[a_j,b_j]$, then $(\Pi\circ T_1)(R)=(T_2\circ D)(R)$ by the final observation in the previous paragraph. Thus, if $A=T_1(R)$ for some hyperrectangle $R=\prod_{j=1}^{d-1}\,[a_j,b_j]\subseteq\R^{d-1}$, then
\begin{align*}
\mu_{d-1}(\Pi(A))=\mu_{d-1}((T_2\circ D)(R))&=\mu_{d-1}(D(R))\\
&=\abs{\alpha}\mu_{d-1}(R)=\abs{\alpha}\mu_{d-1}(T_1(R))=\abs{\alpha}\mu_{d-1}(A).
\end{align*}
Since the Borel $\sigma$-algebra of $H_1$ is generated by the $\pi$-system of sets of the form $T_1(R)$, where $R$ is a hyperrectangle, it follows that the claimed identity $\mu_{d-1}(\Pi(A))=\abs{\alpha}\mu_{d-1}(A)$ holds for all Lebesgue-measurable $A\subseteq H_1$, as required.
\end{proof}
\begin{remark*}
The key fact that underlies this result is that $\abs{\tm{u_1}u_2}$ is the determinant of any matrix which represents $\restr{\Pi}{H_1}$ with respect to orthonormal bases of $H_1$ and $H_2$. This follows from the first paragraph of the proof, 
which amounts to a derivation of the principal angles between $H_1$ and $H_2$ from first principles.
\end{remark*}
\begin{figure}[htb]
\centering
\includegraphics[trim=4cm 13cm 0 0, width=0.7\textwidth]{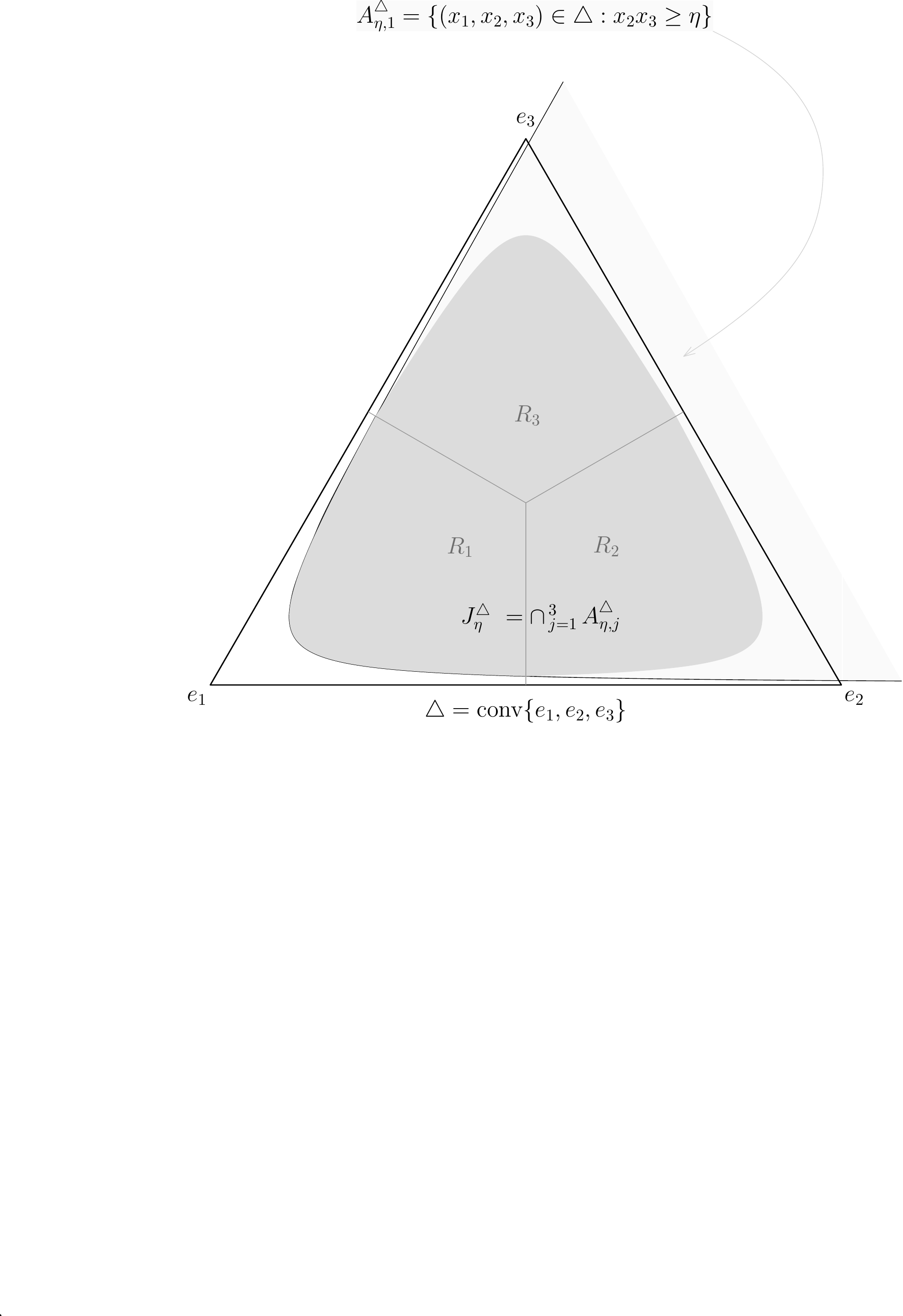}

\vspace{0.3cm}
\caption{Illustration of the polytopes $R_1,R_2,R_3$, and the sets $A_{\eta,1}^\triangle$ (union of the lighter and darker regions) and $J_\eta^\triangle$ (darker region) in \protect{Lemma~\ref{Lem:invelopesimp}} when $d=2$. Lemma~\ref{Lem:invelopesimp}(iii) is a key property that we exploit in the proof of Proposition~\ref{Prop:SimplexEntropy}; see Figure~\ref{Fig:SimplexEntropy}.}
\label{Fig:invelopesimp}
\end{figure}
\begin{lemma}
\label{Lem:invelopesimp}
Let $\Pi\colon\R^{d+1}\to\R^d$ denote the projection onto the first $d$ coordinates, so that $\Pi(x_1,\dotsc,x_{d+1}):=(x_1,\dotsc,x_d)$, and let $Q^\triangle\equiv Q_d^\triangle$ denote the image of $R_{d+1}$ under $\Pi$. Then $Q^\triangle$ is a polytope with $2d$ facets, and
$[0,1/(d+1)]^d\subseteq Q^\triangle\subseteq [0,1/2]^d$. Moreover,
$\mu_d(\Pi(A))=\mu_d(A)/\sqrt{d+1}$ for all Lebesgue-measurable $A\subseteq R_{d+1}$.

In addition, for each $\eta>0$, the sets $A_{\eta,j}^\triangle\equiv A_{d,\eta,j}^\triangle:=\{(x_1,\dotsc,x_{d+1})\in\triangle:\prod_{\,\ell\neq j}x_\ell\geq\eta\}$ and $J_\eta^\triangle\equiv J_{d,\eta}^\triangle:=\bigcap_{\,j=1}^{\,d+1}A_{d,\eta,j}^\triangle$ are convex and have the following properties:
\begin{enumerate}[label=(\roman*)]
\item $\Pi(R_{d+1}\cap J_\eta^\triangle)=Q^\triangle\cap J_\eta$ and $R_j\cap J_\eta^\triangle=R_j\cap A_{\eta,j}^\triangle$ for every $1\leq j\leq d+1$.
\item If $C\subseteq\triangle$ is convex and $J_\eta^\triangle\not\subseteq\relint C$, then $\mu_d(\triangle\setminus C)\geq d!\,\eta\mu_d(\triangle)$.
\item Suppose that $\phi,\phi_0\colon\aff\triangle\to [-\infty,\infty)$ are concave and upper semi-continuous, and that there exists $\theta\in [1,\infty)$ such that $e^{\phi_0}\geq\inv{(\theta\mu_d(\triangle))}$ on $\triangle$. Let $\delta,\eta>0$ be such that $4\delta^2/d!\leq\eta\leq (d+1)^{-d}$ and $\int_\triangle\,(e^{\phi/2}-e^{\phi_0/2})^2\leq\delta^2$. Then\\ $\phi(x)+\log\mu_d(\triangle)\geq -4\delta\{\theta/(d!\,\eta)\}^{1/2}-\log\theta$ for all $x\in J_\eta^\triangle$.
\end{enumerate}
\end{lemma}
\begin{proof}
Note that $H:=\aff\triangle=\{x=(x_1,\dotsc,x_{d+1})\in\R^{d+1}:\sum_{j=1}^{d+1}x_j=1\}$ is an affine hyperplane with unit normal $u:=(1/\sqrt{d+1},\dotsc,1/\sqrt{d+1})\in\R^{d+1}$. For $1\leq k\leq d$, let $H_k^+:=\{x\in\R^{d+1}:x_k\geq 0\}$ and $H_{d+k}^+:=\{x\in\R^{d+1}:x_k\leq x_{d+1}\}$. Then $R_{d+1}=H\cap\bigcap_{\,k=1}^{\,2d}H_k^+\subseteq\triangle$,
and it is easy to verify that $R_{d+1}\subsetneqq H\cap\bigcap_{\,k\neq k'}H_k^+$ for all $1\leq k'\leq 2d$. Therefore, $R_{d+1}$ is a polytope when viewed as a subset of $H$, and we deduce from~\citet[Theorem~1.6]{BG09} that $F$ is a facet of $R_{d+1}$ if and only if $F=R_{d+1}\cap H_k^+$ for some $1\leq k\leq 2d$. It follows that $R_{d+1}$ is a polytope with exactly $2d$ facets, and since $\restr{\Pi}{H}\colon H\to\R^d$ is linear and bijective, the same is true of $Q^\triangle=\Pi(R_{d+1})$.

To see that $[0,1/(d+1)]^d\subseteq Q^\triangle$, fix $x=(x_1,\dotsc,x_d)\in[0,1/(d+1)]^d$ and set $x':=(x_1,\dotsc,x_d,x_{d+1})$, where $x_{d+1}:=1-\sum_{j=1}^d x_j$. Then $x_{d+1}\geq 1/(d+1)\geq x_j$ for all $1\leq j\leq d$, so $x'\in R_{d+1}\subseteq\triangle$ and therefore $x=\Pi(x')\in\Pi(R_{d+1})=Q^\triangle$, as required. In addition, if $x'=(x_1,\dotsc,x_{d+1})\in R_{d+1}$, then $0\leq x_k\leq x_{d+1}$ and $x_k+x_{d+1}\leq\sum_{j=1}^{d+1}x_j=1$ for all $1\leq k\leq d$, so $x_k\in [0,1/2]$ for all such $k$. It follows that $\Pi(x')\in [0,1/2]^d$ for all $x'\in R_{d+1}$ and hence that $Q^\triangle\subseteq [0,1/2]^d$.

Furthermore, to establish that $\mu_d(\Pi(A))=\mu_d(A)/\sqrt{d+1}$ for all Lebesgue-measurable $A\subseteq R_{d+1}$, we apply Lemma~\ref{Lem:detproj} to the hyperplanes $H-e_1=\{x\in\R^{d+1}:\tm{u}x=0\}$ and $\{x\in\R^{d+1}:\tm{e_{d+1}}x=0\}$, whose unit normals $u$ and $e_{d+1}$ satisfy $\abs{\tm{e_{d+1}}u}=1/\sqrt{d+1}$. Since $\Pi$ is linear and Lebesgue measure is translation invariant, we see that \[\mu_d(\Pi(A))=\mu_d(\Pi(A)-e_1)=\mu_d(\Pi(A-e_1))=\mu_d(A-e_1)/\sqrt{d+1}=\mu_d(A)/\sqrt{d+1}\]
for all Lebesgue-measurable $A\subseteq R_{d+1}$, as required.

As for (i), note that if $x'=(x_1,\dotsc,x_{d+1})\in R_j\cap A_{\eta,j}^\triangle$ for some $1\leq j\leq d+1$, then $x_j=\max_{1\leq\ell\leq d+1}x_\ell$ and $\prod_{\,\ell\neq j}x_\ell\geq\eta$, so $\prod_{\,\ell\neq j'}x_\ell\geq\prod_{\,\ell\neq j}x_\ell\geq\eta$ for all $1\leq j'\leq d+1$, and therefore $x'\in R_j\cap A_{\eta,j'}^\triangle$ for all $j'$. This shows that $R_j\cap A_{\eta,j}^\triangle\subseteq R_j\cap J_\eta^\triangle$ for each $1\leq j\leq d+1$, and the reverse inclusion is clear. To see that $\Pi(R_{d+1}\cap J_\eta^\triangle)= Q^\triangle\cap J_\eta$, first note that since $Q^\triangle\subseteq [0,1/2]^d$, it follows from Lemma~\ref{Lem:invelopebox}(i) that $Q^\triangle\cap A_\eta=Q^\triangle\cap J_\eta$. Thus, $x\in Q^\triangle\cap J_\eta=Q^\triangle\cap A_\eta$ if and only if $x=\Pi(x')$ for some $x'=(x_1,\dotsc,x_{d+1})\in R_{d+1}$ with $\prod_{j=1}^d x_j\geq\eta$, i.e.\ precisely when $x\in\Pi(R_{d+1}\cap A_{\eta,d+1}^\triangle)=\Pi(R_{d+1}\cap J_\eta^\triangle)$, as required.

The proof of (ii) is very similar to that of Lemma~\ref{Lem:invelopebox}(iv). Suppose that $C\subseteq\triangle$ is convex and that there exists $x\in J_\eta^\triangle\setminus\relint C$. Then it once again follows from the separating hyperplane theorem and Lemma~\ref{Lem:simplevsets} that there exists a vertex $v$ of $\triangle$ such that $\triangle^v(x)\subseteq\triangle\cap H^+$ for some closed half-space $H^+\subseteq\R^{d+1}$ with $x\in\partial H^+$, $H\neq\partial H^+$ and $C\cap\Int H^+=\emptyset$.
Then $v=e_k$ for some $1\leq k\leq d+1$, and note that $x=\sum_{j=1}^{d+1}x_je_j=e_k+\sum_{j\neq k}x_j(e_j-e_k)$. Thus,
\[\mu_d(\triangle\setminus C)\geq\mu_d(\triangle\cap H^+)\geq\mu_d(\triangle^v(x))=d!\,\mu_d(\triangle)\,\textstyle\prod_{j\neq k}x_j\geq d!\,\eta\mu_d(\triangle)\]
by~\eqref{Eq:pvx} and the fact that $x\in J_\eta^\triangle$, as required. Finally, the final assertion (iii) follows from (ii) and Lemma~\ref{Lem:hellunbds}(i) in much the same way that Lemma~\ref{Lem:invelopebox}(v) follows from Lemma~\ref{Lem:invelopebox}(iv).
\end{proof}
In view of Lemma~\ref{Lem:invelopebox}(iv) and Lemma~\ref{Lem:invelopesimp}(ii), we shall henceforth refer to the sets $J_\eta$ and $J_\eta^\triangle$ as \emph{(convex) invelopes}. Next, we show that the sets $J_\eta\subseteq Q$ can be approximated from within by polytopes $P_\eta$ satisfying $\mu_d(Q\setminus P_\eta)\lesssim_d\mu_d(Q\setminus J_\eta)$, in such a way that the number of vertices of $P_\eta$ does not grow too quickly as $\eta\!\searrow 0$. This is the content of Lemma~\ref{Lem:polyapprox}, whose proof hinges on an inductive construction based on the following fact.
\begin{lemma}
\label{Lem:conveasy}
Let $E\subseteq (0,\infty)^d$ be a convex set with the property that $\lambda E\subseteq E$ for all $\lambda\geq 1$, and let $h\colon [0,\infty)\to [0,\infty]$ be a convex function. Then
\begin{align*}
E^h:=\{(x,z)\in (0,\infty)^d\times (0,\infty):h(z)\in (0,\infty),\,x/h(z)\in E\}\\
\cup\,\{(x,z)\in (0,\infty)^d\times (0,\infty):h(z)=0,\,x\in\textstyle\bigcup_{\lambda>0}\lambda E\}
\end{align*}
is a convex subset of $\R^{d+1}\cong\R^d\times\R$. Suppose further that $E$ is closed and $(0,\infty)^d=\bigcup_{\lambda>0}\lambda E$. Then $\tau_E(x):=\sup\{\lambda>0:x\in\lambda E\}$ lies in $(0,\infty)$ for all $x\in (0,\infty)^d$ and $\tau_E(\lambda x)=\lambda\tau_E(x)$ for all $x\in (0,\infty)^d$ and $\lambda\in (0,\infty)$. Moreover, we have the following:
\begin{enumerate}[label=(\roman*),leftmargin=0.9cm]
\item $\lambda E=\{x\in (0,\infty)^d:\tau_E(x)\geq\lambda\}$ and $\lambda\Int E=\{x\in (0,\infty)^d:\tau_E(x)>\lambda\}$ for all $\lambda>0$, so $\tau_E$ is continuous on $(0,\infty)^d$ and $\lambda E\subsetneqq E$ for all $\lambda>1$.
\item $E^h=\{(x,z)\in (0,\infty)^d\times (0,\infty):\tau_E(x)\geq h(z)\}$.
\end{enumerate}
If in addition $h$ is lower semi-continuous and decreasing on $[0,\infty)$, and if $\lim_{x\searrow\,0}h(x)=\infty$, then
\begin{enumerate}[label=(\roman*),resume]
\item $\inv{h}(c):=\inf\{z\in [0,\infty):h(z)\leq c\}\in (0,\infty]$ for all $c\in [0,\infty)$, where we set $\inf\emptyset:=\infty$. The function $\inv{h}\colon [0,\infty)\to (0,\infty]$ is convex, decreasing and lower semi-continuous. For $z,c\in [0,\infty)$, we have $h(z)\leq c$ if and only if $\inv{h}(c)\leq z$.
\item $\inv{h}\circ\tau_E\colon (0,\infty)^d\to [0,\infty]$ is a convex function whose epigraph is $E^h$. If $h(z)\in (0,\infty)$ for all $z\in (0,\infty)$, then $E^h$ is closed and $\inv{h}\circ\tau_E$ is lower semi-continuous.
\item $\lambda E^h\subseteq E^h$ for all $\lambda\geq 1$, and if $h$ is not identically $\infty$, then $(0,\infty)^{d+1}=\bigcup_{\lambda>0}\lambda E^h$.
\end{enumerate}
\end{lemma}
\begin{proof}
Fix $(x_1,z_1),(x_2,z_2)\in E^h$ and $t\in (0,1)$, and let $(x,z):=t(x_1,z_1)+(1-t)(x_2,z_2)$. Since $h(z_i)<\infty$ for $i=1,2$, we must have $h(z)\leq th(z_1)+(1-t)h(z_2)<\infty$. Moreover, it follows from the definition of $E^h$ that there exist $\lambda_1,\lambda_2\in (0,\infty)$ and $y_1,y_2\in E$ such that $\lambda_i\geq h(z_i)$ and $x_i=\lambda_iy_i$ for each $i=1,2$. Let $\lambda:=t\lambda_1+(1-t)\lambda_2\in (0,\infty)$ and observe that 
\[x':=x/\lambda=\frac{t\lambda_1}{\lambda}y_1+\frac{(1-t)\lambda_2}{\lambda}y_2\in [y_1,y_2]\subseteq E.\]
Thus, if $h(z)=0$, then $x=\lambda x'\in\bigcup_{\lambda>0}\,\lambda E$. Otherwise, if $h(z)\in (0,\infty)$, then \[\lambda':=\frac{\lambda}{h(z)}\geq\frac{th(z_1)+(1-t)h(z_2)}{h(tz_1+(1-t)z_2)}\geq 1,\]
so $x/h(z)=\lambda'x'\in\lambda'E\subseteq E$. In both cases, we deduce that $(x,z)\in E^h$, and this establishes the convexity of $E^h$.

Henceforth, suppose that $E$ is closed and $(0,\infty)^d=\bigcup_{\lambda>0}\lambda E$. 
It follows from these assumptions that $0\notin E$ and moreover that for each $x\in (0,\infty)^d$, there exists $\alpha_x\in (0,\infty)$ such that $R_x:=E\cap\{ax:a\geq 0\}=\{ax:a\geq\alpha_x\}$. Thus, $\tau_E(x)=\max\{\lambda>0:x/\lambda\in E\}=1/\alpha_x$, and it is clear that $\tau_E(\lambda x)=\lambda\tau_E(x)$ for all $\lambda\in (0,\infty)$. Now if $x\notin E$, then $x\notin R_x$, so $\tau_E(x)<1$. On the other hand, if $x\in\Int E$, then there exists $\delta>0$ such that $(1-\delta)x\in E$, so $\tau_E(x)>1$. Otherwise, if $x\in\partial E$, then by the supporting hyperplane theorem \citep[Theorem~1.3.2]{Sch14}, there exists an open half-space $H^+$ such that $x\in\partial H^+$ and $H^+\cap E=\emptyset$. Note that we cannot have $R_x\subseteq\partial H^+$; indeed, this would imply that $0\in\partial H^+$, and since there exists $\eta>0$ such that $B(x,\eta)\subseteq (0,\infty)^d$, it would then follow that $H^+\cap (0,\infty)^d$ is a non-empty cone that is disjoint from $E$. But this contradicts the assumption that $(0,\infty)^d=\bigcup_{\lambda>0}\lambda E$, so it is indeed the case that $R_x\not\subseteq\partial H^+$, as claimed. We conclude that $R_x\cap\partial H^+=\{x\}$ and hence that $\tau_E(x)=1$.

In summary, for $\lambda\in (0,\infty)$, we have $\tau_E(x)=\lambda\tau_E(x/\lambda)\geq\lambda$ if and only if $x/\lambda\in E$, and $\tau_E(x)=\lambda\tau_E(x/\lambda)>\lambda$ if and only if $x/\lambda\in\Int E$. In other words, $\inv{\tau_E}\bigl((\lambda,\infty)\bigr)=\lambda\Int E$ and $\inv{\tau_E}\bigl((-\infty,\lambda)\bigr)=\cm{(\lambda E)}$ for each $\lambda>0$, and since these sets are all open, we conclude that $\tau_E$ is continuous on $(0,\infty)^d$. Moreover, we see that $\partial E\cap\lambda E=\emptyset$ for all $\lambda>1$, which yields the final assertion of (i). 

For (ii), observe that if $h(z)=0$, then $\tau_E(x)\geq h(z)=0$ for all $x\in (0,\infty)^d=\bigcup_{\lambda>0}\lambda E$. On the other hand, if $h(z)\in (0,\infty)$, then (i) implies that $\tau_E(x)\geq h(z)$ if and only if $x/h(z)\in E$. We conclude that $E^h=\{(x,z)\in (0,\infty)^d\times (0,\infty):\tau_E(x)\geq h(z)\}$, as required. 


Suppose further that $h$ is convex, decreasing and lower semi-continuous, and that $h(x)\nearrow\infty$ as $x\searrow 0$. Now for $c\in [0,\infty)$, let $I_c:=\{z\in [0,\infty):h(z)\leq c\}$ and note that either $I_c=\emptyset$, in which case $\inv{h}(c)=\infty$, or $I_c=[z',\infty)$ for some $z'\in (0,\infty)$, in which case $\inv{h}(c)=z'$. For $z,c\in [0,\infty)$, this shows that $h(z)\leq c$ if and only if $\inv{h}(c)\leq z$; in other words, $(z,c)\in [0,\infty)^2$ lies in the epigraph of $h$ if and only if $(c,z)$ lies in the epigraph of $\inv{h}$. Since $h$ is convex and lower semi-continuous, the epigraph of $h$ is closed and convex~\citep[][Theorem~7.1]{Rock97}, so the same is true of the epigraph of $\inv{h}$. This in turn implies that $\inv{h}$ is convex and lower semi-continuous. Moreover, since $h$ is decreasing, the same is true of $\inv{h}$. This yields (iii). 

For (iv), we deduce from (ii) and (iii) that
\[E^h=\{(x,z)\in (0,\infty)^{d+1}:\tau_E(x)\geq h(z)\}=\{(x,z)\in (0,\infty)^d\times\R:z\geq (\inv{h}\circ\tau_E)(x)\},\]
where we have used the fact that $\inv{h}(c)>0$ for all $c\in [0,\infty)$ to obtain the second equality. Thus, $\inv{h}\circ\tau_E$ is a convex function whose epigraph is $E^h$, and if $h(z)\in (0,\infty)$ for all $z\in (0,\infty)$, then it follows from (i) that
\[\{x\in (0,\infty)^d:z\geq (\inv{h}\circ\tau_E)(x)\}=\{x\in (0,\infty)^d:\tau_E(x)\geq h(z)\}=h(z)\cdot E\]
is closed for all $z\in (0,\infty)$. Together with~\citet[Theorem~7.1]{Rock97}, this implies that $\inv{h}\circ\tau_E$ is lower semi-continuous and $E^h$ is closed, as required.

Finally, if $(x,z)\in E^h$ and $\lambda\geq 1$, then $\tau_E(x)\geq h(z)\geq h(\lambda z)/\lambda$ by (ii) and the fact that $h$ is decreasing, so it follows from (ii) that $\lambda (x,z)\in E^h$. Also, for a fixed $(x,z)\in (0,\infty)^{d+1}$, note that since $h(tz)/t\to 0$ as $t\to\infty$ if $h$ is not identically $\infty$,
there exists $\lambda>0$ such that $h(\lambda z)/\lambda\leq\tau_E(x)$. We deduce from (ii) that $\lambda (x,z)\in E^h$, as claimed in (v).
\end{proof}
\begin{lemma}
\label{Lem:polyapprox}
There exists $\alpha_d>0$, depending only on $d\in\N$, such that for every $\eta\in (0,2^{-d}]$, we can construct a $G_R(Q)$-invariant polytope $P_\eta\equiv P_{d,\eta}\subseteq J_{d,\eta}\equiv J_\eta$ with the following properties:
\begin{enumerate}[label=(\roman*)]
\item $P_\eta$ has at most $\alpha_d\log^{d-1}(1/\eta)$ vertices and $\mu_d(Q\setminus P_\eta)\leq\alpha_d\,\mu_d(Q\setminus J_\eta)$.
\item If $0<\eta<\tilde{\eta}\leq 2^{-d}$, then $P_{\tilde{\eta}}\subsetneqq P_\eta$, and the regions $Q\setminus\Int P_\eta$ and $P_\eta\setminus\Int P_{\tilde{\eta}}$ can each be expressed as the union of at most $\alpha_d\log^{d-1}(1/\eta)$ $d$-simplices with pairwise disjoint interiors.
\item $\tilde{P}_\eta\equiv\tilde{P}_{d,\eta}:=Q^\triangle\cap P_{d,\eta}$ is a polytope and $\mu_d(Q^\triangle\setminus\tilde{P}_\eta)\leq\alpha_d\,\mu_d(Q^\triangle\setminus J_\eta)$, where $Q^\triangle$ is defined as in Lemma~\ref{Lem:invelopesimp}. 
\item If $0<\eta<\tilde{\eta}\leq (d+1)^{-d}$, then $\tilde{P}_{\tilde{\eta}}\subsetneqq\tilde{P}_\eta$, and the regions $Q^\triangle\setminus\Int\tilde{P}_{\eta}$ and $\tilde{P}_\eta\setminus\Int\tilde{P}_{\tilde{\eta}}$ can each be expressed as the union of at most $\alpha_d\log^{d-1}(1/\eta)$ $d$-simplices with pairwise disjoint interiors.
\end{enumerate}
\end{lemma}
\begin{figure}[htb]
\centering
\includegraphics[trim=3cm 14cm 0 0.3cm,width=0.7\textwidth]{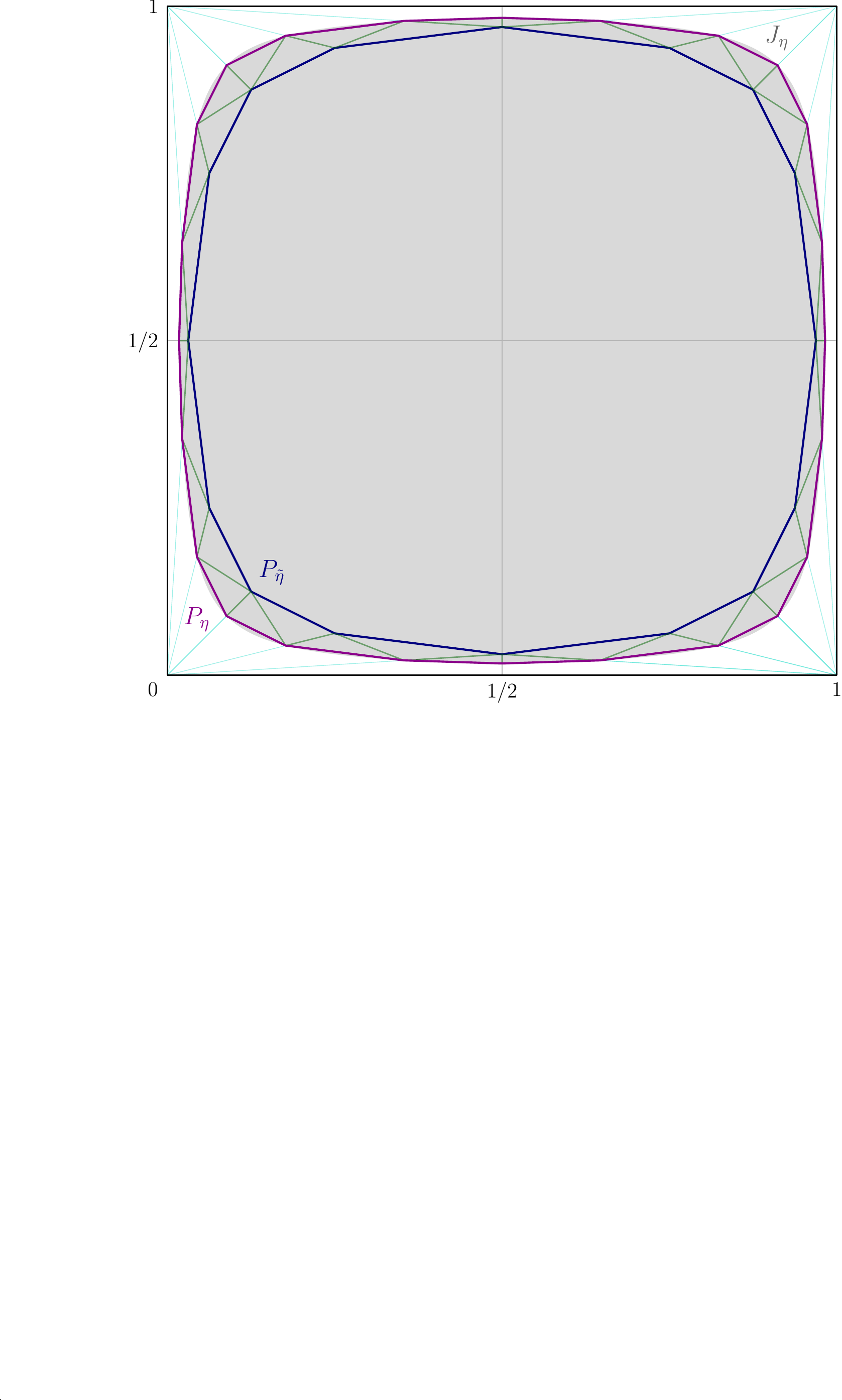}
\caption{Diagram of the nested sets $J_\eta\supseteq P_\eta\supseteq P_{\tilde{\eta}}$ in \protect{Lemma~\ref{Lem:polyapprox}} for $0<\eta<\tilde{\eta}\leq 2^{-d}$ when $d=2$. The grey shaded region is $J_\eta$, and the boundaries of $P_\eta$ and $P_{\tilde{\eta}}$ are outlined in purple and blue respectively. The $x_1$-coordinates of the vertices of $P_\eta$ in $[0,1/2)^2$ take the form $2^k\eta^{1/2}$, where $k\in\Z$. The green line segments indicate a triangulation of $P_\eta\setminus\Int P_{\tilde{\eta}}$ that is constructed using a scaling argument, which we illustrate using faint blue line segments; see properties (iv) and (ix) in the proof.}
\label{Fig:polyapproxbox}
\end{figure}
A key feature of this result is that the bounds on the number of simplices in (ii) and (iv) have a polylogarithmic (rather than polynomial) dependence on $\eta^{-1}$. 
The proof below is constructive and `bare-hands' in that it does not appeal to the general theory of polytopal approximations to compact, convex sets. It is instructive to compare our conclusions with what could be obtained by directly applying an off-the-shelf result such as Lemma~\ref{Lem:BI75} in Section~\ref{Subsec:SmoothnessSupp}, or~\citet[Theorem~3]{GMR95}, which states that for any $K\in\mathcal{K}^b$ and $\zeta>0$, there exists a polytope $P\subseteq K$ with $\lesssim_d\zeta^{-(d-1)/2}$ vertices such that $\mu_d(K\setminus P)\leq\zeta\mu_d(K)$. In (i), we seek a polytope $P\subseteq J_\eta$ such that $\mu_d(Q\setminus P)\leq\alpha_d\,\mu_d(Q\setminus J_\eta)$, and in view of Lemma~\ref{Lem:invelopebox}(iii), the requirement is that \[\mu_d(J_\eta\setminus P)\leq (\alpha_d-1)\,\mu_d(Q\setminus J_\eta)\lesssim_d\eta\log^{d-1}(1/\eta)\lesssim_d\eta\log^{d-1}(1/\eta)\,\mu_d(J_\eta),\]
at least when $\eta\leq 2^{-(d+1)}$. It follows from~\citet[Theorem~3]{GMR95} that there exists a suitable polytope $P$ with $\lesssim_d\{\eta\log^{d-1}(1/\eta)\}^{-(d-1)/2}$ vertices, but this bound is much weaker than what is claimed in (i). Similarly, the bounds of order $\log^{d-1}(1/\eta)$ in (ii) and (iv) do not follow straightforwardly from general schemes for approximating and subdividing the region between two nested convex sets or polytopes~\citep[cf.][page~390]{Lee04}.

Instead, we exploit the special structure of the regions $J_\eta$ and $A_\eta$ defined in Lemma~\ref{Lem:invelopebox}. In particular, the scaling property in Lemma~\ref{Lem:invelopebox}(i) implies that $[0,1/2]^d\cap J_\eta=[0,1/2]^d\cap A_\eta=\eta^{1/d}\,\bigl([0,\eta^{-1/d}/2]^d\cap A_1\bigr)$. In the proof, we aim to construct a set $E\equiv E_d\subseteq A_{d,1}\equiv A_1$ such that $[0,\eta^{-1/d}/2]^d\cap E$ is a polytope satisfying $\mu_d([0,\eta^{-1/d}/2]^d\setminus E)\lesssim_d\mu_d([0,\eta^{-1/d}/2]^d\setminus A_1)$ for all $\eta>0$. It turns out that this can be achieved whilst also ensuring that $[0,\eta^{-1/d}/2]^d\cap E$ has $\lesssim_d\log^{d-1}(1/\eta)$ vertices in $[0,\eta^{-1/d}/2)^d$. Intuitively, the reason for this is that the boundary of $A_1$ becomes much `flatter' away from the origin, as can be seen in Figures~\ref{Fig:invelopebox} and~\ref{Fig:polyapproxbox}. This means that the volume bound above can be satisfied by an approximating polytope whose vertices are spread much more diffusely over the boundary of $A_1$ in regions further away from the origin. 

Returning to the original domain $[0,1/2]^d$, we scale $E$ to get $L_\eta:=\eta^{1/d}E\subseteq\eta^{1/d}A_1=A_\eta$, so that $\mu_d([0,1/2]^d\cap L_\eta)\lesssim_d\mu_d([0,1/2]^d\cap J_\eta)$. The polytope $P_\eta$ is then constructed by applying the isometries in $G_R(Q)$ to $[0,1/2]^d\cap L_\eta$. By elucidating the facial structure of $E$ and scaling $E$ as above, we proceed to obtain simplicial decompositions of $Q\setminus\Int P_\eta$ and $P_\eta\setminus\Int P_{\tilde{\eta}}$ that satisfy the conditions of (ii). 
\begin{proof}[Proof of Lemma~\ref{Lem:polyapprox}]
Throughout, for $d\in\N$, we identify $\R^d$ and $\Z^d$ with $\R^{d-1}\times\R$ and $\Z^{d-1}\times\Z$ respectively. First, we will show by induction on $d$ that for all $d\in\N$ and $\eta>0$, there exists a closed, convex set $L_\eta\equiv L_{d,\eta}$ with the following properties:
\begin{enumerate}[label=(\roman*)]
\item $L_{d,\eta}$ is the epigraph of a continuous, convex function $g_{\eta\,}\colon (0,\infty)^{d-1}\to (0,\infty)$ with the property that $g_\eta(u)\geq s_\eta(u):=\eta\,\prod_{j=1}^{d-1}\inv{u_j}$ for all $u\in\R^d,$ and $\partial L_{d,\eta}=\{(x,g_\eta(x)):x\in (0,\infty)^{d-1}\}$.
\item $\lambda L_{d,\eta}\subseteq L_{d,\eta}\subseteq A_{d,\eta}$ for all $\lambda\geq 1$ and $(0,\infty)^d=\bigcup_{\lambda>0}\lambda L_{d,\eta}$.
\item If $(x_1,\dotsc,x_d)\in L_{d,\eta}$, then $(x_1',\dotsc,x_d')\in L_{d,\eta}$ whenever $x_j'\geq x_j$ for all $1\leq j\leq d$. Also, if $(x_1,\dotsc,x_d)\in\Int L_{d,\eta}$, then $(x_1',\dotsc,x_d')\in\Int L_{d,\eta}$ whenever $x_j'\geq x_j$ for all $j$.
\item $L_{d,\eta}=\eta^{1/d}L_{d,1}=:\eta^{1/d}E_d$ and $(\eta^{1/d},\dotsc,\eta^{1/d})\in\R^d$ lies in $L_{d,\eta}$. If $\tilde{\eta}>\eta$, then $L_{d,\tilde{\eta}}=(\tilde{\eta}/\eta)^{1/d}\,L_{d,\eta}\subsetneqq L_{d,\eta}$. Moreover, $[0,1/2]^d\cap L_{d,\eta}$ is non-empty if and only if $\eta\in (0,2^{-d}]$, and $[0,1/2]^d\cap\Int L_{d,\eta}$ is non-empty if and only if $\eta\in (0,2^{-d})$.
\item Every facet of $L_{d,\eta}$ is a $(d-1)$-dimensional polytope and every $x\in\partial L_{d,\eta}$ lies in some facet. More precisely, if $d\geq 2$ and $F$ is a facet of $L_{d,\eta}$, then there exists $m'=(m,j)\in\Z^{d-2}\times\Z\equiv\Z^{d-1}$ such that $F=F_{\eta,m'}:=\{(\lambda x,g_{\eta}(\lambda x)):x\in G_m,\,\lambda\in [w_{\eta,j},w_{\eta,j-1}]\}$, where $G_m$ is a corresponding facet of $E_{d-1}$ and $w_{\eta,k}\equiv w_{d,\eta,k}:=2^{-k}\eta^{1/d}$ for $k\in\Z$. Moreover, $F_{\eta,m'}=\eta^{1/d}F_{1,m'}=:G_{m'}$ for all $m'\in\Z^{d-1}$, and if $(x_1,\dotsc,x_d)\in F_{\eta,m'}$, then $x_d\in [z_{\eta,j-1},z_{\eta,j}]$, where $z_{\eta,k}\equiv z_{d,\eta,k}:=2^{(d-1)k}\eta^{1/d}$ for $k\in\Z$.
\end{enumerate}
We will then show by induction on $d$ that for all $d\in\N$, there exists $\alpha_d'>0$, depending only on $d$, such that the following hold for all $\eta\in(0,2^{-d}]$:
\begin{enumerate}[resume,label=(\roman*)]
\item $P_\eta\equiv P_{d,\eta}:=\bigcap_{\,g\in G_R(Q)}g(L_{d,\eta}\cap Q)$ is non-empty and $G_R(Q)$-invariant, and $P_{d,\eta}\subseteq J_{d,\eta}$. 
\item $[0,1/2]^d\cap P_{d,\eta}=[0,1/2]^d\cap L_{d,\eta}$ and $[0,1/2]^d\cap\Int P_{d,\eta}=[0,1/2]^d\cap\Int L_{d,\eta}$. Moreover, $\mu_d([0,1/2]^d\setminus P_{d,\eta})\leq\alpha_d'\,\mu_d([0,1/2]^d\setminus J_{d,\eta})$.
\item $P_{d,\eta}$ is a polytope with at most $\alpha_d'\log^{d-1}(1/\eta)$ vertices.
\item If $0<\eta<\tilde{\eta}\leq 2^{-d}$, then $[0,1/2]^d\setminus\Int P_{d,\eta}$ and $[0,1/2]^d\cap (P_{d,\eta}\setminus\Int P_{d,\tilde{\eta}})$ can each be triangulated into at most $\alpha_d'\log^{d-1}(1/\eta)$ $d$-simplices, in such a way that for any $d$-simplex $S$ in the triangulation of $[0,1/2]^d\cap (P_{d,\eta}\setminus\Int P_{d,\tilde{\eta}})$, there exists $m'\in\Z^{d-1}$ such that $S\subseteq\bigcup_{\lambda\in [1,(\tilde{\eta}/\eta)^{1/d}]}\lambda F_{\eta,m'}=\bigcup_{\lambda\in [\eta^{1/d},\,\tilde{\eta}^{1/d}]}\lambda G_{m'}$. 
\end{enumerate}
In view of the $G_R(Q)$-invariance of $P_\eta$, these final four assertions imply parts (i) and (ii) of the desired result. We will address parts (iii) and (iv) of the lemma at the end of the proof.

When $d=1$, we can take $L_{1,\eta}:=A_{1,\eta}=[\eta,\infty)$ and $P_{1,\eta}:=J_{1,\eta}=[\eta,1-\eta]$ for each $\eta\in (0,1/2]$, which trivially have the desired properties. Note that $L_{1,\eta}$ has a single facet $F_{\eta,\varnothing}:=\{\eta\}$, where we write $\varnothing$ for the empty tuple. Next, for fixed $d\geq 2$ and $\eta>0$, let $\overline{h}_\eta(z)\equiv\overline{h}_{d,\eta}(z):=(\eta/z)^{1/(d-1)}$ for each $z\in (0,\infty)$ and let $h_\eta\equiv h_{d,\eta}\colon (0,\infty)\to (0,\infty)$ be the piecewise affine function that is linear on each of the intervals $[z_{\eta,j-1},z_{\eta,j}]$ and satisfies $h_\eta(z_{\eta,j})=w_{\eta,j}=\overline{h}_\eta(z_{\eta,j})$ for all $j\in\Z$. Then $h_\eta$ is a strictly decreasing, convex bijection whose inverse $\inv{h_\eta}\colon (0,\infty)\to(0,\infty)$ is also piecewise affine. Indeed, for $j\in\Z$, let $t_{\eta,j}$ be the (unique) affine function that satisfies $t_{\eta,j}(w_{\eta,j})=z_{\eta,j}$ and $t_{\eta,j}(w_{\eta,j-1})=z_{\eta,j-1}$, and observe that $\inv{h_\eta}(w)\geq t_{\eta,j}(w)$ for all $w\in (0,\infty)$, with equality if and only if $w\in [w_{\eta,j},w_{\eta,j-1}]$.

We now verify that $\overline{h}_\eta\leq h_\eta\leq\gamma_d\,\overline{h}_\eta$ for some $\gamma_d>0$ that depends only on $d$. Indeed, if $z=\lambda z_{\eta,j}\in [z_{\eta,j},z_{\eta,j+1}]$ for some $j\in\Z$ and $\lambda\in [1,2^{d-1}]$, then 
\begin{equation}
\label{Eq:hratio}
\frac{h_\eta(z)}{\overline{h}_\eta(z)}=\frac{\cbr{1-(\lambda-1)/(2^{d}-2)}h_\eta(z_{\eta,j})}{\lambda^{-1/(d-1)}\,\overline{h}_\eta(z_{\eta,j})}=\frac{(2^d-\lambda-1)\,\lambda^{1/(d-1)}}{2^d-2}\,,
\end{equation}
which is independent of $j$ and attains its maximum value when $\lambda=(2^d-1)/d$. Also, it is easy to see that $h_{\eta}(z)=\eta^{1/d\,}h_1(z/\eta^{1/d})$ and $\inv{h_{\eta}}(w)=\eta^{1/d\,}\inv{h_1}(w/\eta^{1/d})$ for all $z,w\in (0,\infty)$. Since $h_1$ is strictly decreasing, it follows that $h_{\eta}<h_{\tilde{\eta}}$ whenever $0<\eta<\tilde{\eta}$.

Using the notation of Lemma~\ref{Lem:conveasy}, we claim that $L_{d,\eta}:=(E_{d-1})^{h_\eta}$ has the required properties (i)--(v), where $E_{d-1}\equiv L_{d-1,1}$. 
\begin{proof}[Properties (i) and (ii)]
\let\qed\relax
By part (ii) of the inductive hypothesis and Lemma~\ref{Lem:conveasy}(iv), it follows that $L_{d,\eta}$ is closed and convex, and that $g_\eta:=\inv{h_\eta}\circ\tau_{E_{d-1}}$ is a convex function whose epigraph is $L_{d,\eta}$. Since $\inv{h_\eta}$ and $\tau_{E_{d-1}}$ are continuous, it follows that $g_\eta$ is continuous and hence that $\partial L_{d,\eta}=\{(x,g_\eta(x)):x\in (0,\infty)^{d-1}\}$. 
In addition, Lemma~\ref{Lem:conveasy}(v) implies that $\lambda L_{d,\eta}\subseteq L_{d,\eta}$ for all $\lambda\geq 1$ and $(0,\infty)^d=\bigcup_{\lambda>0}\lambda L_{d,\eta}$. Now for each $z\in (0,\infty)$, observe that
\begin{align}
\label{Eq:lslice}
L_{d,\eta}\cap\bigl((0,\infty)^{d-1}\times\{z\}\bigr)&=(h_\eta(z)\cdot E_{d-1})\times\{z\}\\
\label{Eq:polyinclusion}
&\subseteq(h_\eta(z)\cdot A_{d-1,1})\times\{z\}\subseteq(\overline{h}_\eta(z)\cdot A_{d-1,1})\times\{z\}\\
&=A_{d-1,\,\eta/z}\times\{z\}=A_{d,\eta}\cap\bigl((0,\infty)^{d-1}\times\{z\}\bigr);\notag
\end{align}
indeed, the first inclusion in~\eqref{Eq:polyinclusion} follows from part (ii) of the inductive hypothesis, which ensures that  $E_{d-1}=L_{d-1,1}\subseteq A_{d-1,1}$, and the second inclusion follows from the definition of $A_{d-1,1}$ and the fact that $\overline{h}_\eta\leq h_\eta$. This shows that $L_{d,\eta}\subseteq A_{d,\eta}$.
\end{proof}
\begin{proof}[Property (iii)]
\let\qed\relax
Now fix $(x,z)\in L_{d,\eta}$ and let $x'\in (0,\infty)^{d-1}$ be such that $x_j'\geq x_j$ for all $1\leq j\leq d-1$. It follows from parts (ii) and (iii) of the inductive hypothesis that $\tau_{E_{d-1}}(x')\geq\tau_{E_{d-1}}(x)$. Since $L_{d,\eta}=\{(x,z):x\in (0,\infty)^d,\,z\geq g_\eta(x)\}$, we see that if $z'\geq z$, then \[z'\geq z\geq g_\eta(x)=(\inv{h}_\eta\circ\tau_{E_{d-1}})(x)\geq (\inv{h}_\eta\circ\tau_{E_{d-1}})(x')=g_\eta(x'),\]
so $(x',z')\in L_{d,\eta}$. This yields the first assertion of (iii). Since $\Int L_{d,\eta}=\{(x,z):x\in (0,\infty)^d,\,z>g_\eta(x)\}$ by property (i), we can argue similarly to obtain the corresponding conclusion when $L_{d,\eta}$ is replaced by $\Int L_{d,\eta}$ throughout.
\end{proof}
\begin{proof}[Property (iv)]
\let\qed\relax
Note that $(x,z)\in (0,\infty)^d$ lies in $L_{d,\eta}$ if and only if $(\eta^{-1/d}x)/h_1(\eta^{-1/d}z)=x/h_\eta(z)\in E_{d-1}$, i.e.\ $\eta^{-1/d\,}(x,z)\in L_{d,1}$. Thus, if $\tilde{\eta}>\eta$, then $L_{d,\tilde{\eta}}=\tilde{\eta}^{1/d}L_{d,1}\equiv\tilde{\eta}^{1/d}E_d=(\tilde{\eta}/\eta)^{1/d}\,L_{d,\eta}\subsetneqq L_{d,\eta}$ by property (ii) and Lemma~\ref{Lem:conveasy}(i). Moreover, $h_\eta(\eta^{1/d})=\overline{h}_\eta(\eta^{1/d})=\eta^{1/d}$, and part (iv) of the inductive hypothesis asserts that $E_{d-1}$ contains $(1,\dotsc,1)\in\R^{d-1}$, so we deduce from~\eqref{Eq:lslice} that $L_{d,\eta}$ contains $(\eta^{1/d},\dotsc,\eta^{1/d})\in\R^d$. It follows from this and property (iii) that $\tau_{E_d}(x)\leq\tau_{E_d}((1/2,\dotsc,1/2))=1/2$ for all $x\in [0,1/2]^d$. Together with property (ii) and Lemma~\ref{Lem:conveasy}(i), this shows that $[0,1/2]^d\cap L_{d,\eta}=\{x\in [0,1/2]^d:\tau_{E_{d-1}}(x)\geq\eta^{1/d}\}$ is non-empty if and only if $\eta\in (0,2^{-d}]$, and that $[0,1/2]^d\cap\Int L_{d,\eta}=\{x\in [0,1/2]^d:\tau_{E_{d-1}}(x)>\eta^{1/d}\}$ is non-empty if and only if $\eta\in (0,2^{-d})$.
\end{proof}
\begin{proof}[Property (v)]
\let\qed\relax
Next, we investigate the facial structure of $L_{d,\eta}$. By part (v) of the inductive hypothesis, every facet of $E_{d-1}$ is a polytope of the form $G_m\equiv F_{1,m}$ for some $m\in\Z^{d-2}$. For each such $m$, let $\theta_m\in\R^{d-1}$ be such that $H_m:=\{u\in\R^{d-1}:\tm{\theta_m}u=1\}$ is a supporting hyperplane to $E_{d-1}$ with $E_{d-1}\cap H_m=G_m$. Since $H_m$ is disjoint from $\Int E_{d-1}$ and $x\in (\tm{\theta_m}x)\,H_m$ for all $x\in (0,\infty)^{d-1}$, it follows from Lemma~\ref{Lem:conveasy}(i) that $\tau_{E_{d-1}}(x)\leq\tm{\theta_m}x$ for all $x\in (0,\infty)^{d-1}$, with equality if and only if $x\in (\tm{\theta_m}x)\,(E_{d-1}\cap H_m)=(\tm{\theta_m}x)\,G_m$, i.e.\ $x\in\bigcup_{\lambda>0}\lambda G_m$. Therefore,
\begin{equation}
\label{Eq:support}
g_{\eta}(x)=(\inv{h_\eta}\circ\tau_{E_{d-1}})(x)\geq\inv{h_\eta}(\tm{\theta_m}x)\geq t_{\eta,j}(\tm{\theta_m}x)
\end{equation}
for all $x\in (0,\infty)^{d-1}$, with equality if and only if $x\in D_{m,j}:=\bigcup_{\lambda\in [w_{\eta,j},w_{\eta,j-1}]}\lambda G_m$. 

Observe that $D_{m,j}$ is a $(d-1)$-dimensional polytope; indeed, writing $V_m$ for the (finite) set of extreme points of the polytope $G_m$, we see that $V_{m,j}:=\{\lambda u:\lambda\in\{w_{\eta,j},\,w_{\eta,j-1}\},\,u\in V_m\}$ is the set of extreme points of $D_{m,j}$. Setting $m'=(m,j)\in\Z^{d-2}\times\Z\equiv\Z^{d-1}$, we deduce from~\eqref{Eq:support} that the restriction of $g_\eta$ to $D_{m,j}$ is an affine function $x\mapsto t_{\eta,j}(\tm{\theta_m}x)$, and moreover that $L_{d,\eta}$ is contained within the closed half-space $H_{\eta,m'}^+:=\{(x,z)\in\R^d:z\geq t_{\eta,j}(\tm{\theta_m}x)\}$. It follows that $H_{\eta,m'}:=\partial H_{\eta,m'}^+=\{(x,t_{\eta,j}(\tm{\theta_m}x)):x\in\R^{d-1}\}\subseteq\R^d$ is a supporting hyperplane to $L_{d,\eta}$ and that 
\[F_{\eta,m'}:=L_{d,\eta}\cap H_{\eta,m'}=\{(x,g_\eta(x)):x\in D_{m,j}\}=\{(x,t_{\eta,j}(\tm{\theta_m}x)):x\in D_{m,j}\}\subseteq\partial L_{d,\eta}\]
is a facet of $L_{d,\eta}$. In addition, the set of extreme points of $F_{\eta,m'}$ is $V_{\eta,m'}:=\{(x,g_\eta(x)):x\in V_{m,j}\}$, so $F_{\eta,m'}$ is also a $(d-1)$-dimensional polytope. 

Since $g_\eta(x)=(\inv{h}_\eta\circ\tau_{E_{d-1}})(x)=\eta^{1/d}\,\inv{h}_\eta(\tau_{E_{d-1}}(x)/\eta^{1/d})=\eta^{1/d}g_1(x/\eta^{1/d})$ for all $x\in (0,\infty)^{d-1}$, it follows that $x'\in F_{\eta,m'}$ if and only if there exists $\lambda'\in [2^{-j},2^{-j+1}]=[w_{1,j},w_{1,j-1}]$ and $x\in G_m$ such that $x'=(\eta^{1/d}\lambda'x,g_\eta(\eta^{1/d}\lambda'x))=\eta^{1/d}(\lambda'x,g_1(\lambda'x))$. This shows that $F_{\eta,m'}=\eta^{1/d}F_{1,m'}\equiv G_{m'}$. Moreover, since $\tm{\theta_m}x=\tau_{E_{d-1}}(x)\in [w_{\eta,j},w_{\eta,j-1}]$ for all $x\in D_{m,j}$, we have $g_\eta(x)=t_{\eta,j}(\tm{\theta_m}x)\in[\inv{h}_\eta(w_{\eta,j-1}),\inv{h}_\eta(w_{\eta,j})]=[z_{\eta,j-1},z_{\eta,j}]$ for all $x\in D_{m,j}$.

By applying parts (ii) and (v) of the inductive hypothesis, we find that
\[\textstyle(0,\infty)^{d-1}=\bigcup_{\lambda>0}\lambda\,\partial E_{d-1}=\bigcup_{m\in\Z^{d-2}}\bigcup_{j\in\Z}\,\bigcup_{\lambda\in [w_{\eta,j},w_{\eta,j-1}]}\lambda G_m=\bigcup_{(m,j)\in\Z^{d-2}\times\Z}D_{m,j}.\]
Thus, for every $x'=(x,g_\eta(x))\in\partial L_{d,\eta}$, there exists $m'=(m,j)\in\Z^{d-1}$ such that $x\in D_{m,j}$, so that $x'\in F_{\eta,m'}$. 
Furthermore, if $F$ is an arbitrary facet of $L_{d,\eta}$, then there exist $m'=(m,j)\in\Z^{d-1}$ and $x\in\Int D_{m,j}$ such that $(x,g_\eta(x))\in\relint F\subseteq\partial L_{d,\eta}$. 
But since $(x,g_\eta(x))\in\relint F_{\eta,m'}$, it follows that $(\relint F)\cap (\relint F_{\eta,m'})\neq\emptyset$, whence $F=F_{\eta,m'}$ by~\citet[Theorem~2.1.2]{Sch14}. 
\end{proof}
Having established properties (i)--(v) of $L_{d,\eta}$, we now verify that $P_{d,\eta}=\bigcap_{\,g\in G_R(Q)}g(L_{d,\eta}\cap Q)$ has the required properties (vi)--(ix). 

\begin{proof}[Property (vi)]
\let\qed\relax
Since $L_{d,\eta}\cap Q$ is compact and convex, it follows that $P_{d,\eta}$ is a compact, convex and $G_R(Q)$-invariant subset of $Q$. Moreover, we have \[\textstyle P_{d,\eta}=\bigcap_{\,g\in G_R(Q)}g(L_{d,\eta}\cap Q)\subseteq\bigcap_{\,g\in G_R(Q)}g(A_{d,\eta}\cap Q)=J_{d,\eta}\]
by property (ii) of $L_{d,\eta}$ and Lemma~\ref{Lem:invelopebox}(ii).
\end{proof}
\begin{proof}[Property (vii)]
\let\qed\relax
Recalling the definition of the function $M\colon Q\to [0,1/2]^d$ from the paragraph before Lemma~\ref{Lem:invelopebox}, we deduce from property (iii) of $L_{d,\eta}$ that for $x\in Q$, we have $g(x)\in L_{d,\eta}$ for all $g\in G_R(Q)$ if and only if $M(x)\in L_{d,\eta}$. Thus, it follows as in the proof of Lemma~\ref{Lem:invelopebox}(ii) that $P_{d,\eta}=\{x\in Q:M(x)\in L_{d,\eta}\}$ and hence that $[0,1/2]^d\cap P_{d,\eta}=[0,1/2]^d\cap L_{d,\eta}$. By a similar argument based on property (iii), we deduce that $\Int P_{d,\eta}=\{x\in Q:M(x)\in\Int L_{d,\eta}\}$ and hence that $[0,1/2]^d\cap\Int P_{d,\eta}=[0,1/2]^d\cap\Int L_{d,\eta}$. 
Turning to the last assertion of (vii), it suffices to show that there exists $\alpha_d''>0$, depending only on $d$, such that
\begin{equation}
\label{Eq:slicevol}
\mu_{d-1}(([0,1/2]^{d-1}\times\{z\})\setminus P_{d,\eta})\leq\alpha_d''\,\mu_{d-1}(([0,1/2]^{d-1}\times\{z\})\setminus J_{d,\eta})
\end{equation}
for all $z\in [0,1/2]$, since we can then integrate this inequality with respect to $z$ to conclude that $\mu_d([0,1/2]^d\setminus P_{d,\eta})\leq\alpha_d''\,\mu_d([0,1/2]^d\setminus J_{d,\eta})$. To this end, fix $z\in [0,1/2]$ and observe that by~\eqref{Eq:lslice} and property (vi), the left-hand side of~\eqref{Eq:slicevol} is equal to
\begin{align*}
\mu_{d-1}(([0,1/2]^{d-1}\times\{z\})\setminus L_{d,\eta})&=\mu_{d-1}([0,1/2]^{d-1}\setminus (h_{\eta}(z)\cdot E_{d-1}))\\
&=\mu_{d-1}([0,1/2]^{d-1}\setminus L_{d-1,\,h_\eta(z)^{d-1}}).
\end{align*}
By applying part (vii) of the inductive hypothesis, Lemma~\ref{Lem:invelopebox}(iii),~\eqref{Eq:hratio} and Lemma~\ref{Lem:invelopebox}(i) in that order, we find that
\begin{align*}
\mu_{d-1}([0,1/2]^{d-1}\setminus L_{d-1,\,h_\eta(z)^{d-1}})&=\mu_{d-1}([0,1/2]^{d-1}\setminus P_{d-1,\,h_\eta(z)^{d-1}})\\
&\leq\alpha_{d-1}'\,\mu_{d-1}([0,1/2]^{d-1}\setminus J_{d-1,\,h_\eta(z)^{d-1}})\\
&\leq \alpha_{d-1}'\,\{h_\eta(z)/\overline{h}_\eta(z)\}^{d-1}\,\mu_{d-1}([0,1/2]^{d-1}\setminus J_{d-1,\,\overline{h}_\eta(z)^{d-1}})\\
&\leq \alpha_{d-1}'\,\gamma_d^{d-1}\mu_{d-1}([0,1/2]^{d-1}\setminus A_{d-1,\,\overline{h}_\eta(z)^{d-1}})\\
&=\alpha_{d-1}'\gamma_d^{d-1}\,\mu_{d-1}(([0,1/2]^{d-1}\times\{z\})\setminus A_{d,\eta}),\\
&=\alpha_{d-1}'\gamma_d^{d-1}\,\mu_{d-1}(([0,1/2]^{d-1}\times\{z\})\setminus J_{d,\eta}),
\end{align*}
which completes the proof of~\eqref{Eq:slicevol} and hence that of (vii).
\end{proof}
\begin{proof}[Property (viii)]
\let\qed\relax
In view of properties (iv) and (vii), it suffices to consider $\eta\in (0,2^{-d}]$, since otherwise $P_{d,\eta}=\emptyset$. Note that for $x\in (0,\infty)^{d-1}$, Lemma~\ref{Lem:conveasy}(i) implies that $g_\eta(x)=(\inv{h}_\eta\circ\tau_{E_{d-1}})(x)\leq 1/2$ if and only if $\tau_{E_{d-1}}(x)\geq h_\eta(1/2)$. By property (iv), this is the case if and only if $x\in h_\eta(1/2)\cdot E_{d-1}=L_{d-1,\,h_\eta(1/2)^{d-1}}$. In view of property (vii), it follows that
\begin{align}
\hspace{-0.1cm}[0,1/2]^d\cap\partial P_{d,\eta}=[0,1/2]^d\cap\partial L_{d,\eta}&=\bigl\{(x,g_\eta(x)):x\in [0,1/2]^{d-1},\,g_\eta(x)\leq 1/2\bigr\}\notag\\
\label{Eq:polybd}
&=\bigl\{(x,g_\eta(x)):x\in [0,1/2]^{d-1}\cap L_{d-1,\,h_\eta(1/2)^{d-1}}\bigr\}.
\end{align}
With the aid of this identity and the inductive hypothesis, we will identify a finite set $U'$ such that $U'\subseteq [0,1/2]^d\cap\partial P_{d,\eta}\subseteq\conv U'$ and $\abs{U'}\lesssim_d\log^{d-1}(1/\eta)$. We will then deduce that $P_{d,\eta}$ is the convex hull of $U:=\bigcup_{g\in G_R(Q)}g(U')$ and hence that (viii) holds.

To this end, fix $\eta\in (0,2^{-d}]$ and let $j_-:=\bigl\lfloor\log_2(2\eta^{1/d})\bigr\rfloor+1$ and $j_+:=\bigl\lceil\log_2\bigl(\eta^{1/d}/h_\eta(1/2)\bigr)\bigr\rceil$, so that $j_-$ is the smallest integer $j$ for which $w_{\eta,j}=2^{-j}\eta^{1/d}<1/2$ and $j_+$ is the smallest integer $j$ for which $w_{\eta,j}\leq h_\eta(1/2)$. Since $h_\eta(1/2)\geq\overline{h}_\eta(1/2)=(2\eta)^{1/(d-1)}$, we have
\begin{align*}
j_+&\leq 1+\log_2\bigl(\eta^{1/d}/h_\eta(1/2)\bigr)\leq 1+\log_2\bigl(\eta^{1/d}/\overline{h}_\eta(1/2)\bigr)\leq 1+\tfrac{1}{d(d-1)}\log_2(1/\eta)\quad\text{and}\\ 
j_-&\geq 1+\log_2(2\eta^{1/d})=2-\tfrac{1}{d}\log_2(1/\eta),
\end{align*}
so $j_+-j_-+1\leq\frac{1}{d-1}\log_2(1/\eta)\leq 2\log(1/\eta)$. For $j=j_-,\dotsc,j_+$, let $L_j':=L_{d-1,\,\{w_{\eta,j}\vee h_\eta(1/2)\}^{d-1}}$ and $P_j':=P_{d-1,\,\{w_{\eta,j}\vee h_\eta(1/2)\}^{d-1}}$, and for convenience, set $P_{j_--1}:=\emptyset$. Then by property (iv) above, we have $L_j'=L_{d-1,\,w_{\eta,j}^d}=w_{\eta,j}E_{d-1}=2L_{j-1}'$ for $j=j_-+1,\dotsc,j_+-1$ and \[L_{j_+}'=L_{d-1,\,h_\eta(1/2)^{d-1}}=h_\eta(1/2)\cdot E_{d-1}.\]
Recall also that by properties (iii) and (iv) above, we have $\tau_{E_{d-1}}(x)\leq\tau_{E_{d-1}}((1/2,\dotsc,1/2))=1/2$ for all $x\in [0,1/2]^{d-1}$. Thus, continuing on from~\eqref{Eq:polybd} and recalling part (vii) of the inductive hypothesis, we can write
\begin{align}
[0,1/2]^{d-1}\cap P_{j_+}'=[0,1/2]^{d-1}\cap L_{j_+}'&=\bigl\{x\in [0,1/2]^{d-1}:h_\eta(1/2)\leq\tau_{E_{d-1}}(x)\leq 1/2\bigr\}\notag\\
\label{Eq:wjshelling}
&=\textstyle\bigcup_{j=j_-}^{\,j_+}\bigl\{[0,1/2]^{d-1}\cap (P_j'\setminus\Int P_{j-1}')\bigr\};
\end{align}
note in particular that by Lemma~\ref{Lem:conveasy}(i) and our choice of $j_\pm$, it follows that
\begin{align}
[0,1/2]^{d-1}\cap (P_j'\setminus\Int P_{j-1}')&=[0,1/2]^{d-1}\cap (L_j'\setminus\Int L_{j-1}')\notag\\
&=\bigl\{x\in [0,1/2]^{d-1}:\tau_{E_{d-1}}(x)\in [w_{\eta,j}\vee h_\eta(1/2),w_{\eta,j-1}\wedge 1/2]\bigr\}\notag\\
\label{Eq:wjshell}
&\textstyle\subseteq [0,1/2]^{d-1}\cap\bigcup_{\lambda\in[w_{\eta,j},w_{\eta,j-1}]}\lambda\,\partial E_{d-1}
\end{align}
for all $j\in\{j_-,\dotsc,j_+\}$, and moreover that the interiors of these sets are non-empty and pairwise disjoint. Since $2^{-(d-1)}=h_{2^{-d}}(1/2)^{d-1}\geq h_\eta(1/2)^{d-1}\geq\overline{h}_\eta(1/2)^{d-1}=2\eta$,
we deduce from part (ix) of the inductive hypothesis and~\eqref{Eq:wjshell} that the following holds for all $j\in\{j_-,\dotsc,j_+\}$: the set $[0,1/2]^{d-1}\cap (P_j'\setminus\Int P_{j-1}')$ can be triangulated into at most $\alpha_{d-1}'\log^{d-2}(1/\eta)$ simplices of dimension $d-1$, in such a way that for every constituent simplex $S$, there exists $m\in\Z^{d-2}$ such that $S\subseteq\bigcup_{[w_{\eta,j},w_{\eta,j-1}]}\lambda G_m=D_{m,j}$. 

Recall from~\eqref{Eq:support} that $\restr{g_\eta}{D_{m,j}}$ is affine on $D_{m,j}$ for all $(m,j)\in\Z^{d-2}\times\Z$. Thus, by the deduction in the previous paragraph and~\eqref{Eq:wjshelling}, it follows that there is a triangulation of $[0,1/2]^{d-1}\cap L_{j_+}'$ into $(d-1)$-simplices $S_1,\dotsc,S_N$, where $N\leq (j_+-j_-+1)\,\alpha_{d-1}'\log^{d-2}(1/\eta)\leq 2\alpha_{d-1}'\log^{d-1}(1/\eta)$ and the restriction of $g_\eta$ to each $S_k$ is an affine function. For each $1\leq k\leq N$, this implies that $R_k:=\{(x,g_\eta(x)):x\in S_k\}$ is a $(d-1)$-simplex and that there exists $m'\in\Z^{d-1}$ such that $R_k\subseteq F_{\eta,m'}\subseteq\partial L_{d,\eta}$.
Writing $U_k$ for the set of vertices of $R_k$ and setting $U':=\bigcup_{k=1}^N U_k$, we deduce from~\eqref{Eq:polybd} that 
\begin{equation}
\label{Eq:polybdsimp}
\textstyle U'\subseteq [0,1/2]^d\cap\partial P_{d,\eta}=\bigcup_{\,k=1}^{\,N}R_k\subseteq\conv U'.
\end{equation}
By applying the (affine) isometries in $G_R(Q)$, we conclude that every $u\in\partial P_{d,\eta}$ lies in the convex hull of $U:=\bigcup_{g\in G_R(Q)}g(U')=\bigcup_{g\in G_R(Q)}\bigcup_{\,k=1}^{\,N}g(U_k)$, and it then follows from the convexity of $P_{d,\eta}$ that
$P_{d,\eta}=\conv\partial P_{d,\eta}=\conv U$. Since $\abs{U_k}=d$ for all $1\leq k\leq N$ and $\abs{G_R(Q)}=2^d$, this implies that $P_{d,\eta}$ is a polytope with at most $\abs{U}\leq 2^d\,dN\leq 2^{d+1}\,d\alpha_{d-1}'\log^{d-1}(1/\eta)$ vertices, so the proof of (viii) is complete.
\end{proof}
\begin{proof}[Property (ix)]
\let\qed\relax
We start by observing that in each of the following cases, the construction from~\citet[][Section~17.5.1]{Lee04} yields a triangulation of any polytope $K\subseteq\R^d$ of the specified form into $d$ simplices of dimension $d$, each of which is the convex hull of $d+1$ vertices of $K$: 
\begin{enumerate}[label=(\alph*)]
\item $K=S\times [0,1]$, where $S\subseteq\R^{d-1}$ is a $(d-1)$-simplex;
\item $K=S_h^{\downarrow}:=\{(x,z)\in S\times\R:0\leq z\leq h(x)\}$, where $S\subseteq\R^{d-1}$ is a $(d-1)$-simplex and $h\colon S\to\R$ is an affine function that is strictly positive on $S$;
\item $K=\bigcup_{\lambda\in [a,b]}\lambda S$, where $0<a<b$ and $S\subseteq\R^d\setminus\{0\}$ is a $(d-1)$-simplex.
\end{enumerate}
Indeed, it is easily seen that all such polytopes are combinatorially isomorphic in the sense of~\citet[Section~16.1.1]{HRGZ04}, so the same explicit construction works in all cases. 

We will now triangulate $[0,1/2]^d\setminus\Int P_{d,\eta}$ by `lifting' an suitable triangulation of $[0,1/2]^{d-1}$ and then applying the fact above. Let $S_1,\dotsc,S_N$ be the $(d-1)$-simplices that constitute the triangulation of $P_{j_+}'=P_{d-1,\,h_\eta(1/2)^{d-1}}$ obtained in the proof of (viii). By part (ix) of the inductive hypothesis, we can also triangulate $[0,1/2]^{d-1}\setminus\Int P_{j_+}'$ into $(d-1)$-simplices $S_{N+1},\dotsc,S_{N+M}$, where $M\leq\alpha_{d-1}'\log^{d-2}(1/\eta)$. 

Recall now that $\Int L_{d,\eta}=\{(x,z)\in (0,\infty)^d:z>g_\eta(x)\}$ by property (i) above. For $x\in (0,\infty)^{d-1}$, Lemma~\ref{Lem:conveasy}(i) implies that $g_\eta(x)\leq 1/2$ if and only if $x\in L_{j_+}'$, and $g_\eta(x)<1/2$ if and only if $x\in\Int L_{j_+}'$; see also~\eqref{Eq:polybd}. Also, by property (vii), we have $[0,1/2]^{d-1}\cap P_{j_+}'=[0,1/2]^{d-1}\cap L_{j_+}'$ and $[0,1/2]^{d-1}\cap\Int P_{j_+}'=[0,1/2]^{d-1}\cap\Int L_{j_+}'$. Therefore, for $x\in [0,1/2]^{d-1}\cap P_{j_+}'$, we have $(x,z)\in [0,1/2]^d\setminus\Int L_{d,\eta}$ if and only if $0\leq z\leq g_\eta(x)$, whereas for $x\in [0,1/2]^{d-1}\setminus\Int P_{j_+}'$, we have $(x,z)\in [0,1/2]^d\setminus\Int L_{d,\eta}$ for all $z\in [0,1/2]$. Recalling from the paragraph before~\eqref{Eq:polybdsimp} that $g_\eta$ is affine on each of the simplices $S_1,\dotsc,S_N$, we deduce that
\begin{align*}
[0,1/2]^d\setminus\Int P_{j_+}'=[0,1/2]^d\setminus\Int L_{j_+}'&=\bigl\{(x,z):x\in [0,1/2]^{d-1}\cap P_{j_+}',\,0\leq z\leq g_\eta(x)\bigr\}\,\cup\\
&\phantom{{}={}}\bigl\{(x,z):x\in [0,1/2]^{d-1}\setminus\Int P_{j_+}',\,z\in [0,1/2]\bigr\}\\
&=\bigl(\textstyle\bigcup_{\,k=1}^{\,N}\,(S_k)_{g_\eta}^\downarrow\bigr)\cup\,\bigl(\textstyle\bigcup_{\,m=1}^{\,M} S_{N+m}\times [0,1/2]\bigr).
\end{align*}
By appealing to cases (a) and (b) of the fact above, we conclude that $[0,1/2]^d\setminus\Int P_{j_+}'$ can be triangulated into $d(N+M)\leq 3d\alpha_{d-1}'\log^{d-1}(1/\eta)$ simplices.

Now if $0<\eta<\tilde{\eta}\leq 2^{-d}$, then it follows from properties (ii) and (iv) above as well as Lemma~\ref{Lem:conveasy}(i) that
\begin{equation}
\label{Eq:polyshells}
\textstyle L_{d,\eta}\setminus\Int L_{d,\tilde{\eta}}=\bigl\{x\in (0,\infty)^d:\tau_{E_d}(x)\in [\eta^{1/d},\tilde{\eta}^{1/d}]\bigr\}=\bigcup_{\lambda\in [1,(\tilde{\eta}/\eta)^{1/d}]}\lambda\,\partial L_{d,\eta}.
\end{equation}
In addition, $[0,1/2]^d\cap(\lambda\,\partial L_{d,\eta})\subseteq\lambda\,([0,1/2]^d\cap\partial L_{d,\eta})$ for all $\lambda>0$, so by property (vii) and~\eqref{Eq:polybdsimp}, we can write 
\begin{align*}
[0,1/2]^d\cap (P_{d,\eta}\setminus\Int P_{d,\tilde{\eta}})&=[0,1/2]^d\cap (L_{d,\eta}\setminus\Int L_{d,\tilde{\eta}})\\
&=\textstyle [0,1/2]^d\cap\bigl\{\bigcup_{\lambda\in [1,(\tilde{\eta}/\eta)^{1/d}]}\lambda\bigl([0,1/2]^d\cap\partial P_{d,\eta}\bigr)\bigr\}\\
&=\textstyle[0,1/2]^d\cap\bigl(\bigcup_{\,k=1}^{\,N}\,\bigcup_{\lambda\in [1,(\tilde{\eta}/\eta)^{1/d}]}\lambda R_k\bigr).
\end{align*}
By appealing to case (c) of the fact above, we deduce that for every $1\leq k\leq N$, there is a triangulation of $\bigcup_{\lambda\in [1,(\tilde{\eta}/\eta)^{1/d}]}\lambda R_k$ into $d$-simplices $T_{k1},\dotsc,T_{kd}$, so that
\begin{equation}
\label{Eq:polyintsimp}
\textstyle [0,1/2]^d\cap (P_{d,\eta}\setminus\Int P_{d,\tilde{\eta}})=\bigcup_{\,k=1}^{\,N}\,\bigcup_{\,\ell=1}^{\,d}\,\bigl([0,1/2]^d\cap T_{k\ell}\bigr).
\end{equation}
Also, for each $1\leq k\leq N$, recall from the paragraph before~\eqref{Eq:polybdsimp} that $R_k\subseteq F_{\eta,m'}$ for some $m'\in\Z^{d-1}$, so by property (iv), it follows that $T_{k\ell}\subseteq\bigcup_{\lambda\in [1,(\tilde{\eta}/\eta)^{1/d}]}\lambda F_{\eta,m'}=\bigcup_{\lambda\in [\eta^{1/d},\,\tilde{\eta}^{1/d}]}\lambda G_{m'}$ for all $1\leq\ell\leq d$.

Before proceeding, we will now verify that the left-hand side of~\eqref{Eq:polyintsimp} is in fact the union of those sets $T_{k\ell}^\dagger:=[0,1/2]^d\cap T_{k\ell}$ for which $\Int T_{k\ell}^\dagger\neq\emptyset$. In view of~\eqref{Eq:polyintsimp} and Lemma~\ref{Lem:clint}, it suffices to show that $[0,1/2]^d\cap (P_{d,\eta}\setminus\Int P_{d,\tilde{\eta}})$ is the closure of its interior. If $[0,1/2]^d\cap\Int P_{d,\tilde{\eta}}=\emptyset$, then $[0,1/2]^d\cap (P_{d,\eta}\setminus\Int P_{d,\tilde{\eta}})$ is convex and therefore has the required property by~\citet[Theorem~1.1.15]{Sch14}, so suppose now that this is not the case. Fix $\tilde{x}\in\Int([0,1/2]^d\setminus P_{d,\tilde{\eta}})$ and let $x\in [0,1/2]^d\cap (P_{d,\eta}\setminus P_{d,\tilde{\eta}})$. Since $[0,1/2]^d\cap P_{d,\eta}$ is convex,~\citet[Lemma~1.1.9]{Sch14} implies that $[\tilde{x},x)\subseteq\Int([0,1/2]^d\cap P_{d,\eta})$. 
Also, since $[0,1/2]^d\cap P_{d,\tilde{\eta}}$ is a closed, convex set that does not contain $x$, there is a unique $x'\in (\tilde{x},x)$ such that $[\tilde{x},x]\cap ([0,1/2]^d\cap P_{d,\tilde{\eta}})=[\tilde{x},x']$. Therefore, \[\emptyset\neq (x',x)\subseteq \Int([0,1/2]^d\cap P_{d,\eta})\cap\cm{P_{d,\tilde{\eta}}}=\Int\{[0,1/2]^d\cap (P_{d,\eta}\setminus\Int P_{d,\tilde{\eta}})\},\]
so $x\in\Cl (x',x)\subseteq\Cl\Int\{[0,1/2]^d\cap (P_{d,\eta}\setminus\Int P_{d,\tilde{\eta}})\}$. If instead $x\in [0,1/2]^d\cap\partial P_{d,\tilde{\eta}}$, then $\tau_{E_d}(\lambda x)=\lambda\tau_{E_d}(x)=\lambda\tilde{\eta}^{1/d}$  and $\lambda x\in (0,1/2)^d$ for all $\lambda\in (0,1)$. Thus, if $\lambda\in ((\eta/\tilde{\eta})^{1/d},1)$, then $\lambda x\in\{u\in (0,1/2)^d:\eta^{1/d}<\tau_{E_d}(u)<\tilde{\eta}^{1/d}\}=\Int\{[0,1/2]^d\cap (P_{d,\eta}\setminus\Int P_{d,\tilde{\eta}})\}$, where the final equality follows from property (iv),~\eqref{Eq:polyshells} and the continuity of $\tau_{E_d}$. This implies that $x=\lim_{\lambda\nearrow\,1}\lambda x\in\Cl\Int\{[0,1/2]^d\cap (P_{d,\eta}\setminus\Int P_{d,\tilde{\eta}})\}$, as required.

Returning to the proof of (ix), note that for all $k,\ell$, the set $T_{k\ell}^\dagger=[0,1/2]^d\cap T_{k\ell}$ is the intersection of a $d$-simplex and at most $2d$ closed half-spaces, so $T_{k\ell}^\dagger$ is a polytope and the number of vertices of $T_{k\ell}^\dagger$ is bounded above by a constant that depends only on $d$. Therefore, there exists $\Gamma_d>0$, depending only on $d$, such that whenever $\Int T_{k\ell}^\dagger\neq\emptyset$, the $d$-dimensional polytope $T_{k\ell}^\dagger$ can be triangulated into at most $\Gamma_d$ $d$-simplices~\citep[e.g.][Corollary~2.3]{RS85}, each of which is the convex hull of $d+1$ vertices of $T_{k\ell}^\dagger$.
It follows from this and the paragraph above that $[0,1/2]^d\cap (P_{d,\eta}\setminus\Int P_{d,\tilde{\eta}})$ can be triangulated into at most $Nd\,\Gamma_d\leq 2d\,\Gamma_d\,\alpha_{d-1}'\log^{d-1}(1/\eta)$ $d$-simplices, in such a way that for each constituent $d$-simplex $T$, there exist $1\leq k\leq N$, $1\leq\ell\leq d$ and $m'\in\Z^{d-1}$ such that $T\subseteq T_{k\ell}\subseteq\bigcup_{\lambda\in [1,(\tilde{\eta}/\eta)^{1/d}]}\lambda F_{\eta,m'}=\bigcup_{\lambda\in [\eta^{1/d},\,\tilde{\eta}^{1/d}]}\lambda G_{m'}$. This completes the inductive step.
\end{proof}
We have now established properties (vi)--(ix), from which we can easily deduce assertions (i) and (ii) of the lemma. To obtain part (iii) of the lemma, recall from Lemma~\ref{Lem:invelopesimp} that $Q^\triangle$ is a polytope with $2d$ facets and that $[0,1/(d+1)]^d\subseteq Q^\triangle\subseteq [0,1/2]^d$. Thus, $\tilde{P}_\eta=Q^\triangle\cap P_\eta$ is a polytope for all $\eta>0$, and since $[0,1/2]^d\setminus J_\eta\subseteq\frac{d+1}{2}\,([0,1/(d+1)]^d\setminus J_\eta)\subseteq \frac{d+1}{2}\,(Q^\triangle\setminus J_\eta)$, it follows from property (vii) above that
\begin{align*}
\mu_d(Q^\triangle\setminus\tilde{P}_\eta)=\mu_d(Q^\triangle\setminus P_\eta)\leq\mu_d([0,1/2]^d\setminus P_\eta)&\leq\alpha_d'\,\mu_d([0,1/2]^d\setminus J_\eta)\\
&\leq\alpha_d'\rbr{\frac{d+1}{2}}^d\,\mu_d(Q^\triangle\setminus J_\eta).
\end{align*}

Turning now to part (iv) of the lemma, we first verify that $\tilde{P}_\eta$ is non-empty if and only if $\eta\in (0,(d+1)^{-d}]$. Indeed, if $x=(x_1,\dotsc,x_d)\in Q^\triangle$, then it follows from the definition of $Q^\triangle$ that $1-\sum_{j=1}^d x_j\leq 1/(d+1)$, so by the AM--GM inequality, we have $\prod_{j=1}^d x_j\leq\bigl(\sum_{j=1}^d x_j/d\bigr)^d\leq (d+1)^{-d}$. Together with the definition of $J_\eta$, this implies that $\tilde{P}_\eta=Q^\triangle\cap P_\eta\subseteq Q^\triangle\cap J_\eta$ is non-empty only if $\eta\in (0,(d+1)^{-d}]$. Conversely, if $\eta\in (0,(d+1)^{-d}]$, then property (iv) above implies that $(\eta^{1/d},\dotsc,\eta^{1/d})\in [0,1/(d+1)]^d\cap P_\eta\subseteq Q^\triangle\cap P_\eta=\tilde{P}_\eta$, as required.

For $0<\eta<\tilde{\eta}\leq (d+1)^{-d}$, we know from property (ix) above that there exist triangulations of $[0,1/2]^d\setminus\Int P_\eta$ and $[0,1/2]^d\cap (P_\eta\setminus\Int P_{d,\tilde{\eta}})$ into at most $\alpha_d'\log(1/\eta)$ $d$-simplices. Since $Q^\triangle\subseteq [0,1/2]^d$, both $Q^\triangle\setminus\Int\tilde{P}_{d,\eta}=Q^\triangle\cap ([0,1/2]^d\setminus\Int P_\eta)$ and $\tilde{P}_{d,\eta}\setminus\Int\tilde{P}_{d,\tilde{\eta}}=Q^\triangle\cap\{[0,1/2]^d\cap (P_\eta\setminus\Int P_{d,\tilde{\eta}})\}$ can be expressed as the union of $Q^\triangle\cap T$ over all $d$-simplices $T$ in the respective triangulations above. By Lemma~\ref{Lem:invelopesimp} and its proof, $Q^\triangle$ is the intersection of $2d$ half-spaces, so each set of the form $Q^\triangle\cap T$ is the intersection of a $d$-simplex and $2d$ half-spaces. Arguing as in the proof of (ix) above, we deduce that if $\Int(Q^\triangle\cap T)\neq\emptyset$, then $Q^\triangle\cap T$ is a $d$-dimensional polytope that can be triangulated into at most $\Gamma_d$ simplices, each of which is the convex hull of $d+1$ vertices of $Q^\triangle\cap T$. 

Moreover, by a very similar argument to that given in the penultimate paragraph of the proof of (ix), the sets $Q^\triangle\setminus\Int\tilde{P}_{d,\eta}$ and $\tilde{P}_{d,\eta}\setminus\Int\tilde{P}_{d,\tilde{\eta}}$ are each equal to the closure of their interior. Indeed, note that in common with $[0,1/2]^d$, the set $Q^\triangle$ has the property that if $x\in Q^\triangle$, then $\lambda x\in\Int Q^\triangle$ for all $\lambda\in (0,1)$. Thus, it follows as in the proof of (ix) that $Q^\triangle\setminus\Int\tilde{P}_{d,\eta}$ and $\tilde{P}_{d,\eta}\setminus\Int\tilde{P}_{d,\tilde{\eta}}$ can each be expressed as the union of those $Q^\triangle\cap T$ for which $T$ is a $d$-simplex in the corresponding original triangulation and $\Int(Q^\triangle\cap T)\neq\emptyset$. We conclude that $Q^\triangle\setminus\Int\tilde{P}_{d,\eta}$ and $\tilde{P}_{d,\eta}\setminus\Int\tilde{P}_{d,\tilde{\eta}}$ can each be triangulated into at most $\Gamma_d\,\alpha_d'\log^{d-1}(1/\eta)$ $d$-simplices, so the proof of part (iv) of the lemma is complete.
\end{proof}
By applying a simple transformation to the polytopes $\tilde{P}_\eta$ from Lemma~\ref{Lem:polyapprox}, we obtain suitable approximating polytopes $P_{\eta,j}^\triangle$ for the closed, convex sets $R_j\cap J_\eta^\triangle\subseteq\triangle$ defined in Lemma~\ref{Lem:invelopesimp}. See Figure~\ref{Fig:SimplexEntropy} for an illustration.
\begin{corollary}
\label{Cor:polyapproxsimp}
There exists $\alpha_d>0$, depending only on $d\in\N$, such that for every $\eta\in(0,(d+1)^{-d}]$ and every $1\leq j\leq d+1$, we can construct a polytope $P_{d,\eta,j}^\triangle\equiv P_{\eta,j}^\triangle\subseteq R_j\cap J_\eta^\triangle$ with the following properties:
\begin{enumerate}[label=(\roman*)]
\item $\mu_d(R_j\setminus P_{\eta,j}^\triangle)\leq\alpha_d\,\mu_d(R_j\setminus J_\eta^\triangle)$.
\item If $0<\eta<\tilde{\eta}\leq (d+1)^{-d}$, then $P_{\tilde{\eta},j}^\triangle\subsetneqq P_{\eta,j}^\triangle$, and the regions $R_j\setminus\Int P_{\eta,j}^\triangle$ and $P_{\eta,j}^\triangle\setminus\Int P_{\tilde{\eta},j}^\triangle$ can each be triangulated into at most $\alpha_d\log^{d-1}(1/\eta)$ $d$-simplices.
\end{enumerate}
\end{corollary}
\begin{proof}
It suffices to establish the result for $j=d+1$, since for any other $1\leq j\leq d$, there is an (affine) isometry of $\triangle$ that sends $J_\eta^\triangle$ to itself and maps $R_{d+1}$ to $R_j$. Writing $\Pi\colon\R^{d+1}\to\R^d$ for the projection onto the first $d$ coordinates, recall from Lemma~\ref{Lem:invelopesimp} that $\tilde{\Pi}:=\restr{\Pi}{R_{d+1}}\colon R_{d+1}\to Q^\triangle=\Pi(R_{d+1})$ is a bijection which preserves Lebesgue measure up to a factor of $1/\sqrt{d+1}$. By Lemma~\ref{Lem:invelopesimp}(i) and the properties of the sets $P_\eta,\tilde{P}_\eta$ constructed in Lemma~\ref{Lem:polyapprox}, it follows that $P_{\eta,\,d+1}^\triangle:=\inv{\tilde{\Pi}}(\tilde{P}_\eta)=\inv{\tilde{\Pi}}(Q^\triangle\cap P_\eta)\subseteq\inv{\tilde{\Pi}}(Q^\triangle\cap J_\eta)=R_{d+1}\cap J_\eta^\triangle$ for all $\eta\in (0,(d+1)^{-d}]$. Moreover, $P_{\eta,\,d+1}^\triangle$ is a polytope since $\Pi$ is linear and $\tilde{P}_\eta$ is a polytope, and we deduce from Lemma~\ref{Lem:polyapprox}(iii) that 
\begin{align*}
\mu_d(R_{d+1}\setminus P_{\eta,\,d+1}^\triangle)&=\mu_d\bigl(\tilde{\Pi}(R_{d+1}\setminus P_{\eta,\,d+1}^\triangle)\bigr)/\sqrt{d+1}\\
&=\mu_d(Q^\triangle\setminus\tilde{P}_\eta)/\sqrt{d+1}\leq\alpha_d\,\mu_d(Q^\triangle\setminus J_\eta)/\sqrt{d+1}=\alpha_d\,\mu_d(R_{d+1}\setminus J_\eta^\triangle),
\end{align*}
where $\alpha_d>0$ is taken from Lemma~\ref{Lem:polyapprox} and depends only on $d$. Finally, if $0<\eta<\tilde{\eta}\leq (d+1)^{-d}$, then by Lemma~\ref{Lem:polyapprox}(iv), there exist triangulations of $Q^\triangle\setminus\Int\tilde{P}_\eta$ and $\tilde{P}_{\eta}\setminus\Int\tilde{P}_{\tilde{\eta}}$ into at most $\alpha_d\log^{d-1}(1/\eta)$ $d$-simplices. 
By applying $\inv{\tilde{\Pi}}$ to all the $d$-simplices in these triangulations, we obtain suitable triangulations of $R_{d+1}\setminus\Int P_{\eta,\,{d+1}}^\triangle=\inv{\tilde{\Pi}}(Q^\triangle\setminus\Int\tilde{P}_\eta)$ and $P_{\eta,\,d+1}^\triangle\setminus\Int P_{\tilde{\eta},\,d+1}^\triangle=\inv{\tilde{\Pi}}(\tilde{P}_\eta\setminus\Int\tilde{P}_{\tilde{\eta}})$ into at most $\alpha_d\log^{d-1}(1/\eta)$ $d$-simplices, as required.
\end{proof}

\section{Supplementary material for Sections~\ref{Sec:Smoothness} and~\ref{Subsec:SmoothnessProofs}}
\label{Sec:SeparationSupp}
\subsection{Technical preparation for Section~\ref{Subsec:SmoothnessProofs}}
\label{Subsec:SmoothnessSupp}
In this subsection, we restrict attention to the case $d=3$ and work towards a proof of Proposition~\ref{Prop:SmoothEntropy} in Section~\ref{Subsec:SmoothnessProofs}, the key local bracketing entropy bound that implies the main result (Theorem~\ref{Thm:Smoothness}) in Section~\ref{Sec:Smoothness}. The following two technical lemmas provide pointwise upper and lower bounds on functions $f$ belonging to a $\delta$-Hellinger neighbourhood of some $f_0\in\mathcal{F}^{(\beta,\Lambda)}\cap\mathcal{F}^{0,I}$. For $g\in\mathcal{F}_d$ and $t\geq 0$, let $U_{g,t}:=\{w\in\R^d:g(w)\geq t\}$, which is closed and convex.
\begin{lemma}
\label{Lem:smoothlb}
Let $d=3$, and fix $\beta\geq 1$ and $\Lambda>0$. If $f_0\in\mathcal{F}^{(\beta,\Lambda)}\cap\mathcal{F}^{0,I}$ and $f\in\tilde{\mathcal{F}}(f_0,\delta)$ for some $\delta>0$, then $f(x)\geq f_0(x)/2$ for all $x\in\R^3$ satisfying $f_0(x)\geq(\tilde{c}\Lambda^3\delta^2)^{\beta/(\beta+3)}$, where $\tilde{c}:=6144\inv{\pi}$.
\end{lemma}
\begin{proof}
Fix $x\in\R^3$ with $f_0(x)\geq(\tilde{c}\Lambda^3\delta^2)^{\beta/(\beta+3)}$. It can be assumed without loss of generality that $f(x)<f_0(x)$, and we begin by setting
\[t:=\frac{f_0(x)-f(x)}{2},\quad U:=U_{f,\,f(x)},\quad U_H:=U_{f_0,\,f_0(x)}\quad\text{and}\quad U_L:=U_{f_0,\,f_0(x)-t}.\]
Then it follows from the defining condition~\eqref{Eq:Separation} for $\mathcal{F}^{(\beta,\Lambda)}$ that
\begin{equation}
\label{Eq:levsep}
\min_{y\in\partial U_L}\min_{z\in\partial U_H}\norm{y-z}\geq\frac{t\inv{\Lambda}}{f_0(x)^{1-1/\beta}}=:\tilde{r}.
\end{equation}
Also, note that there exists $u\in\R^3\setminus\{0\}$ such that $f\leq f(x)$ on $H_{x,u}^+:=\{w\in\R^3:\tm{u}w\geq\tm{u}x\}$. Indeed, this is clear if $f(x)=\max_{w\in\R^3} f(w)$, and in the case where $f(x)<\max_{w\in\R^3} f(w)$, we have $x\in\partial U$ by the concavity of $\log f$, so there exists a suitable supporting half-space to $U$ at $x$. 

Next, we fix $x'\in\argmax_{w\in U_H}\tm{u}w$ (which necessarily lies in $\partial U_H$) and apply~\eqref{Eq:levsep} to deduce that $\bar{B}(x',\tilde{r})\cap H_{x',u}^+\subseteq (U_L\setminus\Int U_H)\cap H_{x,u}^+$, where $H_{x',u}^+:=\{w\in\R^3:\tm{u}w\geq\tm{u}x'\}$. Since $f_0\geq f_0(x)-t$ on $(U_L\setminus\Int U_H)$ and $f\leq f(x)=f_0(x)-2t$ on $H_{x,u}^+$, we have 
\begin{align}
\delta^2\geq\int_{(U_L\setminus\Int U_H)\,\cap\,H_{x,u}^+}\bigl(f_0^{1/2}-f^{1/2}\bigr)^2&\geq\bigl(\sqrt{f_0(x)-t}-\sqrt{f_0(x)-2t}\bigr)^2\,\mu_3\bigl(\bar{B}(x',\tilde{r})\cap H_{x',u}^+\bigr)\notag\\
\label{Eq:lbineq}
&=\frac{2\pi}{3}\,\bigl(\sqrt{f_0(x)-t}-\sqrt{f_0(x)-2t}\bigr)^2\rbr{\frac{t\inv{\Lambda}}{f_0(x)^{1-1/\beta}}}^3\!.
\end{align}

\vspace{-0.5cm}
\noindent Now since
\[\sqrt{f_0(x)-t}-\sqrt{f_0(x)-2t}=\frac{t}{\sqrt{f_0(x)-t}+\sqrt{f_0(x)-2t}}\geq\frac{t}{2\sqrt{f_0(x)}},\]
it follows from~\eqref{Eq:lbineq} that
\[\delta^2\geq\frac{\pi\,t^5\Lambda^{-3}}{6f_0(x)^{4-3/\beta}},\]
so $\{f_0(x)-f(x)\}/2=t\leq (6/\pi)^{1/5}\,\delta^{2/5}\,\Lambda^{3/5}\,f_0(x)^{\frac{4}{5}-\frac{3}{5\beta}}$. Rearranging this, we obtain the bound
\begin{equation}
\label{Eq:lbinterm}
f(x)\geq f_0(x)-2(6/\pi)^{1/5}\,\delta^{2/5}\,\Lambda^{3/5}\,f_0(x)^{\frac{4}{5}-\frac{3}{5 \beta}}. 
\end{equation}
The right-hand side of~\eqref{Eq:lbinterm} is bounded below by $f_0(x)/2$ if and only if
\[f_0(x)\geq(6\times 4^5/\pi)^{\beta/(\beta+3)}\,\Lambda^{3\beta/(\beta + 3)}\,\delta^{2\beta/(\beta+3)}=(\tilde{c}\Lambda^3\delta^2)^{\beta/(\beta+3)}.\]
This completes the proof.
\end{proof}
\begin{lemma}
\label{Lem:smoothub}
Let $d=3$, and fix $\beta\geq 1$ and $\Lambda>0$. There exists a universal constant $\tilde{c}'>0$ such that whenever $0<\delta<\inv{e}\wedge (\tilde{c}'\Lambda^{-3/2}\log_+^{-1}\Lambda)$, $f_0\in\mathcal{F}^{(\beta,\Lambda)}\cap\mathcal{F}^{0,I}$ and $f\in\tilde{\mathcal{F}}(f_0,\delta)\cap\tilde{\mathcal{F}}^{1,\eta}$, we have $f(x)\lesssim f_0(x)\vee\{\Lambda^3\delta^2\log^2(\Lambda^{-3}\delta^{-2})\}^{\beta/(\beta+3)}$ for all $x\in\R^3$.
\end{lemma}
\begin{proof}
Since the bound~\eqref{Eq:nisoenv} applies to $f\in\tilde{\mathcal{F}}^{1,\eta}$, there exists a universal constant $\tilde{C}_3>0$ such that $f(x)\leq\Lambda^3\delta^2$ whenever $\norm{x}>\tilde{C}_3\log(\Lambda^{-3}\delta^{-2})$. Moreover, if $f_0(x)\geq 2^{-8d}/6=2^{-24}/6=:c_0$, then it follows from~\eqref{Eq:nisoenv} that \[f(x)\leq\exp(\tilde{b}_3)\leq\inv{c_0}\exp(\tilde{b}_3)f_0(x).\]
Thus, we may restrict attention to $x\in\R^3$ such that $\norm{x}\leq \tilde{C}_3\log(\Lambda^{-3}\delta^{-2})$,  $f_0(x)<c_0$ and $f(x)>f_0(x)$. Fixing such an $x$, we aim to show that there exist universal constants $C''>C'>0$ such that 
\begin{numcases}{}
\label{Eq:ub1}
\,f(x)\leq 3f_0(x)\quad & if $f_0(x)\geq C't$;\\
\label{Eq:ub2}
\,f(x)\leq 2C''t\quad & otherwise,
\end{numcases}
where $t:=\{\Lambda^3\delta^2\log^2(\Lambda^{-3}\delta^{-2})\}^{\beta/(\beta+3)}$. First, since $f_0\in\mathcal{F}^{0,I}$, it follows from~\citet[Theorem~5.14]{LV06} that $\inf_{\norm{w}\leq 1/9}\,f_0(w)\geq 2^{-8d}=6c_0$, and hence that $\norm{x}>1/9$. Also, Lemma~\ref{Lem:smoothlb} implies that $\inf_{\norm{w}\leq 1/9}\,f(w)\geq\inf_{\norm{w}\leq 1/9}\,f_0(w)/2\geq 3c_0$, provided that
\begin{equation}
\label{Eq:cond1}
6c_0\geq(\tilde{c}\Lambda^3\delta^2)^{\beta/(\beta+3)}.
\end{equation} 
In addition, let $s:=\bigl(f_0(x)+\{f(x)\wedge (3c_0)\}\bigr)/2$ and $H_x:=\{w\in\R^d:\tm{w}x=0\}$. Since $f$ is log-concave, we have $f(w)\geq f(x)\wedge (3c_0)=2s-f_0(x)$ for all $w\in\conv\bigl(\{\bar{B}(0,1/9)\cap H_x\}\cup\{x\}\bigr)=:K$. Note that $K$ is a cone of volume $\mu_3(K)=9^{-2}\,\pi\norm{x}/3$.

Since $f_0\in\mathcal{F}^{(\beta,\Lambda)}$ and $f_0(x)<s$, we have $f_0(w)<s$ for all $w\in\bar{B}(x,r')$, where $r':=\inv{\Lambda}\{s-f_0(x)\}/s^{1-1/\beta}$. Therefore, $f(w)\geq 2s-f_0(x)>s>f_0(w)$ for all $w\in K\cap\bar{B}(x,r')=:K'$. Noting that $K'-x\supseteq\bigl(r'/\sqrt{\norm{x}^2+9^{-2}}\bigr)(K-x)$ and recalling that $\norm{x}>1/9$, we deduce that \[\mu_3(K')\geq\rbr{\frac{r'}{\sqrt{\norm{x}^2+9^{-2}}}}^3\mu_3(K)\gtrsim\rbr{\frac{s-f_0(x)}{\Lambda s^{1-1/\beta}}}^3\norm{x}^{-2}.\]
Now since $\norm{x}\leq\tilde{C}_3\log(\Lambda^{-3}\delta^{-2})$, it follows that 
\[\delta^2\geq\int_{K'}\,\bigl(f^{1/2}-f_0^{1/2}\bigr)^2\gtrsim\bigl(\sqrt{2s-f_0(x)}-\sqrt{s}\bigr)^2\rbr{\frac{s-f_0(x)}{\Lambda s^{1-1/\beta}}}^3\log^{-2}(\Lambda^{-3}\delta^{-2}),\]
which implies that \[t^{(\beta+3)/\beta}=\Lambda^3\,\delta^2\log^2(\Lambda^{-3}\delta^{-2})\gtrsim\frac{\{s-f_0(x)\}^5}{s^{3\,(1-1/\beta)}\bigl(\sqrt{2s-f_0(x)}+\sqrt{s}\bigr)^2}\gtrsim\frac{\{s-f_0(x)\}^5}{s^{4-3/\beta}}.\]
This shows that
\begin{equation}
\label{Eq:ubineq}
s^\alpha-f_0(x)s^{\alpha-1}\leq Ct^\alpha,
\end{equation}
where $\alpha:=(\beta+3)/(5\beta)\in (0,1)$ and $C>0$ is a suitable universal constant. For our fixed $x\in\R^3$, we see that the function $g\colon [f_0(x),\infty)\to [0,\infty)$ defined by $g(w):=w^\alpha-f_0(x)w^{\alpha-1}$ is strictly increasing. Observe also that there exists a universal constant $C'>0$ such that if $f_0(x)\geq C't$, then
$g(2f_0(x))\geq Ct^{\alpha}\geq g(s)$. In this case, it follows from~\eqref{Eq:ubineq} that \[\bigl(f_0(x)+\{f(x)\wedge (3c_0)\}\bigr)/2=s\leq 2f_0(x)\] and hence that $f(x)\wedge (3c_0)\leq 3f_0(x)$. But since it was assumed that $f_0(x)<c_0$, we deduce that $f(x)\leq 3f_0(x)$, which yields~\eqref{Eq:ub1}.

Otherwise, if $f_0(x)<C't$, then $g(w)\geq w^\alpha-C'tw^{\alpha-1}$ for $w\geq f_0(x)$, so there exists a universal constant $C''>C'$ such that $g(C''t)\geq Ct^{\alpha}\geq g(s)$. As above, we deduce from~\eqref{Eq:ubineq} that $s\leq C''t$, and so long as 
\begin{equation}
\label{Eq:cond2}
3c_0>2C''t,
\end{equation}
this yields the desired bound~\eqref{Eq:ub2}. 
It is easy to verify that there exists a universal constant $\tilde{c}'>0$ such that~\eqref{Eq:cond1} and~\eqref{Eq:cond2} are satisfied whenever $0<\delta<\inv{e}\wedge (\tilde{c}'\Lambda^{-3/2}\log_+^{-1}\Lambda)$. This completes the proof.
\end{proof}
In the proof of Proposition~\ref{Prop:SmoothEntropy}, we also apply the following geometric result, which is due to~\citet{BI75}.
\begin{lemma}
\label{Lem:BI75}
For each $d\in\N$, there exist $\bar{\eta}\equiv\bar{\eta}_d>0$ and $\bar{C}^*\equiv\bar{C}_d^*>0$, depending only on $d$, such that for each $\eta\in (0,\bar{\eta})$ and each compact, convex $D\subseteq\bar{B}(0,1)\subseteq\R^d$, we can find a polytope $P$ with at most $\bar{C}^*\eta^{-(d-1)/2}$ vertices with the property that $D\subseteq P\subseteq D+\bar{B}(0,\eta)$. 
\end{lemma}

\subsection{H\"older classes}
\label{Subsec:Holder}
First, we extend the notions of H\"older regularity discussed in Section~\ref{Sec:Smoothness} to general exponents $\beta>1$. It will be helpful to introduce the following additional notation. For finite-dimensional vector spaces $V,W$ over $\R$, let $\mathcal{L}(V,W) \equiv \mathcal{L}^{(1)}(V,W)$ denote the $\{\dim(V) \times \dim(W)\}$-dimensional vector space of all linear maps from $V$ to $W$. For positive integers $k\geq 2$, we inductively define $\mathcal{L}^{(k)}(V,W):=\mathcal{L}(V,\mathcal{L}^{(k-1)}(V,W))$. This may be identified with the space of all $k$-linear maps from $V^k=V\times\dotsm\times V$ to $W$,
or equivalently the space $\mathcal{L}(V^{\otimes k},W)$, where $V^{\otimes k}=V\otimes\,\dotsm\,\otimes V$
denotes the $k$-fold tensor product of $V$ with itself. 
For $k\in\N$ and a linear map $T\colon V\to W$, we write $T^{\otimes k}\colon V^{\otimes k}\to W^{\otimes k}$ for the $k$-fold tensor product of $T$ with itself, which sends $u_1\otimes\dotsm\otimes u_k\in V^{\otimes k}$ to $Tu_1\otimes\dotsm\otimes Tu_k\in W^{\otimes k}$.
When $V=\R^d$ for some $d\in\N$ and $W=\R$, we define the Hilbert--Schmidt norm on $\mathcal{L}\bigl((\R^d)^{\otimes k},\R\bigr)$ for $k\in\N$ by
\vspace{-0.2cm}
\begin{equation}
\label{Eq:TensorNorm}
\norm{\alpha}_{\mathrm{HS}}:=\tr(\alpha\alpha^*)^{1/2}=\tr(\alpha^*\alpha)^{1/2}=\cbr{\sum_{i_1,\dotsc,i_k=1}^d\,\alpha(e_{i_1}\otimes\dotsm\otimes e_{i_k})^2}^{1/2},
\end{equation}
where $\alpha^*\in\mathcal{L}\bigl(\R,(\R^d)^{\otimes k}\bigr)$ denotes the adjoint of $\alpha$ (as a linear map between inner product spaces). In an abuse of notation, we also write $\norm{{\cdot}}_{\mathrm{HS}}$ for the norm this induces on $\mathcal{L}^{(k)}(\R^d,\R)$. This is the natural analogue of the Frobenius norm for general multilinear forms; indeed, when $k=2$, the expression in~\eqref{Eq:TensorNorm} coincides with the Frobenius norm of the matrix that represents the symmetric bilinear form corresponding to $\alpha$ with respect to the standard basis of $\R^d$.

The reason for making these definitions is that if $f\colon V\to W$ is a map between finite-dimensional normed spaces and $f$ is differentiable at $x\in V$, then the derivative of $f$ at $x$, written $Df(x)$, is an element of $\mathcal{L}(V,W)$. In particular, if $f$ is itself linear, then $Df(x)=f$ for all $x\in V$. More generally, if $k\in\N$ and $f\colon V\to W$ is $(k-1)$-times differentiable in a neighbourhood of $x\in V$ and $k$ times differentiable at $x$, then the $k$th derivative of $f$ at $x$, written $D^kf(x)$, is an element of $\mathcal{L}^{(k)}(V,W)$. It is conventional to regard $D^kf(x)$ as a $k$-linear map $(u_1,\dotsc,u_k)\mapsto D^kf(x)(u_1,\dotsc,u_k)$. By repeatedly applying the chain rule, we can establish the following:
\begin{lemma}
\label{Lem:chainrule}
If $f\colon\R^d\to\R$ is $k$ times differentiable for some $k\in\N$ and $T\colon\R^m\to\R^d$ is linear, then $D^k(f\circ T)(x)(u_1,\dotsc,u_k)=D^kf(Tx)(Tu_1,\dotsc,Tu_k)$ for all $x,u_1,\dotsc,u_k\in\R^m$. Equivalently, if we view $D^k(f\circ T)(x)$ and $D^kf(x)$ as elements of $\mathcal{L}\bigl((\R^m)^{\otimes k},\R\bigr)$ and $\mathcal{L}\bigl((\R^d)^{\otimes k},\R\bigr)$ respectively, then $D^k(f\circ T)(x)=D^kf(x)\circ T^{\otimes k}$.
\end{lemma}
\begin{proof}
For $k\in\N$, define the dual map $(T^{\otimes k})^*\colon\mathcal{L}\bigl((\R^d)^{\otimes k},\R\bigr)\to\mathcal{L}\bigl((\R^m)^{\otimes k},\R\bigr)$ to $T^{\otimes k}$ by $(T^{\otimes k})^*(\alpha):=\alpha\circ T^{\otimes k}$, and write $T_{(k)}$ for the corresponding linear map from $\mathcal{L}^{(k)}(\R^d,\R)$ to $\mathcal{L}^{(k)}(\R^m,\R)$. Viewing $D^k(f\circ T)(x)$ and $D^kf(x)$ as elements of $\mathcal{L}^{(k)}(\R^m,\R)$ and $\mathcal{L}^{(k)}(\R^d,\R)$ respectively, we will establish by induction on $k$ that $D^k(f\circ T)(x)=T_{(k)}\circ D^kf(Tx)$ for all $k\in\N$, which implies the desired conclusion. The base case $k=1$ follows by applying the chain rule directly to the function $f\circ T$. For a general $k\geq 2$, the inductive hypothesis asserts that $D^{k-1}(f\circ T)(x)=(T_{(k-1)}\circ g_{k-1})(x)$, where $g_{k-1}(x):=D^{k-1}f(Tx)$, and we deduce by a further application of the chain rule that $Dg_{k-1}(x)=D^k(Tx)\circ T$. Recalling that $T_{(k-1)}$ is linear, we apply the chain rule once again to conclude that
\begin{align*}
D^k(f\circ T)(x)=D(T_{(k-1)}\circ g_{k-1})(x)&=T_{(k-1)}\circ Dg_{k-1}(x)\\
&=T_{(k-1)}\circ D^kf(Tx)\circ T=T_{(k)}\circ D^kf(Tx),
\end{align*}
as required.
\end{proof}

Using the above notation, we are now ready to formulate definitions of $\beta$-H\"older regularity for functions defined on $\R^d$ for general $\beta>1$ and $d\in\N$. For $\beta>1$ and $L>0$, we say that $h\colon\R^d\to\R$ is \emph{$(\beta,L)$-H\"older} on $\R^d$ if $h$ is $k:=\ceil{\beta}-1$ times differentiable on $\R^d$ and 
\begin{equation}
\label{Eq:holdergen}
\norm{D^kh(y)-D^kh(x)}_{\mathrm{HS}}\leq L\norm{y-x}^{\beta-k}
\end{equation}
for all $x,y\in\R^d$. 
We say that $h$ is \emph{$\beta$-H\"older} if there exists $L>0$ such that $h$ is $(\beta,L)$-H\"older. 

In addition, we seek to extend the definition~\eqref{Eq:HolderOriginal} of the affine invariant classes $\mathcal{H}^{\beta,L}$ in Example~\ref{Ex:Holder2} in Section~\ref{Sec:Smoothness} to encompass general $\beta>1$, rather than confine ourselves to working with $\beta\in (1,2]$. To this end, we define a scaled version of~\eqref{Eq:TensorNorm} for each $S\in\mathbb{S}^{d\times d}$ by
\begin{equation}
\label{Eq:TensorNormS}
\norm{\alpha}_S':=\Norm{\alpha\circ (S^{-1/2})^{\otimes k}}_{\mathrm{HS}}\,\det^{-1/2}\!S=\cbr{\sum_{i_1,\dotsc,i_k=1}^d\,\alpha\bigl(S^{-1/2}e_{i_1}\otimes\dotsm\otimes S^{-1/2}e_{i_k}\bigr)^2\,\det^{-1}\!S}^{1/2};
\end{equation}
note that since $(S^{-1/2})^{\otimes k}$ is invertible with inverse $(S^{1/2})^{\otimes k}$, $\norm{{\cdot}}_S'$ does indeed constitute a norm on $\mathcal{L}\bigl((\R^d)^{\otimes k},\R\bigr)$. In a further abuse of notation, we also write $\norm{{\cdot}}_S'$ for the corresponding norm on $\mathcal{L}^{(k)}(\R^d,\R)$.
Now for general $\beta>1$ and $L>0$, let $\mathcal{H}^{\beta,L}$ denote the collection of all $f\in\mathcal{F}_d$ for which $f$ is $k:=\ceil{\beta}-1$ times differentiable on $\R^d$ and 
\begin{equation}
\label{Eq:holderaffinv}
\norm{D^kf(y)-D^kf(x)}_{\inv{\Sigma}_f}'\leq L\norm{y-x}_{\Sigma_f}^{\beta-k}
\end{equation}
for all $x,y\in\R^d$, where the norm on the left-hand side is given by~\eqref{Eq:TensorNormS}. Note that when $\beta\in (1,2]$ or $\beta\in (2,3]$, this specialises to the definitions~\eqref{Eq:HolderOriginal} and~\eqref{Eq:HolderFrob} of the classes $\mathcal{H}^{\beta,L}$ in Example~\ref{Ex:Holder2}.
We now verify that:
\begin{lemma}
For general $\beta>1$ and $L>0$, $\mathcal{H}^{\beta,L}$ is an affine invariant class of densities.
\label{Lem:holderaffinv}
\end{lemma}
\begin{proof}
Fix $f\in\mathcal{H}^{\beta,L}$ and define $g\in\mathcal{F}_d$ by $g(x):=\inv{\abs{\det A}}f(\inv{A}(x-b))$, where $b\in\R^d$ and $A\in\R^{d\times d}$ is invertible. Since $\Sigma_g=A\Sigma_f\tm{A}$, every $\alpha\in\mathcal{L}\bigl((\R^d)^{\otimes k},\R\bigr)$ satisfies
\begin{align}
\Norm{\alpha\circ(\inv{A}\Sigma_g^{1/2})^{\otimes k}}_{\mathrm{HS}}&=\tr\Bigl(\alpha\circ\bigl\{\inv{A}\Sigma_g\tm{(\inv{A})}\bigr\}^{\otimes k}\circ\alpha^*\Bigr)^{1/2}\notag\\
\label{Eq:affinvident}
&=\tr\bigl(\alpha\circ\Sigma_f^{\otimes k}\circ\alpha^*\bigr)^{1/2}=\bigl\|\alpha\circ(\Sigma_f^{1/2})^{\otimes k}\bigr\|_{\mathrm{HS}}.
\end{align}
Setting $k:=\ceil{\beta}-1$ and viewing $D^kf(x)$ and $D^kg(x)$ as elements of $\mathcal{L}\bigl((\R^d)^{\otimes k},\R\bigr)$, we deduce from Lemma~\ref{Lem:chainrule} and~\eqref{Eq:affinvident} that
\begin{align*}
&\norm{D^kg(y)-D^kg(x)}_{\inv{\Sigma}_g}'\\
&\hspace{1cm}=\Norm{\{D^kf(\inv{A}(y-b))-D^kf(\inv{A}(x-b))\}\circ(\inv{A})^{\otimes k}}_{\inv{\Sigma}_g}\\
&\hspace{1cm}=\Norm{\{D^kf(\inv{A}(y-b))-D^kf(\inv{A}(x-b))\}\circ(\inv{A}\Sigma_g^{1/2})^{\otimes k}}_{\mathrm{HS}}\det^{1/2}\Sigma_f\\[2pt]
&\hspace{1cm}=\bigl\|\{D^kf(\inv{A}(y-b))-D^kf(\inv{A}(x-b))\}\circ(\Sigma_f^{1/2})^{\otimes k}\bigr\|_{\mathrm{HS}}\det^{1/2}\Sigma_f\\
&\hspace{1cm}=\bigl\|D^kf(\inv{A}(y-b))-D^kf(\inv{A}(x-b))\bigr\|_{\inv{\Sigma_f}}\\
&\hspace{1cm}\leq L\norm{\inv{A}(y-x)}_{\Sigma_f}^{\beta-k}=L\norm{y-x}_{\Sigma_g}^{\beta-k}
\end{align*}
for all $x,y\in\R^d$, as required.
\end{proof}
Next, we establish that at least when $\beta\in (1,2]$, the classes $\mathcal{H}^{\beta,L}$ are nested with respect to the H\"older exponent $\beta$ in the sense of part (iii) of the following result.
\begin{proposition}
\label{Prop:Nested}
For each $d\in\N$, we have the following:
\begin{enumerate}[label=(\roman*)]
\item If $f\colon\R^d\to\R$ is a differentiable density and $\nabla f\colon\R^d\to\R^d$ is uniformly continuous, then $f$ and $\nabla f$ are in fact bounded.
\item For $\beta\in (1,2]$ and $L>0$, there exists $C\equiv C(d,\beta,L)>0$ such that whenever $f\colon\R^d\to\R$ is a differentiable density with the property that $\norm{\nabla f(x)-\nabla f(y)}\leq L\norm{x-y}^{\beta-1}$ for all $x,y\in\R^d$, we have $f(x)\leq C$ and $\norm{\nabla f(x)}\leq C$ for all $x\in\R^d$.
\item For $\beta\in (1,2]$ and $L>0$, there exists $\tilde{L}\equiv\tilde{L}(d,\beta,L)>0$ such that $\mathcal{H}^{\beta,L}\subseteq\mathcal{H}^{\alpha,\tilde{L}}$ for all $\alpha\in (1,\beta]$.
\end{enumerate}
\end{proposition}
\begin{proof}
For (i), it follows from the uniform continuity of $\nabla f$ that there exists $\delta>0$ such that $\norm{\nabla f(x)-\nabla f(y)}\leq 1$ whenever $\norm{x-y}\leq\delta$. First, we show that $f$ is bounded. Suppose to the contrary that for every $n\in\N$, there exists $x_n\in\R^d$ such that $f(x_n)\geq n$. Fix any $n\in\N$ and let $B_n:=\bar{B}\bigl(x_n,(n V_d/2)^{-1/d}\bigr)$, where $V_d:=\mu_d(\bar{B}(0,1))$. Then $(2/n)\inf_{x\in B_n}f(x)=\mu_d(B_n)\inf_{x\in B_n}f(x)\leq\int_{B_n}f\leq 1$ since $f$ is a density, so we can find $w_n\in B_n$ such that $f(w_n)\leq n/2$. By applying the chain rule and mean value theorem to the function that maps $t\in [0,1]$ to $f(w_n+t(x_n-w_n))$, we deduce that there exists $z_n\in [w_n,x_n]\subseteq B_n$ such that $\tm{\nabla f(z_n)}(x_n-w_n)=f(x_n)-f(w_n)\geq n/2$, whence
\begin{equation}
\label{Eq:largeDeriv}
\norm{\nabla f(z_n)}\geq\frac{n/2}{\norm{x_n-w_n}}\geq\rbr{\frac{n}{2}}^{1+1/d}V_d^{1/d}.
\end{equation}
Now let $S_n:=\{w\in\R^d:\tm{w}x\geq 0\text{ for all }x\in\bar{B}(\nabla f(z_n),1)\}$ and $K_n:=x_n+(\bar{B}(0,\delta/2)\cap S_n)$. We claim that if $n\geq 2\inv{V_d}(2/\delta)^d=:N_1$, then $f(x)\geq n$ for all $x\in K_n$. Indeed, for each fixed $x\in K_n$, there exists $y\in [x_n,x]\subseteq B(x_n,\delta/2)$ such that $f(x)-f(x_n)=\tm{\nabla f(y)}(x-x_n)$, as can be seen by applying the chain rule and mean value theorem to the function that maps $t\in [0,1]$ to $f(x_n+t(x-x_n))$. If $n\geq N_1$, then $(n V_d/2)^{-1/d}\leq\delta/2$, so $z_n\in B_n\subseteq\bar{B}(x_n,\delta/2)$ and $y\in\bar{B}(x_n,\delta/2)\subseteq\bar{B}(z_n,\delta)$. Recalling the first line of the proof, we deduce that $\nabla f(y)\in\bar{B}(\nabla f(z_n),1)$. Since $x-x_n\in S_n$, it follows from the definition of $S_n$ that $f(x)=f(x_n)+\tm{\nabla f(y)}(x-x_n)\geq f(x_n)\geq n$, as required.

Also, if $1<(n/2)^{1+1/d}\,V_d^{1/d}$, then we may define $\theta:=\arcsin\bigl((2/n)^{1+1/d}\,V_d^{-1}\bigr)\in [0,\pi/2]$ and deduce from~\eqref{Eq:largeDeriv} that $\mu_d(K_n)\geq\mu_d\bigl(\bar{B}(x_n,\delta/2)\bigr)(\pi-2\theta)/(2\pi)$. Thus, for any fixed $\gamma\in (0,1/2)$, there exists $N_2>0$ (depending only on $d$ and $\gamma$) such that $\mu_d(K_n)\geq\gamma\mu_d(\bar{B}(x_n,\delta/2))=\gamma(\delta/2)^d\,V_d$ if $n\geq N_2$. Together with the previous claim, this implies that $\int_{K_n}f\geq n\gamma(\delta/2)^d\,V_d=2\gamma n/N_1>1$ whenever $n>\{\inv{(2\gamma)}N_1\}\vee N_2$, which contradicts the fact that $f$ is a density. This shows that $f$ is bounded.

It remains to show that $\nabla f$ is bounded. Assume for a contradiction that for each $n\in\N$, there exists $y_n\in\R^d$ such that $\norm{\nabla f(y_n)}\geq n$. For each fixed $n\geq 2$, let $y_n':=y_n+\delta\,\nabla f(y_n)/\norm{\nabla f(y_n)}$. Then by another application of the chain rule and mean value theorem, there exists $y_n''\in [y_n,y_n']$ such that 
\begin{equation}
\label{Eq:largef}
f(y_n')-f(y_n)=\tm{\nabla f(y_n'')}(y_n'-y_n)=\delta\,\norm{\nabla f(y_n'')}\,\frac{\tm{\nabla f(y_n'')}\nabla f(y_n)}{\norm{\nabla f(y_n'')}\,\norm{\nabla f(y_n)}}.
\end{equation}
Now $y_n''\in\bar{B}(y_n,\delta)$, so $\nabla f(y_n'')\in\bar{B}(\nabla f(y_n),1)$ by the first line of the proof. Also, if $u,v\in\R^d$ are such that $\norm{v}\geq 2$ and $u\in\bar{B}(v,1)$, then $\tm{u}v\geq\sqrt{3}\,\norm{u}\norm{v}/2$. Thus, returning to~\eqref{Eq:largef}, we deduce that $f(y_n')-f(y_n)\geq\sqrt{3}\,\delta\norm{\nabla f(y_n'')}/2\geq\sqrt{3}\,\delta (n-1)/2$ for all $n\geq 2$. But this contradicts the fact that $f$ is bounded, so $\nabla f$ is indeed bounded. This establishes (i).

Under the additional hypothesis that $\norm{\nabla f(x)-\nabla f(y)}\leq L\norm{x-y}^{\beta-1}$ for all $x,y\in\R^d$, note that in the proof above, we can choose suitable values for $\delta,N_1,N_2$ that depend only on $d,\beta,L$ (and not on the specific $f$ under consideration). By carefully tracking through the rest of the argument above, we obtain bounds on $f$ and $\nabla f$ that also depend only on $d,\beta,L$, so (ii) holds.

Finally, to deduce (iii), fix $\beta\in (1,2]$ and $L>0$ and consider any $f\in\mathcal{H}^{\beta,L}\cap\mathcal{F}^{0,I}$. The defining condition \eqref{Eq:HolderOriginal} implies that $\norm{\nabla f(x)-\nabla f(y)}\leq L\norm{x-y}^{\beta-1}$ for all $x,y\in\R^d$, and part (ii) yields $C\equiv C(d,\beta,L)>0$ such that $f(x)\leq C$ and $\norm{\nabla f(x)}\leq C$ for all $x\in\R^d$. Let $\tilde{L}:=L\vee (2C)$ and note that if $\alpha\in (1,\beta]$, then \[\norm{\nabla f(x)-\nabla f(y)}\leq L\norm{x-y}^{\beta-1}\leq L\norm{x-y}^{\alpha-1}\leq\tilde{L}\norm{x-y}^{\alpha-1}\]
whenever $\norm{x-y}\leq 1$. On the other hand, if $\norm{x-y}>1$, then \[\norm{\nabla f(x)-\nabla f(y)}\leq 2C<\tilde{L}\norm{x-y}^{\alpha-1}.\]
These bounds apply to every $f\in\mathcal{H}^{\beta,L}\cap\mathcal{F}^{0,I}$, so by the affine invariance of $\mathcal{H}^{\beta,L}$ (Lemma~\ref{Lem:holderaffinv}), it follows that $\mathcal{H}^{\beta,L}\subseteq\mathcal{H}^{\alpha,\tilde{L}}$ for all $\alpha\in (1,\beta]$, as required.
\end{proof}
\begin{remark*}
For $\beta\in (1,2]$, this argument shows that classes of $\beta$-H\"older \emph{densities} defined on the whole of $\R^d$ are nested with respect to $\beta$. However, the same is not true of classes of general $\beta$-H\"older functions on $\R^d$; see the discussion at the end of Example~\ref{Ex:Holder1} in Section~\ref{Sec:Smoothness}. 
\end{remark*}
The following result shows that the classes $\tilde{\mathcal{H}}^{\gamma,L}$ and $\mathcal{H}^{\beta,L}$ defined in Examples~\ref{Ex:Holder1} and~\ref{Ex:Holder2} respectively (in Section~\ref{Sec:Smoothness}) are contained within the more general classes $\mathcal{F}^{(\beta',\Lambda')}$ for suitably chosen values of $\beta'$ and $\Lambda'$ in each case.
\begin{proposition}
\label{Prop:Holder}
For $L>0$, we have the following:
\begin{enumerate}[label=(\roman*)]
\item If $\beta\in (1,2]$, then every $f\in\mathcal{H}^{\beta,L}$ satisfies~\eqref{Eq:GradBd} and therefore~\eqref{Eq:Separation} with this value of $\beta$ and $\Lambda=\Lambda(\beta,L):=L^{1/\beta}(1-1/\beta)^{-1+1/\beta}$. Consequently, we have $\mathcal{H}^{\beta,L}\subseteq\mathcal{F}^{(\beta,\,\Lambda(\beta,L))}$.
\item For $\beta\in (1,2]$, suppose that $g\colon\R^d\to [0,\infty)$ satisfies $\norm{\nabla g(y)-\nabla g(x)}\leq L\norm{y-x}^{\beta-1}$ for all $x,y\in\R^d$. Then we have $\norm{\nabla g(x)}\leq\Lambda(\beta,L)\,g(x)^{1-1/\beta}$ for all $x\in\R^d$ and $\norm{x-y}\geq\inv{\Lambda}(\beta,L)\{g(x)-g(y)\}/g(x)^{1-1/\beta}$ whenever $x,y\in\R^d$ satisfy $g(x)>g(y)$.
\item If $\beta\in (2,3]$, then there exists $\Lambda\equiv\Lambda(\beta,L)>0$ such that every $f\in\mathcal{H}^{\beta,L}$ satisfies~\eqref{Eq:GradBd} and therefore~\eqref{Eq:Separation} with this value of $\beta$ and $\Lambda=\Lambda(\beta,L)$. Consequently, we have $\mathcal{H}^{\beta,L}\subseteq\mathcal{F}^{(\beta,\,\Lambda(\beta,L))}$.
\item If $\beta'\geq 1$ and $\gamma\in (1,2]$, then there exists $\bar{B}\equiv\bar{B}_d>0$ depending only on $d$ such that every $f\in\tilde{\mathcal{H}}^{\gamma,L}$  satisfies~\eqref{Eq:GradBd} and therefore~\eqref{Eq:Separation} with $\beta=\beta'$, $\Lambda=\beta\Lambda(\gamma,L)\bar{B}$ and any $\tau\leq\inv{e}$. Consequently, for any $\beta\geq 1$, we have $\bigcup_{\gamma \in [1,2]}\tilde{\mathcal{H}}^{\gamma,L}\subseteq\mathcal{F}^{(\beta,\,\Lambda'(\beta,L))}$, where $\Lambda'(\beta,L):=\beta(L\vee L^{1/2})\bar{B}e^{1/e}(eB_d)^{1/\beta}$ and $B_d>0$ is taken from Proposition~\ref{Prop:SepProperties}(ii).
\end{enumerate}
\end{proposition}
\begin{proof}
As in the proof of Proposition~\ref{Prop:Differentiable}, we write $\Sigma\equiv\Sigma_f$ for convenience and let $\ipr{v}{w}':=(\det\Sigma)\,(\tm{v}\Sigma w)$ denote the inner product that gives rise to the norm $\norm{{\cdot}}_{\inv{\Sigma}}'$ on $\R^d$. To verify that~\eqref{Eq:GradBd} holds with the stated values of $\beta,\Lambda$, first fix $x\in\R^d$. The bound~\eqref{Eq:GradBd} holds trivially if $\nabla f(x)=0$, so we may assume that $\nabla f(x)\neq 0$. Let $u:=-\Sigma\,\nabla f(x)$, and for $t\in\R$, let $h(t):=f(x+tu)$. By the chain rule, we have $h'(t)=\tm{\nabla f(x+tu)}u$ for all $t\in\R$, so upon applying the Cauchy--Schwarz inequality and the defining condition~\eqref{Eq:HolderOriginal} for the class $\mathcal{H}^{\beta,L}$, it follows that 
\begin{align*}
\abs{h'(t)-h'(s)}&=\abs{\tm{\{\nabla f(x+tu)-\nabla f(x+su)\}}u}\\
&=\inv{(\det\Sigma)}\,\ipr{\nabla f(x+tu)-\nabla f(x+su)}{\inv{\Sigma}u}'\\
&\leq\inv{(\det\Sigma)}\,\norm{\nabla f(x+tu)-\nabla f(x+su)}_{\inv{\Sigma}}'\,\norm{\inv{\Sigma}u}_{\inv{\Sigma}}'\\
&\leq (\det\Sigma)^{-1}\,L\norm{(t-s)u}_\Sigma^{\beta-1}\,\norm{\nabla f(x)}_{\inv{\Sigma}}'\\
&=(\det\Sigma)^{-(\beta+1)/2}\,L\abs{t-s}^{\beta-1}\,(\norm{\nabla f(x)}_{\inv{\Sigma}}')^\beta
\end{align*}
for all $t,s\in\R$. Hence, writing $\lambda:=(\det\Sigma)^{-1/2}\,\norm{\nabla f(x)}_{\inv{\Sigma}}'$, we deduce that \[h'(t)\leq h'(0)+Lt^{\beta-1}\lambda^\beta\,(\det\Sigma)^{-1/2}=-\lambda^2+Lt^{\beta-1}\lambda^\beta\,(\det\Sigma)^{-1/2}\]
for all $t\geq 0$. Since $h$ takes non-negative values, we can apply the fundamental theorem of calculus to deduce that 
\begin{align*}
-f(x)=-h(0)\leq h(t)-h(0)&=\int_0^t h'(s)\,ds\leq\int_0^t\,\bigl\{-\lambda^2+Ls^{\beta-1}\lambda^\beta\,(\det\Sigma)^{-1/2}\bigr\}\,ds\\
&=(\det\Sigma)^{-1/2}\,\Bigl\{-\norm{\nabla f(x)}_{\inv{\Sigma}}'\,(\lambda t)+\frac{L}{\beta}\,(\lambda t)^\beta\Bigr\}=:G(\lambda t)
\end{align*}
for all $t\geq 0$. We now optimise this bound by setting $t=a^*/\lambda$, where $a^*:=\argmin_{t\geq 0} G(t)=(\norm{\nabla f(x)}_{\inv{\Sigma}}'/L)^{1/(\beta-1)}$, which yields
\[f(x)\geq -G(a^*)=(\det\Sigma)^{-1/2}\,(1-1/\beta)\,L^{-1/(\beta-1)}\,(\norm{\nabla f(x)}_{\inv{\Sigma}}')^{\beta/(\beta-1)}.\]
This establishes the desired bound~\eqref{Eq:GradBd} on $\norm{\nabla f(x)}_{\inv{\Sigma}}'$, and in view of Proposition~\ref{Prop:Differentiable}, the final assertion of (i) now follows immediately. Observe in particular that, in addition to the defining condition~\eqref{Eq:HolderOriginal}, the key property of $f\in\mathcal{H}^{\beta,L}$ that we exploit in the argument above is the non-negativity of $f$. Since we do not appeal to the log-concavity of $f$ or even the fact that $f$ is a density, we may therefore run through the same proof with $\Sigma$ replaced by $I$ throughout in order to establish (ii) for general non-negative functions $g\colon\R^d\to [0,\infty)$ that satisfy $\norm{\nabla g(y)-\nabla g(x)}\leq L\norm{y-x}^{\beta-1}$ for all $x,y\in\R^d$.

For (iii), since $\mathcal{H}^{\beta,L}$ is affine invariant (cf.\ Example~\ref{Ex:Holder2} and Lemma~\ref{Lem:holderaffinv}), it suffices to consider $f\in\mathcal{H}^{\beta,L}\cap\mathcal{F}^{0,I}$. To verify that~\eqref{Eq:GradBd} holds with the stated values of $\beta,\Lambda$, fix $x\in\R^d$ with $\nabla f(x)\neq 0$, and let $u:=-\nabla f(x)/\norm{\nabla f(x)}$. Note that by the log-concavity of the density $f$, the function $h\colon [0,\infty)\to [0,\infty)$ defined by $h(t):=f(x+tu)$ is strictly decreasing and non-negative. Using the chain rule, we find that $h'(t)=\tm{\nabla f(x+tu)}u$ and $h''(t)=\tm{u}Hf(x+tu)u$ for all $t$, and it follows from the defining condition~\eqref{Eq:HolderFrob} for the class $\mathcal{H}^{\beta,L}$ that
\begin{align*}
\abs{h''(t)-h''(s)}&=\abs{\tm{u}\{Hf(x+tu)-Hf(x+su)\}u}\leq\norm{Hf(x+tu)-Hf(x+su)}\\
&\leq\norm{Hf(x+tu)-Hf(x+su)}_{\mathrm{F}}\leq L\abs{t-s}^{\beta-2}
\end{align*}
for all $t,s\geq 0$. Therefore, recalling that $h$ is decreasing, we now apply the fundamental theorem of calculus to deduce that \[-h'(0)\geq h'(t)-h'(0)=\int_0^t h''(s)\,ds\geq\int_0^t\,\bigl(h''(0)-Ls^{\beta-2}\bigr)\,ds=h''(0)t-\frac{L}{\beta-1}t^{\beta-1},\]
for all $t\geq 0$. Setting $\lambda:=\norm{\nabla f(x)}=-h'(0)$ and $\tilde{\Lambda}:=\Lambda(\beta-1,L)$, as defined in (i), we now optimise this bound with respect to $t$ as in the proof of (i) and conclude that $h''(0)\leq\tilde{\Lambda}\lambda^{(\beta-2)/(\beta-1)}$. On the other hand, we have \[h'(t)-h'(0)=\int_0^t h''(s)\,ds\leq\int_0^t\,\bigl(h''(0)+Ls^{\beta-2}\bigr)\,ds=h''(0)t+\frac{L}{\beta-1}t^{\beta-1},\]
for all $t\geq 0$, and a further application of the fundamental theorem of calculus yields
\begin{align*}
h(t)-h(0)\leq\int_0^t h'(s)\,ds&\leq\int_0^t\,\rbr{h'(0)+h''(0)s+\frac{L}{\beta-1}s^{\beta-1}}\,ds\\
&\leq -\lambda t+\frac{\tilde{\Lambda}}{2}\lambda^{(\beta-2)/(\beta-1)}\,t^2+\frac{L}{\beta(\beta-1)}t^\beta
\end{align*}
for all $t\geq 0$. Replacing $t$ by $\lambda^{1/(\beta-1)}t$ and using the fact that $h\geq 0$, we see that \[0\leq h(0)-\rbr{t-\frac{\tilde{\Lambda}}{2}t^2-\frac{L}{\beta(\beta-1)}t^\beta}\lambda^{\beta/(\beta-1)}=:h(0)-\alpha(t)\lambda^{\beta/(\beta-1)}\]
for all $t\geq 0$. Letting $\tilde{\alpha}:=\max_{t\geq 0}\alpha(t)>0$, it follows that $\norm{\nabla f(x)}=\lambda\leq\{\tilde{\alpha}^{-1}h(0)\}^{1-1/\beta}=\{\tilde{\alpha}^{-1}f(x)\}^{1-1/\beta}$. Since $\tilde{\alpha}$ depends only on $\beta$ and $L$, the proof of (iii) is complete.

Finally, in view of the affine invariance of the conditions~\eqref{Eq:GradBd},~\eqref{Eq:LogDensity1},~\eqref{Eq:LogDensity2}, it suffices to prove assertion (iv) for $f\in\mathcal{F}^{0,I}\cap\tilde{\mathcal{H}}^{\gamma,L}$. We begin by recalling that there exists $\tilde{B}\equiv\tilde{B}_d>0$ such that $h(x)\leq e^{\tilde{B}_d}$ for all $h\in\mathcal{F}_d^{0,I}$ and $x\in\R^d$, \citep[e.g.][Theorem~2(a)]{KS16}. Consequently, if $f\in\mathcal{F}^{0,I}\cap\tilde{\mathcal{H}}^{\gamma,L}$, then the function $\psi\colon\R^d\to\R$ defined by $\psi(x)=-\log f(x)+\tilde{B}$ is $(\gamma,L)$-H\"older and takes non-negative values. Thus, if $x,y\in\R^d$ satisfy $f(y)<f(x)<\inv{e}$ and if $\beta\geq 1$, then
\begin{align*}
\norm{x-y}&\geq\inv{\Lambda(\gamma,L)}\,\frac{\psi(y)-\psi(x)}{\psi(y)^{1-1/\gamma}}=\inv{\Lambda(\gamma,L)}\,\frac{\log f(x)-\log f(y)}{\{\log^{\phantom{A}}\!\!\!\{1/f(y)\}+\tilde{B}\}^{1-1/\gamma}}\\[10pt]
&\geq (\tilde{B}+1)^{-1}\,\inv{\Lambda(\gamma,L)}\,\frac{\log f(x)-\log f(y)}{\log^{\!\!\!\!\phantom{A}}\{1/f(y)\}}
\geq(\tilde{B}+1)^{-1}\,\inv{\Lambda(\gamma,L)}\,\inv{\beta}\,\frac{f(x)-f(y)}{f(x)^{1-1/\beta}},
\end{align*}
where we applied assertion (ii) of the proposition and Lemma~\ref{Lem:logineq} below to obtain the first and last inequalities in the display above. This yields the first assertion of (iv). Since $\Lambda(\gamma,L)\leq (L\vee L^{1/2})\max_{\,0< w\leq 1/2} w^{-w}=(L\vee L^{1/2})\,e^{1/e}$ for all $\gamma\in (1,2]$, the final conclusion of (iv) follows immediately upon setting $\tau=\inv{e}$ and invoking Proposition~\ref{Prop:SepProperties}(ii).
\end{proof}
The following elementary bound is used in the proof of Proposition~\ref{Prop:Holder}(iii) above.
\begin{lemma}
\label{Lem:logineq}
If $\beta\geq 1$ and $0<a<b<1$, then
\begin{equation}
\label{Eq:logineq}
\frac{\log b-\log a}{\log\,(1/a)}>\frac{\inv{\beta}(b-a)}{b^{1-1/\beta}}.
\end{equation}
\end{lemma}
\begin{proof}
Setting $u:=\log(e^\beta/b)$, we can differentiate the function $w\mapsto we^{1-w/\beta}$ to deduce that
\begin{equation}
\label{Eq:logbd}
b^{1/\beta}\log(e^{\beta}/b)=ue^{1-u/\beta}<\beta
\end{equation}
for all $b\in (0,1)$. Now fix $b\in (0,1)$ and suppose first that $a\in [be^{-\beta},b]$. By the mean value theorem, there exists $c\in (a,b)$ such that $(\log b-\log a)/(b-a)=1/c$, and it follows from~\eqref{Eq:logbd} that
\[\frac{\log b-\log a}{b-a}=\frac{1}{c}\geq\frac{1}{b}>\frac{\inv{\beta}\log(e^\beta/b)}{b^{1-1/\beta}}\geq\frac{\inv{\beta}\log(1/a)}{b^{1-1/\beta}},\]
so~\eqref{Eq:logineq} holds when $a\in [be^{-\beta},b]$. On the other hand, when $a\in (0,be^{-\beta})$, note that the left-hand side of~\eqref{Eq:logineq} is a decreasing function of $a$ for each fixed value of $b$. Thus, we see that
\[\frac{\log b-\log a}{\log\,(1/a)}\geq\frac{\beta}{\log(e^\beta/b)}\geq b^{1/\beta}>\frac{\inv{\beta}(b-a)}{b^{1-1/\beta}},\]
as required, where we have again used~\eqref{Eq:logbd} to obtain the second inequality above.
\end{proof}
\subsubsection{Review of results on nonparametric density estimation over H\"older classes}
\label{Subsec:HolderRev}
When $d=3$, $\beta\in (1,2]$ and $L>0$, we saw in Example~\ref{Ex:Holder2} that $\sup_{f_0\in\mathcal{H}_d^{\beta,L}}\E_{f_0}\{\dex^2(\hat{f}_n,f_0)\}\lesssim_{\beta,L}n^{-(\beta+3)/(\beta+7)}\log^{\lambda_\beta}n$, where $\lambda_\beta:=(16\beta+39)/(2(\beta+7))$. To relate this result to the existing literature on nonparametric density estimation over H\"older classes of densities (without shape constraints), we work with $\dhell^2$ instead of $\dex^2$ divergence (since the latter is specific to the estimator $\hat{f}_n$), and borrow some definitions from~\citet{GL14}. For $d\in\N$, let $\mathbf{1}_d:=(1,\dotsc,1)\in\R^d$, and for 
$\beta\in (1,2]$ and $L>0$, denote by $\mathscr{H}(\beta,L)\equiv\mathscr{H}_d(\beta,L)$ the anisotropic Nikol'skii class $\mathscr{N}_{\infty\mathbf{1}_d,d}(\beta\mathbf{1}_d,L\mathbf{1}_d)=\mathscr{N}_{\infty\mathbf{1}_d,d}(\beta\mathbf{1}_d,L\mathbf{1}_d,L)$ of densities $g\colon\R^d\to\R$ satisfying $\abs{g(x)}\leq L$ and $\abs{g(x+2h)-2g(x+h)+g(x)}\leq L\norm{h}^\beta$ for all $x,h\in\R^d$; see (3.2) on page~488 of~\citet{GL14}. 
This is slightly different to our H\"older conditions~\eqref{Eq:HolderOriginal} and~\eqref{Eq:holdergen}, but with the aid of Taylor's theorem (with the mean value form of the remainder) and Proposition~\ref{Prop:Nested}, we can verify that whenever $\beta\in (1,2]$ and $L>0$, there exists $L'\equiv L'(d,\beta,L)>0$ such that $\mathcal{H}^{\beta,L}\subseteq\mathscr{H}(\beta,L')$. 

Recalling the standard fact that $\dhell^2(f,g)\geq\norm{f-g}_1^2/4=(\int_{\R^d}\,\abs{f-g})^2/4$ for all densities $f,g$, we can apply the Cauchy--Schwarz inequality and take $p=1$ in~\citet[Theorem~3(i)]{GL14} to deduce that there exists $c\equiv c(\beta,d)>0$ such that
\begin{equation}
\label{Eq:Inconsistent}
\inf_{\tilde{f}_n}\,\sup_{f_0\in\mathscr{H}(\beta,L')}\E_{f_0}\{\dhell^2(\tilde{f}_n,f_0)\}\geq\inf_{\tilde{f}_n}\,\sup_{f_0\in\mathscr{H}(\beta,L')}\E_{f_0}\{\norm{\tilde{f}_n-f_0}_1\}^2/4\geq c>0,
\end{equation}
where the infimum is taken over all estimators $\tilde{f}_n$ based on $n$ observations. In other words, it is not even possible to achieve consistency over $\mathscr{H}(\beta,L')$ with respect to $L^1$ (and hence $\dhell^2$) loss. In view of this, it will be more meaningful to compare the result of Example~\ref{Ex:Holder2} with (a lower bound on) the minimax $\dhell^2$ risk over a carefully chosen subclass of $\mathscr{H}(\beta,L')$. 

Suppose henceforth that $d=3$, and for fixed $\beta\in (1,2]$ and $L>0$, let $L'\equiv L'(3,\beta,L)>0$ and $\mathscr{H}(\beta,L')\equiv\mathscr{H}_3(\beta,L')$ be as above. In the notation of~\citet[Sections~3.3 and 4]{GL14}, the parameters associated with this class are $\beta_1=\beta_2=\beta_3=\beta$ and $s=\infty$, and to avoid confusion, we write $\tilde{\beta}$ for the quantity $\bigl(\sum_{j=1}^d\,\beta_j^{-1}\bigr)^{-1}=\beta/d=\beta/3$ that these authors denote by $\beta$. By Example~\ref{Ex:Holder2} and the affine invariance of $\dhell^2$, we have
\begin{equation}
\label{Eq:Holder3Iso}
\sup_{f_0\in\mathcal{H}^{\beta,L}\cap\mathcal{F}^{0,I}}\E_{f_0}\{\dhell^2(\hat{f}_n,f_0)\}=\sup_{f_0\in\mathcal{H}^{\beta,L}}\E_{f_0}\{\dhell^2(\hat{f}_n,f_0)\}\lesssim_{\beta,L}n^{-(\beta+3)/(\beta+7)}\log^{\lambda_\beta}n.
\end{equation}
We now show that $\mathcal{H}^{\beta,L}\cap\mathcal{F}^{0,I}$ is contained within a subclass of $\mathscr{H}(\beta,L')$ over which the minimax $L^1$ risk is $\tilde{O}(n^{-2\beta/(2\beta+3)})$, rather than $O(1)$ as in~\eqref{Eq:Inconsistent}. The key fact we exploit is that $\mathcal{F}^{0,I}$ has an envelope function that decays exponentially in the following sense. For $a>0$ and $b\in\R$, denote by $\mathscr{G}(a,b)$ the set of $g\colon\R^3\to\R$ such that $g(x)\leq e^{-a\norm{x}+b}$ for all $x\in\R^3$, and recall that there exist universal constants $A>0$ and $B\in\R$ such that $\mathscr{G}(A,B)\supseteq\mathcal{F}^{0,I}\equiv\mathcal{F}_3^{0,I}$~\citep[e.g.][Theorem~2(a)]{KS16}. Now for $g\in\mathscr{G}(a,b)$, define $g^*\colon\R^3\to [0,\infty)$ as in (4.1) on page~493 of~\citet{GL14}, so that $g^*(x):=\sup_{H_x}\mu_3(H_x)^{-1}\int_{H_x}g$ for each $x\in\R^3$, where the supremum is taken over all hyperrectangles $H_x$ of the form $\prod_{j=1}^3\,[x_j-h_j/2,\,x_j+h_j/2]$ with $h_1,h_2,h_3\in (0,2]$. Then $g^*(x)\leq\sup_{h\in [-1,1]^3}g(x+h)\leq\sup_{h\in [-1,1]^3}e^{-a\norm{x+h}+b}\leq e^{-a\norm{x}+(a\sqrt{3}+b)}$ for all $x\in\R^3$. Setting $\theta\equiv\theta(\beta):=1/(2+1/\tilde{\beta})=\beta/(2\beta+3)\in (0,1)$, we deduce that there exists $R=R(a,b,\beta)>0$ such that every $g\in\mathscr{G}(a,b)$ satisfies the `tail dominance' condition $\bigl(\int_{\R^3}\,(g^*)^\theta\bigr)^{1/\theta}\leq R$, which is the defining property (4.2) of the class $\mathscr{G}_\theta(R)$ on page 493 of~\citet{GL14}. Therefore, $\mathcal{F}^{0,I}\subseteq\mathscr{G}(A,B)\subseteq\mathscr{G}_\theta(R')$, where $R':=R(A,B,\beta)$.

Recall that $s=\infty$ and that our $\tilde{\beta}=\beta/3$ corresponds to the parameter $\beta$ in their notation, so that $\nu^*(\theta)=1/(2+1/\tilde{\beta})=\beta/(2\beta+3)$ in the last display on page 493 of~\citet{GL14}. Consequently, by taking $p=1$ in~\citet[Theorem~4(i)]{GL14}, we deduce that there exists $R''\equiv R''(\beta,L)>0$ such that
\begin{equation}
\label{Eq:StandardHolder}
\inf_{\tilde{f}_n}\,\sup_{f_0\in\mathscr{H}(\beta,L')\cap\mathscr{G}_\theta(\tilde{R})}\E_{f_0}\{\dhell^2(\tilde{f}_n,f_0)\}\geq\inf_{\tilde{f}_n}\,\sup_{f_0\in\mathscr{H}(\beta,L')\cap\mathscr{G}_\theta(\tilde{R})}\E_{f_0}\{\norm{\tilde{f}_n-f_0}_1\}^2/4\gtrsim_{\beta,L}n^{-2\beta/(2\beta+3)}
\end{equation}
for all $\tilde{R}\geq R''$, where the first inequality follows as in~\eqref{Eq:Inconsistent} and the second inequality is tight up to logarithmic factors by Remark~3(4) on page 495 of~\citet{GL14}. Since $\mathcal{H}^{\beta,L}\cap\mathcal{F}^{0,I}\subseteq\mathscr{H}(\beta,L')\cap\mathscr{G}_\theta(R'\vee R'')$, the minimax lower bound in~\eqref{Eq:StandardHolder} suggests that under the assumption of log-concavity, it is possible to achieve faster rates of convergence at least when $\beta\in (1,9/5)$. However, this does not rule out the possibility that these accelerated rates could be obtained under a weaker exponential tail condition in place of the stronger constraint of log-concavity. To show that this is not the case, we observe that the proof of~\eqref{Eq:StandardHolder} considers only a subset of densities in $\mathscr{H}(\beta,L')\cap\mathscr{G}_\theta(R')$ whose supports are contained within $[-N,N]^3$ for some $N\equiv N(\beta,L)>0$. There exist $a'\equiv a'(\beta,L)\in (0,A]$ and $b'\equiv b'(\beta,L)\geq B$ such that all such densities are contained within $\mathscr{G}(a',b')$, so we actually have 
\begin{equation}
\label{Eq:StandardHolderExp}
\inf_{\tilde{f}_n}\,\sup_{f_0\in\mathscr{H}(\beta,L')\cap\mathscr{G}(a',b')}\E_{f_0}\{\dhell^2(\tilde{f}_n,f_0)\}\geq\inf_{\tilde{f}_n}\,\sup_{f_0\in\mathscr{H}(\beta,L')\cap\mathscr{G}(a',b')}\E_{f_0}\{\norm{\tilde{f}_n-f_0}_1\}^2/4\gtrsim_{\beta,L}n^{-2\beta/(2\beta+3)}.
\end{equation}
Since $\mathcal{H}^{\beta,L}\cap\mathcal{F}^{0,I}\subseteq\mathscr{H}(\beta,L')\cap\mathscr{G}(A,B)\subseteq\mathscr{H}(\beta,L')\cap\mathscr{G}(a',b')$, we may justifiably conclude on the basis of~\eqref{Eq:Holder3Iso} and~\eqref{Eq:StandardHolderExp} that the improvement in the rates attainable is indeed due to the log-concavity shape constraint rather than the exponential tail decay exhibited by log-concave densities.

\end{document}